\documentclass{amsart}

\usepackage[utf8]{inputenc}

\title{Enriched $\infty$-Categories via Nonsymmetric $\infty$-Operads}
\author{David Gepner}
\address{Department of Mathematics, Purdue University, West Lafayette,
  Indiana, USA}
\urladdr{http://www.math.purdue.edu/~dgepner}
\author{Rune Haugseng}
\address{Max-Planck-Institut für Mathematik, Bonn, Germany}
\email{haugseng@mpim-bonn.mpg.de}
\urladdr{http://people.mpim-bonn.mpg.de/haugseng}

\date{\today}

\usepackage{amsmath,amssymb,amsthm}
\usepackage[utf8]{inputenc}
\usepackage{url}
\usepackage[shortlabels]{enumitem}
  \setitemize[1]{leftmargin=2em}
  \setenumerate[1]{leftmargin=*}
\usepackage[alphabetic]{amsrefs}

\usepackage[colorlinks=true,linktocpage=true,citecolor=blue]{hyperref}

\usepackage{palatino}
\usepackage{mathpazo}
\usepackage{eucal}

\usepackage{mbboard}
\usepackage{todonotes}

\theoremstyle{theorem}
\newtheorem{thm}{Theorem}[subsection]
\newtheorem{lemma}[thm]{Lemma}
\newtheorem{propn}[thm]{Proposition}
\newtheorem{cor}[thm]{Corollary}

\theoremstyle{definition}
\newtheorem{defn}[thm]{Definition}
\newtheorem{ex}[thm]{Example}
\newtheorem{exs}[thm]{Examples}

\newtheorem{warning}[thm]{Warning}

\newtheorem{remark}[thm]{Remark}

\theoremstyle{remark}

\newcommand{\blank}{\text{\textendash}}

\newcommand{\defterm}[1]{\emph{#1}}
\newcommand{\isoto}{\xrightarrow{\sim}}

\newcommand{\IFF}{if and only if}

\newcommand{\catname}[1]{\ensuremath{\text{\textup{#1}}}}
\newcommand{\txt}[1]{\ensuremath{\text{\textup{#1}}}}
\newcommand{\Set}{\catname{Set}}
\newcommand{\sSet}{\Set_{\Delta}}
\newcommand{\Cat}{\catname{Cat}}
\newcommand{\CatI}{\catname{Cat}_\infty}
\newcommand{\Gpd}{\catname{Gpd}}
\newcommand{\LCatI}{\widehat{\catname{Cat}}_\infty}

\newcommand{\Sp}{\catname{Sp}}

\newcommand{\Fun}{\txt{Fun}}

\newcommand{\Map}{\txt{Map}}
\newcommand{\Hom}{\txt{Hom}}

\newcommand{\op}{\txt{op}}

\newcommand{\icat}{$\infty$-category}
\newcommand{\icats}{$\infty$-categories}
\newcommand{\icatl}{$\infty$-categorical}
\newcommand{\itcat}{$(\infty,2)$-category}

\newcommand{\xto}[1]{\xrightarrow{#1}}

\usepackage{color}

\newcommand{\simp}{\bbDelta}

\usepackage{tikz}
\usetikzlibrary{matrix,arrows}

\newcommand{\csquare}[8]{ %
\[ %
\begin{tikzpicture} %
\matrix (m) [matrix of math nodes,row sep=3em,column sep=2.5em,text height=1.5ex,text depth=0.25ex] %
{ #1 \pgfmatrixnextcell #2 \\ %
  #3 \pgfmatrixnextcell #4 \\ }; %
\path[->,font=\footnotesize] %
(m-1-1) edge node[auto] {$#5$} (m-1-2)%
(m-1-1) edge node[left] {$#6$} (m-2-1)%
(m-1-2) edge node[auto] {$#7$} (m-2-2)%
(m-2-1) edge node[below] {$#8$} (m-2-2);%
\end{tikzpicture}%
\]%
}

\newcommand{\nodispcsquare}[8]{ %
\begin{tikzpicture} %
\matrix (m) [matrix of math nodes,row sep=3em,column sep=2.5em,text height=1.5ex,text depth=0.25ex] %
{ #1 \pgfmatrixnextcell #2 \\ %
  #3 \pgfmatrixnextcell #4 \\ }; %
\path[->,font=\footnotesize] %
(m-1-1) edge node[auto] {$#5$} (m-1-2)%
(m-1-1) edge node[left] {$#6$} (m-2-1)%
(m-1-2) edge node[auto] {$#7$} (m-2-2)%
(m-2-1) edge node[below] {$#8$} (m-2-2);%
\end{tikzpicture}%
}

\newcommand{\smallcsquare}[8]{ %
\[ %
\begin{tikzpicture} %
\matrix (m) [matrix of math nodes,row sep=1.5em,column sep=1.25em,text height=1.5ex,text depth=0.25ex] %
{ #1 \pgfmatrixnextcell #2 \\ %
  #3 \pgfmatrixnextcell #4 \\ }; %
\path[->,font=\footnotesize] %
(m-1-1) edge node[auto] {$#5$} (m-1-2)%
(m-1-1) edge node[left] {$#6$} (m-2-1)%
(m-1-2) edge node[auto] {$#7$} (m-2-2)%
(m-2-1) edge node[below] {$#8$} (m-2-2);%
\end{tikzpicture}%
\]%
}

\newcommand{\nolabelcsquare}[4]{\csquare{#1}{#2}{#3}{#4}{}{}{}{}}
\newcommand{\nolabelsmallcsquare}[4]{\smallcsquare{#1}{#2}{#3}{#4}{}{}{}{}}
\newcommand{\nodispnolabelcsquare}[4]{\nodispcsquare{#1}{#2}{#3}{#4}{}{}{}{}}

\newcommand{\opctriangle}[6]{ %
\[ %
\begin{tikzpicture} %
\matrix (m) [matrix of math nodes,row sep=3em,column sep=1.2em,text height=1.5ex,text depth=0.25ex] %
{  #1 \pgfmatrixnextcell \pgfmatrixnextcell #2 \\ %
  \pgfmatrixnextcell #3 \pgfmatrixnextcell \\ %
}; %
\path[->,font=\footnotesize] %
(m-1-1) edge node[above] {$#4$} (m-1-3)%
(m-1-1) edge node[below left] {$#5$} (m-2-2)%
(m-1-3) edge node[below right] {$#6$} (m-2-2);%
\end{tikzpicture}%
\]%
}

\newcommand{\factortriangle}[6]{ %
\[ %
\begin{tikzpicture} %
\matrix (m) [matrix of math nodes,row sep=3em,column sep=1.2em,text height=1.5ex,text depth=0.25ex] %
{  #1 \pgfmatrixnextcell \pgfmatrixnextcell #2 \\ %
  \pgfmatrixnextcell #3 \pgfmatrixnextcell \\ %
}; %
\path[->,font=\footnotesize] %
(m-1-1) edge node[above] {$#4$} (m-1-3)%
(m-1-1) edge node[left] {$#5$} (m-2-2)%
(m-2-2) edge node[right] {$#6$} (m-1-3);%
\end{tikzpicture}%
\]%
}

\newcommand{\nolabelopctriangle}[3]{\opctriangle{#1}{#2}{#3}{}{}{}}

\newcommand{\id}{\txt{id}}

\DeclareMathOperator{\colimP}{colim}
\newcommand{\colim}{\mathop{\colimP}}

\newcommand{\CatIV}{\CatI^{\mathcal{V}}}

\newcommand{\Simp}{\txt{Simp}}

\newcommand{\Seg}{\txt{Seg}}

\newcommand{\SegI}{\Seg_{\infty}}

\newcommand{\Mon}{\txt{Mon}}
\newcommand{\MonS}{\txt{Mon}^{\Sigma}}
\newcommand{\MonI}{\Mon_{\infty}}

\DeclareMathOperator{\ob}{ob}

\newcommand{\GraphIV}{\catname{Graph}_{\infty}(\mathcal{V})}

\newcommand{\Alg}{\catname{Alg}}

\newcommand{\AlgCat}{\Alg_{\txt{cat}}}
\newcommand{\LAlgCat}{\widehat{\Alg}_{\txt{cat}}}
\newcommand{\AlgCatco}{\Alg_{\txt{cat,co}}}
\newcommand{\AlgCatV}{\AlgCat(\mathcal{V})}
\newcommand{\AlgCatS}{\AlgCat(\mathcal{S})}

\newcommand{\Algco}{\Alg_{\txt{co}}}

\newcommand{\AlgS}{\Alg^{\Sigma}}

\newcommand{\Opd}{\catname{Opd}}
\newcommand{\OpdI}{\Opd_{\infty}}

\newcommand{\OPD}{\catname{OPD}}
\newcommand{\OPDI}{\OPD_{\infty}}

\newcommand{\OpdIns}{\OpdI^{\txt{ns}}}
\newcommand{\OpdInsg}{\OpdI^{\txt{ns},\txt{gen}}}
\newcommand{\OPDIns}{\OPDI^{\txt{ns}}}
\newcommand{\OPDInsg}{\OPDI^{\txt{ns},\txt{gen}}}
\newcommand{\OpdIS}{\OpdI^{\Sigma}}

\newcommand{\MonPr}{\MonI^{\txt{Pr}}}

\newcommand{\PresI}{\catname{Pres}_{\infty}}

\newcommand{\iopd}{$\infty$-operad}
\newcommand{\iopds}{$\infty$-operads}

\newcommand{\nsiopd}{non-symmetric $\infty$-operad}
\newcommand{\nsiopds}{non-symmetric $\infty$-operads}

\newcommand{\gnsiopd}{generalized non-symmetric $\infty$-operad}
\newcommand{\gnsiopds}{generalized non-symmetric $\infty$-operads}

\newcommand{\triv}{\txt{triv}}

\newcommand{\Algtriv}{\Alg_{\triv}}

\newcommand{\CoCart}{\txt{CoCart}}

\newcommand{\Po}{\catname{Po}}
\DeclareMathOperator{\Str}{Str}
\newcommand{\StrM}{\Str \mathcal{M}}
\newcommand{\StrMen}{\StrM^{\txt{en}}}

\newcommand{\BM}{\txt{BM}}
\newcommand{\LM}{\txt{LM}}
\newcommand{\RM}{\txt{RM}}

\newcommand{\EnrLur}{\txt{Enr}_{\txt{Lur}}}
\newcommand{\EnrLurV}{\txt{Enr}_{\txt{Lur}}^{\mathcal{V}}}

\makeatletter
\def\@tocline#1#2#3#4#5#6#7{\relax
  \ifnum #1>\c@tocdepth 
  \else
    \par \addpenalty\@secpenalty\addvspace{#2}%
    \begingroup \hyphenpenalty\@M
    \@ifempty{#4}{%
      \@tempdima\csname r@tocindent\number#1\endcsname\relax
    }{%
      \@tempdima#4\relax
    }%
    \parindent\z@ \leftskip#3\relax \advance\leftskip\@tempdima\relax
    \rightskip\@pnumwidth plus4em \parfillskip-\@pnumwidth
    #5\leavevmode\hskip-\@tempdima
      \ifcase #1
       \or \hskip -1em \or \hskip 1em \or \hskip 3em \else \hskip 5em \fi%
      #6\nobreak\relax
    \hfill\hbox to\@pnumwidth{\@tocpagenum{#7}}
      \par
    \nobreak
    \endgroup
  \fi}
\makeatother

\begin{document}

\begin{abstract}
  We set up a general theory of weak or homotopy-coherent enrichment
  in an arbitrary monoidal $\infty$-category $\mathcal{V}$. Our theory
  of enriched $\infty$-categories has many desirable properties; for
  instance, if the enriching $\infty$-category $\mathcal{V}$ is
  presentably symmetric monoidal then
  $\mathrm{Cat}^\mathcal{V}_\infty$ is as well. These features render
  the theory useful even when an $\infty$-category of enriched
  $\infty$-categories comes from a model category (as is often the
  case in examples of interest, e.g. dg-categories, spectral
  categories, and $(\infty,n)$-categories). This is analogous to the
  advantages of $\infty$-categories over more rigid models such as
  simplicial categories --- for example, the resulting
  $\infty$-categories of functors between enriched $\infty$-categories
  automatically have the correct homotopy type.
  
  We construct the homotopy theory of $\mathcal{V}$-enriched
  $\infty$-categories as a certain full subcategory of the
  $\infty$-category of ``many-object associative algebras'' in
  $\mathcal{V}$. The latter are defined using a non-symmetric version
  of Lurie's $\infty$-operads, and we develop the basics of this
  theory, closely following Lurie's treatment of symmetric
  $\infty$-operads. While we may regard these ``many-object'' algebras
  as enriched $\infty$-categories, we show that it is precisely the
  full subcategory of ``complete'' objects (in the sense of Rezk,
  i.e. those whose spaces of objects are equivalent to their spaces of
  equivalences) which are local with respect to the class of fully
  faithful and essentially surjective functors. We also consider an
  alternative model of enriched \icats{} as certain presheaves of
  spaces satisfying analogues of the ``Segal condition'' for Rezk's
  Segal spaces. Lastly, we present some applications of our theory,
  most notably the identification of associative algebras in
  $\mathcal{V}$ as a coreflective subcategory of pointed
  $\mathcal{V}$-enriched $\infty$-categories as well as a proof of a
  strong version of the Baez-Dolan stabilization hypothesis.
\end{abstract}

\maketitle
\tableofcontents

\section{Introduction}

Over the past decade, taking the higher-categorical nature of various
mathematical structures seriously has turned out to be a very fruitful
idea in several areas of mathematics. In particular, the theory of
\emph{\icats{}} (or more precisely \emph{$(\infty,1)$-categories}) has
found many applications in algebraic topology and in other
fields. However, despite the large amount of work that has been
carried out on the foundations of \icat{} theory, above all by Joyal
and Lurie, the theory is in many ways still in its infancy, and the
analogues of many concepts from ordinary category theory remain to be
explored. In this paper we begin to study the natural analogue in the
\icatl{} context of one such concept, namely that of \emph{enriched
  categories}.

Enriched categories in the usual sense are ubiquitous in modern
mathematics: the morphisms between objects in naturally occurring
categories often have more structure than just that of a set. However,
there are a number of important situations where the classical theory
of enriched categories has turned out to be insufficient in ways that
lead us towards considering the higher-categorical version of
enrichment. In algebraic topology, for example, the categories that
arise typically have a \emph{space} of morphisms between any two
objects --- but it is usually only the (weak) homotopy types of these
spaces that matter. Naïvely, we might guess that this means we should
consider these categories as enriched in the \emph{homotopy category}
of spaces, but this turns out to lose information that is important
for most applications. We are therefore forced to consider the
homotopy theory of categories enriched in topological spaces (or any
other model for the homotopy theory of spaces, such as simplicial
sets), with respect to the appropriate notion of weak equivalences,
which takes us outside the usual theory of enriched categories.
It is possible to consider this homotopy theory in the context of
Quillen's model categories (as was originally done by Bergner
\cite{BergnerSimpCat} for simplicial categories), but the resulting
model structures are in some ways not very well-behaved, essentially
because these ``strictly enriched'' categories are in a sense too
rigid. This makes it hard to understand the correct homotopy types of
the spaces of functors between them, and also makes homotopy-invariant
constructions (such as homotopy limits and colimits) problematic to
set up.

An additional problem is that many naturally occurring composition
laws between spaces are not strictly associative, but only associative
up to coherent homotopy. This makes them difficult to model as
simplicial or topological categories. It is therefore often more
convenient to work with a notion of ``category enriched in spaces''
where composition of morphisms is only associative up to coherent
homotopy.  This is the idea behind the theory of
\emph{$\infty$-categories}. Roughly speaking, the notion of
$\infty$-category is a generalization of the notion of category where
in addition to objects and morphisms we also have homotopies between
morphisms, homotopies between homotopies, and so on, and composition
is only associative up to a coherent choice of higher
homotopies. There are several ways to make this idea precise, such as
\emph{Segal categories}, \emph{complete Segal spaces}, and
\emph{quasicategories}. It turns out that working with \icats{} also
avoids the other problems with simplicial or topological categories
mentioned above, such as the difficulty of constructing functor
categories.

A similar situation arises in other areas of mathematics, such as
algebraic geometry or representation theory, where there are many
examples of \emph{derived} categories. These have traditionally been
thought of as \emph{additive} categories, which is to say categories
enriched in abelian groups, equipped with the additional data of a
\emph{triangulation}. Recently, however, it has been found that
derived categories, and other triangulated categories, are not rich
enough for many applications --- the extra structure of the
triangulation must be replaced by the more refined and intrinsic
notion of a \emph{differential graded} structure, i.e. an enrichment
in chain complexes. The correct notion of an equivalence between these
\emph{dg-categories} does not require a dg-functor to be given by
isomorphisms on chain complexes of maps, however --- instead, the
functor need only induce \emph{quasi-isomorphisms}. On the other hand,
it is again not enough to consider these categories as simply enriched
in the homotopy category of chain complexes (i.e. the derived category
of abelian groups) --- just as a differential graded algebra (or more
generally an $A_{\infty}$-algebra) is a much richer and more subtle
object than a homotopy associative multiplication on a chain complex,
the composition in a dg-category contains far more information than an
enrichment in the homotopy category of chain complexes. 

Homotopy-coherent compositions also occur in this context --- a key
example here is the \emph{Fukaya categories} of symplectic
geometry. These can often be described using
\emph{$A_{\infty}$-categories}, but unfortunately the theory of
$A_{\infty}$-categories is not as well-behaved as a replacement for
that of dg-categories as \icats{} are as a replacement for simplicial
categories.



A third example of this type is \emph{spectral} categories (or
categories enriched in spectra), of which there are many interesting
examples in algebraic topology. These are much more general than
dg-categories, and tend to arise in examples where the mapping spectra
can only be extracted up to homotopy. To emphasize the subtleties of
the situation, the very existence of a symmetric monoidal model for
the homotopy theory of spectra (under the smash product) was only
fairly recently resolved, after being an open question for several
decades. Moreover, in this context no notion of homotopy-coherent
enrichment has so far been proposed; this is a problem, for example
because many important functors that are known to preserve
$A_{\infty}$-structures, such as algebraic K-theory or topological
Hochschild homology, cannot be realized as lax monoidal functors to a
model category of spectra.


Now, just as spaces are the higher-categorical analogue of sets,
spectra are the higher categorical analogue of abelian groups or chain
complexes, and the sophisticated nature of these objects means that we
require a more conceptual and less ad hoc approach to the homotopy
theory of spectral categories than what is often sufficient in the
theory of dg-categories. One way to do this is to set up model
category structures on enriched categories --- it is possible to treat
the homotopy theory of dg-categories \cite{TabuadaDGCat}, spectral
categories \cite{TabuadaSpCat}, or even categories enriched in other
sufficiently nice monoidal model categories
\cite{HTT,BergerMoerdijkEnr,StanculescuEnr,MuroEnr} in this way. However, the
resulting model categories suffer from the same problems as that of
simplicial categories. In the case of dg-categories, for example, the
correct spaces of dg-functors have only recently been explicitly
described by Toën \cite{ToenMorita}, using a fairly complex construction;
there are earlier constructions of functor categories between $A_{\infty}$-categories \cite{LyubashenkoAInfty}, but these are also problematic.

In this paper we propose a different approach, namely to set up a
general theory of weak or homotopy-coherent enrichment. Specifically,
we will define and study \icats{} enriched in \emph{monoidal
  \icats{}}, which are \icats{} equipped with a tensor product that is
associative and unital up to coherent homotopy. This theory
encompasses, for example, analogues of spectral categories and
dg-categories where composition is only associative up to coherent
homotopy.  For the former we consider \icats{} enriched in the \icat{}
of spectra, while for the latter we enrich in the \emph{derived
  \icat{}} of abelian groups, in the sense of \cite[\S 1.3.2]{HA},
i.e. the \icat{} obtained by inverting the quasi-isomorphisms between
chain complexes of abelian gropus. The resulting homotopy theories of
enriched \icats{} are much better behaved than those of strictly
enriched categories --- for example, we have naturally defined enriched
$\infty$-categories of functors between enriched
$\infty$-categories. Moreover, the resulting homotopy theories are
\emph{equivalent} to those of ordinary enriched categories, as is
proved in \cite{enrcomp}. Thus, our theory gives a more flexible
approach to the homotopy theory of dg-categories and spectral
categories, which we expect will make many construction in these
settings easier to carry out.

The idea of ``weak'' enrichment is also implicit in the concept of
higher category theory itself: an $n$-category should have
$k$-morphisms between $(k-1)$-morphisms for $k = 1,\ldots, n$, so
there is an $(n-1)$-category of maps between any two objects. As is
well known, however, to obtain a good notion of $n$-category for $n >
2$ it is not sufficient to just consider $n$-categories as
\emph{strictly} enriched in $(n-1)$-categories, as in most naturally
occurring examples composition is only associative up to invertible
higher morphisms. We can avoid this issue by instead applying our
\icatl{} theory of enrichment: iterating the enrichment procedure
starting with the category of sets gives an inductively defined notion
of fully weak $n$-category. Starting instead with spaces we obtain a
theory of $(\infty,n)$-categories, and we can also similarly define
(weak) $(n,k)$-categories, which are $n$-categories where the
$i$-morphisms are all invertible for $i > k$. Moreover, the resulting
homotopy theories are equivalent to those of existing models for
$n$-categories and $(\infty,n)$-categories (as is also proved in
\cite{enrcomp}).

Thanks to the foundational work of Lurie, we are able to set up our
theory of enrichment entirely within the context of \icats{} (rather
than working with model categories, say). Apart from greater
generality, working in this setting gives a theory with many good
properties, including the following:
\begin{enumerate}[(a)]
\item Weak (or homotopy-coherent) enrichment is the only
  natural notion of enrichment which is possible in this language,
  which allows us to define our enriched \icats{} in the obvious way
  as ``many-object associative algebras'' in a given monoidal
  $\infty$-category.  (In other words, the \icatl{} analogue of
  ``strictly enriched'' categories automatically results in the
  appropriate ``weakly enriched'' theory.)
\item It is both easy and natural to consider enriched categories with
  \emph{spaces} of objects rather than just sets of objects, which
  turns out to make the resulting homotopy theory both nicer and
  simpler to set up, analogous to the way in which (complete) Segal
  spaces are better behaved than Segal categories.
\item We automatically get very good naturality properties, some of
  which would have been difficult even to formulate in a
  model-categorical framework --- for example, our \icats{} are
  natural with respect to functors between monoidal \icats{} that are
  lax monoidal in the appropriate \icatl{} sense. This means that we
  can easily apply functors such as group completion, algebraic
  K-theory, and topological Hochschild homology (which are lax
  monoidal as functors of \icats{}, but do not arise from lax monoidal
  functors between model categories) to construct spectral \icats{}.
\item We obtain the correct $\infty$-categories of enriched functors
  between enriched $\infty$-categories simply as the internal Hom
  objects right adjoint to the natural tensor product of enriched
  \icats{}. From the point of view of the model-categorical approach
  to enrichment this is in some sense the most subtle and useful
  feature --- subtle because the homotopically correct internal Hom
  must be invariant under enriched equivalences (the primary defect of
  simplicial categories as a model for $\infty$-categories) and useful
  because the existence of these functor $\infty$-categories makes
  constructions in, and the further development of, enriched higher
  category theory possible.
\item Beyond just constructing a homotopy theory, our theory gives a
  good setting in which to develop \icatl{} analogues of many concepts
  from enriched category theory, as we hope to demonstrate in future
  work.
\end{enumerate}

In addition to setting up the homotopy theory of enriched \icats{}, we
also construct several non-trivial examples: We show that Lurie's
\emph{stable \icats{}} from \cite[\S 1.1]{HA} are all enriched in the
\icat{} of spectra, and that the \emph{$R$-linear \icats{}} of
\cite[\S 6]{DAG7} are enriched in the \icat{} of $R$-modules, where
$R$ is an $\mathbb{E}_{2}$-ring spectrum. Moreover, we prove that
every closed monoidal \icat{} is enriched in itself. This gives us,
for example, the natural $n$-category of functors between any two
$n$-categories, generalizing the familiar fact that the category of
categories is enriched over itself.

We also discuss a number of simple applications of the theory. As
mentioned above, we provide a reasonable definition of the
$\infty$-category of weak $(n,m)$-categories for any $n$ and $m$,
which has the advantage of not relying on families of diagrams
parametrizing coherence conditions and which agrees with those of
Barwick, Bergner-Rezk, Joyal, and others. In this context we give a
proof of ``Baez-Dolan stabilization'' for (weak) $n$-categories
(generalizing that of Lurie for $(n,1)$-categories). This is the idea
that, for $m \geq n+2$, an $m$-tuply monoidal weak $n$-category is
precisely an $(n+2)$-tuply monoidal weak $n$-category (for example,
putting two compatible monoidal structures on a category makes it
a \emph{braided} monoidal category, while three or more monoidal
structures makes it \emph{symmetric} monoidal). We also show that (for
$m \leq \infty$ and $m \geq k \neq \infty$) an
$\mathbb{E}_{n}$-monoidal $(m,k)$-category is the same thing as an
$(m+n,k+n)$-category with a single (distinguished) object and a single
$j$-morphism for $j = 1,\ldots,n-1$. Finally, we show
that we can easily construct, as spectral \icats{}, certain homotopy
theories of presheaves that are expected
to be equivalent to familiar stable $\infty$-categories, such as the
description of genuine $G$-spectra as the $\infty$-category of modules
for the spectral $\infty$-category of stable orbits, or as modules
over a spectral Burnside $\infty$-category.

The theory we set up in this article is the first completely general
theory of weak enrichment. Weak enrichment in {\em Cartesian} monoidal
model categories has previously been defined as \emph{Segal enriched
 categories} as studied by Pellissier~\cite{Pellissier},
Lurie~\cite{LurieGoodwillie}, and Simpson~\cite{SimpsonSegCats}
(generalizing Bergner's model structure on Segal
categories~\cite{Bergner3Mod}).
It is important to note that many of the interesting examples of
enriched categories are cases (such as abelian groups, chain
complexes, and spectra) in which the monoidal structure is not
Cartesian; so, while more complicated to describe, allowing for
non-Cartesian enrichment is necessary to support the standard examples
of interest.

In the non-Cartesian case, there is a theory of
$A_{\infty}$-categories, which gives a notion of weak enrichment in
chain complexes, and more recently
Bacard~\cite{BacardSegEnrI,BacardToward} has set up a
model-categorical theory of weak enrichment in a class of symmetric
monoidal model categories that can be applied to many interesting
examples.  A definition of enriched \icats{} different from ours has
also been given by Lurie~\cite[Definition 4.2.1.28]{HA}, but he does
not develop this theory beyond defining the objects. We will see in
\S\ref{sec:selfenr} that in many cases we can extract an enriched
\icat{} in our sense from one of Lurie's, and we hope to be able to
extend this construction to a comparison between our theory and
Lurie's in the future.

\subsection{Overview}
In \S\ref{sec:fromto} we introduce our definition of enriched \icats{}
in terms of (generalized) non-symmetric \iopds{}, and motivate it by
explaining how it relates to ordinary enriched categories. 

In \S\ref{sec:NSOP} we briefly describe the non-symmetric version of
Lurie's theory of (generalized) \iopds{}, and prove some
(straightforward, for the most part) extensions of Lurie's
results. The most technical results, particularly those building towards
the construction of colimits of algebras, have been relegated to
Appendix~\ref{sec:algcolims}.

The theory of \iopds{} lets us define, for a monoidal \icat{}
$\mathcal{V}$, an \icat{} $\AlgCat(\mathcal{V})$ of
$\mathcal{V}$-enriched \icats{}; this is our object of study in
\S\ref{sec:algcat}. The main result is that if the \icat{}
$\mathcal{V}$ is presentable and its tensor product preserves colimits
in each variable, then this \icat{} is also presentable. We also
compare this model of enriched \icats{} to a certain \icat{} of
presheaves that satisfy analogues of the Segal condition for Segal
spaces.

In \S\ref{sec:CatIV} we construct the correct \icat{} of enriched
\icats{} by inverting the fully faithful and essentially surjective
functors in $\AlgCat(\mathcal{V})$. Here we prove the main theorem of
this article: we can obtain this localization as the full subcategory
of $\AlgCat(\mathcal{V})$ spanned by the \emph{complete}
$\mathcal{V}$-\icats{} --- those $\mathcal{V}$-\icats{} $\mathcal{C}$
such that the underlying space of objects in $\mathcal{C}$ is
equivalent to the classifying space of equivalences in
$\mathcal{C}$. We also prove that the resulting \icat{} has the
expected naturality properties.

In \S\ref{sec:appl} we describe some simple applications of our
construction: First we set up a theory of $(n,k)$-categories and prove
the ``homotopy hypothesis'' in this setting. We then prove that
enriching in an $(n,1)$-category gives an $(n+1,1)$-category of
enriched \icats{}; from this the Baez-Dolan stabilization hypothesis
for $k$-tuply monoidal $n$-categories follows easily if we define
$n$-categories to be $(\infty,n)$-categories enriched in sets. We also
show that $\mathbb{E}_{n}$-algebras in an $\mathbb{E}_{n}$-monoidal
\icat{} $\mathcal{V}$ embed fully faithfully into {\em pointed}
$\mathcal{V}$-enriched $(\infty,n)$-categories. This last result has a
number of interesting corollaries, such as a description of
$\mathbb{E}_{n}$-monoidal \icats{} as $(\infty,n+1)$-categories with a
single object and a single $j$-morphism for $j < n$, and a simple
construction of endomorphism algebras.

In \S\ref{sec:selfenr} we construct an important class of examples of
enriched \icats{}: If an \icat{} $\mathcal{C}$ is right-tensored over
a monoidal \icat{} $\mathcal{V}$ in such a way that the tensor product
$C \otimes (\blank)$ has a right adjoint $F(C,\blank) \in \mathcal{V}$
for all $C \in \mathcal{C}$, we show that $\mathcal{C}$ is enriched in
$\mathcal{V}$ with the maps from $C$ to $D$ given by $F(C,D)$. There
are several interesting special cases: a closed monoidal \icat{} is
enriched in itself, and all stable \icats{} are enriched in the
\icat{} of spectra. We prove this result by considering Lurie's
definition of enriched \icats{} and observing that we can extract an
enriched \icat{} in our sense by means of Lurie's construction of an
\icat{} of ``enriched strings''.

Finally, in Appendix~\ref{sec:algcolims} we prove some more technical
results about \nsiopds{}.

\subsection{Notation and Terminology}
In this article we will work throughout in the setting of
$(\infty,1)$-categories, by which we mean (heuristically) higher
categories in which the $n$-morphisms are invertible for
$n>1$. Specifically, we will make use of the theory of
quasicategories, as, due to the work of Joyal and Lurie, it is
currently by far the most highly developed theory of
$(\infty,1)$-categories. Following Lurie we will refer to these
objects as $\infty$-categories, however.  We generally recycle the
notation and terminology used by Lurie in \cite{HTT,HA}; here are some
exceptions and reminders:
\begin{itemize}
\item $\simp$ is the simplicial indexing category, with objects the
  non-empty finite totally ordered sets $[n] := \{0, 1, \ldots, n\}$
  and morphisms order-preserving functions between them.
\item $\bbGamma^{\op}$ is the category of pointed finite sets (so, by
  our convention, $\bbGamma$ is the {\em opposite} of the category of
  pointed finite sets).
\item Generic categories are generally denoted by single capital
  bold-face letters ($\mathbf{A},\mathbf{B},\mathbf{C}$) and generic
  \icats{} by single caligraphic letters
  ($\mathcal{A},\mathcal{B},\mathcal{C}$). Specific categories and
  \icats{} both get names in the normal text font: thus the category
  of small $\mathbf{V}$-categories is denoted $\Cat^{\mathbf{V}}$ and
  the \icat{} of small $\mathcal{V}$-\icats{} is denoted
  $\CatI^{\mathcal{V}}$.
\item $\sSet$ is the category of simplicial sets, i.e. the category
  $\Fun(\simp^{\op}, \Set)$ of set-valued presheaves on $\simp$.
\item $\mathcal{S}$ is the \icat{} of spaces; this can be defined as
  the coherent nerve $\mathrm{N}\sSet^{\circ}$ of the full simplicial
  subcategory $\sSet^{\circ}$ of $\sSet$ spanned by the Kan complexes.
\item We say a class of morphisms in an \icat{} \emph{satisfies the
    2-out-of-3 property} if given morphisms $f \colon x \to y$ and $g
  \colon y \to z$, if any two out of $f$, $g$, $g\circ f$ are in the
  class, so is the third.
\item If $\mathcal{C}$ is an \icat{} and $A$ and $B$ are objects of
  $\mathcal{C}$, then we write $\Map_{\mathcal{C}}(A, B)$ (or just
  $\Map(A,B)$ if the \icat{} $\mathcal{C}$ is obvious from the
  context) for the space of maps from $A$ to $B$ in $\mathcal{C}$. In
  the context of quasicategories there are a number of explicit models
  for these mapping spaces as simplicial sets (cf. \cite[\S
    1.2.2]{HTT}, \cite{DuggerSpivakMap}), but for our purposes it suffices
  to think of $\Map_{\mathcal{C}}(A, B)$ as an object of the \icat{}
  of spaces. Constructions of such a ``mapping space functor''
  $\Map_{\mathcal{C}} \colon \mathcal{C}^{\op} \times \mathcal{C} \to
  \mathcal{S}$ can be found in \cite[\S 5.1.3]{HTT} and
  \cite[\S 5.2.1]{HA}.
\item To distinguish the \icats{} of non-symmetric \iopds{} and their
  algebras from their symmetric counterparts we use a superscript
  ``$\txt{ns}$'' for the non-symmetric versions and a superscript
  ``$\Sigma$'' for the symmetric versions. Thus the \icat{} of
  non-symmetric \iopds{} is denoted $\OpdIns$ and the \icat{} of
  symmetric \iopds{} $\OpdIS$.  However, we take the non-symmetric
  versions to be the default ones in this paper and thus often do not
  include the superscript --- for example, if $\mathcal{O}$
  and $\mathcal{P}$ are \nsiopds{} we will generally denote
  the \icat{} of $\mathcal{O}$-algebras in
  $\mathcal{P}$ by
  $\Alg_{\mathcal{O}}(\mathcal{P})$.
\item We make use of the elegant theory of \emph{Grothendieck
    universes} to allow us to define ($\infty$-)categories without
  being limited by set-theoretical size issues; specifically, we fix
  three nested universes, and refer to sets contained in them as
  \emph{small}, \emph{large}, and \emph{very large}. When $\mathcal{C}$
  is an \icat{} of small objects of a certain type, we generally refer
  to the corresponding \icat{} of large objects as
  $\widehat{\mathcal{C}}$, without explicitly defining this
  object. For example, $\CatI$ is the (large) \icat{} of small
  \icats{}, and $\widehat{\Cat}_{\infty}$ is the (very large) \icat{}
  of large \icats{}.
\item If $\mathcal{C}$ is an \icat{}, we write $\iota \mathcal{C}$ for
  the \emph{interior} or \emph{underlying space} of $\mathcal{C}$,
  i.e. the largest subspace of $\mathcal{C}$ that is a Kan complex.
\item We write $\txt{LFib}(\mathcal{C})$ for the \icat{} of left
  fibrations over $\mathcal{C}$ (for example obtained from the
  covariant model structure on $(\sSet)_{/\mathcal{C}}$). Similarly, we
  write $\txt{Cart}(\mathcal{C})$ and $\CoCart(\mathcal{C})$ for the
  \icats{} of Cartesian and coCartesian fibrations to $\mathcal{C}$,
  respectively, i.e. the \icats{} associated to the Cartesian and
  coCartesan model structures on $(\sSet^{+})_{/\mathcal{C}}$.
\item We denote by $\PresI$ the \icat{} of presentable \icats{} and
  colimit-preserving functors.
\item If $f \colon \mathcal{C} \to \mathcal{D}$ is left adjoint to a
  functor $g \colon \mathcal{D} \to \mathcal{C}$, we will refer to the
  adjunction as $f \dashv g$.
\item If $K$ is a simplicial set we write $K^{\triangleleft} :=
  \Delta^{0} \star K$ and $K^{\triangleright} := K \star \Delta^{0}$,
  where $\star$ is the \emph{join} operation. If $\mathcal{C}$ is an
  \icat{}, we can interpret $\mathcal{C}^{\triangleleft}$ and
  $\mathcal{C}^{\triangleright}$ as the \icats{} obtained by freely
  adjoining an initial object and a final object to $\mathcal{C}$,
  respectively. We denote the ``cone points'' coming from $\Delta^{0}$
  in $K^{\triangleleft}$ and $K^{\triangleright}$ by $-\infty$ and
  $\infty$, respectively.
\item A simplicial set $K$ is \emph{sifted} if it is non-empty and the
  diagonal map $K \to K \times K$ is cofinal; see \cite[\S 5.5.8]{HTT}
  for alternative characterizations. The key point is that
  sifted colimits are generated by filtered colimits and colimits of
  simplicial objects, and small colimits are generated by sifted
  colimits and finite coproducts.
\end{itemize}

\begin{warning}
  As far as possible we argue using the ``high-level'' language of
  \icats{}, without referring to their specific implementation as
  quasicategories. Following this philosophy we have generally not
  distinguished notationally between categories and their nerves,
  since categories are a special kind of \icat{}. However, we do
  indicate the nerve (using $\mathrm{N}$) when we think of the nerve
  of a category as being a specific simplicial set; by the same
  principle we always indicate the nerves of simplicial
  categories. This should hopefully not cause any confusion.
\end{warning}

\subsection{Acknowledgements}
\emph{David:} I would like to thank John Francis, Charles Rezk, and
Markus Spitzweck for useful discussions and suggestions.

\emph{Rune:} Many of the results of this paper formed part of my
Ph.D. thesis, and I thank Haynes Miller for being a great advisor
(even when his students display an unhealthy interest in higher
category theory). I also thank Clark Barwick for many helpful and
inspiring conversations, and Jacob Lurie for patiently answering
several technical questions about his work. During much of the time
this work was carried out I was partially funded by the
American-Scandinavian Foundation and the Norway-America Association,
and I thank them for their support.

We thank Jeremy Hahn for suggesting the presheaf model for enriched
\icats{} discussed in \S\ref{subsec:presheafalgcat}. This paper would
not have been possible without the extensive foundations for \iopds{}
developed in \cite{HA}. Our work also owes a lot to Rezk's
paper~\cite{RezkCSS}, which inspired many of the arguments in
\S\ref{sec:CatIV}.

\section{From Enriched Categories to Enriched
  $\infty$-Categories}\label{sec:fromto}

The goal of this section is to introduce our definition of enriched
\icats{}, and to motivate it by explaining how it relates to ordinary
enriched categories. In the process, we also give an expository
introduction to (non-symmetric) \iopds{}.

\subsection{Multicategories and Enrichment}\label{subsec:MulticatEnr}
Recall the usual definition of an enriched category: if $\mathbf{V}$
is a monoidal category, a \emph{$\mathbf{V}$-enriched category} (or
\emph{$\mathbf{V}$-category}) $\mathbf{C}$ consists of:
\begin{itemize}
\item a set $\ob \mathbf{C}$ of objects,
\item for all pairs $X,Y \in \ob \mathbf{C}$ an object
$\mathbf{C}(X,Y)$ in $\mathbf{V}$,
\item composition maps $\mathbf{C}(X,Y) \otimes \mathbf{C}(Y,Z) \to
  \mathbf{C}(X,Z)$,
\item units $\id_{X} \colon I \to \mathbf{C}(X,X)$.
\end{itemize}
The composition must be associative (this involves the associator
isomorphism for $\mathbf{V}$) and unital. When formulated in this way,
it is not obvious how this notion ought to be generalized in the
setting of \icats{}. We should therefore look for an alternative, more
conceptual, way of defining enriched categories --- this is provided
by the theory of \emph{multicategories}.

A multicategory is, roughly speaking, a category where a morphism has a
\emph{list} of objects as its source. More precisely, a
\emph{multicategory} (or \emph{non-symmetric coloured operad})
$\mathbf{M}$ consists of
\begin{itemize}
\item a set $\ob \mathbf{M}$ of objects,
\item for objects $X_{1}, \ldots, X_{n}, Y$ (where $n$ can be $0$) a
  set $\mathbf{M}(X_{1},\ldots, X_{n}; Y)$ of ``multimorphisms'' from
  $(X_{1}, \ldots, X_{n})$ to $Y$,
\item an identity multimorphism $\id_{X}\colon (X) \to X$ for all
  objects $X$,
\item an associative and unital composition law, in the sense that we
  can compose multimorphisms
  \[ (Z_{1}, \ldots, Z_{i_{1}}) \to Y_{1},  \quad \ldots, \quad (Z_{i_{n-1}+1}, \ldots,
  Z_{i_{n}}) \to Y_{n}\] with a multimorphism $(Y_{1}, \ldots, Y_{n})
  \to X$ to get a composite multimorphism $(Z_{1}, \ldots, Z_{i_{n}})
  \to X$.
\end{itemize}
A multicategory with a single object is precisely a non-symmetric
operad.\footnote{Note that later we will refer to the \icatl{} version
  of (non-symmetric) coloured operads as just (non-symmetric)
  \iopds{}, for consistency with the terminology used by
  Lurie~\cite{HA} and Barwick~\cite{BarwickOpCat}.} 

If $\mathbf{M}$ and $\mathbf{N}$ are multicategories, a \emph{multifunctor}
 $F \colon \mathbf{M} \to \mathbf{N}$ assigns an
object $F(X)$ in $\mathbf{N}$ to each object $X$ of $\mathbf{M}$, and
to each multimorphism $(X_{1},\ldots,X_{n}) \to Y$ in $\mathbf{M}$ a
multimorphism \[(F(X_{1}),\ldots,F(X_{n})) \to F(Y)\] in $\mathbf{N}$ so that this
assignment is compatible with units and composition. We can view a
monoidal category $\mathbf{V}$ as a multicategory by defining
\[ \mathbf{V}(X_{1}, \ldots, X_{n}; Y) :=
\mathbf{V}(X_{1}\otimes\cdots\otimes X_{n}, Y).\] An \emph{algebra}
for a multicategory $\mathbf{M}$ in a monoidal category $\mathbf{V}$
is then just a multifunctor from $\mathbf{M}$ to
$\mathbf{V}$ viewed as a multicategory.

Given a set $S$, there is a simple multicategory $\mathbf{O}_{S}$ such
that $\mathbf{O}_{S}$-algebras in a monoidal category $\mathbf{V}$ are
precisely $\mathbf{V}$-categories with set of objects $S$: the set of
objects of $\mathbf{O}_{S}$ is $S \times S$, and the multimorphism
sets are defined by \[ \mathbf{O}_{S}((X_{0}, Y_{1}), (X_{1}, Y_{2}),
\ldots, (X_{n-1}, Y_{n}); (Y_{0}, X_{n})) :=
\begin{cases}
  *, & \txt{if $Y_{i} = X_{i}$, $i = 0,\ldots, n$,} \\
  \emptyset, & \txt{otherwise.}
\end{cases}
\]
Thus an $\mathbf{O}_{S}$-algebra $\mathbf{C}$ in $\mathbf{V}$ assigns
an object $\mathbf{C}(X,Y)$ to each pair $(X,Y)$ of elements of $S$,
with a unit $I \to \mathbf{C}(X,X)$ from the unique map $() \to
(X,X)$, and a composition map $\mathbf{C}(X,Y) \otimes
\mathbf{C}(Y,Z) \to \mathbf{C}(X,Z)$ from the unique multimorphism
$((X,Y), (Y,Z)) \to (X,Z)$. Looking at triples of pairs we see that
this composition is associative, and it is also clearly unital, so
$\mathbf{C}$ is a $\mathbf{V}$-category. If $\mathbf{C}$ and
$\mathbf{D}$ are $\mathbf{V}$-categories, with sets of objects $S$ and
$T$, respectively, then from this perspective a $\mathbf{V}$-functor
$\mathbf{C} \to \mathbf{D}$ consists of a function $f \colon S \to T$
and a multicategorical natural transformation from $\mathbf{C}$ to
$f^{*}\mathbf{D}$ of multifunctors $\mathbf{O}_{S} \to \mathbf{V}$, where
$f^{*}\mathbf{D}$ denotes the composite of $\mathbf{D}$ with the
obvious multifunctor $\mathbf{O}_{S} \to \mathbf{O}_{T}$ induced by $f$:
this natural transformation precisely assigns to each pair $X,Y \in S$
a morphism $\mathbf{C}(X,Y) \to \mathbf{D}(f(X), f(Y))$ compatible
with units and composition.

\begin{remark}
  This definition of enriched categories via multicategories is
  certainly classical, and it is not clear to us who first introduced
  it. In more recent work it can be seen, for example, as a starting
  point for Leinster's theory of enrichment in
  \textbf{fc}-multicategories and more general classes of
  multicategories associated to Cartesian
  monads~\cite{LeinsterGenEnr}.
\end{remark}

This construction suggests that we can use an \icatl{} version of
multicategories to define enriched \icats{}. In the next subsection we
will describe such an \icatl{} theory of multicategories, namely a
non-symmetric version of Lurie's \emph{\iopds{}}; this includes as a
special case a notion of monoidal \icat{}, and if $\mathcal{V}$ is a
monoidal \icat{} we will see that we can define a
$\mathcal{V}$-enriched \icat{} with set of objects $S$ as an
$\mathbf{O}_{S}$-algebra in $\mathcal{V}$.

\subsection{$\infty$-Operads}\label{subsec:FTIopds}
To generalize multicategories to the \icatl{} setting it is possible
to use \emph{simplicial multicategories}, i.e. multicategories
enriched in simplicial sets. However, these suffer from the same
technical problems as simplicial categories considered as a model for
\icats{} (most notably, it is difficult to compute the correct space
of simplicial multifunctors between simplicial multicategories in this
rigid setting). Just as for \icats{}, it is better to use a model
where composition is only associative up to coherent homotopy. We will
now introduce one such definition, namely a non-symmetric variant of
Lurie's \iopds{}.\footnote{An alternative approach to \iopds{} is the
  theory of \emph{dendroidal sets} introduced by Moerdijk and Weiss
  \cite{MoerdijkWeiss}, which we will not discuss here.}

Before we state the definition, it is helpful to consider an
alternative definition of ordinary multicategories:
\begin{defn}
  If $\mathbf{M}$ is a multicategory, then the \defterm{category of
    operators} $\mathbf{M}^{\otimes}$ of $\mathbf{M}$ has objects
  lists $(X_{1}, \ldots, X_{n})$ of objects $X_{i}
  \in \mathbf{M}$, $n = 0, 1,\ldots$, and a morphism \[(X_{1},\ldots,X_{n}) \to
  (Y_{1},\ldots, Y_{m})\] is given by a morphism $\phi \colon [m] \to [n]$ in
  $\simp$ and for each $j = 1,\ldots, m$ a multimorphism
  \[ (X_{\phi(j-1)+1}, X_{\phi(j-1)+2}, \ldots, X_{\phi(j)}) \to Y_{j}\]
  in $\mathbf{M}$. There is an obvious projection
  $\mathbf{M}^{\otimes} \to \simp^{\op}$, sending $(X_{1},\ldots,
  X_{n})$ to $[n]$.
\end{defn}
\begin{remark}
  This is the non-symmetric version of the category of operators of a
  symmetric operad introduced by May and Thomason~\cite{MayThomason}.
\end{remark}

We can characterize those categories over $\simp^{\op}$ that are
equivalent to categories of operators of multicategories; to state
this characterization it is convenient to first introduce some
notation:
\begin{defn}
  We say that a morphism $\phi \colon [n] \to [m]$ in $\simp$ is
  \emph{inert} if it is the inclusion of a sub-interval of $[m]$,
  i.e. if $\phi(i) = \phi(0)+i$ for $i = 0,\ldots,n$. We denote the
  inert morphism $[1] \to [n]$ given by the inclusions $\{i-1, i\}
  \hookrightarrow [n]$ by $\rho_{i}$ for $i = 1, \ldots, n$.
\end{defn}

\begin{defn}
  Let $\Cat_{/\simp^{\op}}^{\txt{mult}}$ denote the subcategory of
  $\Cat_{/\simp^{\op}}$ defined as follows: The objects of
  $\Cat_{/\simp^{\op}}^{\txt{mult}}$ are functors $\pi \colon
  \mathbf{C} \to \simp^{\op}$ such that the following conditions hold:
  \begin{enumerate}[(i)]
  \item For every inert morphism $\phi \colon [m] \to [n]$  in
    $\simp^{\op}$ and every $X \in \mathbf{C}_{[n]}$ there exists
    a $\pi$-coCartesian morphism $X \to \phi_{!}X$ over $\phi$.
  \item For every $[n] \in \simp^{\op}$ the functor \[\mathbf{C}_{[n]}
    \to \mathbf{C}_{[1]}^{\times n}\] induced by the coCartesian arrows
    over the inert maps $\rho_{i}$ ($i = 1,\ldots,n$) is an
    equivalence of categories.
  \item For every morphism $\phi \colon [n] \to [m]$ in $\simp^{\op}$
    and $Y \in \mathbf{C}_{[m]}$, composition with coCartesian
    morphisms $Y \to Y_{i}$ over the inert morphisms $\rho_{i}$ gives
    an isomorphism
    \[ \Hom_{\mathbf{C}}^{\phi}(X, Y) \isoto \prod_{i}
    \Hom_{\mathbf{C}}^{\rho_{i} \circ \phi}(X, Y_{i}),\] where
    $\Hom_{\mathbf{C}}^{\phi}(X,Y)$ denotes the subset of
    $\Hom_{\mathbf{C}}(X,Y)$ of morphisms that map to $\phi$ in
    $\simp^{\op}$.
  \end{enumerate}
  The morphisms of $\Cat_{/\simp^{\op}}^{\txt{mult}}$ from
  $\mathbf{C} \to \simp^{\op}$ to $\mathbf{D} \to \simp^{\op}$ are
  the functors $\mathbf{C} \to \mathbf{D}$ over $\simp^{\op}$ that
  preserve the coCartesian morphisms over inert morphisms in
  $\simp^{\op}$.
\end{defn}

\begin{propn}\label{propn:multicatvscatofop}
  The functor $(\blank)^{\otimes}$ from multicategories to categories
  over $\simp^{\op}$ gives an equivalence between the category of
  multicategories and $\Cat_{/\simp^{\op}}^{\txt{mult}}$.
\end{propn}
\begin{proof}
  It is easy to see that the category of operators of a multicategory
  $\mathbf{M}$ satisfies conditions (i)--(iii):
  \begin{enumerate}[(i)]
  \item The coCartesian map from $(X_{1},\ldots,X_{n})$ over an inert
    map $\phi \colon [m] \to [n]$ in $\simp$ is the projection
    $(X_{1},\ldots,X_{n}) \to (X_{\phi(1)},\ldots,X_{\phi(n)})$
    determined by the identity maps of the $X_{i}$'s.
  \item Clearly $\mathbf{M}^{\otimes}_{[n]}$ is equivalent to
    $(\mathbf{M}_{[1]}^{\otimes})^{\times n}$ via these projections.
  \item This is immediate from the definition of the morphisms in
    $\mathbf{M}^{\otimes}$.
  \end{enumerate}
  Moreover, any functor of multicategories $F \colon \mathbf{M} \to \mathbf{N}$
  induces a functor $\mathbf{M}^{\otimes} \to \mathbf{N}^{\otimes}$
  that preserves coCartesian arrows over inert maps: this simply says
  that $(X_{1},\ldots, X_{n})$ is sent to $(F(X_{1}), \ldots,
  F(X_{n}))$. Thus the functor $(\blank)^{\otimes}$ does indeed factor
  through $\Cat_{/\simp^{\op}}^{\txt{mult}}$.

  Conversely, if $\phi \colon \mathbf{M}^{\otimes} \to
  \mathbf{N}^{\otimes}$ is a functor over $\simp^{\op}$ that preserves
  these coCartesian morphisms, then condition (iii) implies that
  $\phi$ is completely determined by the maps
  $\mathbf{M}(X_{1},\ldots,X_{n}; Y) \to
  \mathbf{N}(\phi(X_{1}),\ldots,\phi(X_{n}); \phi(Y))$, and so comes
  from a functor of multicategories. This shows that
  $(\blank)^{\otimes}$ is fully faithful.

  It remains to show that the functor is essentially
  surjective. Suppose $\pi \colon \mathbf{C} \to \simp^{\op}$ is an
  object of $\Cat_{/\simp^{\op}}^{\txt{mult}}$. Then we can define a
  multicategory $\mathbf{M}_{\pi}$ as follows:
  \begin{itemize}
  \item The objects of $\mathbf{M}_{\pi}$ are the objects of
    $\mathbf{C}_{[1]}$. 
  \item By condition (ii) we can think of the objects of $\mathbf{C}_{[n]}$ as
    lists $(X_{1}, \ldots, X_{n})$ where the $X_{i}$'s are objects of
    $\mathbf{C}_{[1]}$. We define the multimorphism set
    $\mathbf{M}_{\pi}(X_{1},\ldots,X_{n}; Y)$ to be
    $\Hom_{\mathbf{C}}^{\alpha_{n}}((X_{1},\ldots,X_{n}), Y)$ where
    $\alpha_{n}$ denotes the map $[1] \to [n]$ that sends $0$ to $0$ and
    $1$ to $n$.
  \item The identity $\id_{X} \in \mathbf{M}_{\pi}(X; X)$ is just the
    identity map of $X$ in $\mathbf{C}_{[1]}$.
  \item To define the composition 
    \[ \mathbf{M}_{\pi}(X_{1}, \ldots, X_{n_{1}}; Y_{1}) \times \cdots
    \times \mathbf{M}_{\pi}(X_{n_{k-1}+1}, \ldots, X_{n_{k}}; Y_{k})
    \times \mathbf{M}_{\pi}(Y_{1},\ldots Y_{k}; Z) \to
    \mathbf{M}_{\pi}(X_{1},\ldots,X_{n_{k}}; Z)\]
    observe that by (iii) we can describe the source as
    \[\Hom_{\mathbf{C}}^{\beta}((X_{1},\ldots,X_{n_{k}}),
    (Y_{1},\ldots,Y_{k})) \times
    \Hom_{\mathbf{C}}^{\alpha_{k}}((Y_{1},\ldots, Y_{k}), Z),\]
    where $\beta \colon [k] \to [n_{k}]$ sends $0$ to $0$ and $i$ to
    $n_{i}$ for $i > 0$. Thus composition in $\mathbf{C}$ gives the
    desired composition in $\mathbf{C}$.
  \item To see that the composition is associative and unital, we
    apply the equivalences from (iii) similarly, and use the
    associativity and unitality of composition in $\mathbf{C}$.
  \end{itemize}
  It is then easy to check that the category of operators
  $\mathbf{M}_{\pi}^{\otimes}$ is equivalent to $\mathbf{C}$ over
  $\simp^{\op}$. Thus the functor $(\blank)^{\otimes}$ is essentially
  surjective, which completes the proof.
\end{proof}

We can thus equivalently \emph{define} a multicategory to be a functor
$\mathbf{C} \to \simp^{\op}$ satisfying (i)--(iii). Using the theory
developed in \cite{HTT}, these conditions moreover have obvious
\icatl{} analogues, which leads us to the following definition:
\begin{defn}\label{defn:nsiopd1}
  A \defterm{non-symmetric $\infty$-operad} is an inner fibration $\pi
  \colon \mathcal{O} \to \simp^{\op}$ such that:
  \begin{enumerate}[(i)]
  \item For every inert morphism $\phi \colon [m] \to [n]$  in
    $\simp^{\op}$ and every $X \in \mathcal{O}_{[n]}$ there exists
    a $\pi$-coCartesian morphism $X \to \phi_{!}X$ over $\phi$.
  \item For every $[n] \in \simp^{\op}$ the functor
    \[\mathcal{O}_{[n]} \to
    (\mathcal{O}_{[1]})^{\times n}\] induced by the
    coCartesian arrows over the inert maps $\rho_{i}$ ($i =
    1,\ldots,n$) is an equivalence of \icats{}.
  \item For every morphism $\phi \colon [n] \to [m]$ in $\simp^{\op}$
    and $Y \in \mathcal{O}_{[m]}$, composition with 
    coCartesian morphisms $Y \to Y_{i}$ over the inert morphisms
    $\rho_{i}$ gives an equivalence
    \[ \Map_{\mathcal{O}}^{\phi}(X, Y) \isoto \prod_{i}
    \Map_{\mathcal{O}}^{\rho_{i} \circ \phi}(X, Y_{i}),\]
    where $\Map_{\mathcal{O}}^{\phi}(X,Y)$ denotes the
    subspace of $\Map_{\mathcal{O}}(X,Y)$ of morphisms that
    map to $\phi$ in $\simp^{\op}$.
  \end{enumerate}
\end{defn}

\begin{remark}
  This is a special case of Barwick's notion of an \iopd{} over an
  operator category \cite{BarwickOpCat}, namely the case where the
  operator category is the category of finite ordered sets.
\end{remark}

\begin{remark}
  Being an inner fibration is a technical condition that does not have
  an analogue for ordinary categories; among other things it implies
  that the simplicial set $\mathcal{O}$ must be an \icat{}.
  Every functor of \icats{} can be replaced by an equivalent one that
  is an inner fibration.
\end{remark}

\begin{remark}
  The proof of Proposition~\ref{propn:multicatvscatofop} indicates how
  to interpret a \nsiopd{} $\mathcal{O} \to \simp^{\op}$ as a
  multicategory ``weakly enriched in spaces'':
  \begin{itemize}
  \item By condition (ii), the objects of $\mathcal{O}$ can be
    identified with lists $(X_{1},\ldots,X_{n})$ where the $X_{i}$'s
    are objects of $\mathcal{O}_{[1]}$ (which we think of as the
    underlying \icat{} of the multicategory)
  \item By condition (iii), the spaces of maps in $\mathcal{O}$ are
    determined by the mapping spaces of the form
    \[ \Map_{\mathcal{O}}^{\alpha_{n}}((X_{1},\ldots,X_{n}), Y), \]
    which we think of as the space  of multimorphisms in $\mathcal{O}$
    from $(X_{1},\ldots,X_{n})$ to $Y$.
  \item The composition of these multimorphisms is determined using
    condition (iii) by ordinary composition of morphisms in
    $\mathcal{O}$, as in the proof of
    Proposition~\ref{propn:multicatvscatofop}.
  \end{itemize}
  Since our definition takes place in the context of \icats{}, which
  already encode the notion of coherently homotopy-associative
  composition of morphisms, this means that the composition of
  multimorphisms in $\mathcal{O}$ is also coherently
  homotopy-associative, as expected.
\end{remark}

\begin{defn}
  If $\mathcal{O}$ and $\mathcal{P}$ are
  \nsiopds{}, a \defterm{morphism of \nsiopds{}} from
  $\mathcal{O}$ to $\mathcal{P}$ is a commutative
  diagram
  \opctriangle{\mathcal{O}}{\mathcal{P}}{\simp^{\op}}{\phi}{}{}
  such that $\phi$ carries coCartesian morphisms in
  $\mathcal{O}$ that map to inert morphisms in $\simp^{\op}$
  to coCartesian morphisms in $\mathcal{P}$. We will also
  refer to a morphism of \nsiopds{} $\mathcal{O} \to
  \mathcal{P}$ as an \emph{$\mathcal{O}$-algebra}
  in $\mathcal{P}$.
\end{defn}

\begin{remark}
  One advantage of working with \iopds{} over simplicial or
  topological multicategories is that they can be described as the
  fibrant objects in a model category where every object is
  cofibrant. This means that we can work with simple objects like the
  associative operad rather than having to use a cofibrant
  replacement, i.e. an $A_{\infty}$-operad: the \icat{} of algebras
  for the associative operad in a \nsiopd{} is always equivalent to
  the \icat{} of $A_{\infty}$-algebras.
\end{remark}

We now want to define monoidal \icats{} as a special class of
\nsiopds{}. The appropriate definition is suggested by the following
observation:
\begin{lemma}\ 
  \begin{enumerate}[(i)]
  \item An object $\pi \colon \mathbf{C} \to \simp^{\op}$ in
    $\Cat_{/\simp^{\op}}^{\txt{mult}}$ is equivalent to the category of
    operators of the multicategory associated to a monoidal category
    \IFF{} $\pi$ is a Grothendieck opfibration.
  \item A morphism $\phi \colon
    \mathcal{C} \to \mathcal{D}$ between two such objects corresponds to
    a lax monoidal functor between the associated monoidal categories.
  \item Under this correspondence the (strong) monoidal functors give
    precisely the morphisms that preserve \emph{all} coCartesian
    morphisms.
  \end{enumerate}
\end{lemma}
\begin{proof}
  Let $\mathbf{M}$ be the multicategory corresponding to $\pi \colon
  \mathbf{C} \to \simp^{\op}$, and write $\mathbf{M}_{0} \cong
  \mathbf{C}_{[1]}$ for its underlying category. The existence of
  coCartesian morphisms for $\alpha_{n} \colon [1] \to [n]$ implies
  that there is a functor $\otimes_{n} \colon \mathbf{C}_{[1]}^{\times
    n} \simeq \mathbf{C}_{[n]} \to \mathbf{C}_{[1]}$ such that
  $\mathbf{M}(X_{1},\ldots,X_{n}; Y) \cong
  \mathbf{M}_{0}(\otimes_{n}(X_{1},\ldots,X_{n}), Y)$. But writing
  $\alpha_{n}$ as a composite of elementary face maps in $\simp$ in
  various ways, we get canonical equivalences between $\otimes_{n}$
  and the various ways of successively applying $\otimes_{2}$ to
  adjacent elements. Moreover, the coCartesian morphism over the
  degeneracy $[1] \to [0]$ in $\simp$ gives a map $* \simeq
  \mathbf{C}_{[0]} \to \mathbf{C}_{[1]}$, which amounts to a unit $I
  \in \mathbf{M}_{0}$. This implies that if we define $X \otimes Y :=
  \otimes_{2}(X,Y)$ then $\otimes$ is a monoidal structure on
  $\mathbf{M}_{0}$ such that $\mathbf{M}$ is the multicategory
  associated to this monoidal category. This proves (i). (ii) is then
  clear, since lax monoidal functors clearly correspond to functors
  between the assocaited multicategories, and (iii) follows since a
  functor preserves all coCartesian arrows precisely if we have
  natural isomorphisms $F(X) \otimes F(Y) \cong F(X \otimes Y)$ and
  $F(I) \cong I$.
\end{proof}

In the \icatl{} case we therefore make the following definitions:
\begin{defn}
  A \defterm{monoidal \icat{}} is a \nsiopd{} $\mathcal{V}^{\otimes}
  \to \simp^{\op}$ that is also a coCartesian fibration. We will
  generally denote the fibre $\mathcal{V}^{\otimes}_{[1]}$ by
  $\mathcal{V}$; by abuse of notation we will allow ourselves to say
  ``let $\mathcal{V}$ be a monoidal \icat{}'' as shorthand
  for ``let $\mathcal{V}^{\otimes} \to \simp^{\op}$ be a monoidal
  \icat{}''.
\end{defn}

\begin{defn}
  If $\mathcal{V}^{\otimes}$ and $\mathcal{W}^{\otimes}$ are monoidal
  \icats{}, we will refer to a morphism of \nsiopds{} from
  $\mathcal{V}^{\otimes}$ to $\mathcal{W}^{\otimes}$ as a \defterm{lax
    monoidal functor}. A \defterm{monoidal functor} from
  $\mathcal{V}^{\otimes}$ to $\mathcal{W}^{\otimes}$ is a commutative
  diagram
  \opctriangle{\mathcal{V}^{\otimes}}{\mathcal{W}^{\otimes}}{\simp^{\op}}{\phi}{}{}
  such that $\phi$ preserves \emph{all} coCartesian morphisms.
\end{defn}

\begin{remark}\label{rmk:monicatsegcond}
  For a coCartesian fibration $\pi \colon \mathcal{C} \to
  \simp^{\op}$, condition (iii) in the definition of \nsiopds{}
  follows from condition (ii), since the coCartesian morphisms in
  $\mathcal{C}$ allow us to identify the space of maps over $\phi
  \colon [n] \to [m]$ in $\simp^{\op}$ with a space of maps in
  $\mathcal{C}_{[n]}$, which decomposes as a product due to condition
  (ii). This means that, under the equivalence between coCartesian
  fibrations over $\simp^{\op}$ and functors $\simp^{\op} \to \CatI$,
  monoidal \icats{} precisely correspond to simplicial \icats{}
  $\mathcal{C}_{\bullet}$ that satisfy the \emph{Segal condition}: the
  map $\mathcal{C}_{n} \to \mathcal{C}_{1}^{\times n}$ induced by the
  maps $\rho_{i} \colon [n] \to [1]$ in $\simp^{\op}$ are
  equivalences. The idea that simplicial objects satisfying this
  condition give a model for $A_{\infty}$-algebras goes back to Segal
  (as an unpublished variant of the definition of
  $E_{\infty}$-algebras using $\bbGamma$-spaces in \cite{SegalCatCohlgy})
  --- thus we can interpret monoidal \icats{} as $A_{\infty}$-algebras
  (or just associative algebras, since we are working in the ``fully
  weak'' context of \icats{}) in $\CatI$.
\end{remark}

\begin{remark}
  A monoidal \icat{} $\mathcal{V}^{\otimes}$ corresponds to the data of a
  homotopy-coherently associative tensor product on $\mathcal{V}$. To
  see this, let us unpack the data we get from a monoidal \icat{},
  interpreted as a simplical \icat{} $\mathcal{V}_{\bullet}$ satisfying the
  Segal condition:
  \begin{itemize}
  \item The map $d_{1} \colon [2] \to [1]$ gives a tensor product
    $\otimes \colon \mathcal{V}^{\times 2} \simeq \mathcal{V}_{2}
    \to \mathcal{V}$.
  \item The map $s_{0} \colon [0] \to [1]$ gives a unit $* \simeq
    \mathcal{V}_{0} \to \mathcal{V}$.
  \item The map $\alpha_{3} \colon [3] \to [1]$ gives a map
    $\otimes_{3} \colon \mathcal{V}^{\times 3} \simeq \mathcal{V}_{3}
    \to \mathcal{V}$. The two factorizations $\alpha_{3} = d_{1} \circ
    d_{1} = d_{1} \circ d_{2}$ give 2-simplices in $\CatI$ that can be
    interpreted as natural equivalences between $\otimes_{3}(A,B,C)$
    and the composites $(A \otimes B) \otimes C$ and $A \otimes (B
    \otimes C)$, respectively. Composing these gives the expected
    natural associator equivalence $(A \otimes B) \otimes C \simeq A
    \otimes (B \otimes C)$.
  \item Similarly, the different ways of decomposing $\alpha_{4}
    \colon [4] \to [1]$ as a composite of 3 face maps gives
    3-simplices in $\CatI$ that determine homotopies between the
    different ways of using the associator to pass between different
    4-fold tensor products.
  \item In general, the different ways of decomposing $\alpha_{n}$ as
    a composite of $n-1$ face maps gives $(n-1)$-simplices in $\CatI$
    that determine the coherence data for $n$-fold tensor products.
  \end{itemize}
\end{remark}

If $\mathbf{M}$ is an ordinary multicategory, then it is clear that
(the nerve of) its category of operators $\mathbf{M}^{\otimes}$ is a
\nsiopd{} --- by abuse of notation we will also refer to this
\nsiopd{} as $\mathbf{M}$ in contexts where this does not cause
confusion. We can then define enriched \icats{} as follows:
\begin{defn}
  If $S$ is a set and $\mathcal{V}^{\otimes}$ is a monoidal \icat{}, a
  \emph{$\mathcal{V}$-enriched \icat{}} (or
  \emph{$\mathcal{V}$-\icat{}}) with set of objects $S$ is an
  $\mathbf{O}_{S}$-algebra in $\mathcal{V}$, i.e. a morphism of
  \nsiopds{} $\mathbf{O}_{S}^{\otimes} \to \mathcal{V}^{\otimes}$. If
  $\mathcal{C}$ and $\mathcal{D}$ are $\mathcal{V}$-\icats{} with sets
  of objects $S$ and $T$, respectively, then a
  \emph{$\mathcal{V}$-functor}
  from $\mathcal{C}$ to $\mathcal{D}$ consists of a function $f \colon
  S \to T$ and a natural transformation $\eta \colon \mathcal{C} \to
  f^{*}\mathcal{D}$ of functors $\mathbf{O}_{S}^{\otimes} \to
  \mathcal{V}^{\otimes}$, where $f^{*}\mathcal{D}$ denotes the
  composite of $\mathcal{D}$ with the functor
  $\mathbf{O}_{S}^{\otimes} \to \mathbf{O}_{T}^{\otimes}$ induced by
  $f$.
\end{defn}

\begin{ex}
  For a one-element set, $\mathbf{O}_{*}$ is just the associative
  operad, and $\mathbf{O}_{*}^{\otimes}$ is $\simp^{\op}$. Thus
  one-object $\mathcal{V}$-\icats{} are precisely \icatl{} associative
  algebras, i.e. $A_{\infty}$-algebras, just as we would expect.
\end{ex}

\begin{remark}
  We saw at the end of \S\ref{subsec:MulticatEnr} that
  $\mathbf{O}_{S}$-algebras in a monoidal category $\mathbf{V}$
  correspond to $\mathbf{V}$-enriched categories with $S$ as their set
  of objects. Similarly, an $\mathbf{O}_{S}$-algebra $\mathcal{C}$ in a monoidal
  \icat{} $\mathcal{V}$ corresponds to the data we would expect to
  have in an enriched \icat{}. Speaking somewhat informally, to make
  the underlying ideas clearer, we have for example the following data:
  \begin{itemize}
  \item The object $(X,Y)$ in $\mathbf{O}_{S}^{\otimes}$ is sent to an object
    $\mathcal{C}(X,Y) \in \mathcal{V}$.
  \item The morphism $((X,Y), (Y,Z)) \to (X,Z)$ in
    $\mathbf{O}_{S}^{\otimes}$ is sent to a morphism
    $\mu_{X,Y,Z} \colon \mathcal{C}((X,Y), (Y,Z)) \to \mathcal{C}(X,Z)$ in
    $\mathcal{V}^{\otimes}$. Since $\mathcal{C}$ preserves coCartesian
    morphisms over inert maps in $\simp^{\op}$, under the equivalence
    $\mathcal{V}^{\otimes}_{[2]} \simeq \mathcal{V}^{\times 2}$ the
    object $\mathcal{C}((X,Y), (Y,Z))$ is equivalent to
    $(\mathcal{C}(X,Y), \mathcal{C}(Y,Z))$, and so using the
    coCartesian morphism over the map $d_{1} \colon [2] \to [1]$, we
    can interpret this as a composition morphism
    $\mathcal{C}(X,Y)\otimes \mathcal{C}(Y,Z) \to \mathcal{C}(X,Z)$.
  \item Similarly, the morphism $() \to (X,X)$ is sent to a morphism
    we may interpret as a map $I \to \mathcal{C}(X,X)$ where $I$ is
    the unit of the tensor product on $\mathcal{V}$.
  \item The morphism $((X,Y), (Y,Z), (Z, W)) \to (X, W)$ in
    $\mathbf{O}_{S}^{\otimes}$ factors as $((X,Y), (Y,Z), (Z,W)) \to
    ((X,Z), (Z,W)) \to (X,W)$ and also as $((X,Y), (Y,Z), (Z,W)) \to
    ((X,Y), (Y,W)) \to (X,W)$. Pushing the associated data in
    $\mathcal{V}^{\otimes}$ into $\mathcal{V}$ using the coCartesian
    morphisms, this gives:
    \begin{itemize}
    \item an object $\otimes_{3}(\mathcal{C}(X,Y), \mathcal{C}(Y,Z),
      \mathcal{C}(Z,W))$ with equivalences $\alpha$ to $\mathcal{C}(X,Y)
      \otimes (\mathcal{C}(Y,Z) \otimes \mathcal{C}(Z,W))$ and $\beta$ to
      $(\mathcal{C}(X,Y) \otimes (\mathcal{C}(Y,Z)) \otimes
      \mathcal{C}(Z,W)$
    \item a morphism $\mu_{X,Y,Z,W} \colon \otimes_{3}(\mathcal{C}(X,Y), \mathcal{C}(Y,Z),
      \mathcal{C}(Z,W)) \to \mathcal{C}(X,W)$
    \item homotopies between $\mu_{X,Y,Z,W} \circ \alpha^{-1}$ and
      $\mu_{X,Y,W} \circ (\id \otimes \mu_{Y,Z,W})$ and between
      $\mu_{X,Y,Z,W} \circ \beta^{-1}$ and $\mu_{X,Z,W} \circ
      (\mu_{X,Y,Z} \otimes \id)$.
    \end{itemize}
    The latter two homotopies can then be composed to get a homotopy
    between $\mu_{X,Y,W} \circ (\id \otimes \mu_{Y,Z,W})$ and
    $\mu_{X,Z,W} \circ (\mu_{X,Y,Z} \otimes \id)$, which is the first
    homotopy-coherence data for the associativity of the composition
    operation.
  \item Similarly, the data derived from the different decompositions
    of $((X,Y), (Y,Z), (Z,W), (W,V)) \to (X,V)$ as composites of
    ``face maps'' gives the coherence data for 3-fold compositions,
    and so forth.
  \end{itemize}
\end{remark}

If $\mathcal{O}$ and $\mathcal{P}$ are \nsiopds{}, we get an \icat{}
$\Alg_{\mathcal{O}}(\mathcal{P})$ of $\mathcal{O}$-algebras in
$\mathcal{P}$ by taking the full subcategory spanned by the morphisms
of \nsiopds{} in the \icat{}
$\Fun_{\simp^{\op}}(\mathcal{O}, \mathcal{P})$ of
functors over $\simp^{\op}$.  By abuse of notation, if $\mathcal{O}$
is a \nsiopd{} and $\mathcal{V}^{\otimes}$ is a monoidal \icat{} we
will usually write $\Alg_{\mathcal{O}}(\mathcal{V})$ instead of
$\Alg_{\mathcal{O}}(\mathcal{V}^{\otimes})$.

In \S\ref{subsec:icatiopds} we will construct an \icat{} $\OpdIns$ of
non-symmetric \iopds{} and see that the \icat{}
$\Alg_{\mathcal{O}}(\mathcal{P})$ is functorial in
$\mathcal{O}$ and $\mathcal{P}$. This allows us to
construct a Cartesian fibration \[\Alg(\mathcal{P}) \to
\OpdIns\] whose fibre at $\mathcal{O}$ is
$\Alg_{\mathcal{O}}(\mathcal{P})$. Pulling this
back along the functor $\Set \to \OpdIns$ that sends a set $S$ to
$\mathbf{O}_{S}^{\otimes}$ we get an \icat{}
$\AlgCat(\mathcal{P})$ with a projection to $\Set$. If
$\mathcal{V}$ is a monoidal \icat{}, the objects of
$\AlgCat(\mathcal{V})$ are clearly $\mathcal{V}$-enriched
\icats{} and the morphisms are precisely $\mathcal{V}$-functors.

A $\mathcal{V}$-functor $\mathcal{C} \to \mathcal{D}$ is given by a
function $f \colon S \to T$ of sets of objects and a morphism $\eta
\colon \mathcal{C} \to f^{*}\mathcal{D}$ of
$\mathbf{O}_{S}$-algebras. This morphism is an equivalence in
$\AlgCat(\mathcal{V})$ \IFF{} $f$ is a \emph{bijection} of sets and
$\eta$ is an equivalence of $\mathbf{O}_{S}$-algebras (i.e. the
morphism is \emph{fully faithful}). This is obviously not the correct
notion of equivalence for $\mathcal{V}$-\icats{} --- we want the
equivalences to be the morphisms that are fully faithful and
\emph{essentially surjective} (in the usual sense that every object of
$\mathcal{D}$ is \emph{equivalent} to an object in the image of $f$;
we will define this precisely below in \S\ref{subsec:FFES} after
discussing equivalences in enriched \icats{} in
\S\ref{subsec:equiv}). We therefore want to invert these morphisms. In
the \icatl{} setting it is always possible to formally invert any
collection of morphisms, but to understand the resulting localization
we need it to be an \emph{accessible} localization. This is the
\icatl{} analogue of left Bousfield localization of model categories,
and means that we can find the localized \icat{} as the full
subcategory of \emph{local} objects inside the original
\icat{}. However, this is easily seen to be impossible using our
current definition of enriched \icats{}: For example, if we enrich in
the monoidal category of sets with the Cartesian product, then
$\AlgCat(\Set)$ is just the ordinary category of small categories and
functors. But if we invert the fully faithful and essentially
surjective functors we get the $(2,1)$-category of categories,
functors, and natural equivalences, which obviously cannot be a full
subcategory of an ordinary category.

To avoid this problem we need another definition of enriched \icats{}
for which this localization is well-behaved. It will turn out that we
get a much nicer \icat{} of enriched \icats{} if we allow them to have
\emph{spaces} of objects rather than just sets --- this is also
aligned with the philosophy of higher category theory, whereby spaces
should be thought of as the \icatl{} analogue of sets in ordinary
category theory. One way to do this would be to define simplicial
multicategories $\mathbf{O}_{S}$ where $S$ is now a simplicial
groupoid, and then work with the associated \iopds{}. We will, in
fact, define such simplicial multicategories and briefly make use of
them below in \S\ref{subsec:OX}, but it turns out that there is an
easier and more natural way to carry out this generalization: We will
base our theory of enriched \icats{} on the \icatl{} version of a
slightly different approach to enriched categories, using
\emph{virtual double categories} instead of multicategories, which we
describe in the next subsection.

\subsection{Virtual Double Categories and Enrichment}
\defterm{Virtual double categories}\footnote{Also known as
  $\mathbf{fc}$-multicategories; note that, for consistency with
  Lurie's terminology, we will refer to their \icatl{} generalization
  as \emph{generalized non-symmetric \iopds{}}.} are a common
generalization of double categories and multicategories. Roughly
speaking, a virtual double category has objects and vertical and
horizontal morphisms between them, but in addition to a collection of
``squares'' there are cells with a list of vertical arrows as source;
we refer the reader to \cite{CruttwellShulman} or
\cite{LeinsterHigherOpds} for an explicit definition along this point
of view.

Here, we will instead consider virtual double categories from the category
of operators perspective: they are exactly what we get if we allow the
fibre $\mathbf{O}_{[0]}$ at $[0]$ in a category of operators
to be non-trivial, and require
$\mathbf{O}_{[n]}$ to be the $n$-fold iterated
fibre product \[\mathbf{O}_{[1]}
\times_{\mathbf{O}_{[0]}} \cdots
\times_{\mathbf{O}_{[0]}}
\mathbf{O}_{[1]}.\] To state the precise definition we
first introduce some notation:
\begin{defn}
  Let $\simp^{\op}_{\txt{int}}$ denote the subcategory of
  $\simp^{\op}$ where the morphisms are the inert morphisms in
  $\simp^{\op}$. We write $\mathcal{G}^{\simp}$ for the full subcategory
  of $\simp^{\op}_{\txt{int}}$ spanned by the objects $[0]$ and $[1]$,
  and $\mathcal{G}^{\simp}_{[n]/}$ for the category
  $(\simp^{\op}_{\txt{int}})_{[n]/} \times_{\simp^{\op}}
  \mathcal{G}^{\simp}$ of inert morphisms from $[n]$ to $[1]$ and
  $[0]$.
\end{defn}

\begin{defn}
  A \defterm{virtual double category} is a functor $\pi \colon \mathbf{M}
  \to \simp^{\op}$ such that:

  \begin{enumerate}[(i)]
  \item For every inert morphism $\phi \colon [m] \to [n]$  in
    $\simp^{\op}$ and every $X \in \mathbf{M}_{[n]}$ there exists
    a $\pi$-coCartesian morphism $X \to \phi_{!}X$ over $\phi$.
  \item For every $[n] \in \simp^{\op}$ the functor \[\mathbf{M}_{[n]}
    \to \lim_{[n] \to [i] \in \mathcal{G}^{\simp}_{[n]/}}
    \mathbf{M}_{[i]} \simeq \mathbf{M}_{[1]} \times_{\mathbf{M}_{[0]}}
    \cdots\times_{\mathbf{M}_{[0]}} \mathbf{M}_{[1]}\] induced by the
    coCartesian arrows over the inert maps in
    $\mathcal{G}^{\simp}_{[n]/}$ is an equivalence of categories.
  \item For every morphism $\phi \colon [n] \to [m]$ in $\simp^{\op}$
    and $Y \in \mathbf{M}_{[m]}$, composition with coCartesian
    morphisms $Y \to Y_{\alpha}$ over the inert morphisms $\alpha
    \colon [m] \to [i]$ in $\mathcal{G}^{\simp}_{[m]/}$ gives
    an isomorphism
    \[ \Hom_{\mathbf{M}}^{\phi}(X, Y) \isoto \lim_{\alpha \in
      \mathcal{G}^{\simp}_{[m]/}} \Hom_{\mathbf{M}}^{\alpha \circ
      \phi}(X, Y_{\alpha}),\] where $\Hom_{\mathbf{M}}^{\phi}(X,Y)$
    denotes the subset of $\Hom_{\mathbf{M}}(X,Y)$ of morphisms that
    map to $\phi$ in $\simp^{\op}$.
  \end{enumerate}
\end{defn}
\begin{remark}
  A virtual double category $\mathbf{M} \to \simp^{\op}$ corresponds
  to a double category precisely when this functor is a Grothendieck
  opfibration.
\end{remark}

\begin{defn}
  If $\mathbf{M} \to \simp^{\op}$ and $\mathbf{N} \to \simp^{\op}$ are
  virtual double categories, a \defterm{functor of virtual double
    categories} from $\mathbf{M}$ to $\mathbf{N}$ is a functor $F
  \colon \mathbf{M} \to \mathbf{N}$ over $\simp^{\op}$ that preserves
  coCartesian morphisms over inert morphisms in $\simp^{\op}$.
\end{defn}

Given a set $S$, we can define a double category with set of objects
$S$ where the vertical morphisms are trivial, and there is a unique
horizontal morphism between any two elements of $S$. In terms of
categories of operators, this corresponds to the category
$\simp^{\op}_{S}$ whose objects are non-empty sequences $(X_{0},
\ldots, X_{n})$ of elements $X_{i} \in S$, with a unique
morphism \[(X_{0}, \ldots, X_{n}) \to (X_{\phi(0)}, \ldots,
X_{\phi(m)})\] for each $\phi \colon [m] \to [n]$ in $\simp$. If
$\mathbf{V}$ is a monoidal category, and $\mathbf{V}^{\otimes}$ is its
category of operators, a functor of virtual double categories
$\mathbf{C} \colon \simp^{\op}_{S} \to \mathbf{V}^{\otimes}$ is a
functor over $\simp^{\op}$ such that $\mathbf{C}(X_{0}, \ldots, X_{n})
= (\mathbf{C}(X_{0}, X_{1}), \ldots, \mathbf{C}(X_{n-1},
X_{n}))$. This is precisely a $\mathbf{V}$-category with set of
objects $S$: for each $X \in S$ the unique map $X \to (X,X)$ gives an
identity $I \to \mathbf{C}(X,X)$, and for objects $X,Y,Z \in S$ the
map $(X,Y,Z) \to (X,Z)$ over $d_{1}\colon [2] \to [1]$ gives a
composition map $\mathbf{C}(X,Y) \otimes \mathbf{C}(Y,Z) \to
\mathbf{C}(X,Z)$, which is associative because the two composite maps
$(X,Y,Z,W) \to (X,Y,W) \to (X,W)$ and $(X,Y,Z,W) \to (X,Z,W) \to
(X,W)$ are equal.

A functor between $\mathbf{V}$-categories $\mathbf{C}$ and
$\mathbf{D}$ with sets of objects $S$ and $T$, respectively, can then
be described as a function $f \colon S \to T$ together with a natural
transformation $\mathbf{C} \to f^{*}\mathbf{D}$ of functors
$\simp^{\op}_{S} \to \mathbf{V}^{\otimes}$, where $f^{*}\mathbf{D}$
denotes the composite of $\mathbf{D}$ with the functor
$\simp^{\op}_{f} \colon \simp^{\op}_{S} \to \simp^{\op}_{T}$ induced
by $f$: this natural transformation precisely gives maps
$\mathbf{C}(X,Y) \to \mathbf{D}(f(X), f(Y))$ compatible with units and
composition.

\begin{remark}
  Using the virtual double categories $\simp^{\op}_{S}$ to define
  enrichment gives the right notion also when considering enrichment
  in more general settings, such as enrichment in double categories or
  in general virtual double categories (cf. \cite{LeinsterGenEnr}).
\end{remark}

\subsection{Generalized $\infty$-Operads}\label{subsec:FTgeniopd}
It is now clear how to generalize the notion of virtual double
category to the \icatl{} setting, analogously to our definition of
non-symmetric \iopds{} above:
\begin{defn}\label{defn:gnsiopd1}
  A \defterm{generalized non-symmetric $\infty$-operad} is an inner
  fibration $\pi \colon \mathcal{M} \to \simp^{\op}$ such
  that:
  \begin{enumerate}[(i)]
  \item For every inert morphism $\phi \colon [m] \to [n]$  in
    $\simp^{\op}$ and every $X \in \mathcal{M}_{[n]}$ there exists
    a $\pi$-coCartesian morphism $X \to \phi_{!}X$ over $\phi$.
  \item For every $[n] \in \simp^{\op}$ the functor
    \[\mathcal{M}_{[n]}
    \to \lim_{[n] \to [i] \in \mathcal{G}^{\simp}_{[n]/}}
    \mathcal{M}_{[i]} \simeq \mathcal{M}_{[1]}
    \times_{\mathcal{M}_{[0]}} \cdots \times_{\mathcal{M}_{[0]}}
    \mathcal{M}_{[1]}\] induced by the coCartesian arrows over the
    inert maps in $\mathcal{G}^{\simp}_{[n]/}$ is an equivalence of \icats{}.
  \item For every morphism $\phi \colon [n] \to [m]$ in $\simp^{\op}$
    and $Y \in \mathcal{M}_{[m]}$, composition with
    coCartesian morphisms $Y \to Y_{\alpha}$ over the inert morphisms $\alpha \colon [m] \to
    [i]$ in $\mathcal{G}^{\simp}_{[m]/}$ 
    gives an equivalence
    \[ \Map_{\mathcal{M}}^{\phi}(X, Y) \isoto \lim_{\alpha \in
      \mathcal{G}^{\simp}_{[m]/}} \Map_{\mathcal{M}}^{\alpha \circ \phi}(X,
    Y_{\alpha}),\] where $\Map_{\mathcal{M}}^{\phi}(X,Y)$ denotes the
    subspace of $\Map_{\mathcal{M}}(X,Y)$ of morphisms that map to
    $\phi$ in $\simp^{\op}$.
  \end{enumerate}
\end{defn}

\begin{defn}
  If $\mathcal{M}$ and $\mathcal{N}$ are
  \gnsiopds{}, a \defterm{morphism of \gnsiopds{}} from
  $\mathcal{M}$ to $\mathcal{N}$ is a commutative
  diagram
  \opctriangle{\mathcal{M}}{\mathcal{N}}{\simp^{\op}}{\phi}{}{}
  such that $\phi$ carries coCartesian morphisms in
  $\mathcal{M}$ that map to inert morphisms in $\simp^{\op}$
  to coCartesian morphisms in $\mathcal{N}$. We will also
  refer to a morphism of \gnsiopds{} $\mathcal{M} \to
  \mathcal{N}$ as an \emph{$\mathcal{M}$-algebra}
  in $\mathcal{N}$.
\end{defn}

\begin{defn}
  A \defterm{double \icat{}} is a \gnsiopd{} $\mathcal{M}
  \to \simp^{\op}$ that is also a coCartesian fibration. 
\end{defn}

\begin{remark}\label{rmk:doubleicatsegcond}
  Again, as in Remark~\ref{rmk:monicatsegcond}, for a coCartesian fibration
  condition (iii) in the definition of a \gnsiopd{} is implied by
  condition (ii). Thus, under the equivalence between coCartesian
  fibrations to $\simp^{\op}$ and functors $\simp^{\op} \to \CatI$,
  double \icats{} correspond to simplicial \icats{}
  $\mathcal{C}_{\bullet}$ satisfying the ``Rezk-Segal
  condition'':
  \[ \mathcal{C}_{n} \to \mathcal{C}_{1} \times_{\mathcal{C}_{0}}
  \cdots \times_{\mathcal{C}_{0}} \mathcal{C}_{1}\] is an
  equivalence. In general simplicial objects in an \icat{}
  $\mathcal{X}$ with finite limits satisfying this condition can be
  thought of as \emph{internal categories} in $\mathcal{X}$ --- in
  particular, taking $\mathcal{X}$ to be the \icat{} of spaces these
  are precisely the \emph{Segal spaces} introduced by
  Rezk~\cite{RezkCSS} as a model for \icats{}. This justifies the
  term double \icat{}, since double categories are precisely internal
  categories in $\Cat$.
\end{remark}

We can now introduce a generalization of the virtual double categories
$\simp^{\op}_{S}$: If $S \in \mathcal{S}$ is a space, there is a
functor $\simp^{\op} \to \mathcal{S}$ that sends $[n]$ to $S^{\times
  n}$, face maps to projections to the corresponding factors, and
degeneracies to the corresponding diagonal maps; a more precise
definition will be given in \S\ref{subsec:DeltaOpX}. It is easy to see
that this simplicial space satisifes the Rezk-Segal condition, so if
we let $\simp^{\op}_{S} \to \simp^{\op}$ be a left fibration
corresponding to this functor then this is a double \icat{} by
Remark~\ref{rmk:doubleicatsegcond}. When $S$ is a set this obviously
agrees with the previous definition.

Using this we can state our improved definition of enriched \icats{}:
\begin{defn}
  Let $S \in \mathcal{S}$ be a space and let $\mathcal{V}$
  be a monoidal \icat{}. A \defterm{$\mathcal{V}$-enriched \icat{}}
  (or \emph{$\mathcal{V}$-\icat{}}) with space of objects $S$ is a
  $\simp^{\op}_{S}$-algebra in $\mathcal{V}$.
\end{defn}

\begin{ex}
  Any associative algebra object in $\mathcal{V}$ can be
  regarded as a $\mathcal{V}$-\icat{} with a contractible space of
  objects. In particular, the unit $I$ of the tensor product in
  $\mathcal{V}$ has a unique associative algebra structure (by
  Proposition~\ref{propn:UnitAlg}) so we can regard $I$
  as a $\mathcal{V}$-\icat{} with a single object whose endomorphisms
  are given by $I$.
\end{ex}

\begin{remark}
  We will define the \gnsiopds{} $\simp^{\op}_{S}$ more carefully
  below in \S\ref{subsec:DeltaOpX}. It will sometimes be useful, for example to
  distinguish our definition from other possible definitions of
  enriched \icats{}, to refer to a $\simp^{\op}_{S}$-algebra in
  $\mathcal{V}$ as a \defterm{categorical algebra} in
  $\mathcal{V}$ with space of objects $S$.
\end{remark}

\begin{remark}
  This definition clearly does not require $\mathcal{V}$ to
  be a monoidal \icat{} --- we can define \icats{} enriched in any
  \gnsiopd{} as $\simp^{\op}_{S}$-algebras. This gives an \icatl{}
  version of Leinster's notion of enrichment in an
  $\mathbf{fc}$-multicategory~\cite{LeinsterGenEnr}. However, as there
  are technical obstacles in the theory of \iopds{} to extending most
  of our results beyond the case of monoidal \icats{}, we will not
  consider this generalization here.
\end{remark}

\begin{defn}
  Suppose $\mathcal{V}$ is a monoidal \icat{}, and
  $\mathcal{C}$ and $\mathcal{D}$ are $\mathcal{V}$-\icats{} with
  spaces of objects $S$ and $T$, respectively. A
  \defterm{$\mathcal{V}$-functor} from $\mathcal{C}$ to $\mathcal{D}$
  consists of a morphism of spaces $f \colon S \to T$ and a natural
  transformation $\mathcal{C} \to f^{*}\mathcal{D}$, where
  $f^{*}\mathcal{D}$ denotes the composite of $\mathcal{D}$ with the
  morphism $\simp^{\op}_{f}\colon \simp^{\op}_{S} \to
  \simp^{\op}_{T}$ induced by $f$.
\end{defn}

If $\mathcal{M}$ and $\mathcal{N}$ are \gnsiopds{} we get an \icat{}
$\Alg_{\mathcal{M}}(\mathcal{N})$ of $\mathcal{M}$-algebras in
$\mathcal{N}$ by taking the full subcategory of the \icat{}
$\Fun_{\simp^{\op}}(\mathcal{M}, \mathcal{N})$ of functors over
$\simp^{\op}$ that is spanned by the morphisms of \gnsiopds{}. Just as
for \iopds{}, we will construct (in \S\ref{subsec:icatiopds}) an \icat{}
$\OpdInsg$ of \gnsiopds{}, and the \icats{}
$\Alg_{\mathcal{M}}(\mathcal{N})$ are functorial in $\mathcal{M}$
and $\mathcal{N}$. As before, we then get a Cartesian fibration
$\Alg(\mathcal{N}) \to \OpdInsg$ whose fibre at $\mathcal{M}$ is
$\Alg_{\mathcal{M}}(\mathcal{N})$. We can pull this back along the
functor $\mathcal{S} \to \OpdInsg$ that sends $S \in \mathcal{S}$ to
$\simp^{\op}_{S}$ to get an \icat{} $\AlgCat(\mathcal{N})$. If
$\mathcal{V}$ is a monoidal \icat{}, the objects of
$\AlgCat(\mathcal{V})$ are $\mathcal{V}$-\icats{} and the
morphisms are $\mathcal{V}$-functors. 

\begin{remark}
  We refer to the \icat{} $\AlgCat(\mathcal{V})$ (which we
  will construct more carefully below in \S\ref{subsec:catalg}) as the
  \emph{\icat{} of categorical algebras} in $\mathcal{V}$,
  reserving the name \emph{\icat{} of $\mathcal{V}$-\icats{}} for the
  localization of this at the fully faithful and essentially
  surjective functors.
\end{remark}

We will prove in \S\ref{subsec:FFES} that inverting the fully faithful
and essentially surjective functors in the \icat{}
$\AlgCat(\mathcal{V})$ as we have just defined it gives the same
\icat{} as inverting them in the version considered above where we
only allowed sets of objects. Now, however, we can find the localized
\icat{} as a full subcategory of $\AlgCat(\mathcal{V})$. The local
objects turn out to be the \emph{complete} $\mathcal{V}$-\icats{},
which are those whose space of objects is equivalent to their
classifying space of equivalences, in a sense we will make precise
below in \S\ref{subsec:equiv}. If we write $\CatIV$ for the full
subcategory of $\AlgCat(\mathcal{V})$ spanned by these complete
$\mathcal{V}$-\icats{}, the main result of this article is the
following:
\begin{thm}
  Let $\mathcal{V}$ be a monoidal \icat{}. The inclusion
  \[\CatIV \hookrightarrow \AlgCat(\mathcal{V})\] has a left
  adjoint, and this exhibits $\CatIV$ as the localization of
  $\AlgCat(\mathcal{V})$ with respect to the fully
  faithful and essentially surjective functors.
\end{thm}

\subsection{Enriched Categories as Presheaves}
As discussed above, our main construction of the \icat{} $\CatIV$ of
\icats{} enriched in $\mathcal{V}$ will be as a localization of
$\AlgCatV$, an \icat{} of algebras for a family of (generalized)
\iopds{}. Although useful for many purposes --- for example, it is
easy to relate $\AlgCatV$ to model categories of strictly enriched
categories (cf. \cite{enrcomp}) --- when working with a
presentable \icat{} it can also often be useful to have a construction
of it as an explicit localization of an \icat{} of presheaves on a
small \icat{} of generators. In much the same way as $\CatI$ itself
embeds into $\mathcal{P}(\simp)$ as the full subcategory of complete
Segal spaces, one might imagine that $\CatIV$ embeds into presheaves
on a $\mathcal{V}$-enriched version of $\simp$ whose objects
classify ``composable strings of morphisms'' in a
$\mathcal{V}$-enriched $\infty$-category $\mathcal{C}$.

In fact, this $\mathcal{V}$-enriched version of $\simp$ nearly comes
to us for free from our monoidal \icat{} 
$p \colon \mathcal{V}^{\otimes} \to \simp^{\op}$. The functor $p$ is a
coCartesian fibration, and so arises as the unstraightening of a functor
$\simp^{\op} \to \CatI$ which satisfies the usual Segal condition.  But
we may also unstraighten $p$ to a Cartesian fibration
$q \colon \mathcal{V}^{\vee}_{\otimes}\to \simp$ --- this is our desired
$\mathcal{V}$-enriched version of $\simp$.  Roughly speaking, the
objects of $\mathcal{V}^{\vee}_{\otimes}$ are ordered tuples $(V_{1},\ldots,V_{n})$
of objects of $\mathcal{V}$, which we can interpret as the free
$\mathcal{V}$-enriched $\infty$-category on the $\mathcal{V}$-enriched
graph
\[
0\xto{V_{1}} 1 \xto{V_{2}} 2 \to \cdots \to n-1 \xto{V_{n}} n,
\]
which we denote $\Delta^{(V_1,\ldots,V_n)}$. The free $\mathcal{V}$-\icat{} on this graph has composition
determined by the monoidal structure on $\mathcal{V}$, so for example
the maps from $i-1$ to $j$ in $\Delta^{(V_1,\ldots,V_n)}$ are given by $V_{i} \otimes V_{i+1} \otimes
\cdots \otimes V_{j}$.

A $\mathcal{V}$-enriched $\infty$-category $\mathcal{C}$ then
determines a presheaf
\[
\Map_{\CatIV}(\blank,\mathcal{C}) \colon (\mathcal{V}^{\vee}_{\otimes})^{\op}\longrightarrow\mathcal{S}
\]
by sending $\Delta^{(V_0,\ldots,V_n)}$ to the space of $\mathcal{V}$-enriched functors from $\Delta^{(V_0,\ldots,V_n)}$ to $\mathcal{C}$.
This construction induces a functor
\[
\CatIV\longrightarrow\mathcal{P}(\mathcal{V}^{\vee}_{\otimes}).
\]
We will rigorously construct this in \ref{subsec:presheafalgcat} and
show that it is fully faithful, from which it follows almost
immediately that $\CatIV$ is an accessible localization of
$\mathcal{P}(\mathcal{V}^{\vee}_{\otimes})$.  Moreover, the essential
image of this embedding can be identified with the ``complete Segal
spaces'' in a sense entirely analgous to that of Rezk \cite{RezkCSS},
and the categorical algebras $\AlgCat(\mathcal{V})$ embed in
$\mathcal{P}(\mathcal{V}^{\vee}_{\otimes})$ as analogues of the Segal
spaces.  We use this ambient presheaf $\infty$-category in
\ref{subsec:completion} to prove a crucial technical result about the
``completion'' functor $\AlgCatV \to \CatIV$.

\section{Non-Symmetric $\infty$-Operads}\label{sec:NSOP}
In this section we give the definitions and results we need about
(generalized) \nsiopds{}. These are a special case of Barwick's
\iopds{} over an operator category \cite{BarwickOpCat}, and are also
studied by Lurie in \cite[\S 4.7.1]{HA} (though in a somewhat
\emph{ad hoc} manner). 

For the most part the theory of non-symmetric \iopds{} is completely
analogous to Lurie's theory of (symmetric) \iopds{} developed in
\cite{HA}, with the category $\bbGamma^{\op}$ of pointed finite sets
replaced by the category $\simp^{\op}$. In order to keep this
article to a reasonable length we only give references to the
corresponding results in \cite{HA} when the proofs are essentially the
same.

\subsection{Basic Definitions Revisited}
In this subsection we restate, in a slightly more technical form, the
basic definitions of (generalized) \nsiopds{} --- see
\S\ref{sec:fromto} for some motivation for these definitions. We begin
by describing a factorization system on the category $\simp^{\op}$.

\begin{defn}\label{defn:activeinert}
  Let $\simp$ be the usual simplicial indexing category. A morphism $f
  \colon [n] \to [m]$ in $\simp$ is \defterm{inert} if it is the
  inclusion of a sub-interval of $[m]$, i.e. $f(i) = f(0)+i$ for all
  $i$, and \defterm{active} if it preserves the extremal elements,
  i.e.  $f(0) = 0$ and $f(n) = m$. We say a morphism in $\simp^{\op}$
  is \emph{active} or \emph{inert} if it is so when considered as a
  morphism in $\simp$, and write $\simp^{\op}_{\txt{act}}$ and
  $\simp^{\op}_{\txt{int}}$ for the subcategories of $\simp^{\op}$
  with active and inert morphisms, respectively. We write $\rho_{i}
  \colon [n] \to [1]$ for the inert map in $\simp^{\op}$ corresponding
  to the inclusion $\{i-1,i\} \hookrightarrow [n]$.
\end{defn}

\begin{lemma}
  The active and inert morphisms form a factorization system on
  $\simp^{\op}$.
\end{lemma}
\begin{proof}
  This is a special case of \cite[Lemma 8.3]{BarwickOpCat}; it is
  also easy to check by hand.
\end{proof}

\begin{defn}\label{defn:nsiopd2}
  A \defterm{non-symmetric $\infty$-operad} is an inner fibration $\pi
  \colon \mathcal{O} \to \simp^{\op}$ such that:
  \begin{enumerate}[(i)]
  \item For each inert map  $\phi \colon [n] \to [m]$ in $\simp^{\op}$ and
    every $X \in \mathcal{O}$ such that $\pi(X) = [n]$, there exists
    a $\pi$-coCartesian edge $X \to \phi_{!}X$ over $\phi$.
  \item For every $[n]$ in $\simp^{\op}$, the functor
    \[ \mathcal{O}_{[n]} \to \prod_{i = 1}^{n}
    \mathcal{O}_{[1]} \] induced by the inert maps $\rho_{i}
    \colon [n] \to [1]$ in $\simp^{\op}$ is an equivalence.
  \item Given $C \in \mathcal{O}_{[n]}$ and a coCartesian map
    $C \to C_{i}$ over each inert map
    $\rho_{i} \colon [n] \to [1]$, the object $C$ is a $\pi$-limit of
    the $C_{i}$'s.
  \end{enumerate}
\end{defn}

\begin{remark}
  It is immediate from the definition of relative limits in
  \cite[\S 4.3.1]{HTT} that Definition~\ref{defn:nsiopd2} is
  equivalent to Definition~\ref{defn:nsiopd1}: Recall that a diagram
  $\bar{p} \colon K^{\triangleleft} \to \mathcal{O}$ is a $\pi$-limit
  \IFF{} the natural map \[\lambda \colon \mathcal{O}_{/\bar{p}} \to \mathcal{O}_{/p}
  \times_{\simp^{\op}_{/\pi p}} \simp^{\op}_{/\pi\bar{p}}\]
  is a categorical equivalence, where $p := \bar{p}|_{K}$. But the
  projections $\mathcal{O}_{/\bar{p}} \to \mathcal{O}$ and $\mathcal{O}_{/p}
  \times_{\simp^{\op}_{/\pi p}} \simp^{\op}_{/\pi\bar{p}} \to
  \mathcal{O}$ are both right fibrations, so the map $\lambda$ is an
  equivalence \IFF{} the induced map on fibres over any $o \in
  \mathcal{O}$ is an equivalence. Since $K^{\triangleleft}$ has an
  initial object, we may identify $\mathcal{O}_{/\bar{p}}$ with
  $\mathcal{O}_{/x}$ where $x = \bar{p}(-\infty)$ and
  $\simp^{\op}_{/\pi\bar{p}}$ with $\simp^{\op}_{/[n]}$ where $[n] =
  \pi(x)$. If $[m] = \pi(o)$ then the induced map on fibres is therefore
  \[ \Map_{\mathcal{O}}(o, x) \to \Map_{\simp^{\op}}([m],[n])
  \times_{\lim_{k \in K} \Map_{\simp^{\op}}([m], \pi p(k))} \lim_{k
    \in K} \Map_{\mathcal{O}}(o, p(k)).\]
  This is an equivalence \IFF{} the commutative square
  \nolabelcsquare{\Map_{\mathcal{O}}(o, x)}{\lim_{k
      \in K} \Map_{\mathcal{O}}(o,
    p(k))}{\Map_{\simp^{\op}}([m],[n])}{\lim_{k \in K}
    \Map_{\simp^{\op}}([m], \pi p(k))}
  is Cartesian, i.e. \IFF{} for every map $\phi \colon [m] \to [n]$
  the map on fibres over $\phi$
  \[ \Map_{\mathcal{O}}^{\phi}(o, x) \to \lim_{k \in K}
  \Map_{\mathcal{O}}^{\bar{p}(\psi_{k})\circ \phi}(o, p(k))\] is an
  equivalence, where $\psi_{k}$ is the unique map $-\infty \to k$ in
  $K^{\triangleleft}$. Applying this to the coCartesian projections
  $c \to c_{i}$ for some $c \in \mathcal{O}_{[n]}$, we get that $c$ is a
  $\pi$-limit of the $c_{i}$'s \IFF{} for every $o \in
  \mathcal{O}_{[m]}$ and every map $\phi \colon [m] \to [n]$ in
  $\simp^{\op}$, the map
  \[ \Map_{\mathcal{O}}^{\phi}(o, c) \to \prod_{i = 1}^{n}
  \Map_{\mathcal{O}}^{\rho_{i}\phi}(o, c_{i})\] is an equivalence,
  which was the condition used in
  Definition~\ref{defn:nsiopd1}. Similarly,
  Definition~\ref{defn:gnsiopd} below is equivalent to
  Definition~\ref{defn:gnsiopd1}.
\end{remark}

\begin{remark}
  We will see below in \S\ref{subsec:symnonsym} that there is a
  natural map $c \colon \simp^{\op} \to \bbGamma^{\op}$ such that if
  $\mathcal{O} \to \bbGamma^{\op}$ is a (generalized) symmetric
  \iopd{}, in the sense of \cite{HA}, then the pullback
  $c^{*}\mathcal{O} \to \simp^{\op}$ along $c$ is a (generalized)
  \nsiopd{}. Moreover, if $\mathcal{O}$ is a symmetric monoidal
  \icat{} then $c^{*}\mathcal{O}$ is a monoidal \icat{}. We will
  occasionally refer to the pullback $c^{*}\mathcal{O}$ also as
  $\mathcal{O}$. For example, if $\mathcal{C}$ is an \icat{} with
  finite products we will denote the monoidal \icat{} pulled back from
  the Cartesian symmetric monoidal structure $\mathcal{C}^{\times} \to
  \bbGamma^{\op}$ by $\mathcal{C}^{\times}$ too.
\end{remark}

A useful way of constructing \nsiopds{} is taking the nerve of the
category of operators associated to a simplicial multicategory:
\begin{defn}\label{defn:simplmulticat}
  A \defterm{simplicial multicategory} $\mathbf{O}$ consists of a set
  $\ob \mathbf{O}$ of objects and simplicial sets
  $\mathbf{O}(X_{1}, \ldots, X_{n};, Y)$ of multimorphisms for all
  $X_{1}, \ldots, X_{n}, Y \in \ob \mathbf{O}$, together with
  composition maps
  \[ \mathbf{O}(X_{1}^{1}, \ldots, X_{n_{1}}^{1}; Y_{1}) \times
  \cdots \times \mathbf{O}(X_{1}^{k}, \ldots, X_{n_{k}}^{k}; Y_{k})
  \times \mathbf{O}(Y_{1}, \ldots, Y_{k}; Z) \to
  \mathbf{O}(X_{1}^{1}, \ldots, X^{k}_{n_{k}}; Z),\] satisfying the
  usual associativity law for multicategories. A simplicial
  multicategory $\mathbf{O}$ is \defterm{fibrant} if all the
  simplicial sets $\mathbf{O}((X_{1}, \ldots, X_{n}), Y)$ are Kan
  complexes.
\end{defn}

\begin{defn}
  Let $\mathbf{O}$ be a simplicial multicategory. Define
  $\mathbf{O}^{\otimes}$ to be the simplicial
  category with objects finite lists $(X_{1}, \ldots, X_{n})$ ($n = 0,
  1,\ldots$) of objects of $\mathbf{O}$ and morphisms given by
  \[ \mathbf{O}^{\otimes}((X_{1}, \ldots, X_{n}), (Y_{1}, \ldots,
  Y_{m})) = \coprod_{\phi \colon [m] \to [n]} \prod_{i = 1}^{m}
  \mathbf{O}(X_{\phi(i-1)+1}, \ldots, X_{\phi(i)}; Y_{i}),\] with
  composition defined using composition in $\mathbf{O}$. The
  simplicial category $\mathbf{O}^{\otimes}$ has an obvious
  projection to $\simp^{\op}$.
\end{defn}

\begin{lemma}\label{lem:multicatopd}
  Suppose $\mathbf{O}$ is a fibrant simplicial multicategory. Then
  the projection $\mathrm{N}\mathbf{O}^{\otimes} \to \simp^{\op}$ is
  a \nsiopd{}.
\end{lemma}
\begin{proof}
  As \cite[Proposition 2.1.1.27]{HA}.
\end{proof}

\begin{remark}
  A non-symmetric variant (using planar trees) of the work of Cisinski
  and Moerdijk~\cite{CisinskiMoerdijkSimplOpd} should give a model
  category structure on simplicial multicategories whose fibrant
  objects are the fibrant simplicial multicategories. The resulting
  homotopy theory of simplicial multicategories is (partially) known
  to be equivalent to that of \iopds{}, at least in the symmetric
  case, but currently the only known relation is via the homotopy
  theory of dendroidal sets: Cisinski and
  Moerdijk~\cite{CisinskiMoerdijkSimplOpd} construct a Quillen
  equivalence between simplicial symmetric multicategories and
  dendroidal sets, and Heuts, Hinich, and
  Moerdijk~\cite{HeutsHinichMoerdijkDendrComp} construct a zig-zag of
  Quillen equivalences between dendroidal sets and symmetric \iopds{}
  (but unfortunately their comparison is currently restricted to the
  special case of \iopds{} without nullary operations). No doubt a
  version of dendroidal sets defined using planar trees would lead to
  a similar comparison between simplicial multicategories and
  \nsiopds{}.
\end{remark}


\begin{defn}
  A \defterm{monoidal \icat{}} is a non-symmetric \iopd{}
  $\mathcal{V}^{\otimes}\to \simp^{\op}$ that is also a
  coCartesian fibration.
\end{defn}

\begin{remark}
  We will see below in \S\ref{subsec:symnonsym} that this is
  equivalent to Lurie's definition of monoidal \icats{} in \cite{HA}.
\end{remark}

\begin{ex}
  Suppose $\mathcal{V}^{\otimes}$ is a monoidal \icat{}. Then $d_{1}
  \colon [2] \to [1]$ induces a functor $d_{1,!} \colon \mathcal{V}
  \times \mathcal{V} \simeq \mathcal{V}^{\otimes}_{[2]} \to
  \mathcal{V}$ --- a tensor product on $\mathcal{V}$. Similarly $s_{0}
  \colon [0] \to [1]$ gives a functor $s_{0,!} \colon * \simeq
  \mathcal{V}^{\otimes}_{[0]} \to \mathcal{V}$ which picks out a unit
  object $I_{\mathcal{V}} := s_{0,!}*$ in $\mathcal{V}$.
\end{ex}

\begin{defn}\label{defn:gnsiopd}
  A \defterm{generalized non-symmetric $\infty$-operad}
  is an inner fibration $\pi \colon \mathcal{M} \to
  \simp^{\op}$ such that:
  \begin{enumerate}[(i)]
  \item For each inert map  $\phi \colon [n] \to [m]$ in $\simp^{\op}$ and
    every $X \in \mathcal{M}$ such that $\pi(X) = [n]$, there exists
    a $\pi$-coCartesian edge $X \to \phi_{!}X$ over $\phi$.
  \item For every $[n]$ in $\simp^{\op}$, the map
    \[ \mathcal{M}_{[n]} \to \mathcal{M}_{[1]}
    \times_{\mathcal{M}_{[0]}} \cdots \times_{\mathcal{M}_{[0]}} \mathcal{M}_{[1]}
    \]
    induced by the inert maps $[n] \to [1],[0]$ is an equivalence.
  \item Given $C \in \mathcal{M}_{[n]}$ and a coCartesian map $C \to
    C_{\alpha}$ over each inert map $\alpha$ in
    $\mathcal{G}^{\simp}_{[n]/}$ (i.e. each inert map from $[n]$ to
    $[1]$ and $[0]$), the object $C$ is a $\pi$-limit of the
    $C_{\alpha}$'s.
  \end{enumerate}
\end{defn}


\begin{defn}\label{defn:doubleicat}
  A \defterm{double \icat{}} is a generalized non-symmetric \iopd{}
  that is also a coCartesian fibration.
\end{defn}

\begin{defn}
  Let $\pi \colon \mathcal{M} \to \simp^{\op}$ be a
  (generalized) non-symmetric \iopd{}. We say that
  a morphism $f$ in $\mathcal{M}$ is \defterm{inert} if it
  is coCartesian and $\pi(f)$ is an inert morphism in
  $\simp^{\op}$. We say that $f$ is \defterm{active} if
  $\pi(f)$ is an active morphism in $\simp^{\op}$.
\end{defn}

\begin{lemma}
  The active and inert morphisms form a factorization system on any
  \gnsiopd{}.
\end{lemma}
\begin{proof}
  This is a special case of \cite[Proposition 2.1.2.5]{HA}.
\end{proof}

\begin{defn}
  A morphism of (generalized) \nsiopds{} is a commutative diagram
  \opctriangle{\mathcal{M}}{\mathcal{N}}{\simp^{\op}}{\phi}{}{} such
  that $\phi$ carries inert morphisms in $\mathcal{M}$ to inert
  morphisms in $\mathcal{N}$. We will also refer to a morphism of
  (generalized) \nsiopds{} $\mathcal{M} \to \mathcal{N}$ as an
  \emph{$\mathcal{M}$-algebra} in $\mathcal{N}$; we write
  $\Alg_{\mathcal{M}}(\mathcal{N})$ for the full subcategory of the
  \icat{} $\Fun_{\simp^{\op}}(\mathcal{M},\mathcal{N})$ of functors
  over $\simp^{\op}$ spanned by the morphisms of (generalized)
  \nsiopds{}.
\end{defn}

\begin{propn}\label{propn:UnitAlg}
  Suppose $\mathcal{V}$ is a monoidal
  \icat{}. Then $\Alg_{\simp^{\op}}(\mathcal{V})$ has an
  initial object $I_{\mathcal{V}} \colon \simp^{\op} \to
  \mathcal{V}^{\otimes}$, which is the unique associative algebra
  structure on the unit object $I_{\mathcal{V}}$ of $\mathcal{V}$.
\end{propn}
\begin{proof}
  As \cite[Corollary 3.2.1.9]{HA}.
\end{proof}

\begin{defn}
  A map of (generalized) \nsiopds{} is a \defterm{fibration of
    (generalized) \nsiopds{}} if it is also a categorical fibration
  and a \defterm{coCartesian fibration of (generalized) \nsiopds{}} if
  it is also a coCartesian fibration.
\end{defn}

\begin{defn}
  We will also refer to a map of \nsiopds{} between
  monoidal \icats{} as a \defterm{lax monoidal functor}. A
  \defterm{monoidal functor} is a lax monoidal functor that
  preserves \emph{all} coCartesian arrows. If $\mathcal{V}$
  and $\mathcal{W}$ are monoidal \icats{}, we denote the full
  subcategory of $\Fun_{\simp^{\op}}(\mathcal{V}^{\otimes},
  \mathcal{W}^{\otimes})$ spanned by the monoidal functors by
  $\Fun^{\otimes}(\mathcal{V}^{\otimes}, \mathcal{W}^{\otimes})$. We
  also use the same notation for the analogous \icat{} of functors
  between double \icats{} that preserve all coCartesian morphisms.
\end{defn}

It will be useful to know that monoidal \icats{} are well-behaved with
respect to certain localizations:
\begin{defn}
  Let $\mathcal{V}$ be a monoidal \icat{} and suppose $\mathcal{W}$ is
  a full subcategory of $\mathcal{V}$ such that the inclusion $i
  \colon \mathcal{W} \hookrightarrow \mathcal{V}$ has a left adjoint
  $L \colon \mathcal{V} \to \mathcal{W}$. We say that the localization
  $L$ is \emph{monoidal} if the tensor product of two $L$-equivalences is
  again an $L$-equivalence.
\end{defn}

\begin{propn}\label{propn:moncomploc}
  Let $\mathcal{V}$ be a monoidal \icat{} and suppose $L
  \colon \mathcal{V} \to \mathcal{W}$ is a monoidal localization with
  fully faithful right adjoint $i \colon \mathcal{W} \hookrightarrow
  \mathcal{V}$. Write $\mathcal{W}^{\otimes}$ for the full subcategory
  of objects $X$ of $\mathcal{V}^{\otimes}$ such that $\rho_{i,!}X \in
  \mathcal{W}$ for $i = 1,\ldots,n$ (if $X \in \mathcal{V}^{\otimes}_{[n]}$). Then
  \begin{enumerate}[(i)]
  \item The inclusion $i^{\otimes} \colon \mathcal{W}^{\otimes} \hookrightarrow
    \mathcal{V}^{\otimes}$ has a left adjoint $L^{\otimes} \colon
    \mathcal{V}^{\otimes} \to \mathcal{W}^{\otimes}$ over $\simp^{\op}$.
  \item The projection $\mathcal{W}^{\otimes} \to \simp^{\op}$
    exhibits $\mathcal{W}^{\otimes}$ as a monoidal \icat{}.
  \item The inclusion $i^{\otimes}$ is a lax monoidal functor and 
    $L^{\otimes}$ is a monoidal functor.
  \end{enumerate}
\end{propn}
\begin{proof}
  As \cite[Proposition 2.2.1.9]{HA}.
\end{proof}

\begin{defn}
  Suppose $\mathcal{V}$ is a
  monoidal \icat{}. If $K$ is a simplicial set, we say that
  $\mathcal{V}$ is \defterm{compatible with $K$-indexed
    colimits} if
  \begin{enumerate}[(1)]
  \item the \icat{} $\mathcal{V}$ has $K$-indexed
    colimits (hence so does $\mathcal{V}^{\otimes}_{[n]} \simeq \prod
    \mathcal{V}$ and $\phi_{!}$ preserves them for any
    inert map $\phi$),
  \item for all (active) maps $\phi \colon [n] \to [m]$ in
    $\simp^{\op}$, the map
     \[ \phi_{!} \colon \prod \mathcal{V} \simeq
     \mathcal{V}^{\otimes}_{[n]} \to \mathcal{V}^{\otimes}_{[m]}\]
     preserves $K$-indexed colimits separately in each variable.
   \end{enumerate}
 \end{defn}

Recall that the \icat{} $\PresI$ of presentable \icats{} and
colimit-preserving functors has a symmetric monoidal structure,
constructed by Lurie in \cite[\S 4.8.1]{HA}. The tensor product has
the universal property that a colimit-preserving functor $\mathcal{C}
\otimes \mathcal{D} \to \mathcal{E}$ corresponds to a functor
$\mathcal{C} \times \mathcal{D} \to \mathcal{E}$ that preserves
colimits separately in each variable. The unit for the tensor product
is thus the \icat{} $\mathcal{S}$ of spaces.
\begin{defn}\label{defn:MonPr}
  Let $\MonPr$ be the \icat{} $\Alg_{\simp^{\op}}(\PresI)$ of
  associative algebra objects in $\PresI$ equipped with the tensor
  product of presentable \icats{}. Thus $\MonPr$ is the \icat{} of
  monoidal \icats{} $\mathcal{C}^{\otimes}$ compatible with small
  colimits such that $\mathcal{C}$ is presentable, with 1-morphisms
  monoidal functors that preserve colimits. We will refer to the
  objects of $\MonPr$ as \defterm{presentably monoidal \icats{}}.
\end{defn}
\begin{remark}\label{rmk:MonPrInitial}
  By Proposition~\ref{propn:UnitAlg} the \icat{} $\MonPr$ has an
  initial object given by the unique presentably monoidal structure on
  the unit $\mathcal{S}$, which is clearly the Cartesian monoidal structure.
\end{remark}

\subsection{The $\infty$-Category of $\infty$-Operads}\label{subsec:icatiopds}
Our goal in this subsection is to construct \icats{} and
$(\infty,2)$-categories of (generalized) \nsiopds{}. For this we make
use of Lurie's theory of \emph{categorical patterns} from \cite[\S
B]{HA}.

A number of important objects in higher category theory can be
regarded as forming (non-full) subcategories of slice categories of
the \icat{} $\CatI$ of \icats{} --- in particular, we have seen above
that this is the case for (non-symmetric) \iopds{} and monoidal
\icats{}, which form subcategories of $(\CatI)_{/\simp^{\op}}$. The
theory of categorical patterns provides a machine for generating model
structures describing \icats{} of this kind. Specifically, these are
model structures on the slice category of marked simplicial sets over
some fixed marked simplicial set --- the marking, which is a
collection of 1-simplices in a simplicial set, allows us to easily
consider subcategories of slice categories where some type of map must
be preserved (the \emph{inert} maps in the case of \iopds{}, and the
\emph{coCartesian} maps in the case of monoidal \icats{}). Although we
could construct the desired \icats{} of \iopds{} or monoidal \icats{}
directly as subcategories of $(\CatI)_{/\simp^{\op}}$, having the
model structure around makes it easy to see that these \icats{} have
all colimits, and indeed are presentable, and also allows us to
construct certain functors as Quillen adjunctions.



\begin{defn}
  A categorical pattern $\mathfrak{P} = (\mathcal{C}, S,
  \{p_{\alpha}\})$ consists of 
  \begin{itemize}
  \item an \icat{} $\mathcal{C}$,
  \item a marking of $\mathcal{C}$, i.e. a collection $S$ of
    1-simplices in $\mathcal{C}$ that includes all the degenerate
    ones,
  \item a collection of diagrams of \icats{} $p_{\alpha} \colon
    K_{\alpha}^{\triangleleft} \to \mathcal{C}$ such that $p_{\alpha}$
    takes every edge in $K_{\alpha}^{\triangleleft}$ to a marked edge
    of $\mathcal{C}$.
  \end{itemize}
\end{defn}

\begin{remark}
  Lurie's definition of a categorical pattern in \cite[\S B]{HA} is
  more general than this: in particular, he includes the data of a
  \emph{scaling} of the simplicial set $\mathcal{C}$, i.e. a
  collection $T$ of 2-simplices in $\mathcal{C}$ that includes all the
  degenerate ones. In all the examples we consider, however, the
  scaling consists of \emph{all} 2-simplices of the simplicial set
  $\mathcal{C}$. We restrict ourselves to this special case as it
  gives a clearer description of the $\mathfrak{P}$-fibrant objects,
  and also simplifies the notation.
\end{remark}

From a categorical pattern, Lurie constructs a model category that
encodes the \icat{} of $\mathfrak{P}$-fibrant objects, in the
following sense:
\begin{defn}
  Suppose $\mathfrak{P} = (\mathcal{C}, S, \{p_{\alpha}\})$ is a
  categorical pattern. A map of simplicial sets $X \to \mathcal{C}$ is
  \emph{$\mathfrak{P}$-fibrant} if the following criteria are satisfied:
  \begin{enumerate}[(1)]
  \item The underlying map $\pi \colon Y \to \mathcal{C}$ is an inner
    fibration. (In particular, $Y$ is an \icat{}.)
  \item $Y$ has all $\pi$-coCartesian edges over the morphisms in $S$.
  \item For every $\alpha$, the coCartesian fibration $\pi_{\alpha}
    \colon Y \times_{\mathcal{C}} K_{\alpha}^{\triangleleft} \to
    K_{\alpha}^{\triangleleft}$, obtained by pulling back $\pi$ along
    $p_{\alpha}$, is classified by a limit diagram
    $K_{\alpha}^{\triangleleft} \to \CatI$.
  \item For every $\alpha$, the composite of any coCartesian section
    $s \colon K_{\alpha}^{\triangleleft} \to Y \times_{\mathcal{C}}
    K_{\alpha}^{\triangleleft}$ of $\pi_{\alpha}$ with the projection
    $Y \times_{\mathcal{C}} K_{\alpha}^{\triangleleft} \to Y$ is a
    $\pi$-limit diagram.
  \end{enumerate}
\end{defn}

\begin{exs}\ 
  \begin{enumerate}[(i)]
  \item Let $\mathfrak{O}_{\txt{ns}}$ be the categorical
    pattern \[(\simp^{\op}, I_{\txt{ns}}, \{p_{[n]} \colon
    K_{[n]}^{\triangleleft} \to \simp^{\op}\}),\] where $I_{\txt{ns}}$
    is the set of inert morphisms and $K_{[n]}$ is the set of inert
    morphisms $[n] \to [1]$ in $\simp^{\op}$. It is immediate from
    Definition~\ref{defn:nsiopd2} that a map $Y \to \simp^{\op}$ is
    $\mathfrak{O}_{\txt{ns}}$-fibrant precisely if it is a \nsiopd{}.
  \item Let $\mathfrak{M}$ denote the categorical
    pattern \[(\simp^{\op}, \mathrm{N}\simp^{\op}_{1}, \{p_{[n]}
    \colon K_{[n]}^{\triangleleft} \to \simp^{\op}\}).\] Then a map $Y
    \to \simp^{\op}$ is $\mathfrak{M}$-fibrant precisely if $Y \to
    \simp^{\op}$ is a monoidal \icat{}.
  \item Let $\mathfrak{O}^{\txt{gen}}_{\txt{ns}}$ be the categorical
    pattern \[(\simp^{\op}, I_{\txt{ns}}, \{
    (\mathcal{G}^{\simp})_{[n]/}^{\triangleleft} \to \simp^{\op}\}).\]
    It is immediate from Definition~\ref{defn:gnsiopd} that a map $Y
    \to \simp^{\op}$ is $\mathfrak{O}^{\txt{gen}}_{\txt{ns}}$-fibrant \IFF{} $Y \to
    \simp^{\op}$ is a \gnsiopd{}.
  \item Let $\mathfrak{D}$ denote the categorical
    pattern \[(\simp^{\op}, \mathrm{N}\simp^{\op}_{1}, \{
    (\mathcal{G}^{\simp})_{[n]/}^{\triangleleft} \to \simp^{\op}\}).\]
    Then a map $Y \to \simp^{\op}$ is $\mathfrak{D}$-fibrant \IFF{} $Y
    \to \simp^{\op}$ is a double \icat{}.
  \end{enumerate}
\end{exs}

\begin{thm}[Lurie, {\cite[Theorem B.0.20]{HA}}]\label{thm:catpatternmodstr}
  Let $\mathfrak{P} = (\mathcal{C}, S, \{p_{\alpha}\})$ be a
  categorical pattern, and let $\overline{\mathcal{C}}$ denote the
  marked simplicial set $(\mathcal{C}, S)$. There is a left proper
  combinatorial simplicial model structure on the category
  $(\sSet^{+})_{/\overline{\mathcal{C}}}$ such that:
  \begin{enumerate}[(1)]
  \item The cofibrations are the morphisms whose underlying maps of
    simplicial sets are monomorphisms. In particular, all objects are
    cofibrant.
  \item An object $(X, T) \to \overline{\mathcal{C}}$ is fibrant
    \IFF{} $X \to \mathcal{C}$ is $\mathfrak{P}$-fibrant and $T$ is
    precisely the collection of coCartesian morphisms over the
    morphisms in $S$.
  \end{enumerate}
  We denote the category $(\sSet^{+})_{/\overline{\mathcal{C}}}$
  equipped with this model structure by $(\sSet^{+})_{\mathfrak{P}}$.
\end{thm}

Applying this in the case $\mathfrak{P} = \mathfrak{O}_{\txt{ns}}$, we
get:
\begin{cor}
  There is a left proper combinatorial simplicial model structure on
  $(\sSet^{+})_{/(\simp^{\op}, I_{\txt{ns}})}$ such that
  \begin{enumerate}[(1)]
  \item The cofibrations are the morphisms whose underlying maps of
    simplicial sets are monomorphisms. In particular, all objects are
    cofibrant.
  \item An object $(X, T) \to \simp^{\op}$ is fibrant
    \IFF{} $X \to \simp^{\op}$ is a \nsiopd{} and $T$ is precisely the
    collection of inert morphisms of $X$.
  \end{enumerate}
  We call this the \emph{\nsiopd{} model structure}.
\end{cor}

\begin{defn}
  The \icat{} $\OpdIns$ of \nsiopds{} is the \icat{} associated to the
  simplicial model category $(\sSet^{+})_{\mathfrak{O}_{\txt{ns}}}$,
  i.e. the coherent nerve of the simplicial category of fibrant
  objects. Thus the objects of $\OpdIns$ can be identified with
  \nsiopds{}. Moreover, since the maps between these in
  $(\sSet^{+})_{\mathfrak{O}_{\txt{ns}}}$ are precisely the maps that
  preserve inert morphisms, it is also easy to see that the space of
  maps from $\mathcal{O}$ to $\mathcal{P}$ in $\OpdIns$ is equivalent
  the subspace of $\Map_{\simp^{\op}}(\mathcal{O}, \mathcal{P})$ given
  by the components corresponding to inert-morphism-preserving maps,
  as expected. This justifies calling $\OpdIns$ the \emph{\icat{} of
    \nsiopds{}}.
\end{defn}

\begin{remark}
  This \icat{} of \nsiopds{} is a special case of the \icats{} of
  \iopds{} over an operator category constructed by Barwick in
  \cite[Theorem 8.15]{BarwickOpCat}. By
  \cite[Proposition 8.17]{BarwickOpCat} a morphism $\mathcal{O} \to
  \mathcal{P}$ in $(\sSet^{+})_{\mathfrak{O}_{\txt{ns}}}$ between
  \nsiopds{} marked by their inert morphisms is a weak equivalence
  \IFF{} the underlying morphism $\mathcal{O} \to \mathcal{P}$ is an
  equivalence of \icats{}, as we would expect.
\end{remark}

\begin{defn}
  Similarly, applying Theorem~\ref{thm:catpatternmodstr} to the
  categorical patterns $\mathfrak{M}$,
  $\mathfrak{O}_{\txt{ns}}^{\txt{gen}}$, and $\mathfrak{D}$ gives
  simplicial model categories $(\sSet^{+})_{\mathfrak{M}}$,
  $(\sSet^{+})_{\mathfrak{O}^{\txt{gen}}_{\txt{ns}}}$, and
  $(\sSet^{+})_{\mathfrak{D}}$ whose fibrant objects are,
  respectively, monoidal \icats{}, \gnsiopds{}, and double
  \icats{}. We write $\MonI$, $\OpdInsg$, and $\txt{Dbl}_{\infty}$ for the \icats{} associated to these
  simplicial model categories, and refer to them as the \emph{\icats{} of
    monoidal \icats{}, \gnsiopds{}, and double \icats{}}.
\end{defn}

\begin{defn}
  The morphisms in $\MonI$ are the (strong) monoidal functors between
  monoidal \icats{}. We write $\MonI^{\txt{lax}}$ for the \icat{} of
  monoidal \icats{} and lax monoidal functors, i.e. the full
  subcategory of $\OpdIns$ spanned by the monoidal \icats{}.
\end{defn}

\begin{exs}
  Several other \icats{} we will encounter can be constructed using
  model categories coming from categorical patterns:
  \begin{itemize}
  \item If $\mathcal{C}$ is an \icat{}, let
    $\mathfrak{P}_{\mathcal{C}}^{\txt{coCart}}$ be the categorical
    pattern $(\mathcal{C}, \mathcal{C}_{1}, \emptyset)$. Then
    $(\mathcal{E}, T) \to \mathcal{C}^{\sharp}$ is
    $\mathfrak{P}_{\mathcal{C}}^{\txt{coCart}}$-fibrant \IFF{} $\pi
    \colon \mathcal{E} \to \mathcal{C}$ is a coCartesian fibration,
    and $T$ is the set of $\pi$-coCartesian edges in
    $\mathcal{E}$. The model category
    $(\sSet^{+})_{\mathfrak{P}_{\mathcal{C}}^{\txt{coCart}}}$ is the
    coCartesian model structure on
    $(\sSet^{+})_{/\mathcal{C}^{\sharp}}$. Thus the associated \icat{}
    is the \icat{} $\txt{CoCart}(\mathcal{C})$ of coCartesian
    fibrations over $\mathcal{C}$, which is equivalent to
    $\Fun(\mathcal{C}, \CatI)$.
  \item If $\mathcal{C}$ is an \icat{}, let
    $\mathfrak{P}_{\mathcal{C}}^{\txt{eq}}$ be the categorical pattern
    $(\mathcal{C}, \iota\mathcal{C}_{1}, \emptyset)$. Then
    $(\mathcal{E}, T) \to \mathcal{C}^{\natural}$ is
    $\mathfrak{P}_{\mathcal{C}}^{\txt{eq}}$-fibrant \IFF{}
    $\mathcal{E}$ is an \icat{}, the map $\pi \colon \mathcal{E} \to
    \mathcal{C}$ is a categorical fibration, and $T$ is the set of
    equivalences in $\mathcal{E}$. (This follows from the description
    of categorical fibrations to \icats{} in \cite[Corollary
      2.4.6.5]{HTT}.) The model category
    $(\sSet^{+})_{\mathfrak{P}_{\mathcal{C}}^{\txt{eq}}}$ is the
    over-category model structure on
    $(\sSet^{+})_{/\mathcal{C}^{\natural}}$ from the model structure
    on $\sSet^{+}$. The associated \icat{} is thus the over-category
    $(\CatI)_{/\mathcal{C}}$.
  \item If $\mathcal{C}$ is an \icat{} and $\mathcal{D}$ is a
    subcategory of $\mathcal{C}$, let
    $\mathfrak{P}_{\mathcal{C},\mathcal{D}}^{\txt{coCart}}$ be the
    categorical pattern $(\mathcal{C}, \mathcal{D}_{1},
    \emptyset)$. Then $(\mathcal{E}, T) \to (\mathcal{C},
    \mathcal{D}_{1})$ is
    $\mathfrak{P}_{\mathcal{C},\mathcal{D}}^{\txt{coCart}}$-fibrant
    \IFF{} $\mathcal{E}$ is an \icat{}, the map $\pi \colon
    \mathcal{E} \to \mathcal{C}$ is an inner fibration, $\mathcal{E}$
    has all $\pi$-coCartesian edges over morphisms in $\mathcal{D}$,
    and $T$ consists precisely of these coCartesian edges. The model
    category
    $(\sSet^{+})_{\mathfrak{P}_{\mathcal{C},\mathcal{D}}^{\txt{coCart}}}$
    gives an \icat{} of functors $\mathcal{E} \to \mathcal{C}$ that
    have coCartesian morphisms over the morphisms in $\mathcal{D}$; we
    write $\txt{CoCart}(\mathcal{C},\mathcal{D})$ for this \icat{}.
\end{itemize}
\end{exs}


\begin{remark}
  For any categorical pattern $\mathfrak{P}$, the model category
  $(\sSet^{+})_{\mathfrak{P}}$ is enriched in the model category of
  marked simplicial sets --- this follows from \cite[Remark
    B.2.5]{HA} (taking $\mathfrak{P}'$ to be the trivial categorical
  pattern on $\Delta^{0}$). Passing to the subcategories of fibrant
  objects we therefore get fibrant marked simplicial categories of
  (generalized) \nsiopds{}. Marked simplicial categories are one model
  for the theory of $(\infty,2)$-categories, so we get
  $(\infty,2)$-categories $\OPDIns$ and $\OPDInsg$ with underlying
  \icats{} $\OpdIns$ and $\OpdInsg$. If $\mathcal{M}$ and
  $\mathcal{N}$ are (generalized) \nsiopds{}, we can identify the
  \icat{} $\Alg_{\mathcal{M}}(\mathcal{N})$ with the \icat{} of maps
  from $\mathcal{M}$ to $\mathcal{N}$ in the fibrant marked simplicial
  category $\OPDInsg$.
\end{remark}

\begin{propn}
  The identity is a left (marked simplicially enriched) Quillen
  functor $(\sSet^{+})_{\mathfrak{O}^{\txt{gen}}_{\txt{ns}}} \to
  (\sSet^{+})_{\mathfrak{O}_{\txt{ns}}}$.
\end{propn}
\begin{proof}
  As \cite[Corollary 2.3.2.6]{HA}.
\end{proof}

\begin{cor}\label{cor:GenOpdLoc}
  The inclusion $\OpdIns \to \OpdInsg$ has a left adjoint
  $L_{\txt{gen}} \colon \OpdInsg \to \OpdIns$.
\end{cor}

\subsection{Filtered Colimits of $\infty$-Operads}
Colimits of (generalized) \nsiopds{} are in general difficult to
describe explicitly. However, we will now show that filtered colimits
can be computed in $\CatI$:

\begin{thm}\label{thm:OpdFiltColim}
  The forgetful functors $\OpdIns$, $\OpdInsg \to \CatI$ detect
  filtered colimits.
\end{thm}

For this we need some preliminary technical results:
\begin{propn}\label{propn:FltColimCoCart}
  Let $p \colon \mathcal{I} \to (\CatI)_{/\mathcal{B}}$ be a filtered
  diagram, and let $f \colon B \to B'$ be a morphism in $\mathcal{B}$
  such that for each $\alpha \in \mathcal{I}$ the functor 
  $p(\alpha) \colon \mathcal{C}_{\alpha} \to \mathcal{B}$ has
  $p(\alpha)$-coCartesian morphisms $C \to f_{!}C$ over $f$ for each
  $C \in (\mathcal{C}_{\alpha})_{B}$, and the functors $p(\phi)$
  preserve these for all morphisms $\phi \colon \alpha \to \beta$ in
  $\mathcal{I}$. Then:
  \begin{enumerate}[(i)]
  \item The colimit $\mathcal{C} \to \mathcal{B}$ of $p$
    also has coCartesian morphisms over $f$.
  \item The functors $\mathcal{C}_{\alpha} \to \mathcal{C}$ preserve
    these coCartesian morphisms for all $\alpha \in \mathcal{I}$.
  \item A functor $\mathcal{C} \to
    \mathcal{D}$ over $\mathcal{B}$ preserves coCartesian morphisms over
    $f$ \IFF{} all the composites $\mathcal{C}_{\alpha} \to \mathcal{C}
    \to \mathcal{D}$ do so.
  \end{enumerate}
\end{propn}
\begin{proof}
  For $\alpha \in \mathcal{I}$, let $r_{\alpha,!}$ denote the
  canonical functor $\mathcal{C}_{\alpha} \to \mathcal{C}$. Suppose $X
  \in \mathcal{C}_{B}$; then there exists $\alpha \in \mathcal{I}$ and
  $X' \in (\mathcal{C}_{\alpha})_{B}$ such that $X \simeq
  r_{\alpha,!}X'$. Let $\bar{f} \colon X' \to f_{!}X'$ be a
  coCartesian morphism over $f$; we wish to prove that
  $r_{\alpha,!}\bar{f}$ is coCartesian in $\mathcal{C}$. To see this
  we must show that for all $Y \in \mathcal{C}_{A}$ the commutative
  square \nolabelcsquare{\Map_{\mathcal{C}}(r_{\alpha,!}f_{!}X',
    Y)}{\Map_{\mathcal{C}}(X, Y)}{\Map_{\mathcal{B}}(B',
    A)}{\Map_{\mathcal{B}}(B, A)} is a pullback diagram.  Changing
  $\alpha$ if necessary, we may without loss of generality assume
  there is a $Y' \in \mathcal{C}_{\alpha}$ such that $r_{\alpha,!}Y'
  \simeq Y$. Since filtered colimits commute with finite limits in
  spaces, and the mapping space $\Map_{\mathcal{C}}(X, Y)$ is the
  fibre of the projection \[\Fun(\Delta^{1}, \mathcal{C}) \simeq
  \colim_{\alpha} \Fun(\Delta^{1}, \mathcal{C}_{\alpha}) \to
  \colim_{\alpha} \mathcal{C}_{\alpha} \times \mathcal{C}_{\alpha}
  \simeq \mathcal{C} \times \mathcal{C}\] at $(X, Y)$, it is easy to
  see that we can describe $\Map_{\mathcal{C}}(X,Y)$ as the filtered
  colimit
  \[\colim_{\phi \colon \alpha \to \beta \in
    \mathcal{I}_{\alpha/}}\Map_{\mathcal{C}_{\beta}}(\phi_{!}X',
  \phi_{!}Y'),
  \]
  and the commutative square as the colimit square
  \nolabelcsquare{\displaystyle{\colim_{\phi \colon \alpha \to \beta
        \in
        \mathcal{I}_{\alpha/}}}\Map_{\mathcal{C}_{\beta}}(\phi_{!}f_{!}X',
    \phi_{!}Y')}{\displaystyle{\colim_{\phi \colon \alpha \to \beta
        \in
        \mathcal{I}_{\alpha/}}}\Map_{\mathcal{C}_{\beta}}(\phi_{!}X',
    \phi_{!}Y')}{\displaystyle{\colim_{\phi \colon \alpha \to \beta
        \in \mathcal{I}_{\alpha/}}}\Map_{\mathcal{B}}(B',
    A)}{\displaystyle{\colim_{\phi \colon \alpha \to \beta \in
        \mathcal{I}_{\alpha/}}}\Map_{\mathcal{B}}(B, A).} Each of the
  squares in this colimit are pullback squares since by assumption
  $\phi_{!}\bar{f}$ is coCartesian in $\mathcal{C}_{\beta}$ for all
  $\phi \colon \alpha \to \beta$. Hence, since filtered colimits in
  $\mathcal{S}$ commute with finite limits, it follows that the
  colimit square is also a pullback. Thus $r_{\alpha,!}\bar{f}$ is
  coCartesian in $\mathcal{C}$, as required. This proves claims (i) and (ii), and (iii) is then clear from this description of the coCartesian
  morphisms in $\mathcal{C}$.
\end{proof}

\begin{cor}\label{cor:CoCartForgerPrFiltColim}
  The forgetful functor $\CoCart(\mathcal{C}) \to
  (\CatI)_{/\mathcal{C}}$ detects filtered colimits.
\end{cor}
\begin{proof}
  We can describe $\CoCart(\mathcal{D})$ as the subcategory of
  $(\CatI)_{/\mathcal{D}}$ whose objects are the coCartesian
  fibrations and whose morphisms are the functors that preserve
  coCartesian morphisms. This is clear if we consider the functor of
  fibrant simplicial categories induced by the functor from the
  coCartesian model structure on $(\sSet^{+})_{/\mathcal{C}}$ to the
  over-category model structure on
  $(\sSet^{+})_{/\mathcal{C}^{\natural}}$ that forgets the markings
  that do not map to equivalences in $\mathcal{C}$. The result then
  follows from   Proposition~\ref{propn:FltColimCoCart}.
\end{proof}

\begin{cor}\label{cor:ForgetCoCartSubcatFltColim}
  Let $\mathcal{C}$ be an \icat{} and $\mathcal{D}$ a subcategory of
  $\mathcal{C}$. The forgetful functor $\CoCart(\mathcal{C},
  \mathcal{D}) \to (\CatI)_{/\mathcal{C}}$ detects filtered colimits
\end{cor}
\begin{proof}
  The \icat{} $\CoCart(\mathcal{C}, \mathcal{D})$ can be identified
  with the full subcategory of the pullback $\CoCart(\mathcal{D})
  \times_{(\CatI)_{/\mathcal{D}}}(\CatI)_{/\mathcal{C}}$ spanned by
  those maps $\mathcal{E} \to \mathcal{C}$ that have coCartesian
  arrows over the morphisms in $\mathcal{D}$ --- this is clear from
  the definition of the mapping spaces in the fibrant simplicial
  categories associated to the corresponding model categories.  The
  result therefore follows from
  Proposition~\ref{propn:FltColimCoCart}.
\end{proof}

\begin{lemma}\label{lem:ladjcompact}
  Suppose $F \colon \mathcal{C} \rightleftarrows \mathcal{D} :\! U$ is
  an adjunction. Then:
  \begin{enumerate}[(i)]
  \item If the right adjoint $U$ preserves $\kappa$-filtered colimits,
    then $F$ preserves $\kappa$-compact objects.
  \item If in addition $\mathcal{C}$ is $\kappa$-accessible, then $U$
    preserves $\kappa$-filtered colimits \IFF{} $F$ preserves
    $\kappa$-compact objects.
  \end{enumerate}
\end{lemma}
\begin{proof}\ 
 For the first claim, suppose $X \in \mathcal{C}$ is a $\kappa$-compact object and
    $p \colon K \to \mathcal{D}$ is a $\kappa$-filtered diagram. Then
    we have \[
  \begin{split}
    \Map_{\mathcal{D}}(F(X), \colim p) &  \simeq \Map_{\mathcal{C}}(X, G(\colim p)) \simeq
  \Map_{\mathcal{C}}(X, \colim G\circ p) \\
  & \simeq \colim \Map_{\mathcal{C}}(X, G \circ p) \simeq \colim
  \Map_{\mathcal{D}}(F(X), p).
  \end{split}
\]
Thus $\Map_{\mathcal{D}}(F(X), \blank)$ preserves $\kappa$-filtered
colimits, i.e. $F(X)$ is $\kappa$-compact.
For the second claim, suppose $F$ preserves $\kappa$-compact objects, and $p \colon K
  \to \mathcal{D}$ is a $\kappa$-filtered diagram; we wish to prove
  that the natural map $\colim G \circ p \to G(\colim p)$ is an
  equivalence. Since $\mathcal{C}$ is $\kappa$-accessible, to prove
  this it suffices to show that the induced map
  \[ \Map_{\mathcal{C}}(X, \colim G \circ p) \to \Map_{\mathcal{C}}(X,
  G(\colim p))\]
  is an equivalence for all $\kappa$-compact objects $X \in
  \mathcal{C}$. But when $X$ is $\kappa$-compact, we have equivalences
  \[ 
  \begin{split}
    \Map_{\mathcal{C}}(X, G(\colim p)) & \simeq
    \Map_{\mathcal{D}}(F(X), \colim p) \simeq
    \colim \Map_{\mathcal{D}}(F(X), p) \\
  & \simeq \colim \Map_{\mathcal{C}}(X, G \circ p) \simeq \Map_{\mathcal{C}}(X, \colim G\circ p),
  \end{split}
  \]
  so this is true.\qedhere
\end{proof}

\begin{lemma}\label{lem:overreflectcolim}
  Let $\mathcal{C}$ be an \icat{} and let $C$ be an object of
  $\mathcal{C}$. Then the forgetful functor $F \colon \mathcal{C}_{/C}
  \to \mathcal{C}$ reflects colimits, i.e. a diagram $\bar{p} \colon
  K^{\triangleright} \to \mathcal{C}_{/C}$ is a colimit diagram if the
  composite $F \circ \bar{p} \colon K^{\triangleright} \to
  \mathcal{C}$ is a colimit diagram. Moreover, if $\mathcal{C}$ has
  finite products, then $F$ creates colimits, i.e. $\bar{p}$ is a
  colimit diagram \IFF{} $F \circ \bar{p}$ is a colimit diagram.
\end{lemma}
\begin{proof}
  Write $C'$ for $\bar{p}(\infty)$ and $p$ for $\bar{p}|_{K}$. For any
  map $f \colon D \to C$ we have a commutative square 
  \nolabelcsquare{\displaystyle{\lim_{x \in K} \Map_{\mathcal{C}}(p(x),
    D)}}{\Map_{\mathcal{C}}(C', D)}{\displaystyle{\lim_{x \in K}
    \Map_{\mathcal{C}}(p(x), C)}}{\Map_{\mathcal{C}}(C', C).}
  If $F \circ \bar{p}$ is a colimit diagram in $\mathcal{C}$ then the
  horizontal morphisms in this square are both equivalences, hence so
  are all induced maps on fibres. But for any object
  $g \colon X \to C$ in $\mathcal{C}_{/C}$ the space
  $\Map_{\mathcal{C}_{/C}}(X, D)$ is the pullback
  \nolabelcsquare{\Map_{\mathcal{C}_{/C}}(X,
    D)}{\Map_{\mathcal{C}}(X, D)}{\{g\}}{\Map_{\mathcal{C}}(X, C),}
  and so since limits commute one map on fibres is
  \[ \lim_{x \in K} \Map_{\mathcal{C}_{/C}}(p(x), D) \to
  \Map_{\mathcal{C}_{/C}}(C', D).\]
  Thus this is an equivalence for all $D \to C$ if $F \circ \bar{p}$
  is a colimit diagram in $\mathcal{C}$, which shows that $\bar{p}$ is a colimit
  diagram in $\mathcal{C}_{/C}$  if $F \circ \bar{p}$ is a colimit
  diagram.

  Conversely, suppose $\bar{p}$ is a colimit diagram, so that 
  \[ \lim_{x \in K} \Map_{\mathcal{C}_{/C}}(p(x), D) \to
  \Map_{\mathcal{C}_{/C}}(C', D).\]
  is an equivalence for all $D \to C$. If $\mathcal{C}$ has finite
  products, then for any $Y \to C$ in
  $\mathcal{C}_{/C}$ and any $X \in \mathcal{C}$ 
  we have a natural equivalence
  \[ \Map_{\mathcal{C}_{/C}}(Y, X \times C) \simeq
  \Map_{\mathcal{C}}(Y, X) \]
  where $X \times C \to C$ is the product projection. Thus, taking $D$
  to be $X \times C$  we get by naturality an equivalence
  \[ \lim_{x \in K} \Map_{\mathcal{C}}(p(x), X) \isoto
  \Map_{\mathcal{C}}(C', X),\]
  and thus $F \circ \bar{p}$ is a colimit diagram in $\mathcal{C}$.
\end{proof}

\begin{propn}\label{propn:AccOverCompact}
  Suppose $\mathcal{C}$ is a $\kappa$-accessible \icat{} with
  finite products such that the Cartesian product
  preserves $\kappa$-filtered colimits separately in each
  variable. Then an object $X \to C$ is $\kappa$-compact in
  $\mathcal{C}_{/C}$ \IFF{} $X$ is a $\kappa$-compact object of
  $\mathcal{C}$.
\end{propn}
\begin{proof}
  The forgetful functor $r_{!} \colon \mathcal{C}_{/C} \to
  \mathcal{C}$ creates colimits by Lemma~\ref{lem:overreflectcolim}
  and admits a right adjoint $r^{*} \colon \mathcal{C} \to
  \mathcal{C}_{/C}$ given by sending $X \in \mathcal{C}$ to the
  projection $X \times C \to C$. By assumption the composite
  $r_{!}r^{*}$, which sends $X$ to $X \times C$, preserves
  $\kappa$-filtered colimits, hence so does $r^{*}$. By
  Lemma~\ref{lem:ladjcompact} the left adjoint $r_{!}$ preserves
  $\kappa$-compact objects. Thus if $X \to C$ is $\kappa$-compact in
  $\mathcal{C}_{/C}$, then $X$ is $\kappa$-compact in $\mathcal{C}$.

  Conversely, suppose $X \to C$ is an object of $\mathcal{C}_{/C}$
  such that $X$ is $\kappa$-compact in $\mathcal{C}$, and $p \colon K
  \to \mathcal{C}_{/C}$ is a $\kappa$-filtered diagram in
  $\mathcal{C}_{/C}$. We then have a diagram \nolabelcsquare{\colim
    \Map_{\mathcal{C}}(X, r_{!}\circ p)}{\Map_{\mathcal{C}}(X, \colim
    r_{!}\circ p)}{\colim \Map_{\mathcal{C}}(X,
    C)}{\Map_{\mathcal{C}}(X, C)} where the horizontal maps are
  equivalences. Since $\kappa$-filtered colimits commute with
  $\kappa$-small limits in $\mathcal{S}$, hence in particular finite
  limits, we have a pullback diagram \nolabelcsquare{\colim
    \Map_{\mathcal{C}_{/C}}(X, p)}{\colim \Map_{\mathcal{C}}(X,
    p)}{\colim *}{\colim \Map_{\mathcal{C}}(X, C)} where the obvious
  map $\colim * \to *$ is an equivalence. Thus the canonical map
  \[\colim \Map_{\mathcal{C}_{/C}}(X, p) \to \Map_{\mathcal{C}_{/C}}(X,
  \colim p)\] can be identified with the pullback along the inclusion
  $\{X \to C\} \to \Map_{\mathcal{C}}(X, C)$ of an equivalence and so
  is itself an equivalence. Hence $X \to C$ is indeed $\kappa$-compact
  in $\mathcal{C}_{/C}$.
\end{proof}

\begin{cor}\label{cor:PullbackFltColim}
  Suppose $\mathcal{C}$ is a $\kappa$-accessible \icat{} with finite
  limits, such that the Cartesian product preserves
  $\kappa$-filtered colimits separately in each variable. Then for
  every morphism $f \colon C \to D$ in $\mathcal{C}$ the pullback
  functor $f^{*} \colon \mathcal{C}_{/D} \to \mathcal{C}_{/C}$
  preserves $\kappa$-filtered colimits.
\end{cor}
\begin{proof}
  The functor $f^{*}$ is right adjoint to the functor $f_{!} \colon
  \mathcal{C}_{/C} \to \mathcal{C}_{/D}$ given by composition with
  $f$.  By Proposition~\ref{propn:AccOverCompact} the functor $f_{!}$
  preserves $\kappa$-compact objects, and so by
  Lemma~\ref{lem:ladjcompact} the right adjoint $f^{*}$ preserves
  $\kappa$-filtered colimits.
\end{proof}

\begin{proof}[Proof of Theorem~\ref{thm:OpdFiltColim}]
  We consider first the case of the forgetful functor $\OpdIns \to
  \CatI$. For any categorical pattern $\mathfrak{P} = (X, S,
  \{p_{\alpha}\})$, it follows from the proof of \cite[Theorem
    B.0.20]{HA} that the model category $(\sSet^{+})_{\mathfrak{P}}$ is a
  left Bousfield localization of the model category
  $(\sSet^{+})_{\mathfrak{P}^{-}}$, where $\mathfrak{P}^{-}$ be the
  categorical pattern $(X,S,\emptyset)$. Thus the \icat{} $\OpdIns$ is
  a localization of $\CoCart(\simp^{\op},\simp^{\op}_{\txt{int}})$,
  and by Corollary~\ref{cor:ForgetCoCartSubcatFltColim} the forgetful
  functor $\CoCart(\simp^{\op},\simp^{\op}_{\txt{int}}) \to
  (\CatI)_{/\simp^{\op}}$ detects filtered colimits. It follows that
  the colimit of a filtered diagram of \iopds{} is the localization of
  the colimit of the corresponding diagram in
  $\CoCart(\simp^{\op},\simp^{\op}_{\txt{int}})$, and this colimit can
  be computed in $(\CatI)_{/\simp^{\op}}$ or equivalently in $\CatI$, by
  Lemma~\ref{lem:overreflectcolim}. Thus, to show that the forgetful functor
  from $\OpdIns$ to $\CatI$ preserves filtered colimits it suffices to
  show that the colimit in $(\CatI)_{/\simp^{\op}}$ of such a diagram
  is also an \iopd{}.
  
  Let $p \colon \mathcal{I} \to \OpdIns, \alpha \mapsto
  \mathcal{O}_{\alpha}$ be a filtered diagram, and let
  $\mathcal{O}$ be the colimit in $\CatI$ of the diagram
  obtained by composing with the forgetful functor. By
  Proposition~\ref{propn:FltColimCoCart} the induced map $\mathcal{O} \to \simp^{\op}$ has
  coCartesian arrows over inert morphisms in $\simp^{\op}$, so it suffices
  to prove that the two other conditions for being an \iopd{} are
  satisfied.

  Since pullbacks in $\CatI$ preserve filtered colimits by
  Corollary~\ref{cor:PullbackFltColim}, and these commute with finite
  limits in $\CatI$, we have a commutative diagram
  \nolabelcsquare{\mathcal{O}_{[n]}}{\colim_{\alpha}
    \mathcal{O}_{\alpha,
      [n]}}{(\mathcal{O}_{[1]})^{\times n}}{\colim_{\alpha}
    (\mathcal{O}_{\alpha,[1]})^{\times n}} where all but the
  left vertical map are known to be equivalences, hence this is also
  an equivalence.

  Now suppose $Y$ is an object of $\mathcal{O}_{[n]}$ and
  $\eta_{i} \colon Y \to Y_{i}$ are coCartesian arrows over the inert
  maps $\rho_{i} \colon [n] \to [1]$ in $\simp^{\op}$. We must show
  that for every $X \in \mathcal{O}_{[m]}$ and every map
  $\phi \colon [m] \to [n]$ in $\simp^{\op}$, the morphism
  \[ \Map^{\phi}_{\mathcal{O}}(X, Y) \to \prod_{i}
  \Map^{\rho_{i}\phi}_{\mathcal{O}}(X, Y_{i})\]
  is an equivalence. We can choose $\alpha \in
  \mathcal{I}$ and objects $X_{\alpha}$ and $Y_{\alpha}$ in
  $\mathcal{O}_{\alpha}$ that map to $X$ and $Y$;
  coCartesian morphisms $Y_{\alpha} \to \rho_{i,!}Y_{\alpha}$ over $\rho_{i}$
  will then map to $\eta_{i}$. 
  As in the proof of Proposition~\ref{propn:FltColimCoCart}, since
  $\mathcal{O}$ is a filtered colimit in $\CatI$ we get a diagram
  \nolabelcsquare{\Map^{\phi}_{\mathcal{O}}(X, Y)}{\displaystyle{\prod_{i}}\,
    \Map^{\rho_{i}\phi}_{\mathcal{O}}(X, Y_{i})}{\displaystyle{\colim_{\psi \colon \alpha \to \beta \in \mathcal{I}_{\alpha/}}}
  \Map^{\phi}_{\mathcal{O}_{\beta}}(\psi_{!}X_{\alpha},
  \psi_{!}Y_{\alpha})}{\displaystyle{\prod_{i}}\,\displaystyle{\colim_{\psi \colon \alpha \to \beta \in \mathcal{I}_{\alpha/}}}
  \Map^{\rho_{i}\phi}_{\mathcal{O}_{\beta}}(\psi_{!}X_{\alpha},
  \psi_{!}\rho_{i,!}Y_{\alpha})}
  where the vertical maps are equivalences. But since filtered
  colimits commute with finite limits in $\mathcal{S}$, the bottom
  horizontal map is also an equivalence, as
  $\mathcal{O}_{\beta}$ is an \iopd{} for all $\beta$. It
  follows that the top horizontal map is also an equivalence, which
  completes the proof that $\mathcal{O}$ is an \iopd{}.
  
  The proof for $\OpdInsg$ is similar --- the only difference is the
  we replace the finite products with limits over the categories
  $\mathcal{G}^{\simp}_{[n]/}$, which are also finite.
\end{proof}

\subsection{Trivial $\infty$-Operads}
In this subsection we will associate to any \nsiopd{}
$\mathcal{O}$ a \emph{trivial} \iopd{}
$\mathcal{O}_{\triv}$ with a map
$\mathcal{O}_{\triv} \to \mathcal{O}$, such that
for any \iopd{} $\mathcal{P}$ the \icat{}
$\Alg_{\mathcal{O}_{\triv}}(\mathcal{P})$ of
$\mathcal{O}_{\triv}$-algebras in $\mathcal{P}$ is
equivalent to the functor \icat{} $\Fun(\mathcal{O}_{[1]}, \mathcal{P}_{[1]})$; an
analogous result also holds for generalized \nsiopds{}.

\begin{defn}
  Let $\mathcal{M}$ be a \gnsiopd{}. Define the \gnsiopd{}
  $\mathcal{M}_{\triv}$ by the pullback diagram
  \csquare{\mathcal{M}_{\triv}}{\mathcal{M}}{\simp^{\op}_{\txt{int}}}{\simp^{\op}}{\tau_{\mathcal{M}}}{}{}{}
  This is the \defterm{trivial \gnsiopd{} over $\mathcal{M}$}.
\end{defn}

\begin{defn}
  Let $\mathfrak{O}_{\txt{ns}}^{\triv}$ denote the categorical
  pattern \[(\simp^{\op}_{\txt{int}},
  \mathrm{N}(\simp^{\op}_{\txt{int}})_{1}, \{
  (\mathcal{G}^{\simp})_{[n]/}^{\triangleleft} \to
  \simp^{\op}\}).\]
\end{defn}
\begin{remark}
  An object $(X, S)$ of $(\sSet^{+})_{/(\simp^{\op}_{\txt{int}},
    \mathrm{N}(\simp^{\op}_{\txt{int}})_{1})}$ is 
  $\mathfrak{O}_{\txt{ns}}^{\txt{triv}}$-fibrant if $X \to
  \simp^{\op}_{\txt{int}}$ is a coCartesian fibration, $S$ is the set
  of coCartesian edges, and the Segal morphisms $X_{[n]} \to X_{[1]}
  \times_{X_{[0]}} \cdots \times_{X_{[0]}} X_{[1]}$ are equivalences.
\end{remark}
Under the equivalence between coCartesian fibrations and functors the
\icat{} associated to the model category
$(\sSet^{+})_{\mathfrak{O}_{\txt{ns}}^{\txt{triv}}}$
corresponds to the full subcategory of $\Fun(\simp^{\op}_{\txt{triv}},
\CatI)$ spanned by the functors that are right Kan extensions along
the inclusion $\gamma \colon \mathcal{G}^{\simp} \to
\simp^{\op}_{\txt{int}}$. Thus we have proved the following:
\begin{lemma}
  The \icat{} associated to the model category
  $(\sSet^{+})_{\mathfrak{O}_{\txt{ns}}^{\txt{triv}}}$ is equivalent
  to $\Fun(\mathcal{G}^{\simp}, \CatI)$.
\end{lemma}
The obvious map of categorical patterns
$\mathfrak{O}_{\txt{ns}}^{\txt{triv}} \to
\mathfrak{O}_{\txt{ns}}^{\txt{gen}}$ then induces an adjoint pair
of functors
\[ \gamma_{!} \colon \Fun(\mathcal{G}^{\simp}, \CatI)
\rightleftarrows \OpdInsg \colon \gamma^{*}.\] Since composition with
the inclusion $\simp^{\op}_{\txt{int}} \to \simp^{\op}$ takes
$\mathfrak{O}_{\txt{ns}}^{\txt{triv}}$-fibrant objects to
$\mathfrak{O}_{\txt{ns}}^{\txt{gen}}$-fibrant objects, the left
adjoint $\gamma_{!}$ sends a functor $\mathcal{G}^{\simp} \to
\CatI$ to its right Kan extension to $\simp^{\op}_{\txt{int}} \to
\CatI$, then to the composite $\mathcal{E} \to \simp^{\op}_{\txt{int}}
\to \simp^{\op}$, where $\mathcal{E} \to \simp^{\op}_{\txt{int}}$ is
the associated coCartesian fibration. In particular, if $\mathcal{M}$
is a \gnsiopd{}, then $\mathcal{M}_{\triv}$ is
$\gamma_{!}\gamma^{*}\mathcal{M}$, and the natural map
$\mathcal{M}_{\triv} \to \mathcal{M}$ is the adjunction morphism.

Taking the $(\infty,2)$-categories associated to the categorical
patterns into account, we get the following:
\begin{propn}\label{propn:TrivAlgEq}
  Let $F \colon \mathcal{G}^{\simp} \to \CatI$ be a functor, and
  $\mathcal{F} \to \mathcal{G}^{\simp}$ the associated
  coCartesian fibration. If $\mathcal{M}$ is a \gnsiopd{} let
  $\mathcal{M}_{\txt{glob}}$ denote the pullback of $\mathcal{M}$
  along $\mathcal{G}^{\simp} \to \simp^{\op}$. Then there is a
  natural equivalence between $\Alg_{\gamma_{!}F}(\mathcal{M})$ and
  the full subcategory
  $\Fun^{\txt{coCart}}_{\mathcal{G}^{\simp}}(\mathcal{F},
  \mathcal{M}_{\txt{glob}})$ of
  $\Fun_{\mathcal{G}^{\simp}}(\mathcal{F},
  \mathcal{M}_{\txt{glob}})$ spanned by functors that preserve
  coCartesian arrows. In particular, if $\mathcal{O}$ is a
  \nsiopd{}, then $\Alg_{\gamma_{!}F}(\mathcal{O}) \simeq
  \Fun(F([1]), \mathcal{O}_{[1]})$.
\end{propn}

\subsection{Monoid and Category Objects} 
We will now observe that if $\mathcal{V}$ is an \icat{} with finite
products and $\mathcal{M}$ is a (generalized) \nsiopd{}, then the
$\mathcal{M}$-algebras in the Cartesian monoidal \icat{}
$\mathcal{V}^{\times}$ are equivalent to a certain class of functors
$\mathcal{M} \to \mathcal{V}$, namely
the $\mathcal{M}$-\emph{monoids}.

\begin{defn}
  Suppose $\mathcal{M}$ is a generalized \nsiopd{} and $\mathcal{V}$
  an \icat{} with finite limits. An
  \defterm{$\mathcal{M}$-monoid object} in $\mathcal{V}$ is
  a functor $F \colon \mathcal{M} \to \mathcal{V}$ such that
  its restriction $F|_{\mathcal{M}_{\triv}}$ is a
  right Kan extension of $F|_{\mathcal{M}_{[1]}}$ along the inclusion
  $\mathcal{M}_{[1]} \hookrightarrow \mathcal{M}_{\triv}$. Write
  $\Mon_{\mathcal{M}}(\mathcal{V})$
  for the full subcategory of $\Fun(\mathcal{M},
  \mathcal{V})$ spanned by the $\mathcal{M}$-monoid
  objects.
\end{defn}

\begin{defn}
  Suppose $\mathcal{M}$ is a generalized \nsiopd{} and $\mathcal{V}$
  is an \icat{} with finite limits. An \defterm{$\mathcal{M}$-category
    object} in $\mathcal{V}$ is a functor $F \colon \mathcal{M} \to
  \mathcal{V}$ such that its restriction $F|_{\mathcal{M}_{\triv}}$ is
  a right Kan extension of $F|_{\mathcal{M}_{\txt{glob}}}$ along the
  inclusion $\mathcal{M}_{\txt{glob}} \hookrightarrow
  \mathcal{M}_{\triv}$. Write
  $\catname{Cat}_{\mathcal{M}}(\mathcal{V})$ for the full subcategory
  of $\Fun(\mathcal{M}, \mathcal{V})$ spanned by the
  $\mathcal{M}$-category objects. When $\mathcal{M}$ is $\simp^{\op}$
  we refer to $\simp^{\op}$-category objects as just \defterm{category
    objects}.
\end{defn}

\begin{propn}\label{propn:monoideqalg}
  Suppose $\mathcal{V}$ is an \icat{} with finite products, and
  consider $\mathcal{V}$ as a monoidal \icat{} via the pullback of the
  Cartesian symmetric monoidal structure. Then for
  any \gnsiopd{} $\mathcal{M}$ we have
  $\Alg_{\mathcal{M}}(\mathcal{V}) \simeq
  \Mon_{\mathcal{M}}(\mathcal{V})$.
\end{propn}
\begin{proof}
  As \cite[Proposition 2.4.2.5]{HA}.
\end{proof}

\begin{propn}\label{propn:moncateqmnd}
  We have equivalences $\MonI \simeq
  \Mon_{\simp^{\op}}(\CatI)$ and
  $\catname{Dbl}_{\infty} \simeq
  \Cat_{\simp^{\op}}(\CatI)$.
\end{propn}
\begin{proof}
  We can identify $\MonI$ with the full
  subcategory of the \icat{} of coCartesian fibrations over
  $\simp^{\op}$ spanned by the
  monoidal \icats{}. Under the equivalence between coCartesian
  fibrations over $\simp^{\op}$ and functors $\simp^{\op} \to \CatI$
  these correspond precisely to those functors satisfying the
  condition for a monoid object. Similarly, the double \icats{}
  correspond to the category objects.
\end{proof}

\subsection{The Algebra Fibration}\label{subsec:AlgFib}
In this subsection we define, given a \nsiopd{} $\mathcal{O}$, a
Cartesian fibration $\Alg(\mathcal{O}) \to \OpdIns$ with fibre
$\Alg_{\mathcal{P}}(\mathcal{O})$ at $\mathcal{P} \in \OpdIns$ --- the
objects of $\Alg(\mathcal{O})$ are thus pairs $(\mathcal{P}, A)$ where
$\mathcal{P}$ is a \nsiopd{} and $A$ is a $\mathcal{P}$-algebra in
$\mathcal{O}$. We then study the \icat{} $\Alg(\mathcal{V})$ in the
special case when $\mathcal{V}$ is a monoidal \icat{} and consider its
behaviour as we vary the monoidal \icat{} $\mathcal{V}$.

\begin{defn}
  Let $\mathcal{O}$ be a \nsiopd{}. Recall that
  $(\sSet^{+})^{\op}_{\mathfrak{O}_{\txt{ns}}}$ is a
  marked simplicial model category, so we have a functor
  \[(\sSet^{+})^{\op}_{\mathfrak{O}_{\txt{ns}}} \to
  \sSet^{+}\]
  represented by $\mathcal{O}$. This restricts to a functor
  between the fibrant objects in these marked simplicial model
  categories; forgetting from the marked simplicial enrichment down to
  enrichment in simplicial sets (by forgetting the unmarked
  1-simplices) and taking nerves we get a functor
  \[ (\OpdIns)^{\op} \to \CatI; \]
  this sends a \nsiopd{} $\mathcal{P}$ to
  $\Alg_{\mathcal{P}}(\mathcal{O})$. We define
  \[ \Alg(\mathcal{O}) \to \OpdIns\]
  to be a Cartesian fibration corresponding to this functor.
\end{defn}

\begin{remark}
  We could also construct $\Alg(\mathcal{O})$ as a full
  subcategory of the source of a Cartesian fibration associated to the
  functor $(\CatI)_{/\simp^{\op}} \to \CatI$ that sends $\mathcal{C}
  \to \simp^{\op}$ to $\Fun_{\simp^{\op}}(\mathcal{C},
  \mathcal{O})$.
\end{remark}

\begin{remark}\label{rmk:MonoidFib}
  Let $\mathcal{V}$ be an \icat{} with finite products. Then we
  can similarly define a fibration $\Mon(\mathcal{V}) \to
  \OpdIns$ with fibre
  $\Mon_{\mathcal{O}}(\mathcal{V})$ at
  $\mathcal{O}$. The proof of \cite[Proposition 2.4.1.7]{HA}
  implies that the equivalence
  $\Alg_{\mathcal{O}}(\mathcal{V}) \simeq
  \Mon_{\mathcal{O}}(\mathcal{V})$ is natural in
  $\mathcal{O}$, which gives an equivalence
  $\Alg(\mathcal{V}) \isoto \Mon(\mathcal{V})$ when $\mathcal{V}$ is
  considered as a monoidal \icat{} via the Cartesian product.
\end{remark}

\begin{defn}
  For $\mathcal{O}$ a \nsiopd{}, let
  \[\Algtriv(\mathcal{O}) \to \OpdIns\] be the pullback of
  $\Alg(\mathcal{O})$ along the functor
  $\gamma_{!}\gamma^{*}$ from $\OpdIns$ to
  itself that sends $\mathcal{P}$ to
  $\mathcal{P}_{\triv}$. The natural maps
  $\tau^{*}_{\mathcal{P}} \colon \mathcal{P}_{\triv} \to \mathcal{P}$ then
  induce a functor \[\tau^{*} \colon \Alg(\mathcal{O}) \to
  \Algtriv(\mathcal{O}).\]
\end{defn}

\begin{remark}\label{rmk:AlgTrivPullback}
  The natural equivalence $\Alg_{\mathcal{P}_{\txt{triv}}}(\mathcal{V}) \simeq \Fun(\mathcal{P}_{[1]}, \mathcal{V})$ of
  Proposition~\ref{propn:TrivAlgEq} implies that there is a pullback
  diagram
  \nolabelcsquare{\Algtriv(\mathcal{V})}{\mathcal{F}_{\mathcal{V}}}{\OpdIns}{\CatI,}
  where the lower horizontal map sends an \iopd{}
  $\mathcal{O}$ to $\mathcal{O}_{[1]}$, and the right vertical map
  is a Cartesian fibration associated to the functor $\CatI^{\op} \to
  \CatI$ that sends $\mathcal{C}$ to $\Fun(\mathcal{C}, \mathcal{V})$.
\end{remark}

\begin{lemma}\label{lem:AlgCoCart}
  Suppose $\mathcal{V}$ is a monoidal \icat{} compatible
  with small colimits. Then the projection
  $\Alg(\mathcal{V}) \to \OpdIns$ is both
  Cartesian and coCartesian.
\end{lemma}
\begin{proof}
  By \cite[Corollary 5.2.2.5]{HTT} it suffices to prove that for each
  map $f \colon \mathcal{O} \to \mathcal{P}$ in $\OpdIns$ the map
  $f^{*} \colon \Alg^{\txt{ns}}_{\mathcal{P}}(\mathcal{V}) \to
  \Alg^{\txt{ns}}_{\mathcal{O}}(\mathcal{V})$ has a left adjoint. This
  is precisely the content of Theorem~\ref{thm:FreeFtrExists}.
\end{proof}

\begin{lemma}
  Suppose $\mathcal{V}$ is a monoidal \icat{} compatible
  with small colimits. Then the functor $\tau^{*}$ has a left
  adjoint \[\tau_{!} \colon \Algtriv(\mathcal{V}) \to
  \Alg(\mathcal{V})\] relative to $\OpdIns$.
\end{lemma}
\begin{proof}
  By \cite[Proposition 7.3.2.6]{HA} it suffices to prove that
  $\tau^{*}$ admits fibrewise left adjoints, which we showed in
  Theorem~\ref{thm:FreeFtrExists}, and that $\tau^{*}$ preserves
  Cartesian arrows, which is clear since it is the functor
  associated to a natural transformation between the corresponding
  functors to $\CatI$.
\end{proof}

\begin{lemma}\label{lem:AlgOprlim}
  The functor $\Alg_{(\blank)}(\mathcal{V}) \colon
  (\OpdIns)^{\op} \to \CatI$ takes colimits in $\OpdIns$ to limits.
\end{lemma}
\begin{proof}
  For any categorical pattern $\mathfrak{P}$, the product
  \[ \sSet^{+} \times (\sSet^{+})_{\mathfrak{P}} \to
  (\sSet^{+})_{\mathfrak{P}}\] is a left Quillen bifunctor by
  \cite[Remark B.2.5]{HA}. Thus the induced functor of \icats{}
  preserves colimits in each variable. In particular, the product
  \[ \CatI \times \OpdIns \to \OpdIns \] preserves colimits in each
  variable. Now $\Alg_{(\blank)}(\blank)$ is defined as a right
  adjoint to this, so for any \icat{} $\mathcal{C}$ we have
  \[
  \begin{split}
\Map_{\CatI}(\mathcal{C}, \Alg_{\colim_{\alpha}
    \mathcal{O}_{\alpha}}(\mathcal{P})) &   \simeq \Map_{\OpdIns}(\mathcal{C} \times
  \colim_{\alpha} \mathcal{O}_{\alpha}, \mathcal{P}) \\ & \simeq
  \Map_{\OpdIns}(\colim_{\alpha} (\mathcal{C} \times \mathcal{O}_{\alpha}), \mathcal{P})
  \\ & \simeq \lim_{\alpha} \Map_{\OpdIns}(\mathcal{C} \times \mathcal{O}_{\alpha},
  \mathcal{P}) \\ & \simeq \lim_{\alpha} \Map_{\CatI}(\mathcal{C},
  \Alg_{\mathcal{O}_{\alpha}}(\mathcal{P})) \\ & \simeq \Map_{\CatI}(\mathcal{C},
  \lim_{\alpha} \Alg_{\mathcal{O}_{\alpha}}(\mathcal{P})).
  \end{split}\]
  Thus $\Alg_{\colim \mathcal{O}_{\alpha}}(\mathcal{P}) \simeq \lim_{\alpha}
  \Alg_{\mathcal{O}_{\alpha}}(\mathcal{P})$.
\end{proof}

\begin{propn}\label{propn:AlgColim}
  Suppose $\mathcal{V}$ is a monoidal \icat{} compatible
  with small colimits. Then $\Alg(\mathcal{V})$
  admits small colimits. 
\end{propn}


\begin{proof}[Proof of Proposition~\ref{propn:AlgColim}]
  By Lemma~\ref{lem:AlgCoCart}, the fibration $\pi \colon
  \Alg(\mathcal{V}) \to \OpdIns$ is coCartesian. Moreover, its fibres
  have all colimits by Corollary~\ref{cor:AlgOcolim} and the functors
  $f_{!}$ induced by morphisms $f$ in $\OpdIns$ preserve colimits,
  being left adjoints. Thus $\pi$ satisfies the conditions of
  \cite[Lemma 9.8]{freepres}.
\end{proof}

\begin{propn}\label{propn:StrMonColim2}
  Let $\mathcal{V}$ and $\mathcal{W}$ be monoidal
  \icats{} compatible with small colimits. Suppose $F \colon
  \mathcal{V}^{\otimes} \to \mathcal{W}^{\otimes}$ is a
  monoidal functor such that $F_{[1]} \colon \mathcal{V} \to \mathcal{W}$
  preserves colimits. Then $F_{*} \colon
  \Alg(\mathcal{V}) \to
  \Alg(\mathcal{W})$ preserves colimits.
\end{propn}
\begin{proof}
  Since $\mathcal{V}$ and $\mathcal{W}$ are
  compatible with small colimits, the projections
  \[\Alg(\mathcal{V}),\,\,
  \Alg(\mathcal{W}) \to \OpdIns\] are coCartesian
  fibrations. Thus a diagram in $\Alg(\mathcal{W})$ is a
  colimit diagram \IFF{} it is a relative colimit diagram whose
  projection to $\OpdIns$ is a colimit diagram.

  It therefore suffices to prove that $F_{*}$ preserves coCartesian
  arrows and preserves colimits fibrewise. The former follows from
  Lemma \ref{lem:StrMonPrFree}, and the latter from
  Proposition~\ref{propn:StrMonColim1}.
\end{proof}

\begin{propn}\label{propn:AlgVPres}
  Suppose $\mathcal{V}$ is a presentably monoidal \icat{}. Then the
  \icat{} $\Alg(\mathcal{V})$ is presentable and the projection
  $\Alg(\mathcal{V}) \to \OpdIns$ is an accessible functor.
\end{propn}
\begin{proof}
  This follows from \cite[Theorem 9.3]{freepres} together with
  Theorem~\ref{thm:FreeFtrExists}, Corollary~\ref{cor:AlgOcolim}, and
  Lemma~\ref{lem:AlgOprlim}.
\end{proof}

Next we observe that the \icat{} $\Alg(\mathcal{O})$ is
functorial in $\mathcal{O}$:
\begin{defn}
  Since the model category $(\sSet^{+})_{\mathfrak{O}_{\txt{ns}}}$
  is enriched in marked simplicial sets, the enriched Yoneda functor
 \[ \mathfrak{H} \colon (\sSet^{+})_{\mathfrak{O}_{\txt{ns}}}^{\op} \times
  (\sSet^{+})_{\mathfrak{O}_{\txt{ns}}} \to \sSet^{+}\]
  sending $(\mathcal{O}, \mathcal{P})$ to
  $\Alg_{\mathcal{O}}(\mathcal{P})$ induces a
  functor of \icats{} $(\OpdIns)^{\op} \times
  \OpdIns \to \CatI$. Let $\Algco \to \OpdIns \times
  (\OpdIns)^{\op}$ be a Cartesian fibration corresponding to this functor.
\end{defn}

The fibre of $\Algco$ at $\mathcal{O}$ in the second
component is $\Alg(\mathcal{O})$. The composite $\Algco \to
(\OpdIns)^{\op}$ with projection to the second factor is then a
Cartesian fibration corresponding to a functor $\OpdIns \to \CatI$
that sends $\mathcal{O}$ to
$\Alg(\mathcal{O})$. Thus we see that
$\Alg(\mathcal{O})$ is functorial in
$\mathcal{O}$.

\begin{defn}
  Let $\Alg \to \OpdIns$ be a coCartesian fibration corresponding to
  the functor $\mathcal{O} \mapsto \Alg(\mathcal{O})$.
\end{defn}

Next we show that the algebra fibration is compatible with products of
\nsiopds{}:
\begin{propn}\label{propn:AlgLaxMon}
  $\Alg(\blank)$ is lax monoidal with respect to the Cartesian
  product of \nsiopds{}.
\end{propn}
\begin{proof}
  The Cartesian product on $(\sSet^{+})_{\mathfrak{O}_{\txt{ns}}}$ gives a symmetric monoidal
  structure on $(\sSet^{+})_{\mathfrak{O}_{\txt{ns}}}^{\op} \times
  (\sSet^{+})_{\mathfrak{O}_{\txt{ns}}}$  by taking products in
  both variables. The functor $\mathfrak{H}$ is lax monoidal with
  respect to this, and so 
  induces an $((\OpdIns)^{\op} \times \OpdIns)^{\times}$-monoid in
  $\CatI$. From this we get a Cartesian fibration $\Algco^{\times} \to
  (((\OpdIns)^{\op} \times \OpdIns)^{\times})^{\op}$.  Projecting to the
  second factor gives a Cartesian fibration that corresponds to a
  monoid $(\OpdIns)^{\times} \to \CatI$, and so a lax monoidal functor
  $(\OpdIns)^{\times} \to \CatI^{\times}$. This shows that
  $\Alg(\blank)$ is a lax monoidal functor.
\end{proof}

This construction gives an ``external product''
\[\boxtimes \colon \Alg(\mathcal{O}) \times \Alg(\mathcal{P}) \to
\Alg(\mathcal{O} \times_{\simp^{\op}}
\mathcal{P}).\] Our next result is that for algebras in
monoidal \icats{} compatible with colimits this preserves colimits in
each variable; this requires some preliminary results:
\begin{lemma}\label{lem:FibBoxProdCoCart}
  Suppose $\mathcal{V}$ and $\mathcal{W}$ are
  monoidal \icats{} compatible with small colimits. Then the external
  product $\boxtimes$ preserves free algebras, i.e. given \nsiopds{}
  $\mathcal{O}$ and $\mathcal{P}$, algebras $A \in
  \Alg_{\mathcal{O}}(\mathcal{V})$ and $B \in
  \Alg_{\mathcal{P}}(\mathcal{W})$, and morphisms
  of \nsiopds{} $f \colon \mathcal{O} \to
  \mathcal{Q}$ and $g \colon \mathcal{P} \to
  \mathcal{R}$, we have $f_{!}A \boxtimes g_{!}B \simeq (f
  \times g)_{!}(A \boxtimes B)$ in $\Alg_{\mathcal{Q}
    \times_{\simp^{\op}} \mathcal{R}}(\mathcal{V}
  \times \mathcal{W})$.
\end{lemma}
\begin{proof}
  This follows from Lemma~\ref{lem:opdcoliminproduct}.
\end{proof}

\begin{propn}\label{propn:FibBoxProdColim}
  Suppose $\mathcal{V}$ and $\mathcal{W}$ are
  monoidal \icats{} compatible with small colimits, and let
  $\mathcal{O}$ and $\mathcal{P}$ be \nsiopds{}
  and $A
  \in \Alg_{\mathcal{O}}(\mathcal{V})$ be an $\mathcal{O}$-algebra. Then \[A
  \boxtimes (\blank) \colon
  \Alg_{\mathcal{P}}(\mathcal{W}) \to
  \Alg_{\mathcal{O} \times_{\simp^{\op}}
    \mathcal{P}}(\mathcal{V}
  \times \mathcal{W})\] preserves colimits.
\end{propn}
\begin{proof}
  First we consider the case of trival \nsiopds{}. Suppose $A'$ is
  an $\mathcal{O}_{\triv}$-algebra. Then \[A'
  \boxtimes \blank \colon
  \Alg_{\mathcal{P}_{\triv}}(\mathcal{W})
  \to \Alg_{\mathcal{O}_{\triv} \times_{\simp^{\op}}
    \mathcal{P}_{\triv}}(\mathcal{V}
  \times \mathcal{W})\] clearly preserves colimits,
  since it is equivalent to the functor \[A'|_{\mathcal{O}_{[1]}}
  \times \blank \colon \Fun(\mathcal{P}, \mathcal{W}) \to
  \Fun(\mathcal{O}_{[1]} \times \mathcal{P}_{[1]}, \mathcal{V} \times
  \mathcal{W}).\]
  Since we have $\tau^{*}_{\mathcal{V} \times \mathcal{W}}(A \boxtimes
  B) \simeq \tau^{*}_{\mathcal{V}}A \boxtimes \tau^{*}_{\mathcal{W}}B$
  and $\tau^{*}_{\mathcal{V}\times \mathcal{W}}$ detects sifted
  colimits by Corollary~\ref{cor:AlgSiftedColim}, it follows that $A
  \boxtimes \blank$ preserves sifted colimits for any $A$.

  Next we consider the case where $A$ is a free algebra
  $\tau_{\mathcal{V},!}A'$ for some 
  $\mathcal{O}_{\triv}$-algebra $A'$ in $\mathcal{V}$. By
  Lemma~\ref{lem:FibBoxProdCoCart} we have
  \[\tau_{\mathcal{V},!}A' \boxtimes \tau_{\mathcal{W},!}B'
  \simeq \tau_{\mathcal{V} \times \mathcal{W},!}(A' \boxtimes B'),\]
  so the functor $\tau_{\mathcal{V},!}A \boxtimes \blank$ preserves
  colimits of free algebras. Thus it must preserve all colimits, by
  monadicity (Corollary~\ref{cor:AlgMonad}).
  
  Finally, suppose $A_{\bullet}$ is a free resolution of $A$, and
  $\alpha \mapsto B_{\alpha}$ is any diagram. Then since $\boxtimes$
  preserves sifted colimits we have
  \[ A \boxtimes \colim B_{\alpha} \simeq |A_{\bullet}| \boxtimes
  \colim B_{\alpha} \simeq |A_{\bullet} \boxtimes \colim
  B_{\alpha}|.\] From the case of free algebras we then get that this
  is equivalent to \[ |\colim (A_{\bullet} \boxtimes B_{\alpha})| \simeq
  \colim |A_{\bullet} \boxtimes B_{\alpha}|.\] But since $\boxtimes$
  preserves sifted colimits in each variable, this is $\colim
  (|A_{\bullet}| \boxtimes B_{\alpha}) \simeq \colim (A \boxtimes
  B_{\alpha})$.
\end{proof}

\begin{remark}
  The Cartesian product of \nsiopds{} does not in general preserve
  colimits, so it is not possible for the external product, considered
  as a functor $A \boxtimes (\blank) \colon
  \Alg(\mathcal{W}) \to \Alg(\mathcal{V}
  \times \mathcal{W})$ to preserve colimits.
\end{remark}

Finally, we observe that the algebra fibration is well-behaved with
respect to adjunctions and monadic localizations:
\begin{propn}\label{propn:rightadjlaxmonalgfib}
  Suppose $\mathcal{V}$ and $\mathcal{W}$ are
  presentably monoidal \icats{}
  and $F \colon \mathcal{V}^{\otimes} \to
  \mathcal{W}^{\otimes}$ is a  monoidal functor such that the
  underlying functor $F_{[1]} \colon \mathcal{V} \to \mathcal{W}$ preserves
  colimits. Let $g \colon \mathcal{W} \to \mathcal{V}$ be a right
  adjoint of $F_{[0]}$. Then there exists a lax monoidal functor
  $G \colon \mathcal{W}^{\otimes} \to \mathcal{V}^{\otimes}$
  extending $g$ such that we have an adjunction
  \[ F_{*} :
  \Alg(\mathcal{V}) \rightleftarrows
  \Alg(\mathcal{W}) : G_{*}.\]
  over $\OpdIns$.
\end{propn}
\begin{proof}
  This is immediate from (the dual of) \cite[Proposition
  7.3.2.6]{HA} as its hypotheses are satisfied by
  Lemma~\ref{lem:StrMonPrFree} and
  Proposition~\ref{propn:rightadjlaxmon}.
\end{proof}

\begin{cor}\label{cor:monlocadjalg}
  Suppose $\mathcal{V}$ is a presentably monoidal \icat{}
  and $L \colon \mathcal{V} \to \mathcal{W}$ is an accessible monoidal
  localization with fully faithful right adjoint $i \colon \mathcal{W}
  \hookrightarrow \mathcal{V}$. Then we have an adjunction
  \[ L^{\otimes}_{*} : 
  \Alg(\mathcal{V}) \rightleftarrows
  \Alg(\mathcal{W}) : i^{\otimes}_{*}\]
  over $\OpdIns$. Moreover, $i^{\otimes}_{*}$ is fully faithful.
\end{cor}
\begin{proof}
  This follows from combining
  Proposition~\ref{propn:rightadjlaxmonalgfib} and Lemma~\ref{lem:monlocadj}.
\end{proof}

\subsection{Non-Symmetric and Symmetric $\infty$-Operads}\label{subsec:symnonsym}
In this subsection we briefly discuss the relation between
non-symmetric and symmetric $\infty$-operads and their algebras. We
will use the terminology and notation of \cite{HA} for (symmetric)
$\infty$-operads, except that we use superscript $\Sigma$'s to
distinguish the symmetric case from the non-symmetric case discussed
so far.

\begin{defn}
  Let $c \colon \simp^{\op} \to \bbGamma^{\op}$ be the functor defined
  as in \cite[Construction 4.1.2.5]{HA} (this is the same as the
  functor introduced by Segal in \cite{SegalCatCohlgy}). This takes
  inert morphisms in $\simp^{\op}$ to inert morphisms in
  $\bbGamma^{\op}$, and moreover induces a morphism of categorical
  patterns from $\mathfrak{O}_{\txt{ns}}$ to the analogous categorical
  pattern $\mathfrak{O}_{\Sigma}$ for symmetric \iopds{}. Thus $c$
  induces adjoint functors \[ c_{!} : \OpdIns \rightleftarrows \OpdIS
  : c^{*}.\] Moreover, since the induced Quillen functors are enriched
  in marked simplicial sets, we get equivalences
  \[ \Alg_{\mathcal{O}}(c^{*}\mathcal{P}) \simeq
  \Alg^{\Sigma}_{c_{!}\mathcal{O}} (\mathcal{P}),\] where
  $\mathcal{O}$ is a \nsiopd{} and $\mathcal{P}$ is a symmetric
  \iopd{}.
\end{defn}

\begin{remark}
  This Quillen adjunction is a special case of the Quillen adjunction
  induced by a morphism of operator categories defined in
  \cite[Proposition 8.18]{BarwickOpCat}.
\end{remark}

\begin{propn}\label{propn:nssymmeq}\ 
  \begin{enumerate}[(i)]
  \item The symmetric \iopd{} $c_{!}\simp^{\op}$ is equivalent to the
    symmetric \iopd{} $\mathbb{E}_{1} \simeq \txt{Ass}$
    of \cite[Definition 4.1.1.3]{HA}.
  \item The \icat{} $\MonI$ of monoidal \icats{} is equivalent to the
    \icat{} $\MonI^{\Sigma,\mathbb{E}_{1}}$ of
    $\mathbb{E}_{1}$-monoidal \icats{}.
  \item The \icat{} $\MonI^{\Sigma,\mathbb{E}_{n}}$ of
    $\mathbb{E}_{n}$-monoidal (also called $n$-tuply monoidal)
    \icats{} is equivalent to the \icat{}
    $\AlgS_{\mathbb{E}_{n-1}}(\MonI)$ of $\mathbb{E}_{n-1}$-algebras
    in monoidal \icats{}.
  \end{enumerate}
\end{propn}
\begin{proof}\ 
  \begin{enumerate}[(i)]
  \item This follows from \cite[Proposition 4.1.2.15]{HA}.
  \item We have an equivalence
  \[
  \begin{split}
    \MonI & \simeq \Mon_{\simp^{\op}}(\CatI) \simeq
  \Alg_{\simp^{\op}}(\CatI) \simeq
  \AlgS_{c_{!}\simp^{\op}}(\CatI) \\ & \simeq 
  \MonS_{\mathbb{E}_{1}}(\CatI) \simeq
  \MonI^{\Sigma,\mathbb{E}_{1}}.
  \end{split}
\]
\item Since $\mathbb{E}_{n} \simeq
  \mathbb{E}_{n-1} \otimes \mathbb{E}_{1}$, using
  the equivalences from (ii) we get an equivalence
  \[
  \AlgS_{\mathbb{E}_{n-1}}(\MonI)
  \simeq
  \AlgS_{\mathbb{E}_{n-1}}(\AlgS_{\mathbb{E}_{1}}(\CatI))
  \simeq \AlgS_{\mathbb{E}_{n}}(\CatI) \simeq
  \MonI^{\Sigma,\mathbb{E}_{n}}.\qedhere\]
  \end{enumerate}
\end{proof}

\begin{remark}
  In fact, though we do not need it here, the functor $c_{!}$ induces
  an equivalence $\OpdIns \simeq
  (\OpdIS)_{/\mathbb{E}_{1}}$ --- this is
  \cite[Proposition 4.7.1.1]{HA}.
\end{remark}

\begin{remark}\label{rmk:MonPrMon}
  By Proposition~\ref{propn:nssymmeq}, the \icat{} $\MonPr$ of
  presentably monoidal \icats{} is equivalent to the \icat{}
  $\Alg_{\mathbb{E}_{1}}(\PresI)$ of $\mathbb{E}_{1}$-algebras in
  $\PresI$. Using \cite[Proposition 3.2.4.3]{HA} we therefore see that
  the tensor product on $\PresI$ induces a symmetric monoidal
  structure on $\MonPr$. The unit for this tensor product is given by the
  unique presentably monoidal structure on the unit $\mathcal{S}$,
  namely the Cartesian monoidal structure.
\end{remark}

On the \iopds{} corresponding to ordinary multicategories, the functor
$c_{!}$ corresponds to the usual symmetrization, i.e. it adds free
actions by the symmetric groups:
\begin{defn}
  Let $\mathbf{M}$ be a multicategory. The \emph{symmetrization}
  $\txt{Sym}(\mathbf{M})$ is the symmetric multicategory with objects
  those of $\mathbf{M}$, and multimorphism sets
  \[\txt{Sym}(\mathbf{M})(X_{1},\ldots,X_{n}; Y) =
  \coprod_{\sigma \in \Sigma_{n}}
  \mathbf{M}(X_{\sigma(1)},\ldots,X_{\sigma(n)}; Y);\]
  composition in $\txt{Sym}(\mathbf{M})$ is defined using the usual
  maps $\Sigma_{n} \times \Sigma_{m} \to \Sigma_{n+m}$. The units in
  $\Sigma_{n}$ give an obvious map $\mu \colon \mathbf{M}^{\otimes} \to
  \txt{Sym}(\mathbf{M})^{\otimes}$.
\end{defn}
\begin{propn}\label{propn:SymApprox}
  Let $\mathbf{M}$ be a multicategory. The map $\mu \colon \mathbf{M}^{\otimes}
  \to \txt{Sym}(\mathbf{M})^{\otimes}$ over $\bbGamma^{\op}$ is an
  approximation of symmetric \iopds{} (cf. \cite[Definition 2.3.3.6]{HA}).
\end{propn}
\begin{proof}
  This follows by a variant of the argument in the proof of \cite[Proposition 4.1.2.10]{HA}.
\end{proof}

\begin{cor}\label{cor:SymEquiv}
  The map $\mathbf{M}^{\otimes} \to \txt{Sym}(\mathbf{M})^{\otimes}$
  induces an equivalence of symmetric \iopds{}
  \[c_{!}\mathbf{M}^{\otimes} \isoto
  \txt{Sym}(\mathbf{M})^{\otimes}.\] In particular, if
  $\mathcal{O}$ is any symmetric \iopd{} we have a natural
  equivalence
  \[ \Alg_{\mathbf{M}}(c^{*}\mathcal{O}) \simeq
  \AlgS_{\txt{Sym}(\mathbf{M})}(\mathcal{O}). \]
\end{cor}

\section{Categorical Algebras}\label{sec:algcat}
Our main goal in this section is to define the \icat{}
$\AlgCat(\mathcal{V})$ of categorical algebras in a monoidal \icat{}
$\mathcal{V}$ and prove that this has various good properties. First,
in \S\ref{subsec:DeltaOpX}, we carefully define the double \icats{}
$\simp^{\op}_{S}$ for $S$ a space, and make some observations about
the functor $S \mapsto \simp^{\op}_{S}$.  Next, in \S\ref{subsec:OX},
we identify the \nsiopd{} associated to $\simp^{\op}_{S}$ as one
arising from a certain simplicial multicategory; this allows us to
prove a crucial property of the double \icats{} $\simp^{\op}_{S}$. We
are then ready, in \S\ref{subsec:catalg}, to use the algebra fibration
from \S\ref{subsec:AlgFib} to construct the \icats{}
$\AlgCat(\mathcal{V})$ and study these; in particular, we will prove
that $\AlgCat(\mathcal{V})$ is a lax monoidal functor of
$\mathcal{V}$, and that it is presentable if $\mathcal{V}$ is
presentable and equipped with a colimit-preserving monoidal
product. In \S\ref{subsec:catalgspace} we then prove that categorical
algebras in spaces are equivalent to Segal spaces, which will prove
useful in the next section as it allows us to reduce several proofs to
the known case of Segal spaces. Finally, in \S\ref{subsec:presheafalgcat}
we show that categorical algebras are equivalent to an alternative
model for enriched \icats{} as certain presheaves.

\subsection{The Double $\infty$-Categories $\simp^\op_{S}$}\label{subsec:DeltaOpX}

We begin with an abstract definition of double \icats{}
$\simp^{\op}_{\mathcal{C}}$, where $\mathcal{C}$ is any \icat{}:
\begin{defn}
  Let $i$ denote the inclusion $\{[0]\} \hookrightarrow
  \simp^{\op}$. Taking right Kan extensions along $i$ gives a functor
  $i_{*} \colon \CatI \to \Fun(\simp^{\op}, \CatI)$. If $\mathcal{C}$
  is an \icat{}, we write $\simp^{\op}_{\mathcal{C}} \to \simp^{\op}$
  for a coCartesian fibration corresponding to the functor
  $i_{*}\mathcal{C}$. 
\end{defn}

\begin{remark}
  If $\mathcal{C}$ is an \icat{}, then $i_{*}\mathcal{C}$ is the
  simplicial \icat{} with $n$th space $\mathcal{C}^{\times n+1}$, face
  maps given by the appropriate projections, and degeneracies by the
  appropriate diagonal maps.
\end{remark}

\begin{lemma}
  Let $\mathcal{C}$ be an \icat{}. The coCartesian fibration
  $\simp^{\op}_{\mathcal{C}} \to \simp^{\op}$ is a double \icat{}.
\end{lemma}
\begin{proof}
  It is clear that $i_{*}\mathcal{C}$ is a category object, hence
  $\simp^{\op}_{\mathcal{C}}$ is a double \icat{} by
  Proposition~\ref{propn:moncateqmnd}.
\end{proof}

\begin{remark}\label{rmk:DopSdescr}
  We can also give a more explicit description of the simplicial sets
  $\simp^{\op}_{\mathcal{C}}$, as follows: Consider the forgetful
  functor $\simp \to \Set$ that sends $[n]$ to the set $\{0,\ldots,
  n\}$, and let $\mathbf{P} \to \simp^{\op}$ be an associated
  Grothendieck fibration. Then define $E_{\mathcal{C}} \to \simp^{\op}$ to be
  the simplicial set satisfying the universal property
  \[ \Hom_{\simp^{\op}}(K, E_{\mathcal{C}}) \cong \Hom(\mathbf{P}
  \times_{\simp^{\op}} K, \mathcal{C}) \] The map $E_{\mathcal{C}} \to
  \simp^{\op}$ is a coCartesian fibration by \cite[Proposition
  3.2.2.13]{HTT}, and the corresponding functor is that sending $[n]$
  to $\Fun(\mathbf{P}_{[n]}, \mathcal{C}) \simeq \mathcal{C}^{\times
    (n+1)}$ by \cite[Proposition 7.3]{freepres}. Thus
  $E_{\mathcal{C}} \to \simp^{\op}$ is the same as the
  coCartesian fibration $\simp^{\op}_{\mathcal{C}} \to \simp^{\op}$.
\end{remark}




\begin{remark}\label{rmk:simpopXrightadj}
  The functor \[\simp^{\op}_{(\blank)} \colon \CatI \to \OpdInsg\] is a
  right adjoint to the functor $\OpdInsg \to \CatI$ that sends a
  \gnsiopd{} $\mathcal{M}$ to its fibre $\mathcal{M}_{[0]}$ at $[0]$:
  it is a composite of the right Kan extension functor $i_{*} \colon
  \CatI \to \txt{Dbl}_{\infty}$, which is right adjoint to the
  fibre-at-$[0]$ functor, and the inclusion $\txt{Dbl}_{\infty}
  \hookrightarrow \OpdInsg$, right adjoint to the monoidal envelope
  functor, which preserves fibres at $[0]$
  (cf. \S\ref{subsec:monenv}).
\end{remark}

\begin{remark}
  It follows from Remark~\ref{rmk:simpopXrightadj} that the functor
  $\simp^{\op}_{(\blank)} \colon \CatI \to \OpdInsg$ is fully
  faithful, since using the adjunction we have
  \[ \Map(\simp^{\op}_{\mathcal{C}}, \simp^{\op}_{\mathcal{D}}) \simeq
  \Map((\simp^{\op}_{\mathcal{C}})_{[0]}, \mathcal{D}) \simeq
  \Map(\mathcal{C}, \mathcal{D}).\]
\end{remark}

\begin{propn}\label{propn:DopXfilt}
  The functor $\simp^{\op}_{(\blank)} \colon \CatI \to \OpdInsg$
  preserves filtered colimits.
\end{propn}
\begin{proof}
  Suppose we have a filtered diagram of \icats{} $p \colon \mathcal{I}
  \to \CatI$ with colimit $\mathcal{C}$. Since
  $\simp^{\op}_{\mathcal{C}}$ is a \gnsiopd{}, by
  Theorem~\ref{thm:OpdFiltColim} it suffices to show that
  $\simp^{\op}_{\mathcal{C}}$ is the colimit of
  $\simp^{\op}_{p(\blank)}$ in $\CatI$. Now this composite functor \[
  \CatI \xto{\simp^{\op}_{(\blank)}} \OpdInsg \to \CatI\] factors as 
  \[ \CatI \xto{i_{*}} \Fun(\simp^{\op}, \CatI) \isoto
  \txt{CoCart}(\simp^{\op}) \xto{q} \CatI, \] where
  $\txt{CoCart}(\simp^{\op})$ is the \icat{} of coCartesian fibrations
  over $\simp^{\op}$ and the rightmost functor $q$ is the forgetful
  functor that sends a fibration $\mathcal{E} \to \simp^{\op}$ to the
  \icat{} $\mathcal{E}$. The functor $q$ preserves filtered colimits
  by Corollary~\ref{cor:ForgetCoCartSubcatFltColim}, so it suffices to prove
  that $i_{*}$ preserves them. Colimits in functor categories are
  computed pointwise, so to see this it suffices to show that for each
  $[n]$ the composite functor $\CatI \to \CatI$ induced by composing
  with evaluation at $[n]$ preserves filtered colimits. This functor
  sends $\mathcal{D}$ to the product $\mathcal{D}^{\times (n+1)}$, and
  so preserves filtered (and even sifted) colimits by
  \cite[Proposition 5.5.8.6]{HTT}, since the Cartesian product of
  \icats{} preserves colimits separately in each variable.
\end{proof}

\subsection{The $\infty$-Operad Associated to
  $\simp^{\op}_{S}$}\label{subsec:OX}
By Corollary~\ref{cor:GenOpdLoc} there is a universal \nsiopd{}
$L_{\txt{gen}}\simp^{\op}_{S}$ receiving a map from the double \icat{}
$\simp^{\op}_{S}$. In this subsection we describe a concrete model for
$L_{\txt{gen}}\simp^{\op}_{S}$ as the \iopd{} associated to a
simplicial multicategory. We will use this below in \S\ref{subsec:FFES} to
see that our theory of enriched \icats{} is equivalent to the definition
sketched in \S\ref{subsec:FTIopds}, and it will also allow us to
conclude that the functor that sends $S$ to
$L_{\txt{gen}}\simp^{\op}_{S}$ preserves products.

\begin{remark}
  Although it is obvious that the functor $\simp^{\op}_{(\blank)}$
  preserves products, since it's a right adjoint by
  Remark~\ref{rmk:simpopXrightadj}, it is not clear that the
  localization functor \[L_{\txt{gen}} \colon \OpdInsg \to \OpdIns\]
  preserves products --- in fact, this may well be false in general.
\end{remark}

First we define simplicial categories $\mathcal{D}(\mathcal{C})$ that
model $\simp^{\op}_{\mathrm{N}\mathcal{C}}$ when $\mathcal{C}$ is a
simplicial category:
\begin{defn}
  Given a simplicial category $\mathcal{C}$, the simplicial category
  $\mathcal{D}(\mathcal{C})$ has objects finite sequences $(c_{0},
  \ldots, c_{n})$ of objects of $\mathcal{C}$; morphisms are given by
  \[ \mathcal{D}(\mathcal{C})((c_{0}, \ldots, c_{n}), (d_{0}, \ldots,
  d_{m})) := \coprod_{\phi \colon [m] \to [n]} \prod_{i = 0}^{m}
  \mathcal{C}(c_{\phi(i)}, d_{i}),\]
  with the obvious composition maps induced by those in $\mathcal{C}$.
\end{defn}

\begin{propn} Suppose $\mathcal{C}$ is a fibrant simplicial category. Then:
  \begin{enumerate}[(i)]
  \item The projection $\mathrm{N}\mathcal{D}(\mathcal{C}) \to
    \mathrm{N}\simp^{\op}$ is a coCartesian fibration.
  \item The fibre $\mathrm{N}\mathcal{D}(\mathcal{C})_{[0]}$ is
    equivalent to $\mathrm{N}\mathcal{C}$.
  \item There is a natural map $\mathrm{N}\mathcal{D}(\mathcal{C}) \to
    \simp^{\op}_{\mathrm{N}\mathcal{C}}$, and this preserves
    coCartesian edges.
  \item This map is an equivalence of \icats{}.
  \end{enumerate}
\end{propn}
\begin{proof}\
  \begin{enumerate}[(i)]
  \item It is clear that $\mathcal{D}(\mathcal{C}) \to \simp^{\op}$ is
    a fibration in the model structure on simplicial categories; since
    $\mathrm{N}$ is a right Quillen functor, it follows that
    $\mathrm{N}\mathcal{D}(\mathcal{C}) \to \mathrm{N}\simp^{\op}$ is
    a categorical fibration. It therefore suffices to check that
    $\mathrm{N}\mathcal{D}(\mathcal{C})$ has coCartesian morphisms.
    Given an object $C = (c_{0},\ldots, c_{n})$ in
    $\mathcal{D}(\mathcal{C})$ and a map $\phi \colon [m] \to [n]$ in
    $\simp$, let $\overline{\phi}$ denote the obvious map $C \to C' =
    (c_{\phi(0)}, \ldots, c_{\phi(m)})$ in
    $\mathcal{D}(\mathcal{C})$. We apply the criterion of
    \cite[Proposition 2.4.1.10]{HTT} to see that $\overline{\phi}$ is
    coCartesian in $\mathrm{N}\mathcal{D}(\mathcal{C})$; thus we need
    to show that for every $X \in \mathcal{D}(\mathcal{C})$ over $[k]
    \in \simp^{\op}$ the commutative diagram
    \nolabelcsquare{\mathcal{D}(\mathcal{C})(C',
      X)}{\mathcal{D}(\mathcal{C})(C, X)}{\Hom_{\simp^{\op}}([m],
      [k])}{\Hom_{\simp^{\op}}([n], [k])} is a homotopy Cartesian
    square of simplicial sets. Since the simplicial category
    $\mathcal{C}$ is fibrant, so is $\mathcal{D}(\mathcal{C})$, hence
    the vertical maps are Kan fibrations. It therefore suffices to
    show that the induced maps on fibres are equivalences, which is
    clear from the definition of $\mathcal{D}(\mathcal{C})$.

  \item We have a pullback diagram of simplicial categories
    \nolabelcsquare{\mathcal{C}}{\mathcal{D}(\mathcal{C})}{\{[0]\}}{\simp^{\op}.}
    Since the simplicial nerve is a right adjoint, it follows that
    $\mathrm{N}\mathcal{C}$ is the fibre of the map of simplicial sets
    $\mathrm{N}\mathcal{D}(\mathcal{C}) \to \simp^{\op}$ at
    $[0]$. Since this map is a coCartesian fibration, by
    \cite[Corollary 3.3.1.4]{HTT} $\mathrm{N}\mathcal{C}$ is also the
    homotopy fibre in the Joyal model structure.

  \item By definition $\simp^{\op}_{\mathrm{N}\mathcal{C}}$
    corresponds to the right Kan extension
    $i_{*}\mathrm{N}\mathcal{C}$ of $\mathrm{N}\mathcal{C}$ along the
    inclusion $i \colon \{[0]\} \hookrightarrow \simp^{\op}$. The
    functor $i_{*}$ is right adjoint to the fibre-at-$[0]$ functor
    $i^{*}$, and from (ii) we know that
    $i^{*}\mathrm{N}\mathcal{D}(\mathcal{C}) \simeq
    \mathrm{N}\mathcal{C}$. The adjunction $i^{*} \dashv i_{*}$ then
    gives the required map $\mathcal{D}(\mathcal{C}) \to
    \simp^{\op}_{\mathrm{N}\mathcal{C}}$ (which preserves coCartesian
    edges since by definition $i_{*}$ lands in the \icat{} of
    coCartesian fibrations and coCartesian-morphism-preserving
    functors).

  \item By \cite[Corollary 2.4.4.4]{HTT} it suffices to show that for
  each $[n]$ in $\simp^{\op}$ the induced map on fibres
  \[ (\mathrm{N}\mathcal{D}(\mathcal{C}))_{[n]} \to
  (\simp^{\op}_{\mathrm{N}\mathcal{C}})_{[n]}\] is a categorical
  equivalence. As in (ii) we can identify the fibre
  $(\mathrm{N}\mathcal{D}(\mathcal{C}))_{[n]}$ with
  $\mathrm{N}\mathcal{C}^{\times n}$, via the Segal maps, so by
  naturality we have a commutative diagram
  \nolabelcsquare{(\mathrm{N}\mathcal{D}(\mathcal{C}))_{[n]}}{(\simp^{\op}_{\mathrm{N}\mathcal{C}})_{[n]}}{\mathrm{N}\mathcal{C}^{\times
      n}}{\mathrm{N}\mathcal{C}^{\times n},} where all but the top
  horizontal map are known to be categorical equivalences. Hence this
  must also be a categorical equivalence, by the 2-out-of-3
  property. \qedhere
  \end{enumerate}
\end{proof}

\begin{defn}
  Let $\mathcal{C}$ be a simplicial category. The simplicial
  multicategory $\mathcal{O}_{\mathcal{C}}$ has objects $\ob
  \mathcal{C} \times \ob \mathcal{C}$ and multimorphism spaces defined
  by
 \[
 \mathcal{O}_{\mathcal{C}}((x_{0},y_{1}),
  \ldots, (x_{n-1}, y_{n}); (y_{0},x_{n})) := \mathcal{C}(y_{0}, x_{0}) \times
  \mathcal{C}(y_{1},x_{1}) \times \cdots \times \mathcal{C}(y_{n-1},
  x_{n-1}) \times \mathcal{C}(y_{n}, x_{n}).
  \]
 
  Composition is defined in the obvious way, using composition in $\mathcal{C}$.
  Write $\mathcal{O}^{\otimes}_{\mathcal{C}}$ for the associated
  simplicial category of operators over $\simp^{\op}$.
\end{defn}
If $\mathcal{C}$ is a fibrant simplicial category then
$\mathcal{O}_{\mathcal{C}}$ is a fibrant simplicial multicategory in
the sense of Definition~\ref{defn:simplmulticat}, and so
$\mathrm{N}\mathcal{O}^{\otimes}_{\mathcal{C}}$ is a \nsiopd{} by
Lemma~\ref{lem:multicatopd}.

\begin{remark}
  If $S$ is a set (regarded as a category with no non-identity
  morphisms), then the multicategory $\mathcal{O}_{S}$ is clearly the
  same as $\mathbf{O}_{S}$ as defined in \S\ref{subsec:MulticatEnr}.
\end{remark}

The simplicial multicategory $\mathcal{O}_{\mathcal{C}}$ is only a
model for $\simp^{\op}_{\mathrm{N}\mathcal{C}}$ when
$\mathrm{N}\mathcal{C}$ is a space, but is easier to define than the
version that works more generally. Indeed there is not even a natural
map from $\mathcal{D}(\mathcal{C})$ to $\mathcal{O}^{\otimes}_{\mathcal{C}}$ in
general; however, we can construct one if we restrict ourselves to
simplicial \emph{groupoids}.

A simplicial category can be viewed as a simplicial object in
categories whose simplicial set of objects is constant, so by analogy
we take a \emph{simplicial groupoid} to be a simplicial object in
groupoids with constant set of objects. There is a model structure on
simplicial groupoids, due to Dwyer and Kan~\cite[Theorem
2.5]{DwyerKanSimplGpd}, where the weak equivalences are the usual
Dwyer-Kan equivalences of simplicial categories, restricted to
groupoids. The simplicial nerve functor restricts to a right Quillen
equivalence from this to the usual model structure on simplicial sets
by \cite[Theorem 3.3]{DwyerKanSimplGpd}. In particular, it follows
that every space is modelled by a fibrant object in simplicial
groupoids, which is a simplicial groupoid whose mapping spaces are Kan
complexes.

Since a simplicial category can be viewed as a simplicial object in
categories with constant set of objects, a simplicial groupoid
$\mathcal{G}$ can also be regarded as a simplicial category with an
involution $i \colon \mathcal{G} \to \mathcal{G}^{\op}$ such that
$i^{\op} \circ i = \id_{\mathcal{G}}$, which sends a morphism to its
inverse. Using this we can construct a functor
$\mathcal{D}(\mathcal{G}) \to \mathcal{O}^{\otimes}_{\mathcal{G}}$:
\begin{defn}
  Suppose $\mathcal{G}$ is a simplicial groupoid. Let $\Phi \colon
  \mathcal{D}_{\mathcal{G}} \to \mathcal{O}^{\otimes}_{\mathcal{G}}$
  be the functor that sends an object $(c_{0}, \ldots, c_{n})$ of
  $\mathcal{D}(\mathcal{G})$ to $((c_{0}, c_{1}),(c_{1},c_{2}),\ldots,
  (c_{n-1},c_{n}))$ and is given on morphisms by applying $i$ on the
  first factor and inserting identities into the factors that are
  missing in $\mathcal{D}(\mathcal{G})$ in the obvious way.
\end{defn}

\begin{thm}\label{thm:opdlocDopX}
  Let $\mathcal{G}$ be a fibrant simplicial groupoid. Then the map
  \[\mathrm{N}\Phi \colon \mathrm{N}\mathcal{D}(\mathcal{G}) \to
  \mathrm{N}\mathcal{O}^{\otimes}_{\mathcal{G}}\] exhibits
  $\mathrm{N}\mathcal{O}^{\otimes}_{\mathcal{G}}$ as the operadic
  localization $L_{\txt{gen}}\mathrm{N}\mathcal{D}(\mathcal{G})$ of
  $\mathrm{N}\mathcal{D}(\mathcal{G})$.
\end{thm}
\begin{proof}
  By Corollary~\ref{cor:OpdLocCof} it suffices to show that for all
  $(x,y) \in \mathcal{G} \times \mathcal{G}$ the induced map \[g \colon
  (\mathrm{N}\mathcal{D}(\mathcal{G})_{\txt{act}})_{/(x,y)} \to
  (\mathrm{N}(\mathcal{O}^{\otimes}_{\mathcal{G}})_{\txt{act}})_{/(x,y)}\]
  is cofinal. We will prove that $g$ is a categorical equivalence; to
  see this we show that $g$ is essentially surjective and induces
  equivalences on mapping spaces.

  We first observe that $g$ is essentially surjective: an
  active morphism to $(x,y)$ in $\mathcal{O}^{\otimes}_{\mathcal{G}}$
  is determined by an object $T = ((t_{0}, s_{1}), (t_{1}, s_{2}), \ldots,
  (t_{n-1}, s_{n}))$ and morphisms  $\alpha \colon x \to t_{0}$,
  $\beta_{1} \colon s_{1} \to t_{1}$, \ldots, $\beta_{n-1} \colon
  s_{n-1} \to t_{n-1}$, $\gamma \colon s_{n} \to
  y$ in $\mathcal{G}$. Such a morphism is in the
  image of $g$ \IFF{} the $\beta_{i}$'s are all identities. Since
  $\mathcal{G}$ is by assumption a simplicial groupoid all
  morphisms in $\mathcal{G}$ are equivalences, and so
  the morphism
  \[((t_{0}, s_{1}), (s_{1}, s_{2}), \ldots, (s_{n-1},
  s_{n})) \to ((t_{0}, s_{1}), (t_{1}, s_{2}), \ldots,
  (t_{n-1}, s_{n}))
  \]
  given by $(\id, \id, \beta_{1}, \id, \beta_{2}, \ldots, \id)$ is an
  equivalence from an object in the image of $g$ to $T$.

  It remains to show that $g$ is fully faithful. Given objects $Z =
  (z_{0}, \ldots, z_{n})$ and $Z' = (z'_{0}, \ldots, z'_{m})$ in
  $\mathcal{D}(\mathcal{G})$ we must show that for each active map
  $\phi \colon [m] \to [n]$ in $\simp^{\op}$ the map
  \[ \Map_{\mathrm{N}\mathcal{D}(\mathcal{G})_{/(x,y)}}^{\phi}(Z, Z')
  \to
  \Map_{(\mathrm{N}\mathcal{O}^{\otimes}_{\mathcal{G}})_{/(x,y)}}^{\phi}(g(Z),
  g(Z')) \] is an equivalence, where the superscripts denote the
  fibres over $\phi$ in $\simp^{\op}$. Let $\alpha$ be the unique
  active map $[1] \to [n]$ in $\simp$; then we can identify this as a
  map of homotopy fibres from the commutative square
  \nolabelcsquare{\mathcal{D}(\mathcal{G})^{\phi}(Z,
    Z')}{\mathcal{D}(\mathcal{G})^{\alpha}(Z,
    (x,y))}{(\mathcal{O}^{\otimes}_{\mathcal{G}})^{\phi}(g(Z),
    g(Z'))}{(\mathcal{O}^{\otimes}_{\mathcal{G}})^{\alpha}(g(Z),
    (x,y)),} where the superscripts again denote the fibres of these
  spaces over maps in $\simp^{\op}$. To see that our map of homotopy
  fibres is an equivalence it suffices to show that this diagram is
  homotopy Cartesian.
  
  We have equivalences
\[ 
\begin{split}
\mathcal{D}(\mathcal{G})^{\phi}(Z,
Z') \simeq &\,\,  \prod_{i = 0}^{m} \mathcal{G}(z_{\phi(i)}, z'_{i}),   \\
\mathcal{D}(\mathcal{G})^{\alpha}(Z,
(x,y))  \simeq &\,\,  \mathcal{G}(z_{0}, x) \times \mathcal{G}(z_{n}, y),\\
(\mathcal{O}^{\otimes}_{\mathcal{G}})^{\phi}(g(Z), g(Z')) \simeq &\,\, \mathcal{G}(z'_{0}, z_{\phi(0)})
\times \mathcal{G}(z_{\phi(0)+1}, z_{\phi(0)+1}) \times \cdots \times
\mathcal{G}(z_{\phi(1)-1}, z_{\phi(1)-1}) \\ & \times
\mathcal{G}(z_{\phi(1)}, z'_{1}) \times \mathcal{G}(z'_{1},
z_{\phi(1)}) \times \cdots \times \mathcal{G}(z_{\phi(m)}, z'_{m}), \\
(\mathcal{O}^{\otimes}_{\mathcal{G}})^{\alpha}(g(Z),
(x,y))  \simeq &\,\,  \mathcal{G}(x, z_{0}) \times \mathcal{G}(z_{1}, z_{1})
\times \cdots \times \mathcal{G}(z_{n-1},z_{n-1}) \times
\mathcal{G}(z_{n}, y).
\end{split}
\]
Under these equivalences our commutative square is the product of the
squares
\nolabelcsquare{*}{*}{\mathcal{G}(z_{j},z_{j})}{\mathcal{G}(z_j, z_j)}
for $j$ not in the image of $\phi$, \csquare{\mathcal{G}(z_{0},
  z'_{0}) \times \mathcal{G}(z_{n},z'_{m})}{\mathcal{G}(z_0,x) \times
  \mathcal{G}(z_n, y)}{\mathcal{G}(z'_{0}, z_{0}) \times
  \mathcal{G}(z_{n},z'_{m})}{\mathcal{G}(x,z_0) \times
  \mathcal{G}(z_n, y),}{}{(i, \id)}{(i, \id)}{} and
\csquare{\mathcal{G}(z_{\phi(i)}, z'_{i})}{*}
{\mathcal{G}(z_{\phi(i)}, z'_{i}) \times \mathcal{G}(z'_{i},
  z_{\phi(i)})}{\mathcal{G}(z_{\phi(i)}, z_{\phi(i)})}{}{(\id, i)}{}{}
for $i = 1, \ldots, m-1$.  The squares of the first kind are clearly
homotopy Cartesian, the second square is homotopy Cartesian since the
maps induced by the involution $i$ are equivalences, and the 
squares of the third kind are homotopy Cartesian since $\mathcal{G}$ is a simplicial
groupoid.
\end{proof}

\begin{cor}\label{cor:LDopXisOX}
  Let $X$ be a space and $\mathcal{X}$ a fibrant simplicial groupoid
  such that the Kan complex $\mathrm{N}\mathcal{X}$ is equivalent to
  $X$. Then the composite map \[\simp^{\op}_{X} \simeq
  \mathrm{N}\mathcal{D}(\mathcal{X}) \to
  \mathrm{N}\mathcal{O}^{\otimes}_{\mathcal{X}}\] induces an
  equivalence of \nsiopds{} $L_{\txt{gen}}\simp^{\op}_{X} \isoto
  \mathrm{N}\mathcal{O}^{\otimes}_{\mathcal{X}}$.
\end{cor}

\begin{cor}\label{cor:LSimpProd}
  The functor $L_{\txt{gen}}(\simp^{\op}_{(\blank)}) \colon \mathcal{S} \to \OpdIns$
  preserves products.
\end{cor}
\begin{proof}
  Given spaces $X$ and $Y$, there exist fibrant simplicial groupoids
  $\mathcal{X}$ and $\mathcal{Y}$ such that $\mathrm{N}\mathcal{X}
  \simeq X$ and $\mathrm{N}\mathcal{Y} \simeq Y$. Then by
  Corollary~\ref{cor:LDopXisOX} we have a commutative diagram
  \nolabelcsquare{L_{\txt{gen}}\simp^{\op}_{X \times Y}}{L_{\txt{gen}}\simp^{\op}_{X}
    \times_{\simp^\op}
    L_{\txt{gen}}\simp^{\op}_{Y}}{\mathrm{N}\mathcal{O}^{\otimes}_{\mathcal{X}
      \times
      \mathcal{Y}}}{\mathrm{N}(\mathcal{O}^{\otimes}_{\mathcal{X}}
    \times_{\simp^\op}\mathcal{O}^{\otimes}_{\mathcal{Y}})} where the
  vertical maps are equivalences. It is clear from the definition that
  $\mathcal{O}_{\mathcal{X} \times \mathcal{Y}} \simeq
  \mathcal{O}_{\mathcal{X}} \times \mathcal{O}_{\mathcal{Y}}$, so the
  natural map $\mathcal{O}^{\otimes}_{\mathcal{X} \times \mathcal{Y}}
  \to \mathcal{O}^{\otimes}_{\mathcal{X}} \times_{\simp^{\op}}
  \mathcal{O}^{\otimes}_{\mathcal{Y}}$ is a weak equivalence of fibrant
  simplicial categories. By the 2-out-of-3 property the top horizontal
  map in the commutative square is therefore an equivalence of
  \icats{}.
\end{proof}

\subsection{The $\infty$-Category of Categorical
  Algebras}\label{subsec:catalg}
We are now ready to define and study the \icats{}
$\AlgCat(\mathcal{V})$ of categorical algebras:
\begin{defn}
  Suppose $\mathcal{V}$ is a monoidal \icat{}. The \icat{}
  $\AlgCatV$ is defined by the pullback square
  \csquare{\AlgCatV}{\Alg(\mathcal{V})}{\mathcal{S}}{\OpdIns.}{}{}{}{L_{\txt{gen}}\simp^\op_{(\blank)}}
  where the right vertical map is the algebra fibration from
  \S\ref{subsec:AlgFib} and the lower horizontal map sends a space $S$
  to the \nsiopd{} $L_{\txt{gen}}\simp^{\op}_{S}$ associated to the
  \gnsiopd{} $\simp^{\op}_{S}$. The objects of $\AlgCatV$ are thus
  categorical algebras in $\mathcal{V}$ and its 1-morphisms
  are $\mathcal{V}$-functors as defined in
  \S\ref{subsec:FTgeniopd}. We will refer to $\AlgCatV$ as the
  \defterm{\icat{} of categorical algebras}.
\end{defn}

\begin{remark}
  Since $\mathcal{V}$ is a monoidal \icat{}, and so in
  particular a \nsiopd{}, we could equivalently have defined
  $\AlgCatV$ using the analogue of the algebra fibration
  over the base $\OpdInsg$, since there is natural equivalence
  $\Alg_{L_{\txt{gen}}\simp^{\op}_{S}}(\mathcal{V}) \isoto
  \Alg_{\simp^{\op}_{S}}(\mathcal{V})$ for every space $S$.
\end{remark}

Pulling back the fibration of trivial algebras in the same way, we get
the functor that forms the free $\mathcal{V}$-\icat{} on a \emph{graph}:
\begin{defn}
  Let $\mathcal{V}$ be a monoidal \icat{}. The \icat{}
  $\GraphIV$ of \emph{$\mathcal{V}$-graphs} is defined by the pullback
  \csquare{\GraphIV}{\Algtriv(\mathcal{V})}{\mathcal{S}}{\OpdIns.}{}{}{}{L\simp^\op_{(\blank)}}
  Thus the fibre of $\GraphIV$ at $X \in \mathcal{S}$ is $\Fun(X\times
  X, \mathcal{V})$. By Remark~\ref{rmk:AlgTrivPullback} we also get a
  pullback square
  \csquare{\GraphIV}{\mathcal{F}_{\mathcal{V}}}{\mathcal{S}}{\mathcal{S},}{}{}{}{\Delta}
  where $\Delta$ is the diagonal functor that sends $S$ to $S \times
  S$, and $\mathcal{F}_{\mathcal{V}} \to \mathcal{S}$ is the Cartesian
  fibration associated to the functor $\mathcal{S}^{\op} \to \CatI$
  sending $S$ to $\Fun(S, \mathcal{V})$.
\end{defn}

\begin{remark}\label{rmk:GraphAdj} 
  The pullback of the left adjoint $\tau_{!}$ of $\tau^{*}$ gives a
  functor \[F \colon \GraphIV \to \AlgCatV\] left adjoint to $U$.
\end{remark}

\begin{propn}\label{propn:AlgCatPres}
  Suppose $\mathcal{V}$ is a presentably monoidal \icat{}, i.e. the
  \icat{} $\mathcal{V}$ is presentable and the tensor product on
  $\mathcal{V}$ preserves small colimits separately in each
  variable. Then $\AlgCatV$ is a presentable \icat{}.
\end{propn}

\begin{remark}
  Proposition~\ref{propn:AlgCatPres} can be seen as an \icatl{} version of a
  theorem of Kelly and Lack~\cite[Theorem 4.5]{KellyLackVCatPres}. The
  fact that this 1-categorical result is comparatively recent, whereas
  the \icatl{} variant is one of the first steps in our setup,
  underscores the importance of presentability in the \icatl{}
  context.
\end{remark}

Using Corollary~\ref{cor:AlgOcolim} it is easy to show that $\AlgCatV$ has
colimits:
\begin{lemma}\label{lem:AlgCatcolim}
  Suppose $\mathcal{V}$ is a monoidal \icat{} compatible
  with small colimits (i.e. the tensor product on $\mathcal{V}$
  preserves colimits separately in each variable). Then $\AlgCatV$
  has all small colimits.
\end{lemma}
\begin{proof}
  By Lemma~\ref{lem:AlgCoCart}, the fibration $\pi \colon
  \Alg(\mathcal{V}) \to \OpdIns$ is both Cartesian and
  coCartesian, hence the same is true of its pullback $p \colon
  \AlgCatV \to \mathcal{S}$. Moreover, the fibres
  $\Alg_{\simp^{\op}_{X}}(\mathcal{V})$ have all colimits
  by Corollary~\ref{cor:AlgOcolim} and the functors $f_{!}$ induced by
  morphisms $f$ in $\mathcal{S}$ preserve colimits, being left
  adjoints. Thus $p$ satisfies the conditions of
  \cite[Lemma 9.8]{freepres}, which implies that
  $\AlgCat(\mathcal{V})$ has small colimits.
\end{proof}

\begin{proof}[Proof of Proposition~\ref{propn:AlgCatPres}]
  The \icat{} $\AlgCatV$ has colimits by Lemma~\ref{lem:AlgCatcolim},
  so it remains to prove that it is accessible. But in the pullback
  square
  \csquare{\AlgCatV}{\Alg(\mathcal{V})}{\mathcal{S}}{\OpdIns.}{}{}{}{L_{\txt{gen}}\simp^\op_{(\blank)}}
  the right vertical morphism is an accessible functor between
  accessible \icats{} by Proposition~\ref{propn:AlgVPres}. Moreover,
  by Proposition~\ref{propn:DopXfilt} the bottom horizontal morphism
  preserves filtered colimits (since $L_{\txt{gen}}$ is a left adjoint
  and so preserves all colimits), and thus is in particular also an
  accessible functor. It then follows from \cite[Proposition
  5.4.6.6]{HTT} that $\AlgCatV$ is also an accessible \icat{}.
\end{proof}

Our next goal is to prove that the \icat{} $\AlgCatV$ is a lax
monoidal functor in $\mathcal{V}$, with respect to the
Cartesian product of monoidal \icats{} and the Cartesian product of
\icats{}. Knowing this will allow us to conclude,
for example, that if $\mathcal{V}$ is a symmetric monoidal
\icat{} then there is an induced symmetric monoidal structure on
$\AlgCatV$. We first observe that $\AlgCat(\mathcal{V})$
is indeed functorial in $\mathcal{V}$:

\begin{defn}
  As in \S\ref{subsec:AlgFib}, let $\Algco \to \OpdIns \times
  (\widehat{\txt{Opd}}_{\infty}^{\txt{ns}})^{\op}$ be a Cartesian
  fibration classifying the functor $\Alg_{(\blank)}(\blank)$. Then
  we define $\AlgCatco$ by the pullback square
  \nolabelcsquare{\AlgCatco}{\Algco}{\mathcal{S} \times
    (\widehat{\txt{Mon}}_{\infty}^{\txt{lax}})^{\op}}{\OpdIns
    \times (\widehat{\txt{Opd}}_{\infty}^{\txt{ns}})^{\op},} where
  the bottom horizontal functor is the product of the functor
  $\simp^{\op}_{(\blank)}$ and the opposite of the inclusion of the
  full subcategory of large monoidal \icats{} into
  $\widehat{\txt{Opd}}_{\infty}^{\txt{ns}}$.
\end{defn}

\begin{lemma}
  $\AlgCat(\mathcal{V})$ is functorial in
  $\mathcal{V}$ with respect to lax monoidal functors.
\end{lemma}
\begin{proof}
  The composite $\AlgCatco \to
  (\widehat{\txt{Mon}}_{\infty}^{\txt{lax}})^{\op}$ is a Cartesian
  fibration classifying a functor $\mathcal{V}^{\otimes} \mapsto
  \AlgCat(\mathcal{V})$.
\end{proof}

\begin{remark}\label{rmk:ordinaryasenricat}
  If $\mathbf{V}$ is an ordinary monoidal category, we can identify
  the usual category of $\mathbf{V}$-enriched categories with the full
  subcategory of $\AlgCat(\mathbf{V})$ spanned by the
  $\mathbf{V}$-enriched \icats{} with sets of objects. In particular,
  taking $\mathbf{V}$ to be the category $\Set$ of sets, with the
  Cartesian product as monoidal structure, we can identify the usual
  category $\Cat$ of categories with a full subcategory of
  $\AlgCat(\Set)$. Since the inclusion $\Set \hookrightarrow
  \mathcal{S}$ preserves products, this allows us to consider ordinary
  categories as $\mathcal{S}$-\icats{}.
\end{remark}

\begin{propn}\label{propn:AlgCatLaxMon}
  $\AlgCat(\blank)$ is lax monoidal with respect to the Cartesian
  product of monoidal \icats{}.
\end{propn}
\begin{proof}
  The functor $L_{\txt{gen}}\simp^{\op}_{(\blank)}$ is  monoidal
  with respect to the Cartesian products of spaces and \nsiopds{}, by
  Corollary~\ref{cor:LSimpProd}. The result therefore follows by the
  same proof as that of Proposition~\ref{propn:AlgLaxMon}.
\end{proof}

\begin{cor}\label{cor:AlgCatOMon}
  Let $\mathcal{O}$ be a symmetric \iopd{}, and
  suppose $\mathcal{V}$ is an $\mathcal{O} \otimes
  \mathbb{E}_{1}$-monoidal \icat{}. Then $\AlgCatV$ is an
  $\mathcal{O}$-monoidal \icat{}.
\end{cor}
\begin{proof}
  It follows from Proposition~\ref{propn:AlgCatLaxMon} that for any
  symmetric \iopd{} $\mathcal{O}$, the functor
  $\AlgCat(\blank)$ takes an $\mathcal{O}$-algebra in
  $\widehat{\txt{Mon}}_{\infty}$ to an
  $\mathcal{O}$-algebra in $\LCatI$. The inclusion
  $\widehat{\txt{Mon}}_{\infty} \to
  \widehat{\txt{Mon}}_{\infty}^{\txt{lax}}$ clearly preserves
  Cartesian products, and by Proposition~\ref{propn:nssymmeq} we can
  identify $\widehat{\txt{Mon}}_{\infty}$ with the \icat{}
  $\Alg^{\Sigma}_{\mathbb{E}_{1}}(\LCatI)$ of
  $\mathbb{E}_{1}$-monoidal \icats{}. By \cite[Remark 2.4.2.6]{HA} a
  large $\mathcal{O} \otimes \mathbb{E}_{1}$-monoidal \icat{} is the
  same thing as an $\mathcal{O}$-algebra in
  $\Alg^{\Sigma}_{\mathbb{E}_{1}}(\LCatI)$, and so
  the functor $\AlgCat(\blank)$ indeed takes $\mathcal{O} \otimes
  \mathbb{E}_{1}$-monoidal \icats{} to $\mathcal{O}$-monoidal
  \icats{}.
\end{proof}

\begin{cor}\label{cor:AlgCatEnMon}\ 
    \begin{enumerate}[(i)]
  \item Suppose $\mathcal{V}$ is an
    $\mathbb{E}_{n}$-monoidal \icat{}. Then $\AlgCatV$ is an
    $\mathbb{E}_{n-1}$-monoidal \icat{}.
  \item Suppose $\mathcal{V}$ is a symmetric monoidal
    \icat{}. Then $\AlgCatV$ is a symmetric monoidal \icat{}.
  \end{enumerate}
\end{cor}
\begin{proof}
  This follows from combining Corollary~\ref{cor:AlgCatOMon} with
  \cite[Theorem 5.1.2.2]{HA}.
\end{proof}

Our next goal is to show that the functor $\AlgCat(\blank)$, when
restricted to presentably monoidal \icats{}, is lax monoidal with
respect to the tensor product of presentable \icats{}. We first
observe that restricting in this way does indeed give a functor to
presentable \icats{}:
\begin{propn}\label{propn:AlgCatPresFtr}
  The restriction of $\AlgCat(\blank)$ to the \icat{} $\MonPr$ of
  presentably monoidal \icats{} factors through the subcategory $\PresI$
  of $\LCatI$ of presentable \icats{} and colimit-preserving functors.
\end{propn}
\begin{proof}
  If $\mathcal{V}$ is presentably monoidal, then
  $\AlgCat(\mathcal{V})$ is presentable by
  Proposition~\ref{propn:AlgCatPres}. Moreover, it follows by the same proof
  as that of Proposition~\ref{propn:StrMonColim2} that a
  monoidal functor $F \colon \mathcal{V}^{\otimes} \to
  \mathcal{W}^{\otimes}$ such that $F_{[1]}$ preserves
  colimits induces a colimit-preserving functor
  $\AlgCat(\mathcal{V}) \to
  \AlgCat(\mathcal{W})$.
\end{proof}

Next we see that when restricted to categorical algebras, the external
product $\boxtimes$ of \S\ref{subsec:AlgFib} preserves colimits in
each variable:
\begin{propn}\label{propn:BoxProdColim}
  Let $\mathcal{V}$ be a monoidal \icat{}, and suppose that
  $\mathcal{C}$ is a categorical algebra in $\mathcal{V}$. 
  Then $\mathcal{C} \boxtimes \blank \colon
  \AlgCat(\mathcal{W}) \to \AlgCat(\mathcal{V}
  \times \mathcal{W})$ preserves colimits.
\end{propn}
\begin{proof}
  Since the Cartesian product of spaces preserves colimits in each
  variable, it suffices to prove that $\mathcal{C} \boxtimes (\blank)$
  preserves colimits fibrewise and preserves coCartesian arrows. This
  follows from Lemma~\ref{lem:FibBoxProdCoCart} and
  Proposition~\ref{propn:FibBoxProdColim}.
\end{proof}

\begin{cor}\label{cor:AlgCatLaxMonPr}
  The functor $\AlgCat(\blank) \colon \MonPr \to \PresI$
  is lax monoidal with respect to the tensor product of presentable
  \icats{}.
\end{cor}
\begin{proof}
  We have constructed a lax monoidal functor \[\AlgCat(\blank)
  \colon (\widehat{\txt{Mon}}_{\infty}^{\txt{lax}})^{\times}
  \to \LCatI^{\times}.\] By Proposition \ref{propn:BoxProdColim} and
  Propositon \ref{propn:AlgCatPresFtr}, the
  composite \[(\MonPr)^{\otimes} \to
  (\widehat{\txt{Mon}}_{\infty}^{\txt{lax}})^{\times} \to
  \LCatI^{\times}\] factors through $\PresI^{\otimes}$ as defined in
  \cite[Notation 4.8.1.2]{HA}.
\end{proof}

\begin{cor}\label{cor:prestensAlgCatS}
  If $\mathcal{V}$ is presentably monoidal, then
  $\AlgCatV$ is tensored over $\AlgCat(\mathcal{S})$, and
  the tensoring operation
  \[ \AlgCat(\mathcal{S}) \times \AlgCatV \to
  \AlgCatV \]
  preserves colimits separately in each variable.
\end{cor}
\begin{proof}
  By Remark~\ref{rmk:MonPrMon}, the unit of the tensor product of
  presentably monoidal \icats{} is $\mathcal{S}^{\times}$, and so this
  is a commutative algebra object in the \icat{} $\MonPr$ by
  \cite[Corollary 3.2.1.9]{HA}. Any presentably monoidal \icat{}
  $\mathcal{V}^{\otimes}$ is moreover canonically a module over this
  commutative algebra object. Since $\AlgCat(\blank)$ is lax monoidal
  with respect to $\otimes$, it follows that the \icat{}
  $\AlgCat(\mathcal{V})$ is a module over the presentably symmetric
  monoidal \icat{} $\AlgCat(\mathcal{S})$ in $\PresI$. In other words,
  the \icat{} $\AlgCat(\mathcal{V})$ is tensored over
  $\AlgCat(\mathcal{S})$ and the tensoring operation preserves
  colimits separately in each variable.
\end{proof}
\begin{defn}
  If $\mathcal{V}$ is a presentably monoidal \icat{},
  $\mathcal{C}$ is a $\mathcal{V}$-\icat{}, and $\mathcal{X}$ is an
  $\mathcal{S}$-\icat{}, then we denote their tensor by $\mathcal{C}
  \otimes \mathcal{X}$. For fixed $\mathcal{X}$ the functor
  $\mathcal{C} \mapsto \mathcal{C} \otimes \mathcal{X}$ preserves
  colimits, and hence has a right adjoint ---
  i.e. $\AlgCat(\mathcal{V})$ is also cotensored over
  $\AlgCat(\mathcal{S})$; we denote the cotensor of
  $\mathcal{C}$ and $\mathcal{X}$ by $\mathcal{C}^{\mathcal{X}}$. If
  $\mathcal{D}$ is another $\mathcal{V}$-\icat{} we thus have a
  canonical equivalence 
  \[ \Map(\mathcal{D}, \mathcal{C}^{\mathcal{X}}) \simeq
  \Map(\mathcal{D} \otimes \mathcal{X}, \mathcal{C}).\]
\end{defn}

The \icat{} of categorical algebras is well-behaved with respect to
adjunctions:
\begin{lemma}\label{lem:StrMonAdjAlgCat}
  Suppose $\mathcal{V}$ and $\mathcal{W}$ are
  presentably monoidal \icats{}
  and $F \colon \mathcal{V}^{\otimes} \to
  \mathcal{W}^{\otimes}$ is a monoidal functor such that the
  underlying functor $f \colon \mathcal{V} \to \mathcal{W}$ preserves
  colimits. Let $g \colon \mathcal{W} \to \mathcal{V}$ be a right
  adjoint of $f$, and let $G \colon \mathcal{W}^{\otimes} \to
  \mathcal{V}^{\otimes}$ be the lax monoidal structure on $g$ given by
  Proposition~\ref{propn:rightadjlaxmon}. Then the functors
  \[ F_{*} : \AlgCatV \rightleftarrows \AlgCat(\mathcal{W})
  : G_{*} \] are adjoint.
\end{lemma}
\begin{proof}
  Let $\mathcal{C}$ be a $\mathcal{V}$-\icat{} with space of objects
  $S$, and let $\mathcal{D}$ be a $\mathcal{W}$-\icat{} with space of
  objects $T$. We must show that the natural map $\Map(\mathcal{C},
  G_{*}\mathcal{D}) \to \Map(F_{*}\mathcal{C}, \mathcal{D})$ is an
  equivalence. We have a commutative triangle of spaces
  \nolabelopctriangle{\Map(\mathcal{C},
  G_{*}\mathcal{D})}{\Map(F_{*}\mathcal{C}, \mathcal{D})}{\Map(X, Y)}
  so it suffices to show that we have an equivalence on the fibres
  over each $\phi \colon X \to Y$. But we can identify the map on this
  fibre with
  \[
  \Map_{\Alg_{\simp^{\op}_{S}}(\mathcal{V})}(\mathcal{C},
  G_{*}\phi^{*}\mathcal{D}) \to
  \Map_{\Alg_{\simp^{\op}_{S}}(\mathcal{W})}(F_{*}\mathcal{C},
  \phi^{*}\mathcal{D}), \] which is an equivalence since $F_{*}$ and
  $G_{*}$ are adjoint functors on $\simp^{\op}_{S}$-algebras by
  Proposition~\ref{propn:rightadjlaxmon}.
\end{proof}

\begin{ex}\label{ex:TisTensorIV}
  Suppose $\mathcal{V}$ is a presentably monoidal \icat{}. Then
  there is a unique colimit-preserving functor $t \colon \mathcal{S}
  \to \mathcal{V}$ that sends $*$ to $I_{\mathcal{V}}$. This has right
  adjoint $u \colon \mathcal{V} \to \mathcal{S}$ given by
  $\Map_{\mathcal{V}}(I_{\mathcal{V}}, \blank)$. Using the monoidal
  structure on $\MonPr$ we get a monoidal functor
  \[ T \colon \mathcal{S}^{\times} \simeq \mathcal{S}^{\times}
  \times_{\simp^{\op}} \simp^{\op} \xto{\id \times_{\simp^{\op}}
      I_{\mathcal{V}}} \mathcal{S}^{\times} \times_{\simp^{\op}}
    \mathcal{V}^{\otimes} \to \mathcal{S}^{\times} \otimes
    \mathcal{V}^{\otimes} \isoto \mathcal{V}^{\otimes} \]
  extending $t$. By Proposition~\ref{propn:rightadjlaxmon} there is a
  lax monoidal functor $U \colon \mathcal{V}^{\otimes} \to
  \mathcal{S}^{\times}$ extending $u$ such that for any \nsiopd{}
  $\mathcal{O}$ we have an adjunction
  \[ T_{*} : \Alg_{\mathcal{O}}(\mathcal{S})
  \rightleftarrows
  \Alg_{\mathcal{O}}(\mathcal{V}) : U_{*}. \]
  Then by Lemma~\ref{lem:StrMonAdjAlgCat} we
  have an adjunction
  \[ T_{*} : \AlgCat(\mathcal{S}) \rightleftarrows
  \AlgCat(\mathcal{V}) : U_{*}. \] Unravelling the
  definitions, it is clear that we can identify the functor $T_{*}$
  with the operation of tensoring with the unit $I_{\mathcal{V}} \in
  \AlgCat(\mathcal{V})$.
\end{ex}

Our final goal in this subsection is to show that
$\AlgCat(\mathcal{V})$ behaves very nicely with respect to monoidal
localizations of $\mathcal{V}$. First we must introduce some notation:
\begin{defn}\label{defn:SigmaVCat}
  Let $\mathcal{V}$ be a presentably monoidal \icat{}. The
  functor \[\Fun(\{0,1\}\times\{0,1\}, \mathcal{V}) \to \mathcal{V}\]
  given by evaluation at $(0,1)$ clearly has a left adjoint given by
  sending $V \in \mathcal{V}$ to the functor $\{0,1\}\times\{0,1\} \to
  \mathcal{V}$ that takes $(0,1)$ to $V$ and the other elements to
  $\emptyset$. Let $\Sigma \colon \mathcal{V} \to
  \Alg_{\simp^{\op}_{\{0,1\}}}(\mathcal{V})$ denote the
  composite of this with the free algebra functor $\tau_{!} \colon
  \Fun(\{0,1\}\times\{0,1\}, \mathcal{V}) \to
  \Alg_{\simp^{\op}_{\{0,1\}}}(\mathcal{V})$. Thus for any
  categorical algebra $\mathcal{C}$ in $\mathcal{V}$ with space of
  objects $\{0,1\}$ we have 
  \[\Map_{\Alg_{\simp^{\op}_{\{0,1\}}}(\mathcal{V})}(\Sigma V,
  \mathcal{C}) \simeq \Map_{\mathcal{V}}(V, \mathcal{C}(0,1)).\] We
  also write $\Sigma$ for the functor $\mathcal{V} \to \AlgCatV$
  obtained by composing this with the inclusion of the fibre at
  $\{0,1\}$. Thus for any $\mathcal{V}$-\icat{} $\mathcal{C}$ with
  space of objects $S$ the fibre of
  \[ \Map(\Sigma V, \mathcal{C}) \to \Map(\{0,1\}, S) \simeq S \times
  S\] at $(X,Y)$ is $\Map_{\mathcal{V}}(V, \mathcal{C}(X,Y))$.
\end{defn}

\begin{propn}\label{propn:monlocalgcat}
  Let $\mathcal{V}$ be a presentably monoidal \icat{} and
  suppose $L \colon \mathcal{V} \to \mathcal{W}$ is a monoidal
  accessible localization with fully faithful right adjoint $i \colon
  \mathcal{W} \hookrightarrow \mathcal{V}$. Let 
  $i^{\otimes} \colon \mathcal{W}^{\otimes} \hookrightarrow
  \mathcal{V}^{\otimes}$ and $L^{\otimes} \colon
  \mathcal{V}^{\otimes} \to \mathcal{W}^{\otimes}$ be as in
  Proposition~\ref{propn:moncomploc}. Suppose $L$ exhibits
  $\mathcal{W}$ as the localization of $\mathcal{V}$ with respect to a
  set of morphisms $S$. Then there is an adjunction
  \[ L^{\otimes}_{*} : \AlgCat(\mathcal{V}) \rightleftarrows
  \AlgCat(\mathcal{W}) : i^{\otimes}_{*} \] which exhibits $\AlgCat(\mathcal{W})$ as the
  localization of $\AlgCat(\mathcal{V})$ with respect to
  $\Sigma(S)$. Moreover, if $\mathcal{V}$ is at least
  $\mathbb{E}_{2}$-monoidal then this localization is again monoidal.
\end{propn}
\begin{proof}
  It follows from Lemma~\ref{lem:monlocadj} that the lax monoidal
  structure on $i$ provided by Proposition~\ref{propn:rightadjlaxmon}
  is given by $i^{\otimes}$, so by Lemma~\ref{lem:StrMonAdjAlgCat} we
  indeed have an adjunction $L^{\otimes}_{*} \dashv i^{\otimes}_{*}$. 

  To see that this is a localization we must show that
  $i^{\otimes}_{*}$ is fully faithful. To see this it suffices to
  prove that for every categorical algebra $\mathcal{C} \in \AlgCatV$
  with space of objects $X$ the counit
  $L^{\otimes}_{*}i^{\otimes}_{*}\mathcal{C} \to \mathcal{C}$ is an equivalence
  in $\Alg_{\simp^{\op}_{X}}(\mathcal{V})$. By
  Lemma~\ref{lem:AlgCons} this is equivalent to the induced morphism
  of underlying graphs being an equivalence, i.e. to
  $Li\mathcal{C}(C,D) \to \mathcal{C}(C,D)$ being an equivalence in
  $\mathcal{V}$ for all $C,D \in \mathcal{C}$. But this is true since
  $i$ is fully faithful.

  Next we must show that $\mathcal{C} \in
  \AlgCat(\mathcal{V})$ lies in
  $\AlgCat(\mathcal{W})$ \IFF{} it is local with respect to
  the morphisms in $\Sigma(S)$. Consider a map $f \colon A \to B$ in
  $\mathcal{V}$. Then the induced map 
  \[ \Map_{\AlgCatV}(\Sigma B, \mathcal{C}) \to \Map_{\AlgCatV}(\Sigma
  A, \mathcal{C}) \]
  is an equivalence in $\mathcal{S}$
  \IFF{} it induces an equivalence on the fibres over all points of
  $\Map_{\mathcal{S}}(S^{0}, X)$. Using the universal property of
  $\Sigma$ we conclude that it is an equivalence
  \IFF{} for all objects $C, D \in \mathcal{C}$ the induced map 
  \[ \Map_{\mathcal{V}}(B, \mathcal{C}(C,D)) \to \Map_{\mathcal{V}}(A,
  \mathcal{C}(C,D) \] 
  is an equivalence. Thus $\mathcal{C}$ is local with respect to the
  maps in $\Sigma(S)$ \IFF{} all the mapping objects
  $\mathcal{C}(C,D)$ are local with respect to the maps in $S$,
  i.e. \IFF{} these all lie in $\mathcal{W}$. Thus
  $\AlgCat(\mathcal{W})$ is indeed the localization of
  $\AlgCat(\mathcal{V})$ with respect to $\Sigma(S)$.

  Finally, it is clear from the construction of the monoidal structure
  on $\AlgCat(\mathcal{V})$ that the localization will again
  be monoidal when this exists.
\end{proof}

\subsection{Categorical Algebras in Spaces}\label{subsec:catalgspace}
In this subsection we will prove that the \icat{}
$\AlgCat(\mathcal{S})$ of categorical algebras in spaces is
equivalent to the \icat{} $\SegI$ of Segal spaces. These are an
alternative definition of $(\infty,1)$-categories introduced by
Rezk~\cite{RezkCSS}. We begin by briefly reviewing the definition in
the \icatl{} context:

\begin{defn}
  Suppose $\mathcal{C}$ is an \icat{} with finite limits. A
  \defterm{category object} in $\mathcal{C}$ is a simplicial object $F
  \colon \simp^{\op}\to \mathcal{C}$ such that for each $n$ the map 
  \[ F_{n} \to F_{1} \times_{F_{0}} \cdots \times_{F_{0}} F_{1} \]
  induced by the inclusions $\{i,i+1\} \hookrightarrow [n]$ and $\{i\}
  \hookrightarrow [n]$ is an equivalence.  A \defterm{Segal space} is
  a category object in the \icat{} $\mathcal{S}$ of spaces.
\end{defn}
Let $\delta_{n}$ denote the simplicial space obtained from the
simplicial set $\Delta^{n}$ by composing with the inclusion
$\Set \hookrightarrow \mathcal{S}$. A simplicial space is then a Segal
space \IFF{} it is local with respect to the map
\[ \txt{seg}_{n} \colon \delta_{n} \to \delta_{1} \amalg_{\delta_{0}} \cdots
\amalg_{\delta_{0}} \delta_{1}. \]
\begin{defn}
  Let $\Seg(\mathcal{S})$ denote the full subcategory of
  $\Fun(\simp^{\op}, \mathcal{S})$ spanned by the Segal spaces; this
  is the localization of $\Fun(\simp^{\op}, \mathcal{S})$ with respect
  to the maps $\txt{seg}_{*}$.
\end{defn}

The key result for the comparison is the following:
\begin{propn}\label{propn:KanExtSegalSp}
  Let $S$ be a space, and let $\pi \colon \simp^{\op}_{S} \to
  \simp^{\op}$ be the usual projection. Let $\pi_{!} \colon
  \Fun(\simp^{\op}_{S}, \mathcal{S}) \to \Fun(\simp^{\op},
  \mathcal{S})$ be the functor given by left Kan extension along
  $\pi$. Then a functor $F \colon \simp^{\op}_{S}\to \mathcal{S}$ is a
  $\simp^{\op}_{S}$-monoid \IFF{} $\pi_{!}F$ is a Segal space.
\end{propn}
\begin{proof}
  It is clear that $\pi_{!}F([0])$ is equivalent to $S$. We must thus
  show that the Segal morphism
  \[ \pi_{!}F([n]) \to \pi_{!}F([1]) \times_{S} \cdots \times_{S}
  \pi_{!}F([1]) =: (\pi_{!}F)^{\Seg}_{[n]} \] is an equivalence \IFF{}
  $F$ is a $\simp^{\op}_{S}$-monoid. Since $\pi$ is a coCartesian
  fibration, we have an equivalence $\pi_{!}F([n]) \simeq \colim_{\xi
    \in S^{\times (n+1)}} F(\xi)$. It thus suffices to show that
  $(\pi_{!}F)^{\Seg}_{[n]}$ is also a colimit of this diagram \IFF{}
  $F$ is a $\simp^{\op}_{S}$-monoid. There is a natural transformation
  $(S^{\times (n+1)})^{\triangleright} \to \Fun(\Delta^{1},
  \mathcal{S})$ that sends $\xi \in S^{\times (n+1)}$ to $F(\xi) \to
  \xi$ and $\infty$ to $(\pi_{!}F)^{\Seg}_{[n]} \to S^{\times (n+1)}$;
  since $\mathcal{S}$ is an $\infty$-topos, by \cite[Theorem
  6.1.3.9]{HTT} the colimit is $(\pi_{!}F)^{\Seg}_{[n]}$ \IFF{} this
  natural transformation is Cartesian. Since $S^{\times (n+1)}$ is a
  space, this is equivalent to the square
  \nolabelsmallcsquare{F(\xi)}{(\pi_{!}F)^{\Seg}_{[n]}}{\xi}{S^{\times
      (n+1)}} being a pullback for all $\xi$, so we are reduced to
  showing that the fibre of $(\pi_{!}F)^{\Seg}_{[n]} \to S^{\times
    (n+1)}$ at $\xi$ is $F(\xi)$ \IFF{} $F$ is a
  $\simp^{\op}_{S}$-monoid. Since limits commute, if $\xi$ is $(s_{0},
  \ldots, s_{n})$ this fibre is the iterated fibre product
  \[ (\pi_{!}F[1])_{(s_{0},s_{1})} \times_{(\pi_{!}F[0])_{(s_{1})}}
  \cdots \times_{(\pi_{!}F[0])_{(s_{n-1})}}
  (\pi_{!}F[1])_{(s_{n-1},s_{n})}.\]
  But using \cite[Theorem 6.1.3.9]{HTT} again it is clear that the
  natural maps $F(x,y) \to (\pi_{!}F[1])_{(x,y)}$ and $* \simeq F(x)
  \to (\pi_{!}F)_{(x)}$ are equivalences for all $x,y \in S$. Thus the
  map $F(\xi) \to (\pi_{!}F)^{\Seg}_{[n],\xi}$ is equivalent to the
  natural map
  \[ F(\xi) \to F(s_{0},s_{1}) \times \cdots \times F(s_{n-1},
  s_{n}).\] By definition this is an equivalence for all $\xi \in
  \simp^{\op}_{S}$ \IFF{} $F$ is a $\simp^{\op}_{S}$-monoid, which
  completes the proof.
\end{proof}

\begin{cor}\label{cor:DopSMndEq} 
  Let $S$ be a space, and let $\pi \colon \simp^{\op}_{S} \to
  \simp^{\op}$ denote the canonical projection. By
  \cite[Corollary 8.6]{freepres} the functor
  \[ \pi_{!} \colon \Fun(\simp^{\op}_{S}, \mathcal{S}) \to
  \Fun(\simp^{\op}, \mathcal{S})_{/i_{*}S} \] given by left Kan
  extension is an equivalence. Under this equivalence, the full
  subcategory $\Mon_{\simp^{\op}_{S}}(\mathcal{S})$ of
  $\simp^{\op}_{S}$-monoids corresponds to the full subcategory of
  $\Fun(\simp^{\op}, \mathcal{S})_{/i_{*}S}$ spanned by the Segal
  spaces $Y_{\bullet}$ such that $Y_{0}\simeq S$ and the map
  $Y_{\bullet} \to i_{*}S$ is given by the adjunction unit
  $Y_{\bullet} \to i_{*}i^{*}Y_{\bullet} \simeq i_{*}S$. 
\end{cor}
\begin{proof}
  It is clear that $\pi_{!}$ takes
  $\Mon_{\simp^{\op}_{S}}(\mathcal{S})$ into the full
  subcategory of $\Fun(\simp^{\op}, \mathcal{S})_{/i_{*}S}$ spanned by
  simplicial spaces $Y_{\bullet}$ with $Y_{0} \simeq S$ and the map
  $Y_{\bullet} \to i_{*}S$ given by the adjunction unit $Y_{\bullet}
  \to i_{*}i^{*}Y \simeq i_{*}S$. The result therefore follows by
  Proposition~\ref{propn:KanExtSegalSp}.
\end{proof}

\begin{cor}\label{cor:DopSMndEq2} 
  Let $S$ be a space, and let $\pi \colon \simp^{\op}_{S} \to
  \simp^{\op}$ denote the canonical projection. The functor $\pi_{!}
  \colon \Fun(\simp^{\op}_{S}, \mathcal{S}) \to \Fun(\simp^{\op},
  \mathcal{S})$ given by left Kan extension along $\pi$ gives an
  equivalence of the full subcategory
  $\Mon_{\simp^{\op}_{S}}(\mathcal{S})$ of $\simp^{\op}_{S}$-monoids
  with the subcategory $(\SegI)_{S}$ of Segal spaces with $0$th space
  $S$ and morphisms that are the identity on the $0$th space.
\end{cor}

\begin{lemma}\label{lem:adjtocart}
  Let $\mathcal{E}$ and $\mathcal{B}$ be \icats{} and $p \colon
  \mathcal{E} \to \mathcal{B}$ an inner fibration. Suppose
  \begin{enumerate}[(1)]
  \item $\mathcal{E}$ has finite limits and $p$ preserves these,
  \item $p$ has a right adjoint $r \colon \mathcal{B} \to
    \mathcal{E}$ such that $p \circ r \simeq \id_{\mathcal{B}}$.
  \end{enumerate}
  Then $p$ is a Cartesian fibration.
\end{lemma}
\begin{proof}
  Given $x \in \mathcal{E}$ and a morphism $f \colon b \to p(x)$, we
  must show there exists a Cartesian arrow in $\mathcal{E}$ lying over
  $f$ with target $x$. Define $\overline{f} \colon y \to x$ by the pullback
  diagram
  \csquare{y}{x}{r(b)}{rp(x).}{\overline{f}}{}{}{r(f)}
  Since $p$ preserves pullbacks, the morphism $p(\overline{f})$ is
  equivalent to $f$. Moreover, for any $z \in \mathcal{E}$ we have a
  pullback diagram
  \nolabelcsquare{\Map_{\mathcal{E}}(z, y)}{\Map_{\mathcal{E}}(z,
    x)}{\Map_{\mathcal{E}}(z, r(b))}{\Map_{\mathcal{E}}(z, rp(x)).}
  Under the adjunction this corresponds to the commutative diagram 
  \nolabelcsquare{\Map_{\mathcal{E}}(z, y)}{\Map_{\mathcal{E}}(z,
    x)}{\Map_{\mathcal{B}}(p(z), b)}{\Map_{\mathcal{E}}(p(z), p(x))}
  induced by the functor $p$. But then $\overline{f}$ is Cartesian by \cite[Proposition 2.4.4.3]{HTT}.
\end{proof}

\begin{thm}\label{thm:AlgCatSisSeg}
  There is an equivalence $\AlgCat(\mathcal{S}) \isoto
  \SegI$, given by sending a $\simp^{\op}_{S}$-algebra $\mathcal{C}$
  to the left Kan extension $\pi_{!}\mathcal{C}'$ of the
  composite 
  \[ \mathcal{C}' \colon \simp^{\op}_{S} \xto{\mathcal{C}}
  \mathcal{S}^{\times} \to \mathcal{S} \] along $\pi \colon
  \simp^{\op}_{S} \to \simp^{\op}$, where the second map (which sends
  $(S_{1},\ldots, S_{n}) \in \mathcal{S}^{\times}_{[n]}$ to $S_{1}
  \times \cdots \times S_{n}$) comes from a Cartesian structure in the
  sense of \cite[Definition 2.4.1.1]{HA}.
\end{thm}

\begin{proof}
  If $\mathcal{V}$ is an \icat{} with finite products, pulling
  back the monoid fibration $\Mon(\mathcal{V}) \to \OpdIns$ of
  Remark~\ref{rmk:MonoidFib} along $\simp^{\op}_{(\blank)}$ gives a
  Cartesian fibration $\Mon_{\txt{cat}}(\mathcal{V})$ with an
  equivalence \[\AlgCat(\mathcal{V}) \isoto
  \Mon_{\txt{cat}}(\mathcal{V})\] over $\mathcal{S}$. Taking left Kan
  extensions along the projections $\simp^{\op}_{S}\to \simp^{\op}$
  for all $S \in \mathcal{S}$ we get (using
  Proposition~\ref{propn:KanExtSegalSp}) a commutative square
  \opctriangle{\Mon_{\txt{cat}}(\mathcal{S})}{\SegI}{\mathcal{S}.}{\Phi}{}{\txt{ev}_{[0]}}
  By Lemma~\ref{lem:adjtocart} it is clear that $\txt{ev}_{[0]} \colon
  \SegI \to \mathcal{S}$ is a Cartesian fibration, and the functor
  $\Phi$ preserves Cartesian morphisms by \cite[Theorem
  6.1.3.9]{HTT}. It thus suffices to prove that for each $S \in
  \mathcal{S}$ the functor on fibres
  $\Mon_{\simp^{\op}_{S}}(\mathcal{S}) \to (\SegI)_{S}$ is an
  equivalence, which is the content of Corollary~\ref{cor:DopSMndEq2}.
\end{proof}

\subsection{A Presheaf Model for Categorical Algebras}\label{subsec:presheafalgcat}
In this subsection we will give an alternative characterization of the
\icat{} $\AlgCatV$ (for $\mathcal{V}$ a presentably monoidal \icat{})
as a localization of an \icat{} of presheaves. We thank Jeremy Hahn
for suggesting this model. Throughout this subsection we assume that
$\mathcal{V}$ is a presentably monoidal \icat{}.

\begin{defn}
  Let $\mathcal{V}^{\vee}_{\otimes} \to \simp$ be a Cartesian fibration
  corresponding to the same functor as the coCartesian fibration
  $\mathcal{V}^{\otimes} \to \simp^{\op}$. A presheaf $\phi \colon
  (\mathcal{V}^{\vee}_{\otimes})^{\op} \to \mathcal{S}$ is a \defterm{Segal
    presheaf} if it satisfies the following conditions:
  \begin{enumerate}[(1)]
  \item The functor $\mathcal{V}^{\op} \simeq
    (\mathcal{V}^{\vee}_{\otimes})^{\op}_{[1]} \to \mathcal{S}_{/\phi()^{\times
        2}}$, induced by the Cartesian morphisms over the face maps
    $[0] \to [1]$ in $\simp$, takes colimit diagrams in $\mathcal{V}$
    to limit diagrams in $\mathcal{S}_{/\phi()^{\times 2}}$.
  \item For every object $X \in \mathcal{V}^{\vee}_{\otimes}$, lying over $[n]
    \in \simp$, the diagram
    \nolabelcsquare{\phi(X)}{\phi(d_{n}^{*}X)}{\phi(\alpha^{*}(X))}{\phi(),}
    where $\alpha \colon [1] \to [n]$ is the map sending $0$ to $n-1$
    and $1$ to $n$, is a pullback square.
  \end{enumerate}
  Write $\mathcal{P}(\mathcal{V}^{\vee}_{\otimes})^{\Seg}$ for the full
  subcategory of $\mathcal{P}(\mathcal{V}^{\vee}_{\otimes})$ spanned by
  the Segal presheaves.
\end{defn}

\begin{remark}
  If $\phi \colon (\mathcal{V}^{\vee}_{\otimes})^{\op} \to \mathcal{S}$ is a
  Segal presheaf, then for every $n$ the functor
  \[ (\mathcal{V}^{\times n})^{\op} \simeq
  (\mathcal{V}^{\vee}_{\otimes})_{[n]}^{\op} \to \mathcal{S}_{/\phi()^{\times
      (n+1)}}, \]
  induced by the Cartesian morphisms over the inclusions $[0]
  \hookrightarrow [n]$, takes colimits in $\mathcal{V}^{\times n}$ to
  limits in $\mathcal{S}_{/\phi()^{\times (n+1)}}$.

  Since filtered \icats{} are contractible, it is easy to see that a
  filtered diagram in $\mathcal{S}_{/\phi()^{\times (n+1)}}$ is a
  limit diagram
  \IFF{} the diagram in $\mathcal{S}$ obtained by composing with the
  forgetful functor $\mathcal{S}_{/\phi()^{\times (n+1)}} \to
  \mathcal{S}$ is a limit diagram. Thus if $\phi$ is a Segal presheaf
  the functors $(\mathcal{V}^{\times n})^{\op} \simeq
  (\mathcal{V}^{\vee}_{\otimes})^{\op}_{[n]} \to \mathcal{S}$ all take filtered
  colimits in $\mathcal{V}^{\times n}$ to limits in $\mathcal{S}$. If
  $\mathcal{V}$ is a $\kappa$-presentable \icat{} we
  may therefore regard a Segal presheaf on $\mathcal{V}$ as a presheaf
  on the full subcategory $(\mathcal{V}^{\vee}_{\otimes})^{\kappa}$ spanned by
  the objects that lie in the image of $(\mathcal{V}^{\kappa})^{\times
    n}$ in $(\mathcal{V}^{\vee}_{\otimes})_{[n]} \simeq \mathcal{V}^{\times
    n}$ for all $n$. Moreover, a presheaf $\phi \colon
  (\mathcal{V}^{\vee}_{\otimes})^{\kappa,\op} \to \mathcal{S}$ 
  corresponds to a Segal presheaf \IFF{} it is local with respect to a
  set of maps in $\mathcal{P}((\mathcal{V}^{\vee}_{\otimes})^{\kappa})$, hence
  $\mathcal{P}(\mathcal{V}^{\vee}_{\otimes})^{\Seg}$ is an accessible
  localization of $\mathcal{P}((\mathcal{V}^{\vee}_{\otimes})^{\kappa})$.
\end{remark}

We now prove that Segal presheaves give an alternative model for
categorical algebras:
\begin{thm}\label{thm:Segpsheaves}
  There is an equivalence between
  $\mathcal{P}(\mathcal{V}^{\vee}_{\otimes})^{\Seg}$ and $\AlgCatV$.
\end{thm}

\begin{proof}[Proof of Theorem~\ref{thm:Segpsheaves}]
  Given a Cartesian fibration of \icats{} $p \colon \mathcal{E} \to
  \mathcal{B}$, let $\mathcal{E}^{\triangleright}_{\mathcal{B}}$ be
  the pushout $\mathcal{B} \amalg_{\mathcal{E} \times \{0\}}
  \mathcal{E} \times \Delta^{1}$ and let $j \colon \mathcal{B} \to
  \mathcal{E}^{\triangleright}_{\mathcal{B}}$ be the obvious
  inclusion. By \cite[\S 8]{freepres}, the functor $j^{*} \colon
  \mathcal{P}(\mathcal{E}^{\triangleright}_{\mathcal{B}}) \to
  \mathcal{P}(\mathcal{B})$ is a Cartesian fibration corresponding to
  the functor $\mathcal{P}(\mathcal{B}) \simeq
  \text{RFib}(\mathcal{B}) \to \CatI$ that sends a right fibration
  $\mathcal{Y} \to \mathcal{B}$ to
  $\Fun_{\mathcal{B}^{\op}}(\mathcal{Y}^{\op},
  \mathcal{P}_{\mathcal{B}}(\mathcal{E}))$, where
  $\mathcal{P}_{\mathcal{B}}(\mathcal{E}) \to \mathcal{B}^{\op}$ is
  the coCartesian fibration corresponding to the functor
  $\mathcal{B}^{\op} \to \CatI$ that sends $b \in \mathcal{B}^{\op}$
  to $\mathcal{P}(\mathcal{E}_{b})$. Let $\delta \colon \mathcal{S} \to \txt{LFib}(\simp^{\op}) \simeq
  \mathcal{P}(\simp)$ denote the functor that sends $X$ to
  $\simp^{\op}_{X} \to \simp^{\op}$. Write $\mathcal{Q}$ for the
  pullback
  \csquare{\mathcal{Q}}{\mathcal{P}((\mathcal{V}^{\vee}_{\otimes})^{\triangleleft}_{\simp})}{\mathcal{S}}{\mathcal{P}(\simp).}{}{q}{}{\delta}
  Then by \cite[Corollary 8.7]{freepres} the functor $q$ is the Cartesian
  fibration corresponding to the functor that sends $X \in
  \mathcal{S}$ to $\Fun_{\simp^{\op}}(\simp^{\op}_{X},
  \mathcal{P}_{\simp}(\mathcal{V}^{\vee}_{\otimes}))$.

  Let $\mathcal{Q}_{1}$ be the full subcategory of $\mathcal{Q}$
  spanned by presheaves $\phi \colon
  ((\mathcal{V}^{\vee}_{\otimes})^{\triangleleft}_{\simp})^{\op} \to \mathcal{S}$
  whose restriction to $(\mathcal{V}^{\vee}_{\otimes})^{\op}$ are Segal
  presheaves and for which the restriction $\phi|_{\{()\} \times
    \Delta^{1}} \colon \Delta^{1} \to \mathcal{S}$ is an
  equivalence. Then the restriction functor
  $\mathcal{P}((\mathcal{V}^{\vee}_{\otimes})^{\triangleleft}_{\simp}) \to
  \mathcal{P}(\mathcal{V}^{\vee}_{\otimes})$ gives an equivalence between
  $\mathcal{Q}_{1}$ and $\mathcal{P}(\mathcal{V}^{\vee}_{\otimes})^{\Seg}$ ---
  this is clear, since for every space $X$ the composite
  \[(\mathcal{V}^{\vee}_{\otimes})^{\op} \to \simp^{\op} \xto{\delta(X)}
  \mathcal{S} \] is the final Segal presheaf that sends $()$ to $X$.

  We can identify $\mathcal{V}^{\otimes}$ with the full subcategory of
  $\mathcal{P}_{\simp}(\mathcal{V}^{\vee}_{\otimes})$ spanned fibrewise by the
  representable presheaves. Let $\mathcal{Q}_{2}$
  denote the full subcategory of $\mathcal{Q}$ spanned by the
  presheaves that correspond to categorical algebras in $\mathcal{V}$,
  i.e. that under the identification above correspond to functors
  $\simp^{\op}_{X} \to \mathcal{P}_{\simp}(\mathcal{V}^{\vee}_{\otimes})$ that
  land in the full subcategory $\mathcal{V}^{\otimes}$ and preserve
  inert morphisms. Then we can identify the \icat{} $\mathcal{Q}_{2}$
  with $\AlgCatV$.

  It remains to observe that the full subcategories $\mathcal{Q}_{1}$
  and $\mathcal{Q}_{2}$ have the same objects. It is clear that a
  presheaf $\phi \colon
  (\mathcal{V}^{\vee,\triangleleft}_{\simp})^{\op} \to \mathcal{S}$
  whose restriction to $\{()\} \times \Delta^{1}$ is an equivalence
  corresponds to a functor $F \colon \simp^{\op}_{X} \to
  \mathcal{V}^{\otimes}$ \IFF{} for every $[n]$ the functor
  $(\mathcal{V}^{\times n})^{\op} \simeq
  (\mathcal{V}^{\vee}_{\otimes})_{[n]}^{\op} \to \mathcal{S}_{/\phi()^{\times
      (n+1)}}$ takes colimits in $\mathcal{V}^{\times n}$ to limits in
  $\mathcal{S}_{/\phi()^{\times (n+1)}}$.  Moreover, the functor $F$
  preserves inert morphisms \IFF{} for every object $T \in
  \simp^{\op}_{X}$, the morphism $F(T) \to F(\alpha_{!}T)$ is
  coCartesian, where $\alpha \colon [1] \to [n]$ is the morphism in
  $\simp$ that sends $0$ to $n-1$ and $1$ to $n$, or equivalently,
  under the identification $\mathcal{V}^{\otimes}_{[n]} \simeq
  \mathcal{V}^{\times n}$, the objects $F(T)$ and $(F(d_{n,!}T),
  F(\alpha_{!}T))$ are equivalent. In terms of $\phi$, this condition,
  for all $T \in \phi()^{\times (n+1)}$, is precisely the condition
  that the diagram
  \nolabelcsquare{\phi(X)}{\phi(d_{n}^{*}X)}{\phi(\alpha^{*}(X))}{\phi()}
  is a pullback square for all $X \in (\mathcal{V}^{\vee}_{\otimes})_{[n]}$. Thus
  $\phi$ is a Segal presheaf \IFF{} $F$ is a categorical algebra.
\end{proof}

\section{The $\infty$-Category of Enriched
  $\infty$-Categories}\label{sec:CatIV}
Our goal in this section is to prove our main result: we can always
localize the \icat{} of categorical algebras at the fully faithful and
essentially surjective functors by restricting to the full subcategory
of \emph{complete} objects. Along the way, we will introduce analogues
of a number of familiar concepts from ordinary enriched category
theory in our setting.

In \S\ref{subsec:basic} we define objects, morphisms, and equivalences
in enriched \icats{}. Then in \S\ref{subsec:equiv} we study the
classifying space of equivalences in an enriched \icat{}; the
\emph{complete} enriched \icats{} are those whose classifying spaces
of equivalences are equivalent to their underlying spaces of objects.
Next we study three types of equivalences of $\mathcal{V}$-\icats{}:
in \S\ref{subsec:FFES} we define \emph{fully faithful} and
\emph{essentially surjective} functors, in \S\ref{subsec:localeq}
\emph{local equivalences} (those in the saturated class of a certain
map) and finally in \S\ref{subsec:cateq} \emph{categorical
  equivalences} (those with an inverse up to natural equivalence). In
\S\ref{subsec:completion} we prove that for \icats{} enriched in a
presentably monoidal \icat{} the fully faithful and essentially
surjective functors are the same as the local equivalences, hence the
full subcategory of complete objects gives the localization; we can
extend this result to \icats{} enriched in a general large monoidal
\icat{} by embedding this in a presentable \icat{} in a larger
universe. Finally in \S\ref{subsec:proploccat} we prove that the
localized \icat{} inherits the functoriality properties of $\AlgCatV$.

\subsection{Some Basic Concepts}\label{subsec:basic}
In this subsection we define the basic notions of objects,
morphisms, and equivalences in an enriched \icat{}, an observe that
these have the expected properties. We first consider objects:
\begin{defn}
  Suppose $\mathcal{V}$ is a monoidal \icat{}. The unit of
  $\mathcal{V}$ defines an (essentially unique) associative algebra
  object $I_{\mathcal{V}} \colon \simp^{\op} \to
  \mathcal{V}^{\otimes}$ by Proposition~\ref{propn:UnitAlg}. We write
  $[0]_{\mathcal{V}}$ (or sometimes $I_{\mathcal{V}}$ or
  $E^{0}_{\mathcal{V}}$ depending on context) for this associative
  algebra regraded as an enriched \icat{}. We regard this as the
  trivial $\mathcal{V}$-\icat{} with one object, and so we refer to a
  map $[0]_{\mathcal{V}} \to \mathcal{C}$ as an \emph{object} of the
  $\mathcal{V}$-\icat{} $\mathcal{C}$.
\end{defn}

This definition justifies calling the mapping space
$\Map_{\AlgCatV}([0]_{\mathcal{V}}, \mathcal{C})$ the \emph{space of objects} in
$\mathcal{C}$. However, if $\mathcal{C}$ is a
$\simp^{\op}_{X}$-algebra in $\mathcal{V}$ then we also think of $X$
as being the space of objects of $\mathcal{C}$. Luckily, it is easy to
see that the two concepts agree:
\begin{lemma}\label{lem:iota0}
  Let $\mathcal{C} \colon \simp^{\op}_{X} \to \mathcal{V}^{\otimes}$
  be a $\mathcal{V}$-\icat{}. Then the map
  \[\Map_{\AlgCat(\mathcal{V})}([0]_{\mathcal{V}}, \mathcal{C}) \to
  \Map_{\mathcal{S}}(*, X) \simeq X\] induced by the Cartesian
  fibration $\AlgCat(\mathcal{V}) \to \mathcal{S}$ is an
  equivalence.
\end{lemma}
\begin{proof}
  It suffices to check that the fibres of this map are
  contractible. By \cite[Proposition 2.4.4.2]{HTT} the fibre
  at a point $p \colon * \to X$ is
  \[\Map_{\Alg_{\simp^{\op}}(\mathcal{V})}(I_{\mathcal{V}}, p^{*}\mathcal{C}),\]
  which is contractible since the unit $I_{\mathcal{V}}$ is the inital
  associative algebra object of $\mathcal{V}$.
\end{proof}

Next, we consider morphisms in an enriched \icat{}:
\begin{defn}
  Write $[1]$ for the category corresponding to the ordered set
  $\{0,1\}$, regarded as an $\mathcal{S}$-\icat{} by
  Remark~\ref{rmk:ordinaryasenricat}. Suppose $\mathcal{V}$ is a
  presentably monoidal \icat{}; then $\AlgCatV$ is tensored over
  $\AlgCatS$ by Corollary~\ref{cor:prestensAlgCatS}. We write
  $[1]_{\mathcal{V}}$ for the $\mathcal{V}$-\icat{} $[1] \otimes
  I_{\mathcal{V}}$. A \emph{morphism} in a $\mathcal{V}$-\icat{}
  $\mathcal{C}$ is a map $[1]_{\mathcal{V}} \to \mathcal{C}$.
\end{defn}

\begin{lemma}
  Suppose $\mathcal{V}$ is a presentably monoidal \icat{} and
  $\mathcal{C}$ is a $\mathcal{V}$-\icat{}. The two objects $0$ and
  $1$ of $[1]_{\mathcal{V}}$ induce two maps $i_{0},i_{1}\colon
  [0]_{\mathcal{V}} \to [1]_{\mathcal{V}}$; composing with these gives for any
  $\mathcal{V}$-\icat{} $\mathcal{C}$ a map of spaces
  \[ \Map_{\AlgCatV}([1]_{\mathcal{V}}, \mathcal{C}) \to
  \Map_{\AlgCatV}([0]_{\mathcal{V}}, \mathcal{C})^{\times 2}.\] The
  fibre $\Map([1]_{\mathcal{V}}, \mathcal{C})_{X,Y}$ of this map
  $\Map([1]_{\mathcal{V}}, \mathcal{C})$ at points $X,Y \in
  \Map([0]_{\mathcal{V}}, \mathcal{C})$ is equivalent to $\Map(I,
  \mathcal{C}(X,Y))$.
\end{lemma}
\begin{proof}
  Let $U \colon \mathcal{V}^{\otimes} \to \mathcal{S}^{\times}$ be the
  lax monoidal functor defined in Example~\ref{ex:TisTensorIV}.  We then
  have
  \[ \Map([1]_{\mathcal{V}}, \mathcal{C})_{X,Y} \simeq \Map([1],
  U_{*}\mathcal{C})_{X,Y}.\] Since $[1]_{\mathcal{S}}$ is the free
  $\mathcal{S}$-\icat{} on the $\mathcal{S}$-graph having a single
  edge from $0$ to $1$, using the adjunction between
  $\mathcal{S}$-\icats{} and $\mathcal{S}$-graphs from
  Remark~\ref{rmk:GraphAdj} we see that this is given by
  $U_{*}\mathcal{C}(X,Y) \simeq \Map(I_{\mathcal{V}}, \mathcal{C}(X,Y))$.
\end{proof}
\begin{remark}
  This means that a morphism in $\mathcal{C}$ from $X$ to $Y$ is the
  same thing as a map $I \to \mathcal{C}(X,Y)$. This definition, of
  course, makes sense for any monoidal \icat{} $\mathcal{V}$.
\end{remark}

We now define \emph{equivalences} in enriched \icats{},
and prove that these satisfy some of the expected properties. We will
define an equivalence in a $\mathcal{V}$-\icat{} $\mathcal{C}$ to be a
functor $E^{1} \to \mathcal{C}$ where $E^{1}$ is the generic
$\mathcal{V}$-\icat{} with two objects and an equivalence between
them. More precisely, $E^{1}$ is a special case of a more general
notion of a \emph{trivial} enriched \icat{}, which we now define:
\begin{defn}
  For any space $S$, the \defterm{trivial $\mathcal{V}$-\icat{}}
  $E_{S}^{\mathcal{V}}$ with objects $S$ is given by the
  composite \[\simp^{\op}_{S} \to \simp^{\op} \xto{I_{\mathcal{V}}}
  \mathcal{V}^{\otimes}.\] We will generally drop the $\mathcal{V}$
  from the notation and just write $E_{S}$ when the monoidal \icat{}
  in question is obvious from the context. The $\mathcal{V}$-\icats{}
  $E_{S}$ are functorial in $S$. We abbreviate $E^{n} :=
  E_{\{0,\ldots, n\}}$; restricting to order-preserving maps between
  the sets $\{0,\ldots,n\}$ ($n= 0,1,\ldots$) we then have a
  cosimplicial $\mathcal{V}$-\icat{} $E^{\bullet}$.
\end{defn}

\begin{remark}
  When $S$ is a set, $E_{S}$ is the enriched \icat{} associated to the
  trivial category with objects $S$ and a unique morphism $A \to B$
  for any pair of objects $A, B \in S$. This is also known as the
  \emph{coarse} category with objects $S$, to distinguish it from the
  ``discrete'' trivial category with objects $S$ (which has only
  identity morphisms).
\end{remark}

We think of $E^{n}$ as the generic $\mathcal{V}$-\icat{} with $n+1$
equivalent objects, so a map $E^{n} \to \mathcal{C}$
for some $\mathcal{V}$-\icat{} $\mathcal{C}$ is a choice of $n+1$
equivalent objects of $\mathcal{C}$. In particular, we have:
\begin{defn}
  Suppose $\mathcal{C}$ is a $\mathcal{V}$-\icat{}. An
  \defterm{equivalence} in $\mathcal{C}$ is a $\mathcal{V}$-functor
  $E^{1} \to \mathcal{C}$.
\end{defn}

\begin{remark}
  We will see below, in Proposition~\ref{propn:eqsareeqonmaps}, that
  this is equivalent to other reasonable definitions of an equivalence
  in a $\mathcal{V}$-\icat{}.
\end{remark}

\begin{defn}
  Let
  \[ T : \mathcal{S}^{\times} \rightleftarrows \mathcal{V}^{\otimes} :
  U \] be the adjoint functors described in
  Example~\ref{ex:TisTensorIV}, which induce an adjunction
  \[ T_{*} : \AlgCat(\mathcal{S}) \rightleftarrows
  \AlgCat(\mathcal{V}) : U_{*}. \] by
  Lemma~\ref{lem:StrMonAdjAlgCat}. If $\mathcal{C}$ is a
  $\mathcal{V}$-\icat{}, we refer to $U_{*}\mathcal{C}$ as the
  \emph{underlying $\mathcal{S}$-\icat{}} of $\mathcal{C}$. By
  Theorem~\ref{thm:AlgCatSisSeg} we can identify $U_{*}\mathcal{C}$
  with a Segal space.
\end{defn}

We now make the very useful observation that the equivalences in a
$\mathcal{V}$-\icat{} $\mathcal{C}$ only depend on the underlying
Segal space $U_{*}\mathcal{C}$:
\begin{propn}\label{propn:eqCiseqUC}
  Let $\mathcal{V}$ be a presentably monoidal \icat{}. Then for any
  space $S$ there is a natural equivalence
  \[ \Map_{\AlgCatV}(E_{S}^{\mathcal{V}}, \mathcal{C}) \simeq
  \Map_{\AlgCat(\mathcal{S})}(E_{S}^{\mathcal{S}}, U_{*}\mathcal{C}).\]
\end{propn}

This follows from the following Lemma:
\begin{lemma}\label{lem:ESastensor}\ 
  \begin{enumerate}[(i)]
  \item Let $\mathcal{V}$ be a monoidal \icat{}. By
    Proposition~\ref{propn:AlgCatLaxMon}, the \icat{} $\AlgCatV$ is
    tensored over $\AlgCat(*)$, since the unique monoidal structure on
    the trivial one-object \icat{} $*$ is the unit for the Cartesian
    product of monoidal \icats{}. There is a natural equivalence
    between the
    $\mathcal{V}$-\icat{} $E_{S}^{\mathcal{V}}$ and the tensor
    $E_{S}^{*} \otimes I_{\mathcal{V}}$.
  \item Let $\mathcal{V}$ be a presentably monoidal \icat{}; then the
    \icat{} $\AlgCatV$ is tensored over $\AlgCat(\mathcal{S})$ by
    Corollary~\ref{cor:prestensAlgCatS}. In this case there is a
    natural equivalence between $E_{S}^{\mathcal{V}}$ and the tensor
    $E_{S}^{\mathcal{S}} \otimes I_{\mathcal{V}}$.
  \end{enumerate}  
\end{lemma}
\begin{proof}
  We first prove (i). Considering the construction of the external
  product in $\Alg$, we see that $E_{S}^{*} \otimes I_{\mathcal{V}}$
  is given by \[E_{S}^{*} \times_{\simp^{\op}} I_{\mathcal{V}} \colon
  \simp^{\op}_{S} \times_{\simp^{\op}} \simp^{\op} \to \simp^{\op}
  \times_{\simp^{\op}} \mathcal{V}^{\otimes} \simeq
  \mathcal{V}^{\otimes}.\] We can factor this as
  \[ \simp^{\op}_{S} \times_{\simp^{\op}} \simp^{\op} \xto{E_{S}^{*}
    \times_{\simp^{\op}} \id} \simp^{\op} \times_{\simp^{\op}}
  \simp^{\op} \xto{\id \times_{\simp^{\op}} I_{\mathcal{V}}}
  \simp^{\op} \times_{\simp^{\op}} \mathcal{V}^{\otimes}, \] which is
  clearly the same as $E_{S}^{\mathcal{V}}$.

  Now in the situtation of (ii), part (i) then gives an
  equivalence \[E_{S}^{\mathcal{S}} \otimes I_{\mathcal{V}} \simeq
  (E_{S}^{*} \otimes I_{\mathcal{S}}) \otimes I_{\mathcal{V}} \simeq
  E_{S}^{*} \otimes (I_{\mathcal{S}} \otimes I_{\mathcal{V}}) \simeq
  E_{S}^{*} \otimes I_{\mathcal{V}} \simeq E_{S}^{\mathcal{V}},\]
  since it is easy to see that the tensorings with $\AlgCat(*)$ and
  $\AlgCat(\mathcal{S})$ are compatible.
\end{proof}

\begin{proof}[Proof of Proposition~\ref{propn:eqCiseqUC}]
  By Lemma~\ref{lem:ESastensor}, the $\mathcal{V}$-\icat{}
  $E_{S}^{\mathcal{V}}$ is naturally equivalent to
  $T_{*}E_{S}^{\mathcal{S}}$. We now complete the proof by recalling
  that the functor $T_{*}$ is left adjoint to $U_{*}$.
\end{proof}

\begin{defn}
  We write $\iota_{1}\mathcal{C} := \Map_{\AlgCatV}(E^{1},
  \mathcal{C})$ for the \emph{space of equivalences} in a
  $\mathcal{V}$-\icat{} $\mathcal{C}$. More generally, we write
  $\iota_{n}\mathcal{C}$ for the mapping space $\Map_{\AlgCatV}(E^{n},
  \mathcal{C})$ --- we can think of this as the space of $n$
  composable equivalences in $\mathcal{C}$, together with all the
  coherence data for the compositions. These spaces form a simplicial
  space $\iota_{\bullet}\mathcal{C}$ --- here the face maps can be
  thought of as composing equivalences, and the degeneracy maps as
  inserting identity maps.
\end{defn}
\begin{remark}
  By Proposition~\ref{propn:eqCiseqUC} there is a natural equivalence
  $\iota_{\bullet}\mathcal{C} \simeq
  \iota_{\bullet}U_{*}\mathcal{C}$. This will allow us to reduce many
  of our arguments below to the case of spaces, where we can make use
  of results of Rezk from \cite{RezkCSS}.
\end{remark}

In particular, we will now use this to prove that our definition of
equivalence agrees with a number of other reasonable definitions:
\begin{propn}\label{propn:eqsareeqonmaps}
  Suppose $\mathcal{V}$ is a monoidal \icat{}
  and $\mathcal{C}$ is a $\mathcal{V}$-\icat{}. Let $X,Y$ be objects
  of $\mathcal{C}$ and $\alpha \colon I_{\mathcal{V}}
  \to \mathcal{C}(X,Y)$ a morphism in $\mathcal{C}$. Then the
  following are equivalent:
  \begin{enumerate}[(i)]
  \item $\alpha$ is an equivalence, i.e. it extends to a functor $E^{1}
    \to \mathcal{C}$.
  \item For all $Z \in \iota_{0}\mathcal{C}$, the composite map in
    $\mathcal{V}^{\otimes}$ 
    \[ \mathcal{C}(Y,Z) \to (I_{\mathcal{V}}, \mathcal{C}(Y, Z)) \to
    (\mathcal{C}(X,Y), \mathcal{C}(Y,Z)) \to \mathcal{C}(X, Z) \]\
    given by composing with $\alpha$ is an equivalence.
  \item For all $Z \in \iota_{0}\mathcal{C}$, the composite map in
    $\mathcal{V}^{\otimes}$ 
    \[ \mathcal{C}(Z,X) \to (\mathcal{C}(Z, X), I_{\mathcal{V}}) \to
    (\mathcal{C}(Z,X), \mathcal{C}(X,Y)) \to \mathcal{C}(X, Y) \]\
    given by composing with $\alpha$ is an equivalence.
  \item $\alpha$ has an inverse, i.e. a map $I \to \mathcal{C}(Y,X)$
    such that the composites
    \[ I \to (I, I) \xto{(\beta,\alpha)} (\mathcal{C}(X,Y), \mathcal{C}(Y,X)) \to
    \mathcal{C}(X,X),\]
    \[ I \to (I, I) \xto{(\alpha,\beta)} (\mathcal{C}(Y,X),
    \mathcal{C}(X,Y)) \to \mathcal{C}(Y,Y)\] are homotopic to the
    identity maps of $X$ and $Y$, respectively.
  \end{enumerate}
\end{propn}
\begin{proof}
  We first show that (i) is equivalent to (ii). Suppose (i) holds,
  and let $\hat{\alpha} \colon E^{1} \to \mathcal{C}$ be an
  equivalence extending $\alpha$. Composing with the inverse
  equivalence from $Y$ to $X$ gives an inverse to composition with
  $\alpha$, since the composite map is composing with the composite $X
  \to Y \to X$, which is the identity.

  Now suppose (ii) holds. Without loss of generality, we may assume
  that $\mathcal{V}$ is presentably monoidal (by embedding in a
  presentably monoidal \icat{} of presheaves in a larger universe, if
  necessary). Then a map $E^{1}_{\mathcal{V}} \to \mathcal{C}$ is adjoint to a map
  $E^{1}_{\mathcal{S}} \to U_{*}\mathcal{C}$ where $U \colon
  \mathcal{V} \to \mathcal{S}$ is again as in Example~\ref{ex:TisTensorIV}.
  If (ii) holds for $\alpha$ then the analogous
  condition also holds for $\alpha$ considered as a morphism in
  $U_{*}\mathcal{C}$. It thus suffices to show that (ii) implies (i)
  in the case where $\mathcal{V}$ is $\mathcal{S}$. We again use the
  equivalence between $\mathcal{S}$-\icats{} and Segal spaces of
  Theorem~\ref{thm:AlgCatSisSeg}; the map
  $\alpha$ is clearly a ``homotopy equivalence'' in the sense of
  \cite[\S 5.5]{RezkCSS}, and so extends to a map from $E^{1}$ by
  \cite[Theorem 6.2]{RezkCSS}.

  The proof that (i) is equivalent to (iii) is similar, so it remains
  to prove that (i) is equivalent to (iv). Since equivalences are
  detected in $U_{*}\mathcal{C}$, this is immediate from
  \cite[Theorem 6.2]{RezkCSS}.
\end{proof}

The inclusion $[1]_{\mathcal{V}} \to E^{1}$ of
the map from $0$ to $1$ induces a map $\iota_{1}\mathcal{C} \to
\Map([1]_{\mathcal{V}}, \mathcal{C})$. The two inclusions of
$E^{0} \simeq [0]_{\mathcal{V}}$ into $[1]_{\mathcal{V}}$ and
$E^{1}$ then give a commutative triangle
\opctriangle{\iota_{1}\mathcal{C}}{\Map([1]_{\mathcal{V}},
  \mathcal{C})}{\iota_{0}\mathcal{C} \times
  \iota_{0}\mathcal{C}.}{}{}{}
We end this section by showing that on fibres, this map is an
inclusion of components:

\begin{defn}
  Suppose $\mathcal{C}$ is a $\mathcal{V}$-\icat{} and $X, Y$ are
  objects of $\mathcal{C}$. We let $\Map(I_{\mathcal{V}},
  \mathcal{C}(X,Y))_{\text{eq}}$ be the subspace of $\Map(I_{\mathcal{V}},
  \mathcal{C}(X,Y))$ consisting of the components in the image of
  $\iota_{1}\mathcal{C}_{X,Y}$ under the induced map on fibres in the
  diagram above.
\end{defn}

\begin{propn}\label{propn:eqeq}
  Suppose $\mathcal{V}$ is a presentably monoidal \icat{},
  $\mathcal{C}$ is a $\mathcal{V}$-\icat{}, and $X,Y$ are objects of
  $\mathcal{C}$. Then the map $(\iota_{1}\mathcal{C})_{X,Y} \to \Map(I_{\mathcal{V}},
  \mathcal{C}(X,Y))_{\txt{eq}}$ is an equivalence.
\end{propn}
\begin{proof}
  By Proposition~\ref{propn:eqCiseqUC} it again suffices to
  prove this for $\mathcal{S}$-\icats{}. Using the identification of
  $\mathcal{S}$-\icats{} with Segal spaces of
  Theorem~\ref{thm:AlgCatSisSeg}, this therefore follows from
  the corresponding statement in that setting. The latter is a
  consequence of \cite[Theorem 6.2]{RezkCSS}, since if $\mathcal{C}$
  is a Segal space with objects $X$, $Y$, a point of
  $\mathcal{C}(X,Y)$ is a ``homotopy equivalence'' in the sense of
  \cite[\S 5.5]{RezkCSS} \IFF{} it extends to a map from
  $E^{1}_{\mathcal{S}}$, by \cite[Proposition 11.1]{RezkCSS}.
\end{proof}

\subsection{The Classifying Space of Equivalences}\label{subsec:equiv}
In this section we define the classifying space of equivalences in an
enriched \icat{}, and use this to define \emph{complete} enriched
\icats{}. We then prove that the simplicial space of equivalences is
always a groupoid object, which allows us to give a simpler
description of the completeness condition.

\begin{defn}
  Let $\mathcal{C}$ be a $\mathcal{V}$-\icat{}. The
  \defterm{classifying space of equivalences} $\iota \mathcal{C}$ of
  $\mathcal{C}$ is the geometric realization $|\iota_{\bullet}
  \mathcal{C}|$ of the simplicial space $\iota_{\bullet}\mathcal{C} :=
  \Map(E^{\bullet}, \mathcal{C})$.
\end{defn}

We regard $\iota \mathcal{C}$ as the ``correct'' space of objects of
$\mathcal{C}$, and by analogy with Rezk's notion of complete
Segal space we say that an enriched \icat{} is \emph{complete} if its
underlying space is the correct one:
\begin{defn}\label{defn:complete}
  A $\mathcal{V}$-\icat{} $\mathcal{C}$ is \defterm{complete} if the
  natural map $\iota_{0}\mathcal{C} \to \iota \mathcal{C}$ is an
  equivalence.
\end{defn}

Our next goal is to prove that the simplicial space $\iota
\mathcal{C}$ is always a groupoid object; we prove this by showing
that the cosimplicial object $E^{\bullet}$ satisfies the dual
condition of being a \emph{cogroupoid} object:
\begin{defn}\label{defn:cogroupoid}
  A cosimplicial object $X \colon \simp \to \mathcal{C}$ in an \icat{}
  $\mathcal{C}$ is a \defterm{cogroupoid object} if for every
  partition $[n] = S \cup S'$ such that $S \cap S'$ consists of a
  single element, the diagram \nolabelsmallcsquare{X(S \cap
    S')}{X(S)}{X(S')}{X([n])} is a pushout square.
\end{defn}

\begin{lemma} 
  If $X \colon \simp \to \mathcal{C}$ is a cogroupoid object in an
  \icat{} $\mathcal{C}$, then for every object $Y \in \mathcal{C}$ the
  simplicial space $\Map_{\mathcal{C}}(X, Y)$ is a groupoid object in
  spaces.
\end{lemma}

\begin{thm}\label{thm:Ecogpd}
  If $\mathcal{V}$ is a presentably monoidal \icat{} then
  the cosimplicial object $E^{\bullet}$ is a cogroupoid object.
\end{thm}
\begin{proof}
  We will show that $E^{N} \amalg_{E_{\{N\}}} E_{\{N,N+1\}} \to
  E^{N+1}$ is an equivalence; by induction this will imply that
  $E^{\bullet}$ is a cogroupoid object, as the ordering of the objects
  is arbitrary. Since $\mathcal{V}$ is presentably monoidal, by
  Proposition~\ref{propn:eqCiseqUC} it suffices to prove this when
  $\mathcal{V}$ is $\mathcal{S}$.

  Under the equivalence $\AlgCat(\mathcal{S}) \isoto \SegI$
  of Theorem~\ref{thm:AlgCatSisSeg}
  the $\mathcal{S}$-\icat{} $E_{S}$ clearly corresponds to the Segal
  space $i_{*}S$. If $S$ is a set it follows that in the model
  category structure on bisimplicial sets modelling Segal spaces,
  $E_{S}$ corresponds to $\pi^{*}N\mathcal{I}_{S}$ where
  $\mathcal{I}_{S}$ is the ordinary category with objects $S$ and a
  unique morphism between any pair of objects, and $\pi \colon
  \simp^{\op} \times \simp^{\op} \to \simp^{\op}$ is the projection
  onto the first factor.

  Define $G_{N} := \mathrm{N}\mathcal{I}_{\{0, \ldots, N\}}$. By
  \cite[Remark 10.2]{RezkCSS}, for $0 < i < n$ the map
  $\pi^{*}\Lambda^{n}_{k} \to \pi^{*}\Delta^{n}$ is a Segal
  equivalence, so (since $\pi^{*}$ is a left adjoint and thus
  preserves colimits) it suffices to prove that
  $G_{N}\amalg_{G_{\{N\}}} G_{\{N,N+1\}} \hookrightarrow G_{N+1}$ is
  an inner anodyne morphism of simplicial sets. To prove this we
  consider a series of nested filtrations of the simplices of
  $G_{N+1}$. First we must introduce some notation:

  An $n$-simplex $\sigma$ of $G_{N+1}$ can be described by a list
  $a_{0}\cdots a_{n}$ of elements $a_{i} \in \{0, \ldots, N+1\}$; it
  is non-degenerate if $a_{i} \neq a_{i+1}$ for all $i$. If $\sigma$
  is a non-degenerate simplex, let $\beta(\sigma)$ be the number of
  times the sequence jumps between $\{0, \ldots, N\}$ and $\{N,
  N+1\}$. 

  Also let $\tau(\sigma)$ is the position of the first $N+1$ where the
  sequence jumps from $\{N,
  N+1\}$ to $\{0, \ldots, N\}$; if there is no such jump let
  $\tau(\sigma) = \infty$ and let $\tau'(\sigma)$ denote the position
  of the first jump from $\{0, \ldots, N\}$ to $\{N,
  N+1\}$. Then we make the following definitions:
  \begin{itemize}
  \item If $t \neq \infty$, let $S^{b,t}_{n}$ be the set of non-degenerate $n$-simplices
    $\sigma$ in $G_{N+1}$ such that $\beta(\sigma) = b$, $\tau(\sigma)
    = t$, and $a_{t+1} \neq N$. Let $S^{1,\infty,t}_{n}$ be the set of
    non-degenerate $n$-simplices in $G_{N+1}$ such that $\beta(\sigma)
    = 1$, $\tau(\sigma) = \infty$, $\tau'(\sigma) = t$, and $a_{t-1}
    \neq N$.
  \item If $t \neq \infty$, let $T^{b,t}_{n}$ be the set of
    non-degenerate $(n+1)$-simplices $\sigma$ in $G_{N+1}$ such that
    $\beta(\sigma) = b$, $\tau(\sigma) = t$ and $a_{t+1} = N$. Let
    $T^{1,\infty,t}_{n}$ be the set of non-degenerate
    $(n+1)$-simplices $\sigma$ in $G_{N+1}$ such that $\beta(\sigma) =
    1$, $\tau(\sigma) = \infty$, $\tau'(\sigma)=t+1$, and $a_{t} = N$.
  \end{itemize}
  Define a filtration
  \[G_{N}\amalg_{G_{\{N\}}} G_{\{N,N+1\}} =: \mathcal{F}_{0} \subseteq
  \mathcal{F}_{1} \subseteq \cdots \subseteq G_{N+1} \] by letting
  $\mathcal{F}_{n}$ be the subspace of $G_{N+1}$ whose non-degenerate
  simplices are those of $\mathcal{F}_{0}$ together with all the
  non-degenerate $i$-simplices for $i \leq n$ and the
  $(n+1)$-simplices in $T^{b,t}_{n}$ and $T^{1,\infty,t}_{n}$ for all
  $b,t$. Then $G_{N+1} =
  \bigcup_{n} \mathcal{F}_{n}$, so to prove that
  $G_{N}\amalg_{G_{\{N\}}} G_{\{N,N+1\}} \hookrightarrow G_{N+1}$ is
  inner anodyne it suffices to prove that the inclusions
  $\mathcal{F}_{n-1} \hookrightarrow \mathcal{F}_{n}$ are inner
  anodyne.

  Next define a filtration
  \[ \mathcal{F}_{n-1} =: \mathcal{F}_{n}^{0} \subseteq
  \mathcal{F}_{n}^{1} \subseteq \cdots \subseteq \mathcal{F}_{n}^{n-1}
  := \mathcal{F}_{n}\] by setting $\mathcal{F}_{n}^{b}$ to be the
  subspace of $\mathcal{F}_{n}$ containing $\mathcal{F}_{n-1}$
  together with the simplices in $S_{n}^{i,t}$ and $T_{n}^{i,t}$ for
  all $i \leq b$ together with $S_{n}^{1,\infty,t}$ and
  $T_{n}^{1,\infty,t}$ for all $t$. To prove that the inclusions
  $\mathcal{F}_{n-1} \hookrightarrow \mathcal{F}_{n}$ are inner
  anodyne it suffices to prove that the inclusions
  $\mathcal{F}_{n}^{b-1} \hookrightarrow \mathcal{F}_{n}^{b}$ are all
  inner anodyne.

  Finally define a filtration
  \[ \mathcal{F}_{n}^{b-1} =: \mathcal{F}_{n}^{b,n+1} \subseteq
  \mathcal{F}_{n}^{b, n} \subseteq \cdots \subseteq
  \mathcal{F}_{n}^{b,0} := \mathcal{F}_{n}^{b},\] by setting
  $\mathcal{F}_{n}^{b, t}$ to be the subspace of $\mathcal{F}_{n}^{b}$
  containing $\mathcal{F}_{n}^{b-1}$ together with the simplices in
  $S_{n}^{b,j}$ and $T_{n}^{b,j}$ (as well as $S_{n}^{1,\infty,j}$ and
  $T_{n}^{1,\infty,j}$ if $b = 1$) for all $j \geq t$. To prove that
  the inclusions $\mathcal{F}_{n}^{b-1} \hookrightarrow
  \mathcal{F}_{n}^{b}$ are inner anodyne it suffices to show that the
  inclusions $\mathcal{F}_{n}^{b,t-1} \hookrightarrow
  \mathcal{F}_{n}^{b,t}$ are all inner anodyne.

  Now observe that (for $b > 1$) if $\sigma \in T_{n}^{b,t}$ then
  $d_{t}\sigma \in S_{n}^{b,t}$ and $d_{i}\sigma \in
  \mathcal{F}_{n}^{b,t-1}$ for $i \neq t$, and $\sigma$ is uniquely
  determined by $d_{t}\sigma$.  Thus we get a pushout diagram
  \nolabelcsquare{\coprod_{\sigma \in T_{n}^{b,t}}
    \Lambda^{n+1}_{t}}{\coprod_{\sigma \in T_{n}^{b,t}}
    \Delta^{n+1}}{\mathcal{F}_{n}^{b,t-1}}{\mathcal{F}_{n}^{b,t}}
  where we always have $0 < t < n+1$. Thus the bottom horizontal map
  is inner anodyne. The proof is similar when $b = 1$, expect that we
  must also consider the simplices in $S_{n}^{1,\infty,t}$, so we
  conclude that $G_{N}\amalg_{G_{\{N\}}} G_{\{N,N+1\}} \to G_{N+1}$ is
  indeed inner anodyne.
\end{proof}

\begin{remark}
  We can generalize this to the case of an arbitrary large monoidal
  \icat{} $\mathcal{V}$ as follows: by \cite[Remark
  4.8.1.8]{HA} there exists a presentably monoidal structure on the
  (very large) presentable \icat{}
  $\widehat{\mathcal{P}}(\mathcal{V})$ of presheaves of large spaces
  on $\mathcal{V}$, such that the Yoneda embedding $\mathcal{V} \to
  \widehat{\mathcal{P}}(\mathcal{V})$ is a monoidal functor. This induces
  a fully faithful embedding \[\AlgCat(\mathcal{V}) \to
  \LAlgCat(\widehat{\mathcal{P}}(\mathcal{V}));\] moreover,
  if $X$ a small space then
  $E_{X}^{\widehat{\mathcal{P}}(\mathcal{V})}$ is clearly the image of
  $E_{X}^{\mathcal{V}}$. Thus if a diagram of $E_{X}^{\mathcal{V}}$'s
  is a colimit diagram in
  $\LAlgCat(\widehat{\mathcal{P}}(\mathcal{V}))$ it must
  also be a colimit diagram in $\AlgCat(\mathcal{V})$ --- in
  particular $E^{\bullet}_{\mathcal{V}}$ is a cogroupoid object in
  $\AlgCat(\mathcal{V})$.
\end{remark}

\begin{cor}\label{cor:iotagpd}
  The simplicial space $\iota_{\bullet}\mathcal{C}$ is a groupoid
  object in spaces for all $\mathcal{V}$-\icats{} $\mathcal{C}$.
\end{cor}

\begin{lemma}\label{lem:segspconst}
  Suppose $X_{\bullet}$ is a category object in an \icat{}
  $\mathcal{C}$. Then the following are equivalent:
  \begin{enumerate}[(i)]
  \item The functor $X_{\bullet}$ is constant.
  \item The map $s_{0} \colon X_{0} \to X_{1}$ is an equivalence.
  \end{enumerate}
\end{lemma}
\begin{proof}
  It is obvious that (i) implies (ii). To show that (ii) implies (i) first observe that if $s_{0} \colon
  X_{0} \to X_{1}$ is an equivalence, then by the 2-out-of-3 property
  $d_{0},d_{1} \colon X_{1} \to X_{0}$ are also equivalences. Since
  $X_{\bullet}$ is a category object we have pullback diagrams
  \csquare{X_n}{X_{n-1}}{X_1}{X_0,}{d_0}{}{}{} and so the face maps $d_{0} \colon X_{n} \to X_{n-1}$ are 
  equivalences for all $i$ and $n$. Combining this with the simplicial
  identities we see inductively that all face maps and degeneracies
  are equivalences.
\end{proof}

\begin{lemma}\label{lem:gpdcolim}
  Suppose $U_{\bullet}$ is a groupoid object in $\mathcal{S}$. The
  following are equivalent:
  \begin{enumerate}[(i)]
  \item The map $U_{0} \to |U_{\bullet}|$ is an equivalence.
  \item The map $s_{0} \colon U_{0} \to U_{1}$ is an equivalence.
  \item The simplicial object $U_{\bullet}$ is constant, i.e. for
    every map $\phi \colon [n] \to [m]$ in $\simp^{\op}$ the
    induced map $\iota_{n}\mathcal{C} \to \iota_{m}\mathcal{C}$ is an
    equivalence.
  \end{enumerate}
\end{lemma}
\begin{proof}
  We first show that (i) implies (ii): Since $\mathcal{S}$ is an
  $\infty$-topos, the groupoid object $U_{\bullet}$ is effective,
  i.e. it is equivalent to the Čech nerve of the map $U_{0} \to
  |U_{\bullet}|$. Thus we have a pullback diagram \csquare{U_1
  }{U_0}{U_0}{{|U_\bullet|,}}{d_0}{d_1}{}{} so the maps $d_{0}$,
  $d_{1}$ are equivalences. From the 2-out-of-3 property it follows
  that $s_{0}$ is also an equivalence. It follows from
  Lemma~\ref{lem:segspconst} that (ii) implies (iii).  Finally (iii)
  implies (i) since the simplicial set $\simp^{\op}$ is weakly
  contractible.
\end{proof}

We can now give a simpler characterization of the completeness
conditionf or $\mathcal{V}$-\icats{}:
\begin{cor}\label{cor:completeifE1}
  Let $\mathcal{C}$ be a $\mathcal{V}$-\icat{}. The following are
  equivalent:
  \begin{enumerate}[(i)]
  \item $\mathcal{C}$ is complete.
  \item The natural map $s_{0} \colon \iota_{0}\mathcal{C} \to
    \iota_{1}\mathcal{C}$ is an equivalence.
  \item The simplicial space $\iota_{\bullet}\mathcal{C}$ is constant
    (i.e. for every map $\phi \colon [n] \to [m]$ in $\simp^{\op}$ the
    induced map $\iota_{n}\mathcal{C} \to \iota_{m}\mathcal{C}$ is an
    equivalence).
  \end{enumerate}
\end{cor}
\begin{proof}
  Apply Lemma~\ref{lem:gpdcolim} to the groupoid object
  $\iota_{\bullet}\mathcal{C}$.
\end{proof}

\subsection{Fully Faithful and Essentially Surjective
  Functors}\label{subsec:FFES}
In this subsection we introduce analogues of \emph{fully faithful} and
\emph{essentially surjective} functors in the context of enriched
\icats{}, and show that these have the expected properties.

\begin{defn}
  Let $\mathcal{V}$ be a monoidal \icat{}. A
  $\mathcal{V}$-functor $F \colon \mathcal{C} \to \mathcal{D}$ is
  \emph{fully faithful} if the maps $\mathcal{C}(X,Y) \to
  \mathcal{D}(FX,FY)$ are equivalences in $\mathcal{V}$ for all $X,Y$
  in $\iota_{0}\mathcal{C}$.
\end{defn}

\begin{lemma}\label{lem:FullyFaithfulEq}
  A $\mathcal{V}$-functor $F \colon \mathcal{C} \to \mathcal{D}$ is
  fully faithful \IFF{} it is a Cartesian morphism in $\AlgCatV$
  with respect to the projection $\AlgCat(\mathcal{V}) \to
  \mathcal{S}$.
\end{lemma}
\begin{proof}
  If $f \colon S \to \iota_{0}\mathcal{D}$ is a map of spaces, then
  a Cartesian morphism over $f$ with target $\mathcal{D}$ has source
  $f^{*}\mathcal{D} = \mathcal{D} \circ \simp^{\op}_{f}$; in
  particular a Cartesian morphism induces equivalences
  $f^{*}\mathcal{D}(x,y) \to \mathcal{D}(f(x),f(y))$ for all $x,y \in
  X$.

  Conversely, suppose $F \colon \mathcal{C} \to \mathcal{D}$ gives an
  equivalence on all mapping spaces. The functor $F$ factors as 
  \[ \mathcal{C} \xto{F'} (\iota_{0}F)^{*}\mathcal{D} \xto{F''}
  \mathcal{D},\] where $F''$ is Cartesian. The morphism $F'$ induces
  an equivalence on underlying spaces and is given by equivalences
  $\mathcal{C}(X,Y) \to \mathcal{D}(F(X), F(Y))$ for all $X,Y\in
  \iota_{0}\mathcal{C}$. By Lemma~\ref{lem:AlgCons} it follows that
  $F'$ is an equivalence in
  $\Alg_{\simp^{\op}_{\iota_{0}\mathcal{C}}}(\mathcal{V})$
  and so in $\AlgCat(\mathcal{V})$. In particular $F'$ is a
  Cartesian morphism and hence so is the composite $F \simeq F'' \circ
  F'$.
\end{proof}

\begin{defn}
  A functor $F \colon \mathcal{C} \to \mathcal{D}$ is
  \defterm{essentially surjective} if for every point $X \in
  \iota_{0}\mathcal{D}$ there exists an equivalence $E^{1} \to
  \mathcal{D}$ from $X$ to a point in the image of $\iota_{0}F$.
\end{defn}

\begin{lemma}\label{lem:esssurjcond}
  A functor $F \colon \mathcal{C} \to \mathcal{D}$ is essentially
  surjective \IFF{} the induced map $\pi_{0}\iota F \colon
  \pi_{0}\iota \mathcal{C} \to \pi_{0}\iota \mathcal{D}$ is
  surjective.
\end{lemma}
\begin{proof}
  Since $\iota_{\bullet}\mathcal{D}$ is a groupoid object, the set
  $\pi_{0}\iota \mathcal{D}$ is the quotient of
  $\pi_{0}\iota_{0}\mathcal{D}$ where we identify two components of
  $\iota_{0}\mathcal{D}$ if there exists a point of
  $\iota_{1}\mathcal{D}$, i.e. an equivalence $E^{1} \to \mathcal{D}$,
  connecting them. Thus $F \colon \mathcal{C} \to \mathcal{D}$ is
  essentially surjective \IFF{} $\pi_{0}\iota F$ is surjective.
\end{proof}

\begin{lemma}\label{lem:FFeqiota1}
  Suppose $F \colon \mathcal{C} \to \mathcal{D}$ is a fully faithful
  functor of $\mathcal{V}$-\icats{}. Then for every
  $X,Y \in \mathcal{C}$ the induced map $(\iota_{1}\mathcal{C})_{X,Y}
  \to (\iota_{1}\mathcal{D})_{FX,FY}$ is an equivalence.
\end{lemma}
\begin{proof}
  By Proposition~\ref{propn:eqeq}, we can identify the map
  $(\iota_{1}\mathcal{C})_{X,Y} \to (\iota_{1}\mathcal{D})_{FX,FY}$
  with the map \[\Map(I, \mathcal{C}(X,Y))_{\txt{eq}} \to \Map(I,
  \mathcal{D}(FX,FY))_{\txt{eq}}\] induced by $F$. Since $F$ is fully
  faithful the map $\mathcal{C}(X,Y) \to \mathcal{D}(FX,FY)$ is an
  equivalence in $\mathcal{V}$, hence $\Map(I, \mathcal{C}(X,Y)) \to
  \Map(I, \mathcal{D}(FX,FY))$ is an equivalence in $\mathcal{S}$. To
  complete the proof it therefore suffices to show that $\Map(I,
  \mathcal{C}(X,Y))_{\txt{eq}} \to \Map(I,
  \mathcal{D}(FX,FY))_{\txt{eq}}$ is surjective on components ---
  i.e. if $\alpha \colon I \to \mathcal{D}(FX,FY)$ is an equivalence
  then it is the image of an equivalence $\beta \colon I \to
  \mathcal{C}(X,Y)$. We know that $\alpha$ is the image of some map
  $\beta$, so it suffices to show that such a $\beta$ must be an
  equivalence. By Proposition~\ref{propn:eqsareeqonmaps} the map
  $\beta$ is an equivalence \IFF{} for every $Z \in
  \iota_{0}\mathcal{C}$ the map $\mathcal{C}(Z, X) \to
  \mathcal{C}(Z,Y)$ induced by composition with $\beta$ is an
  equivalence. Consider the diagram
  \nolabelcsquare{\mathcal{C}(Z,X)}{\mathcal{D}(FZ,FX)}{\mathcal{C}(Z,Y)}{\mathcal{D}(FZ,FY),}
  where the vertical maps are given by composition with $\beta$ and
  $\alpha$, respectively.  Since $F$ is fully faithful and $\alpha$ is
  an equivalence, all morphisms in this diagram except the left
  vertical map are known to be equivalences. By the 2-out-of-3
  property this must also be an equivalence for all $Z$, so $\beta$ is
  indeed an equivalence.
\end{proof}

\begin{propn}\label{propn:FFESiotaEq}
  If a $\mathcal{V}$-functor $F \colon \mathcal{C} \to \mathcal{D}$ is
  fully faithful and essentially surjective, then the induced map
  $\iota F \colon \iota \mathcal{C} \to \iota \mathcal{D}$ is an
  equivalence.
\end{propn}
\begin{proof}
  The simplicial spaces $\iota_{\bullet}\mathcal{C}$ and
  $\iota_{\bullet}\mathcal{D}$ are groupoid objects by
  Corollary~\ref{cor:iotagpd}, and since $F$ is essentially surjective
  the map $\iota F$ is surjective on $\pi_{0}$ by
  Lemma~\ref{lem:esssurjcond}. By \cite[Remark
  1.2.17]{LurieGoodwillie} it therefore suffices to show that the
  diagram \nolabelcsquare{\iota_1 \mathcal{C}}{ \iota_1
    \mathcal{D}}{\iota_0 \mathcal{C} \times \iota_0
    \mathcal{C}}{\iota_0 \mathcal{D} \times \iota_0 \mathcal{D}} is a
  pullback square. To prove this we must show that for all $X,Y \in
  \mathcal{C}$ the map on fibres is an equivalence, which we proved in
  Lemma~\ref{lem:FFeqiota1}.
\end{proof}

\begin{cor}\label{cor:ffesiff}
  A fully faithful $\mathcal{V}$-functor $F$ is essentially surjective
  \IFF{} $\iota F$ is an equivalence.
\end{cor}

\begin{cor}\label{cor:ffescompliseq}
  A fully faithful and essentially surjective functor between complete
  $\mathcal{V}$-\icats{} is an equivalence in
  $\AlgCat(\mathcal{V})$.
\end{cor}
\begin{proof}
  This follows by combining Proposition~\ref{propn:FFESiotaEq} and
  Lemma~\ref{lem:AlgCons}.
\end{proof}

\begin{propn}\label{propn:ffes2of3}
  Fully faithful and essentially surjective $\mathcal{V}$-functors
  satisfy the 2-out-of-3 property.
\end{propn}
\begin{proof}
  Suppose we have $\mathcal{V}$-functors $F \colon \mathcal{C} \to
  \mathcal{D}$ and $G \colon \mathcal{D} \to \mathcal{E}$. There are
  three cases to consider:
  \begin{enumerate}[(1)]
  \item Suppose $F$ and $G$ are fully faithful and essentially
    surjective. It is obvious that $G \circ F$ is fully
    faithful. Since $\pi_{0}\iota F$ and $\pi_{0}\iota G$ are
    surjective, so is their composite $\pi_{0}\iota (G\circ F)$, thus
    $G \circ F$ is also essentially surjective.
  \item Suppose $G$ and $G \circ F$ are fully faithful and essentially
    surjective. Then $F$ is also Cartesian, i.e. fully faithful, by
    \cite[Proposition 2.4.1.7]{HTT}. By
    Proposition~\ref{propn:FFESiotaEq} the maps $\iota G$ and $\iota
    (G \circ F)$ are equivalences, hence so is $\iota F$, thus $F$ is also
    essentially surjective.
  \item Suppose $F$ and $G \circ F$ are fully faithful and essentially
    surjective. By Proposition \ref{propn:FFESiotaEq} the maps $\iota
    F$ and $\iota (G \circ F)$ are equivalences, hence so is $\iota
    G$, and thus $G$ is essentially surjective. To see that $G$ is
    fully faithful, we must show that for any $X, Y$ in $\iota_{0}G$
    the map $\mathcal{D}(X,Y) \to \mathcal{E}(GX, GY)$ is an
    equivalence. But since $F$ is essentially surjective there exist
    objects $X'$, $Y'$ in $\iota_{0}\mathcal{C}$ and equivalences $FX'
    \simeq X$, $FY' \simeq Y$ in $\mathcal{D}$. Then we have a
    commutative diagram \nolabelcsquare{\mathcal{D}(FX',
      FY')}{\mathcal{E}(GFX',
      GFY')}{\mathcal{D}(X,Y)}{\mathcal{E}(GX,GY),} where the vertical
    maps are given by composition with the chosen equivalences and so
    are equivalences in $\mathcal{V}$ by
    Proposition~\ref{propn:eqsareeqonmaps}. The top horizontal map is
    also an equivalence, since in the commutative triangle
    \factortriangle{\mathcal{C}(X',Y')}{\mathcal{E}(GFX',GFY')}{\mathcal{D}(FX',FY')}{}{}{}
    the other two maps are equivalences. Thus by the 2-out-of-3
    property the bottom horizontal map $\mathcal{D}(X,Y) \to
    \mathcal{E}(GX,GY)$ is also an equivalence, and so $G$ is also
    fully faithful.\qedhere
  \end{enumerate}
\end{proof}

\begin{remark}
  Under the equivalence $\AlgCat(\mathcal{S}) \isoto \SegI$
  of Theorem~\ref{thm:AlgCatSisSeg}, the fully faithful and
  essentially surjective functors correspond to the Dwyer-Kan
  equivalences in the sense of \cite[\S 7.4]{RezkCSS}.
\end{remark}

The ``correct'' \icat{} of $\mathcal{V}$-\icats{} is obtained by
inverting the fully faithful and essentially surjective morphisms in
$\AlgCatV$. We will now show that doing this produces the same \icat{}
as inverting the fully faithful and essentially surjective functors in
the subcategory of $\AlgCatV$ where we only have sets of
objects. First, we will briefly review the general notion of
localization of \icats{} and prove a basic fact about these
(generalizing \cite[Corollary 3.6]{DwyerKanCalc}):
\begin{defn}
  The inclusion $\mathcal{S} \hookrightarrow \CatI$ has left and right
  adjoints. The right adjoint, $\iota \colon \CatI \to \mathcal{S}$,
  sends an \icat{} $\mathcal{C}$ to its maximal Kan complex, i.e. its
  subcategory of equivalences. The left adjoint $\kappa \colon \CatI
  \to \mathcal{S}$ sends an \icat{} $\mathcal{C}$ to a Kan complex
  $\kappa \mathcal{C}$ such that $\mathcal{C} \to \kappa \mathcal{C}$
  is a weak equivalence in the usual model structure on simplicial sets.
\end{defn}

\begin{defn}
  Suppose $\mathcal{C}$ is an \icat{} and $\mathcal{W}$ is a
  subcategory of $\mathcal{C}$ that contains all the equivalences. The
  \defterm{localization} $\mathcal{C}[\mathcal{W}^{-1}]$ of
  $\mathcal{C}$ with respect to $\mathcal{W}$ is the \icat{} with the
  universal property that for any \icat{} $\mathcal{E}$, a functor
  $\mathcal{C}[\mathcal{W}^{-1}] \to \mathcal{E}$ is the same thing as
  a functor $\mathcal{C} \to \mathcal{E}$ that sends morphisms in
  $\mathcal{W}$ to equivalences in $\mathcal{E}$. More precisely,
  we have for every $\mathcal{E}$ a pullback square
  \nolabelcsquare{\Map(\mathcal{C}[\mathcal{W}^{-1}],
    \mathcal{E})}{\Map(\mathcal{W},\iota\mathcal{E})}{\Map(\mathcal{C},
    \mathcal{E})}{\Map(\mathcal{W}, \mathcal{E}).}
\end{defn}

\begin{remark}
  It follows that, in the situation above, the \icat{}
  $\mathcal{C}[\mathcal{W}^{-1}]$ is given by the pushout square in
  $\CatI$
  \nolabelcsquare{\mathcal{W}}{\kappa\mathcal{W}}{\mathcal{C}}{\mathcal{C}[\mathcal{W}^{-1}].}
\end{remark}

\begin{lemma}\label{lem:locadj}
  Suppose $\mathcal{C}$ and $\mathcal{D}$ are \icats{} and
  $\mathcal{W}_{\mathcal{C}} \subseteq \mathcal{C}$ and $\mathcal{W}_{\mathcal{D}} \subseteq
  \mathcal{D}$ are subcategories containing all the equivalences. Let
  $\mathcal{C}[\mathcal{W}_{\mathcal{C}}^{-1}]$ and $\mathcal{D}[\mathcal{W}_{\mathcal{D}}^{-1}]$
  be localizations with respect to $\mathcal{W}_{\mathcal{C}}$ and $\mathcal{W}_{\mathcal{D}}$.
  Suppose
  \[ F \colon \mathcal{C} \leftrightarrows \mathcal{D} : G \]
  is an adjunction such that
  \begin{enumerate}[(1)]
  \item $F(\mathcal{W}_{\mathcal{C}}) \subseteq \mathcal{W}_{\mathcal{D}}$,
  \item $G(\mathcal{W}_{\mathcal{D}}) \subseteq \mathcal{W}_{\mathcal{C}}$,
  \item the unit morphism $\eta_{C}\colon C \to GFC$ is in $\mathcal{W}_{\mathcal{C}}$ for all $C
    \in \mathcal{C}$,
  \item the counit morphism $\gamma_{D} \colon FGD \to D$ is in $\mathcal{W}_{\mathcal{D}}$ for all $D
    \in \mathcal{D}$.
  \end{enumerate}
  Then $F$ and $G$ induce an equivalence
  $\mathcal{C}[\mathcal{W}_{\mathcal{C}}^{-1}] \simeq
  \mathcal{D}[\mathcal{W}_{\mathcal{D}}^{-1}]$.
\end{lemma}
\begin{proof}
  Let $\kappa\mathcal{W}_{\mathcal{C}}$ and $\kappa\mathcal{W}_{\mathcal{D}}$ be Kan complexes
  that are fibrant replacements for $\mathcal{W}_{\mathcal{C}}$ and $\mathcal{W}_{\mathcal{D}}$ in
  the usual model structure on simplicial sets. Then the \icats{}
  $\mathcal{C}[\mathcal{W}_{\mathcal{C}}^{-1}]$ and $\mathcal{D}[\mathcal{W}_{\mathcal{D}}^{-1}]$
  can be described as the homotopy pushouts
  \[
  \nodispnolabelcsquare{\mathcal{W}_{\mathcal{C}}}{\kappa\mathcal{W}_{\mathcal{C}}}{\mathcal{C}}{\mathcal{C}[\mathcal{W}_{\mathcal{C}}^{-1}],} \qquad
  \nodispnolabelcsquare{\mathcal{W}_{\mathcal{D}}}{\kappa\mathcal{W}_{\mathcal{D}}}{\mathcal{D}}{\mathcal{D}[\mathcal{W}_{\mathcal{D}}^{-1}]}   
\]
in the Joyal model structure. Then from (1) and (2) it is clear that
the functors $F$ and $G$ induce functors $F' \colon
\mathcal{C}[\mathcal{W}_{\mathcal{C}}^{-1}] \to \mathcal{D}[\mathcal{W}_{\mathcal{D}}^{-1}]$ and
$G' \colon \mathcal{D}[\mathcal{W}_{\mathcal{D}}^{-1}] \to
\mathcal{C}[\mathcal{W}_{\mathcal{C}}^{-1}]$, and the natural transformations $\eta$
and $\gamma$ induce natural transformations $\eta' \colon \id \to
G'F'$ and $\gamma' \colon F'G' \to \id$. The objects of
$\mathcal{C}[\mathcal{W}_{\mathcal{C}}^{-1}]$ and $\mathcal{D}[\mathcal{W}_{\mathcal{D}}^{-1}]$
can be taken to be the same as those of $\mathcal{C}$ and
$\mathcal{D}$, so by (3) and (4) the morphisms $\eta'_{c}$ and
$\gamma'_{d}$ are equivalences for all $c \in
\mathcal{C}[\mathcal{W}_{\mathcal{C}}^{-1}]$ and $d \in
\mathcal{D}[\mathcal{W}_{\mathcal{D}}^{-1}]$. Thus $\eta'$ and $\gamma'$ are natural
equivalences and $F'$ and $G'$ are hence equivalences of \icats{}.
\end{proof}

\begin{lemma}\label{lem:CartPBLocKanCx}
  Suppose $\mathcal{W}$ is an \icat{} and $\pi \colon \mathcal{E} \to \kappa\mathcal{W}$ is a
  Cartesian fibration. Let $\pi'
  \colon \mathcal{E}' \to \mathcal{W}$ denote the pullback of $\pi$
  along the canonical map $\eta \colon \mathcal{W} \to
  \kappa\mathcal{W}$. Then $\mathcal{E}$ is the localization of
  $\mathcal{E}'$ with respect to $\mathcal{W} \times_{\kappa\mathcal{W}} \iota
  \mathcal{E}$, i.e. the morphisms in $\mathcal{E}'$ that map to equivalences in $\mathcal{E}$.
\end{lemma}
\begin{proof}
  Let $F \colon \kappa\mathcal{W}^{\op} \to \CatI$ be a functor classified by
  $\pi$. Then $\pi'$ is classified by the composite functor
  $\eta^{\op} \circ F \colon \mathcal{W}^{\op} \to \CatI$. By
  \cite[Corollary 4.1.2.6]{HTT}, the functor $\eta^{\op}$ is cofinal,
  hence by \cite[Proposition 4.1.1.8]{HTT} the functors $F$ and
  $\eta^{\op} \circ F$ have the same colimit. But by \cite[Corollary
  3.3.4.3]{HTT}, the colimit of $F$ is the localization of
  $\mathcal{E}$ with respect to the $\pi$-Cartesian morphisms, and the
  colimit of $\eta^{\op} \circ F$ is the localization of
  $\mathcal{E}'$ with respect to the $\pi'$-Cartesian morphisms. But
  since $\kappa\mathcal{W}$ is a Kan complex, the $\pi$-Cartesian morphisms in
  $\mathcal{E}$ are precisely the equivalences, hence it follows that
  $\mathcal{E}$ is the localization of $\mathcal{E}'$ with respect to
  the $\pi'$-Cartesian morphisms. But the $\pi'$-Cartesian morphisms
  in $\mathcal{E}'$ are precisely the morphisms that map to
  equivalences in $\mathcal{E}$, by \cite[Remark 2.4.1.12]{HTT}.
\end{proof}

\begin{propn}\label{propn:CartPBLoc}
  Let $\mathcal{C}$ be an \icat{} and $\mathcal{W}$ a subcategory of
  $\mathcal{C}$ containing the equivalences. Suppose we have a pushout
  square in $\CatI$
  \nolabelcsquare{\mathcal{W}}{\kappa\mathcal{W}}{\mathcal{C}}{\mathcal{C}[\mathcal{W}^{-1}],}
  and a Cartesian fibration $\pi \colon \mathcal{E} \to
  \mathcal{C}[\mathcal{W}^{-1}]$. Write $\pi' \colon \mathcal{E}' \to
  \mathcal{C}$ for the pullback of $\pi$ along $\mathcal{C} \to
  \mathcal{C}[\mathcal{W}^{-1}]$. Then the map $\mathcal{E}' \to
  \mathcal{E}$ exhibits $\mathcal{E}$ as the localization of
  $\mathcal{E}'$ with respect to $\mathcal{W}
  \times_{\mathcal{C}[\mathcal{W}^{-1}]} \iota \mathcal{E}$, i.e. the
  morphisms in $\mathcal{E}'$ that map to equivalences in $\mathcal{E}$ and to $\mathcal{W}$ under the projection to $\mathcal{C}$.
\end{propn}
\begin{proof}
  Since $\pi$ is a Cartesian fibration it follows from \cite[Corollary 2.4.4.5]{HTT} that the given pushout
  square pulls back along $\pi$ to a pushout square
  \nolabelcsquare{\mathcal{W} \times_{\mathcal{C}[\mathcal{W}^{-1}]}
    \mathcal{E}}{\kappa \mathcal{W}
    \times_{\mathcal{C}[\mathcal{W}^{-1}]}
    \mathcal{E}}{\mathcal{E}'}{\mathcal{E}.}  It therefore suffices to
  show that we have a pushout square \nolabelcsquare{\mathcal{W}
    \times_{\mathcal{C}[\mathcal{W}^{-1}]} \iota \mathcal{E}}{\kappa
    \mathcal{W} \times_{\mathcal{C}[\mathcal{W}^{-1}]} \iota
    \mathcal{E}}{\mathcal{W} \times_{\mathcal{C}[\mathcal{W}^{-1}]}
    \mathcal{E}}{\kappa \mathcal{W}
    \times_{\mathcal{C}[\mathcal{W}^{-1}]} \mathcal{E},} which follows
  from Lemma~\ref{lem:CartPBLocKanCx}.
\end{proof}

\begin{thm}\label{thm:FFESloc}
  Suppose $\mathcal{V}$ is a monoidal \icat{}. Define
  $\AlgCatV_{\Set}$ by the pullback
  \csquare{\AlgCatV_{\Set}}{\AlgCatV}{\Set}{\mathcal{S}}{i}{}{}{}
  where the bottom horizontal map is the obvious inclusion. Then the
  functor $i$ induces an equivalence
  \[\AlgCatV_{\Set}[\txt{FFES}^{-1}] \isoto \AlgCatV[\txt{FFES}^{-1}]\]
  after inverting the fully faithful and essentially surjective
  functors.
\end{thm}
\begin{proof}
  Considering $\mathcal{S}$ as the \icat{} associated to the usual
  model structure on simplicial sets, we get a functor $j \colon \sSet
  \to \mathcal{S}$ that exhibits $\mathcal{S}$ as the localization of
  $\sSet$ with respect to the weak equivalences. Let
  $\AlgCat(\mathcal{V})_{\Delta}$ be the \icat{} defined by
  the pullback square
  \csquare{\AlgCat(\mathcal{V})_{\Delta}}{\AlgCat(\mathcal{V})}{\sSet}{\mathcal{S}.}{j'}{}{}{j}
  Then $\AlgCat(\mathcal{V})_{\Set}$ is the pullback of
  $\AlgCat(\mathcal{V})_{\Delta}$ along the inclusion $\Set
  \to \sSet$ of the constant simplicial sets. This has a right adjoint
  $(\blank)_{0} \colon \sSet \to \Set$ that sends a simplicial set to
  its set of 0-simplices. The inclusion \[i' \colon
  \AlgCat(\mathcal{V})_{\Set} \hookrightarrow
  \AlgCat(\mathcal{V})_{\Delta}\] therefore 
  has a right adjoint \[s \colon
  \AlgCat(\mathcal{V})_{\Delta} \to
  \AlgCat(\mathcal{V})_{\Set}\] that sends an object $(X \in
  \sSet, \mathcal{C} \in \AlgCat(\mathcal{V}))$ to the
  pullback of $\mathcal{C}$ along the morphism $X_{0} \to X \to
  \iota_{0}\mathcal{C}$. It is clear that $i'$ preserves fully
  faithful and essentially surjective functors, as does $s$ by the
  2-out-of-3 property. Moreover, $s i \simeq \id$ and the counit
  $is(\mathcal{C}) \to \mathcal{C}$ is fully faithful and
  essentially surjective for all $\mathcal{C}$. It then follows from
  Lemma~\ref{lem:locadj} that $i'$ induces an equivalence
  \[ \AlgCat(\mathcal{V})_{\Set}[\txt{FFES}^{-1}] \isoto
  \AlgCat(\mathcal{V})_{\Delta}[\txt{FFES}^{-1}] \] after
  inverting the fully faithful and essentially surjective
  functors. Moreover, by Proposition~\ref{propn:CartPBLoc} the \icat{}
  $\AlgCat(\mathcal{V})$ is the localization of
  $\AlgCat(\mathcal{V})_{\Delta}$ with respect to the morphisms that
  induce weak equivalences in $\sSet$ and project to equivalences in
  $\AlgCat(\mathcal{V})$. These are obviously among the fully faithful
  and essentially surjective functors, and so $j'$ induces an
  equivalence
  \[ \AlgCat(\mathcal{V})_{\Delta}[\txt{FFES}^{-1}] \isoto
  \AlgCat(\mathcal{V})[\txt{FFES}^{-1}].\] Composing these
  two equivalences gives the result.
\end{proof}
\begin{remark}
  Combining this result with Corollary~\ref{cor:LDopXisOX} it follows
  that the localized \icat{} $\AlgCatV[\txt{FFES}^{-1}]$ is equivalent to the
  preliminary definition of an \icat{} of $\mathcal{V}$-\icats{} we
  discussed in \S\ref{subsec:FTIopds}, using the \iopds{} associated
  to the multicategories $\mathbf{O}_{S}$ with $S$ a set.
\end{remark}

\subsection{Local Equivalences}\label{subsec:localeq}
In this subsection we consider the strongly saturated class of maps
generated by $s^{0} \colon E^{1} \to E^{0}$; we call these the
\emph{local equivalences}. We assume throughout that
$\mathcal{V}$ is a presentably monoidal \icat{}, so that
$\AlgCatV$ is a presentable \icat{} by
Proposition~\ref{propn:AlgCatPres}.

\begin{defn}
  The \defterm{local equivalences} in $\AlgCat(\mathcal{V})$
  are the elements of the strongly saturated class of morphisms generated by
  the map $s^{0} \colon E^{1} \to E^{0}$.
\end{defn}

\begin{propn}\label{propn:completeislocal}
  The following are equivalent, for a $\mathcal{V}$-\icat{}
  $\mathcal{C}$:
  \begin{enumerate}[(i)]
  \item $\mathcal{C}$ is complete.
  \item $\mathcal{C}$ is local with respect to $E^{1} \to E^{0}$,
    i.e. the map $\Map(E^{0}, \mathcal{C}) \to \Map(E^{1},
    \mathcal{C})$ is an equivalence.
  \item For every local equivalence $\mathcal{A} \to \mathcal{B}$, the
    induced map \[\Map(\mathcal{B}, \mathcal{C}) \to \Map(\mathcal{A},
    \mathcal{C})\] is an equivalence.
  \end{enumerate}
\end{propn}
\begin{proof}
  (i) is equivalent to (ii) by Corollary~\ref{cor:completeifE1}, and
  (ii) is equivalent to (iii) by \cite[Proposition 5.5.4.15(4)]{HTT}.
\end{proof}

\begin{defn}
  Write $\CatIV$ for the full subcategory of $\AlgCatV$
  spanned by the complete $\mathcal{V}$-\icats{}.
\end{defn}

\begin{propn}
  The inclusion $\CatIV \hookrightarrow \AlgCatV$ has a left adjoint,
  which exhibits $\CatIV$ as the localization of $\AlgCatV$ with
  respect to the local equivalences.
\end{propn}
\begin{proof}
  The \icat{} $\AlgCat(\mathcal{V})$ is presentable by
  Proposition~\ref{propn:AlgCatPres}, and the local equivalences are
  generated by a set of maps. The existence of the left adjoint
  therefore follows from \cite[Proposition 5.5.4.15(4)]{HTT} 
  and Proposition~\ref{propn:completeislocal}.
\end{proof}

\begin{cor}\label{cor:CatIVpres}
  The \icat{} $\CatIV$ is presentable.
\end{cor}
\begin{proof}
  This follows from \cite[Proposition 5.5.4.15(3)]{HTT}.
\end{proof}

\begin{thm}\label{thm:CatISisCatI}
  $\CatI^{\mathcal{S}}$ is equivalent to $\CatI$.
\end{thm}
\begin{proof}
  Under the equivalence $\AlgCat(\mathcal{S}) \isoto \SegI$
  of Theorem~\ref{thm:AlgCatSisSeg}, the subcategory
  $\CatI^{\mathcal{S}}$ corresponds to the subcategory of
  \emph{complete} Segal spaces. It is proved in \cite{JoyalTierney}
  that this is equivalent to $\CatI$.
\end{proof}

%

\begin{lemma}\label{lem:E1xE1localeq}
  The map $\id \otimes s^{0} \colon E^{1}\otimes E^{1} \to E^{1}
  \otimes E^{0} \simeq E^{1}$ is a local equivalence.
\end{lemma}
\begin{proof}
  Using Proposition~\ref{propn:eqCiseqUC} it suffices to prove this when
  $\mathcal{V}$ is $\mathcal{S}$. We can then identify
  $E^{1} \otimes E^{1}$ with $E^{\{0,1\} \times \{0,1\}} \simeq
  E^{3}$; under this identification the map $E^{1}\otimes E^{1} \to
  E^{1}$ is induced by the map from $\{0,1,2,3\}$ to $\{0,1\}$ that sends $0,1$
  to $0$ and $2,3$ to $1$. Under the equivalence $E^{3} \simeq
  E^{\{0,1\}} \amalg_{E^{\{1\}}} E^{\{1,2\}} \amalg_{E^{\{2\}}}
  E^{\{2,3\}}$ implied by Theorem~\ref{thm:Ecogpd}
  this corresponds to \[s^{0} \cup \id \cup s^{0} \colon
  E^{1}\amalg_{E_{0}} E^{1} \amalg_{E^{0}} E^{1} \to E^{0}
  \amalg_{E^{0}} E^{1} \amalg_{E^{0}} E^{0},\] which is clearly in the
  strongly saturated class generated by $s^{0}$.
\end{proof}

\begin{lemma}\label{lem:E1complete}
  If $\mathcal{C}$ is a complete $\mathcal{V}$-\icat{}, then the
  $\mathcal{V}$-\icat{} $\mathcal{C}^{E^{1}}$ is also complete.
\end{lemma}
\begin{proof}
  We need to show that the natural map $\iota_{0}
  \mathcal{C}^{E^{1}} \to \iota_{1} \mathcal{C}^{E^{1}}$ is an
  equivalence. Using the adjunction between cotensoring and tensoring
  we can identify this with the map $\Map(E^{1},
  \mathcal{C}) \to \Map(E^{1} \otimes E^{1}, \mathcal{C})$ induced by
  composition with $\id \otimes s^{0}$. This map is an equivalence
  since $\mathcal{C}$ is complete and $\id \otimes s^{0}$ is a local
  equivalence by Lemma~\ref{lem:E1xE1localeq}.
\end{proof}

\subsection{Categorical Equivalences}\label{subsec:cateq}
In this subsection we study \defterm{categorical equivalences} between
enriched \icats{}, which are functors with an inverse up to natural
equivalence. Our main result is that categorical equivalences are
always local equivalences as well as fully faithful and essentially
surjective. We begin by defining natural equivalences between
$\mathcal{V}$-functors:

\begin{defn}
  Suppose $\mathcal{A}$ and $\mathcal{B}$ are $\mathcal{V}$-\icats{}
  and $F, G \colon \mathcal{A} \to \mathcal{B}$ are
  $\mathcal{V}$-functors. A \defterm{natural equivalence} from $F$ to
  $G$ is a functor $H \colon \mathcal{A} \otimes
  E^{1} \to \mathcal{B}$ such that $H \circ (\id \otimes d^{1}) \simeq F$
  and $H \circ (\id \otimes d^{0}) \simeq G$. We say that $F$ and $G$
  are \emph{naturally equivalent} if there exists a natural
  equivalence from $F$ to $G$.
\end{defn}

\begin{defn}
  A functor $F \colon \mathcal{A} \to \mathcal{B}$ is a
  \defterm{categorical equivalence} if there exists a functor $G
  \colon \mathcal{B} \to \mathcal{A}$ and natural equivalences $\phi$
  from $F \circ G$ to $\id_{\mathcal{B}}$ and $\psi$ from $G \circ F$
  to $\id_{\mathcal{A}}$. Such a functor $G$ is called a
  \defterm{pseudo-inverse} of $F$; we refer to $(F,G,\phi,\psi)$ as a
  \emph{categorical equivalence datum}.
\end{defn}

\begin{propn}\label{propn:cateqisffes}
  Categorical equivalences are fully faithful and essentially
  surjective.
\end{propn}
\begin{proof}
  Suppose $F \colon \mathcal{C} \to \mathcal{D}$ is a categorical
  equivalence, and let $(F,G,\phi,\psi)$ be a categorical equivalence
  datum. For each object $X$ in $\iota_{0} \mathcal{D}$ the natural
  equivalence $\psi$ supplies an equivalence between $X$ and $FG(X)$,
  which is in the image of $F$, so $F$ is essentially surjective.

  To prove that $F$ is fully faithful, we must show that, given $X, Y$
  in $\mathcal{C}$, the map $\alpha \colon \mathcal{C}(X,Y) \to
  \mathcal{D}(FX,FY)$ induced by $F$ is an equivalence in
  $\mathcal{V}$. 

  The natural equivalence $\phi$ supplies an equivalence \[\beta
  \colon \mathcal{C}(GFX, GFY) \to \mathcal{C}(X,Y)\] and a
  commutative diagram
  \opctriangle{\mathcal{C}(X,Y)}{\mathcal{C}(GFX,GFY)}{\mathcal{C}(X,Y).}{}{\id}{\beta}
  The top map is the composite \[ \mathcal{C}(X,Y) \xto{\alpha}
  \mathcal{D}(FX,FY) \xto{\gamma} \mathcal{C}(GFX,GFY), \] where
  $\gamma$ is the map induced by $G$, and so we see that $\beta \circ
  \gamma \circ \alpha \simeq \id$.

  From $F \circ \phi$ we likewise get an equivalence \[\epsilon \colon
  \mathcal{D}(FGFX,FGFY) \to \mathcal{D}(FX,FY)\] and a commutative
  diagram
  \opctriangle{\mathcal{D}(FX,FY)}{\mathcal{D}(FGFX,FGFY)}{\mathcal{D}(FX,FY).}{}{\id}{\epsilon}
  where the top map is the composite
  \[ \mathcal{D}(FX,FY) \xto{\gamma} \mathcal{C}(GFX,GFY) \xto{\delta}
  \mathcal{D}(FGFX,FGFY),\] and so $\epsilon \circ \delta \circ \gamma
  \simeq \id$.  Moreover, we have a commutative square
  \csquare{\mathcal{C}(GFX,GFY)}{\mathcal{D}(FGFX,FGFY)}{\mathcal{C}(X,Y)}{\mathcal{D}(FX,FY),}{\delta}{\beta}{\epsilon}{\alpha}
  thus we get $\alpha \circ \beta \circ \gamma \simeq \epsilon \circ
  \delta \circ \gamma \simeq \id$. This shows that $\beta \circ
  \gamma$ is an inverse of $\alpha$, and so $\alpha$ is an equivalence
  in $\mathcal{V}$. Thus $F$ is fully faithful.
\end{proof}

\begin{cor}\label{cor:cateqcompliseq}
  A categorical equivalence between complete $\mathcal{V}$-\icats{} is
  an equivalence.
\end{cor}
\begin{proof}
  Combine Proposition~\ref{propn:cateqisffes} and
  Corollary~\ref{cor:ffescompliseq}.
\end{proof}

Our next goal is to prove that categorical equivalences are local
equivalences; this will require some preliminary results:
\begin{propn}\label{propn:cateq2of3}
  Categorical equivalences satisfy the 2-out-of-3 property.
\end{propn}
\begin{proof}
  Suppose we have functors $F \colon \mathcal{C} \to \mathcal{D}$ and
  $F' \colon \mathcal{D} \to \mathcal{E}$. There are three cases to
  consider:
  \begin{enumerate}[(1)]
  \item Suppose $(F,G,\phi,\psi)$ and $(F',G',\phi',\psi')$ are
    categorical equivalence data. Then $G \circ \phi' \circ (F \otimes
    \id)$ is a natural equivalence from $GG'F'F$ to $GF$. Combining
    this with $\phi$ gives a map \[(\mathcal{C} \otimes E^{1})
    \amalg_{\mathcal{C} \otimes E^{0}} (\mathcal{C} \otimes E^{1}) \to
    \mathcal{C}.\] But tensoring with $\mathcal{C}$ preserves colimits,
    and $E^{1} \amalg_{E^{0}} E^{1} \simeq E^{2}$ by
    Theorem~\ref{thm:Ecogpd}, so we get a map $\mathcal{C} \otimes
    E^{2} \to \mathcal{C}$. Composing with $\id \otimes d^{1} \colon
    \mathcal{C} \otimes E^{1} \to \mathcal{C} \otimes E^{2}$ we get a
    natural equivalence from $GG'F'F$ to the identity. Using the same
    argument we can also combine $F'\circ \psi \circ (G' \otimes \id)$
    and $\psi'$ to get a natural equivalence from $F'FGG'$ to the
    identity. Thus $F'F$ is a categorical equivalence with
    pseudo-inverse $GG'$.
  \item Suppose $(F,G,\phi,\psi)$ and $(F'F, H, \alpha, \beta)$ are
    categorical equivalence data. We will show that $FH$ is a
    pseudo-inverse of $F'$. Since $\beta$ is a natural equivalence
    from $F'(FH)$ to $\id$ it remains to construct a natural
    equivalence from $FHF'$ to $\id$. Let $\overline{\psi}$ denote $\psi
    \circ (\id \otimes E_{\sigma})$, where $\sigma \colon \{0,1\} \to
    \{0,1\}$ is the map that interchanges $0$ and $1$ (thus
    $\overline{\psi}$ is $\psi$ considered as a natural equivalence from $\id$ to
    $FG$). Combining $FHF' \circ \overline{\psi}$, $F \circ \alpha \circ g$
    and $\psi$ we get a map
    \[ \mathcal{D} \otimes E^{3} \simeq \mathcal{D} \otimes E^{1}
    \amalg_{\mathcal{D}} \mathcal{D}\otimes E^{1} \amalg_{\mathcal{D}}
    \mathcal{D} \otimes E^{1} \to \mathcal{D}\]
    and composing with $\mathcal{D} \otimes E_{\{0,3\}} \to
    \mathcal{D} \otimes E^{3}$ we get the required natural
    equivalence.
  \item Suppose $(F',G',\phi',\psi')$ and $(F'F,H,\alpha,\beta)$ are
    categorical equivalence data. We will show that $HF'$ is a
    pseudo-inverse of $F$. Since $\alpha$ is a natural equivalence
    from $HF'F$ to $\id$ it remains to construct a natural equivalence
    from $FHF'$ to $\id$. Let $\overline{\phi}'$ denote $\phi \circ (\id
    \otimes E_{\sigma})$; combining $\overline{\phi}' \circ FHF'$, $G'
    \circ \beta \circ F'$ and $\phi'$ we get a map
    \[ \mathcal{D} \otimes E^{3} \simeq \mathcal{De} \otimes E^{1}
    \amalg_{\mathcal{D}} \mathcal{D}\otimes E^{1} \amalg_{\mathcal{D}}
    \mathcal{D} \otimes E^{1} \to \mathcal{D},\] and composing with
    $\mathcal{D} \otimes E_{\{0,3\}} \to \mathcal{D} \otimes E^{3}$ we
    get the required natural equivalence.\qedhere
  \end{enumerate}
\end{proof}

For the rest of this subsection we will for convenience assume that
$\mathcal{V}$ is a presentably monoidal \icat{}.
\begin{cor}\label{cor:Ecateq}
  Suppose $f \colon S \to T$ is a map of sets. Then $E_{f} \colon
  E_{S} \to E_{T}$ is a categorical equivalence.
\end{cor}
\begin{proof}
  By Proposition~\ref{propn:eqCiseqUC} it suffices to prove this in
  $\mathcal{S}$. First suppose $f$ is surjective; let $g \colon T
  \hookrightarrow S$ be a section of $f$. We claim that $E_{g}$ is a
  pseudo-inverse to $E_{f}$. We have $E_{f} \circ E_{g} \simeq E_{f
    \circ g} \simeq \id$, so it suffices to construct a natural
  equivalence $E_{S} \times E^{1} \simeq E_{S \times \{0,1\}} \to
  E_{S}$ from $E_{g \circ f}$ to the identity. This is given by
  $E_{h}$ where $h \colon S \times \{0,1\} \to S$ sends $(s, 0)$ to
  $gf(s)$ and $(s,1)$ to $s$.

  By the dual argument the result holds if $f$ is injective. By
  Proposition~\ref{propn:cateq2of3} we can therefore conclude that it holds
  for a general $f$.
\end{proof}

\begin{lemma}\label{lem:expcateq}
  Suppose $F \colon \mathcal{A} \to \mathcal{B}$ is a categorical
  equivalence of $\mathcal{S}$-\icats{}. Then for any
  $\mathcal{V}$-\icat{} $\mathcal{C}$ the induced map
  $\mathcal{C}^{F}\colon \mathcal{C}^{\mathcal{B}} \to
  \mathcal{C}^{\mathcal{A}}$ is a categorical equivalence.
\end{lemma}
\begin{proof}
  A natural equivalence $\mathcal{A} \otimes E^{1} \to \mathcal{A}$
  induces a natural equivalence \[\mathcal{C}^{\mathcal{A}} \otimes
  E^{1} \to \mathcal{C}^{\mathcal{A}}\] by taking the adjoint of the
  induced map $\mathcal{C}^{\mathcal{A}} \to \mathcal{C}^{\mathcal{A}
    \otimes E^{1}} \simeq (\mathcal{C}^{\mathcal{A}})^{E^{1}}$.
\end{proof}

\begin{lemma}\label{lem:CeqCE1}
  If $\mathcal{C}$ is a complete $\mathcal{V}$-\icat{}, then the
  natural map \[\mathcal{C}^{s^{0}} \colon \mathcal{C} \simeq
  \mathcal{C}^{E^{0}} \to \mathcal{C}^{E^{1}}\] is an equivalence.
\end{lemma}
\begin{proof}
  The map $s^{0} \colon E^{1} \to E^{0}$ is a categorical equivalence
  by Corollary~\ref{cor:Ecateq}, so it follows by Lemma~\ref{lem:expcateq}
  that $\mathcal{C} \to \mathcal{C}^{E^{1}}$ is also a categorical
  equivalence. But $\mathcal{C}^{E^{1}}$ is complete by
  Lemma~\ref{lem:E1complete}, and a categorical equivalence between
  complete objects is an equivalence by
  Corollary~\ref{cor:cateqcompliseq}.
\end{proof}

\begin{propn}\label{propn:CtensE1loceq}
  For any $\mathcal{V}$-\icat{} $\mathcal{C}$, the map $\id \otimes
  s^{0} \colon \mathcal{C}
  \otimes E^{1} \to \mathcal{C} \otimes E^{0} \simeq \mathcal{C}$ is a
  local equivalence.
\end{propn}
\begin{proof}
  We must show that for any complete $\mathcal{V}$-\icat{}
  $\mathcal{D}$ the map
  \[ \Map(\mathcal{C}, \mathcal{D}) \to \Map(\mathcal{C} \otimes
  E^{1}, \mathcal{D})\]
  is an equivalence. Using the adjunction between tensoring and
  cotensoring with $E^{1}$, we see that this map is equivalent to the map
  \[\Map(\mathcal{C}, \mathcal{D}) \to \Map(\mathcal{C},  \mathcal{D}^{E^{1}})\]
  given by composing with $\mathcal{D}^{s^{0}} \colon \mathcal{D} \to \mathcal{D}^{E^{1}}$. This
  is an equivalence by Lemma~\ref{lem:CeqCE1}.
\end{proof}

\begin{cor}\label{cor:MapComplRight}
  Suppose $\mathcal{D}$ is a complete $\mathcal{V}$-\icat{}; then for
  any $\mathcal{V}$-\icat{} $\mathcal{C}$ we have \[|\Map(\mathcal{C}
  \otimes E^{\bullet}, \mathcal{D})| \simeq \Map(\mathcal{C}, \mathcal{D}).\]
\end{cor}
\begin{proof}
  The simplicial space $\Map(\mathcal{C} \otimes E^{\bullet},
  \mathcal{D})$ is a groupoid object in spaces, since $E^{\bullet}$ is
  a cogroupoid object by Theorem~\ref{thm:Ecogpd} and tensoring preserves colimits. By
  Lemma~\ref{lem:gpdcolim} it therefore suffices to show that
  $\Map(\mathcal{C} \otimes E^{0}, \mathcal{D}) \to \Map(\mathcal{C}
  \otimes E^{1}, \mathcal{D})$ is an equivalence, which holds by
  Proposition~\ref{propn:CtensE1loceq}.
\end{proof}

\begin{remark}
  The left-hand side here is what we would expect the mapping space to
  be in the \icat{} underlying an $(\infty,2)$-category of
  $\mathcal{V}$-\icats{}, functors, and natural transformations. This
  shows that the mapping spaces between complete
  $\mathcal{V}$-\icats{} are the correct ones.
\end{remark}

\begin{lemma}\label{lem:complE1htpy}
  Suppose $\mathcal{D}$ is a complete $\mathcal{V}$-\icat{}. Then for
  any $\mathcal{V}$-\icat{} $\mathcal{C}$ the two maps
  \[ (\id \otimes d^{0})^{*}, (\id \otimes d^{1})^{*} \colon
  \Map(\mathcal{C} \otimes E^{1}, \mathcal{D}) \to \Map(\mathcal{C},
  \mathcal{D})\]
  are homotopic.
\end{lemma}
\begin{proof}
  Clearly $(\id \otimes s^{0})^{*} \circ (\id \otimes d^{i})^{*}
  \colon \Map(\mathcal{C}, \mathcal{D}) \to \Map(\mathcal{C},
  \mathcal{D})$ is homotopic to the identity for $i = 0,1$. But by
  Proposition~\ref{propn:CtensE1loceq}, the map $(\id \otimes s^{0})$
  is a local equivalence, hence $(\id \otimes s^{0})^{*}$ is an
  equivalence since $\mathcal{D}$ is complete. Composing with its
  inverse we get that \[(\id \otimes d^{0})^{*} \simeq (\id \otimes
  d^{1})^{*},\] as required.
\end{proof}

\begin{thm}\label{thm:cateqislocaleq}
  Categorical equivalences are local equivalences.
\end{thm}
\begin{proof}
  Suppose $F \colon \mathcal{C} \to \mathcal{D}$ is a categorical
  equivalence and $(F,G, \phi,\psi)$ is a categorical equivalence
  datum. If $\mathcal{E}$ is a complete $\mathcal{V}$-\icat{} we must
  show that the map
  \[ F^{*} \colon \Map(\mathcal{C}, \mathcal{E}) \to \Map(\mathcal{D},
  \mathcal{E}) \] given by composition with $F$ is an equivalence of spaces. By
  Lemma~\ref{lem:complE1htpy} we have equivalences
  \[ G^{*}F^{*} \simeq \phi^{*} \circ (\id \otimes d^{1})^{*} \simeq
  \phi^{*} \circ (\id \otimes d^{0})^{*} \simeq \id,\]
  \[ F^{*}G^{*} \simeq \psi^{*} \circ (\id \otimes d^{1})^{*} \simeq
  \psi^{*} \circ (\id \otimes d^{0})^{*} \simeq \id.\] Thus $G^{*}$ is
  an inverse of $F^{*}$, and so $F^{*}$ is indeed an equivalence.
\end{proof}

\subsection{Completion}\label{subsec:completion}
We will now construct an explicit completion functor, analogous to
Rezk's completion functor for Segal spaces in \cite[\S 14]{RezkCSS},
when $\mathcal{V}$ is a presentably monoidal \icat{}. Using
this we can then show that the local equivalences are precisely the
fully faithful and essentially surjective functors.

\begin{defn}
  If $\mathcal{C}$ is a $\mathcal{V}$-\icat{}, let
  $\widehat{\mathcal{C}}$ denote the geometric realization
  $|\mathcal{C}^{E^{\bullet}}|$.
\end{defn}

\begin{thm}\label{thm:completion}
  Suppose $\mathcal{V}$ is a presentably monoidal \icat{}
  and $\mathcal{C}$ is a $\mathcal{V}$-\icat{}. The natural map
  $\mathcal{C} \to \widehat{\mathcal{C}}$ is both a local equivalence
  and fully faithful and essentially surjective. Moreover, the
  $\mathcal{V}$-\icat{} $\widehat{\mathcal{C}}$ is complete.
\end{thm}

\begin{proof}
  The functors $E^{n} \to E^{m}$ induced by the maps $[n] \to [m]$ in
  $\simp$ are categorical equivalences by Corollary~\ref{cor:Ecateq},
  so the induced functors $\mathcal{C}^{E^{m}} \to
  \mathcal{C}^{E^{n}}$ are also categorical equivalences by
  Lemma~\ref{lem:expcateq}. These functors are therefore all fully
  faithful and essentially surjective by
  Proposition~\ref{propn:cateqisffes}, and local equivalences by
  Theorem~\ref{thm:cateqislocaleq}. Local equivalences are by
  definition closed under colimits, so it follows that the map
  $\mathcal{C} \to \widehat{\mathcal{C}}$ is a local equivalence.

  Since $\iota_{0}$ preserves colimits, the map $\iota_{0}\mathcal{C}
  \to \iota_{0}\widehat{\mathcal{C}} \simeq \iota \mathcal{C}$ is
  surjective on $\pi_{0}$, and so the functor $\mathcal{C} \to
  \widehat{\mathcal{C}}$ is essentially surjective. To see that this
  functor is also fully faithful, we consider the model for categorical
  algebras as Segal presheaves from \S\ref{subsec:presheafalgcat}. If
  $\mathcal{V}$ is $\kappa$-presentable, then the colimit
  $\widehat{\mathcal{C}}$ in $\AlgCatV \simeq
  \mathcal{P}(\mathcal{V}^{\vee}_{\otimes})^{\Seg}$ can be described as a
  localization of the colimit $\widehat{F}$ of the diagram
  $F_{\bullet} \colon \simp^{\op} \to
  \mathcal{P}((\mathcal{V}^{\vee}_{\otimes})^{\kappa})$ corresponding
  to $\mathcal{C}^{E^{\bullet}}$. The colimit $\widehat{F}$ can be
  computed objectwise, and in fact is already local: Given $X \in
  (\mathcal{V}^{\vee}_{\otimes})_{[k]}$, we know that for every $\phi \colon [m] \to [n]$
  in $\simp^{\op}$ the diagram
  \nolabelcsquare{F_{m}(X)}{F_{n}(X)}{F_{m}()^{\times
      (k+1)}}{F_{n}()^{\times (k+1)}} is a pullback square. Since
  $\mathcal{S}$ is an $\infty$-topos, by
  \cite[Theorem 6.1.3.9]{HTT} it follows that the square
  \nolabelcsquare{F_{0}(X)}{\widehat{F}(X)}{F_{0}()^{\times
      (k+1)}}{\widehat{F}()^{\times (k+1)}} is also a pullback
  square. From this we conclude that $\widehat{F}$ is also a Segal
  presheaf, since the map $\widehat{F}(X) \to \widehat{F}()^{\times
    (k+1)} \simeq |F_{\bullet}()|^{\times (k+1)}$ has the same fibres
  as $F_{0}(X) \to F_{0}()^{\times (k+1}$ and $F_{0}() \to
  \widehat{F}()$ is surjective on $\pi_{0}$. Using the equivalence
  between Segal presheaves and categorical algebras of
  Theorem~\ref{thm:Segpsheaves} we conclude that $\mathcal{C} \to
  \widehat{\mathcal{C}}$ is fully faithful, as the object
  $\widehat{\mathcal{C}}(x,y)$ is determined by the fibre of
  $\widehat{F}(A) \to \widehat{F}()^{\times 2} \simeq
  (\iota\mathcal{C})^{\times 2}$ at $(x,y)$ for all $A \in
  \mathcal{V}$.

  It remains to prove that $\widehat{\mathcal{C}}$ is complete,
  i.e. that the map $\iota_{0}\widehat{\mathcal{C}} \to
  \iota_{1}\widehat{\mathcal{C}}$ is an equivalence. We have a
  commutative diagram \nolabelcsquare{{|\iota_{0}
      \mathcal{C}^{E^{\bullet}}|}}{\iota_{0}\widehat{\mathcal{C}}}{{|\iota_{1}
      \mathcal{C}^{E^{\bullet}}|}}{\iota_{1}\widehat{\mathcal{C}},}
  where the top horizontal morphism is an equivalence since
  $\iota_{0}$ preserves colimits. The left vertical map is also an
  equivalence: We have equivalences $\iota_{1}\mathcal{C}^{E^{n}}
  \simeq \Map(E^{1} \otimes E^{n}, \mathcal{C}) \simeq
  \iota_{n}\mathcal{C}^{E^{1}}$, so
  $|\iota_{1}\mathcal{C}^{E^{\bullet}}| \simeq \iota
  \mathcal{C}^{E^{1}}$, and under this equivalence the left vertical
  map corresponds to that induced by the natural map $\mathcal{C} \to
  \mathcal{C}^{E^{1}}$; we know that this is fully faithful and
  essentially surjective, and so induces an equivalence on $\iota$ by
  Proposition~\ref{propn:FFESiotaEq}. In order to show that
  $\widehat{\mathcal{C}}$ is complete, it thus suffices to show that
  the bottom horizontal map $|\iota_{1}\mathcal{C}^{E^{\bullet}}| \to
  \iota_{1}\widehat{\mathcal{C}}$ is an equivalence.

  Consider the commutative diagram \nolabelcsquare{{|\iota_{1}
      \mathcal{C}^{E^{\bullet}}|}}%
  {\iota_{1}\widehat{\mathcal{C}}}%
  {{|\iota_{0} \mathcal{C}^{E^{\bullet}}|^{\times 2}}}%
  {\iota_{0} \widehat{\mathcal{C}}^{\times 2},} with the vertical maps
  coming from the maps $d^{0},d^{1}\colon E^{0} \to E^{1}$.  Here the
  bottom horizontal map is an equivalence, so to prove that the top
  horizontal map is an equivalence it suffices to prove that this is a
  pullback square. Since $\mathcal{C} \to \widehat{\mathcal{C}}$ is
  essentially surjective, to see this we need only show that for all
  $(X, Y) \in \iota_{0}\mathcal{C}^{\times 2}$ the induced map on
  fibres $|\iota_{1}\mathcal{C}^{E^{\bullet}}|_{(X,Y)} \to
  \iota_{1}\widehat{\mathcal{C}}_{(X,Y)}$ is an equivalence.

  Since $\mathcal{C}^{E^{m}} \to \mathcal{C}^{E^{n}}$ is fully
  faithful and essentially surjective for all $[n] \to [m]$ in
  $\simp^{\op}$, the map $\iota \mathcal{C}^{E^{m}}\to \iota
  \mathcal{C}^{E^{n}}$ is an equivalence by Proposition
  \ref{propn:FFESiotaEq}. Therefore, as the groupoid objects
  $\iota_{\bullet}\mathcal{C}^{E^{m}}$ and
  $\iota_{\bullet}\mathcal{C}^{E^{n}}$ are effective, the diagram
  \nolabelcsquare{\iota_{1} \mathcal{C}^{E^{m}}}{\iota_{1}
    \mathcal{C}^{E^{n}}}{(\iota_{0} \mathcal{C}^{E^{m}})^{\times
      2}}{(\iota_{0} \mathcal{C}^{E^{n}})^{\times 2}} is a pullback
  square. In other words, the natural transformation
  $\iota_{1}\mathcal{C}^{E^{\bullet}} \to
  (\iota_{0}\mathcal{C}^{E^{\bullet}})^{\times 2}$ is
  Cartesian. Applying \cite[Theorem 6.1.3.9]{HTT} again, we see that
  the extended natural transformation of functors
  $(\simp^{\op})^{\triangleright} \to \mathcal{S}$ that includes the
  colimits is also Cartesian. Thus we have a pullback square
  \nolabelcsquare{\iota_{1}
    \mathcal{C}}{{|\iota_{1}\mathcal{C}^{E^{\bullet}}|}}{\iota_{0}
    \mathcal{C}^{\times
      2}}{{|\iota_{0}\mathcal{C}^{E^{\bullet}}|^{\times 2}}.} In
  particular, for $(X,Y) \in \iota_{0}\mathcal{C}^{\times 2}$ the
  induced map on fibres $\iota_{1}\mathcal{C}_{(X,Y)} \to
  |\iota_{1}\mathcal{C}^{E^{\bullet}}|_{(X,Y)}$ is an equivalence.
  Since $\mathcal{C} \to \widehat{\mathcal{C}}$ is fully faithful and
  essentially surjective, the map $\iota_{1}\mathcal{C}_{(X,Y)} \to
  \iota_{1}\widehat{\mathcal{C}}_{(X,Y)}$ is also an equivalence by
  Lemma~\ref{lem:FFeqiota1}. By the 2-out-of-3 property it then
  follows that $|\iota_{1}\mathcal{C}^{E^{\bullet}}|_{(X,Y)} \to
  \iota_{1}\widehat{\mathcal{C}}_{(X,Y)}$ is an equivalence too. This
  completes the proof that $\widehat{\mathcal{C}}$ is complete.
\end{proof}

\begin{cor}\label{cor:ffeseqloceq}
  Suppose $\mathcal{V}$ is a presentably monoidal
  \icat{}. The following are equivalent, for a functor $F \colon
  \mathcal{C} \to \mathcal{D}$ of $\mathcal{V}$-\icats{}:
  \begin{enumerate}[(i)]
  \item $F$ is a local equivalence.
  \item $F$ is fully faithful and essentially surjective.
  \end{enumerate}
\end{cor}
\begin{proof}
  By Theorem~\ref{thm:completion} we have a commutative diagram
  \csquare{\mathcal{C}}{\mathcal{D}}{\widehat{\mathcal{C}}}{\widehat{\mathcal{D}},}{F}{}{}{\widehat{F}}
  where the vertical maps are both local equivalences and fully
  faithful and essentially surjective, and $\widehat{\mathcal{C}}$ and
  $\widehat{\mathcal{D}}$ are complete. 

  Since local equivalences form a strongly saturated class of
  morphisms, it follows from the 2-out-of-3 property that $F$ is a
  local equivalence \IFF{} $\widehat{F}$ is a local equivalence,
  i.e. \IFF{} $\widehat{F}$ is an equivalence, since
  $\widehat{\mathcal{C}}$ and $\widehat{\mathcal{D}}$ are complete.
  
  Fully faithful and essentially surjective functors also satisfy the
  2-out-of-3 property, by Proposition~\ref{propn:ffes2of3}, so $F$ is
  fully faithful and essentially surjective \IFF{} $\widehat{F}$ is.
  But by Corollary \ref{cor:ffescompliseq} the functor $\widehat{F}$ is fully
  faithful and essentially surjective \IFF{} it is an equivalence,
  since $\widehat{\mathcal{C}}$ and $\widehat{\mathcal{D}}$ are
  complete. Thus $F$ is a local equivalence \IFF{} it is fully
  faithful and essentially surjective.
\end{proof}

\begin{cor}\label{cor:CatIVFFESLoc}
  Suppose $\mathcal{V}$ is a presentably monoidal
  \icat{}. The \icat{} $\CatIV$ is the localization of
  $\AlgCat(\mathcal{V})$ with respect to the fully faithful
  and essentially surjective functors.
\end{cor}

\begin{remark}\label{rmk:cateqNOTffes}
  We might expect that the fully faithful and essentially surjective
  functors also coincide with the categorical equivalences, but this
  turns out \emph{not} to be the case when we allow spaces of
  objects. To see this, first observe that if $F \colon \mathcal{A}
  \to \mathcal{B}$ is a categorical equivalence, then for every
  $\mathcal{V}$-\icat{} $\mathcal{C}$ the map \[F_{*} \colon
  |\Map(\mathcal{C} \otimes E^{\bullet}, \mathcal{A})| \to
  |\Map(\mathcal{C} \otimes E^{\bullet}, \mathcal{B})|\] is surjective
  on $\pi_{0}$: suppose $G \colon \mathcal{B} \to \mathcal{A}$ is a
  pseudo-inverse of $F$, then given a functor $\phi \colon \mathcal{C}
  \to \mathcal{B}$ the natural equivalence from $F \circ G$ to $\id$
  gives a natural equivalence from $F \circ G \circ \phi$ to $\phi$, so
  up to natural equivalence $\phi$ is in the image of $F_{*}$. Now if
  $\mathcal{B} \to \widehat{\mathcal{B}}$ is a categorical equivalence
  where $\widehat{\mathcal{B}}$ is complete, then by
  Corollary~\ref{cor:MapComplRight} we have $|\Map(\mathcal{C} \otimes
  E^{\bullet}, \widehat{\mathcal{B}})| \simeq \Map(\mathcal{C},
  \widehat{\mathcal{B}})$, and since $\Map(\mathcal{C} \otimes
  E^{\bullet}, \mathcal{B})$ is a groupoid object the map
  $\Map(\mathcal{C}, \mathcal{B}) \to |\Map(\mathcal{C} \otimes
  E^{\bullet}, \mathcal{B})|$ is surjective on $\pi_{0}$. Thus
  $\Map(\mathcal{C}, \mathcal{B}) \to \Map(\mathcal{C},
  \widehat{\mathcal{B}})$ is surjective on $\pi_{0}$. 

  Now suppose $\iota_{0}\mathcal{B}$ is discrete and
  $\iota\mathcal{B}$ is not; then there clearly exists for some $n >
  0$ a map from the $n$-sphere $S^{n}\to \iota \mathcal{B}$ that does
  not factor through $\iota_{0}\mathcal{B}$. But we have a
  $\mathcal{V}$-\icat{} $S^{n}\otimes E^{0}$ such that
  $\Map(S^{n}\otimes E^{0}, \mathcal{B}) \simeq \Map(S^{n},
  \iota_{0}\mathcal{B})$ --- so if $\mathcal{B} \to
  \widehat{\mathcal{B}}$ were a categorical equivalence then
  $\Map(S^{n}, \iota_{0}\mathcal{B}) \to \Map(S^{n}, \iota
  \mathcal{B})$ would have to be surjective on $\pi_{0}$, a
  contradiction. Thus completion maps $\mathcal{B} \to
  \widehat{\mathcal{B}}$ therefore cannot be categorical equivalences in
  general.
\end{remark}

We now deduce our main result for a general large monoidal \icat{}
$\mathcal{V}$ from the presentable case, by embedding in a
larger universe:
\begin{thm}
  Let $\mathcal{V}$ be a large monoidal \icat{}. The
  inclusion of the full subcategory of complete $\mathcal{V}$-\icats{}
  $\CatIV \hookrightarrow \AlgCat(\mathcal{V})$ has a left
  adjoint that exhibits $\CatIV$ as the localization of
  $\AlgCat(\mathcal{V})$ with respect to the fully faithful
  and essentially surjective functors.
\end{thm}
\begin{proof}
  Let $\widehat{\mathcal{P}}(\mathcal{V})$ be the \icat{} of
  presheaves of large spaces on $\mathcal{V}$. By \cite[Proposition
  4.8.1.10]{HA} there exists a monoidal structure on
  $\widehat{\mathcal{P}}(\mathcal{V})$ such that the Yoneda embedding
  $j \colon \mathcal{V} \to \widehat{\mathcal{P}}(\mathcal{V})$ is a
  monoidal functor. Let
  $\widehat{\txt{Alg}}_{\txt{cat}}(\widehat{\mathcal{P}}(\mathcal{V}))$
  be the (very large) \icat{} of large categorical algebras in
  $\widehat{\mathcal{P}}(\mathcal{V})$; this is a presentable
  \icat{}, and writing $\LCatI^{\widehat{\mathcal{P}}(\mathcal{V})}$
  for its subcategory of complete
  $\widehat{\mathcal{P}}(\mathcal{V})$-\icats{} we know from
  Corollary~\ref{cor:CatIVFFESLoc} that the
  inclusion \[\LCatI^{\widehat{\mathcal{P}}(\mathcal{V})}
  \hookrightarrow
  \widehat{\txt{Alg}}_{\txt{cat}}(\widehat{\mathcal{P}}(\mathcal{V}))\]
  has a left adjoint $\widehat{L}$ that exhibits
  $\LCatI^{\widehat{\mathcal{P}}(\mathcal{V})}$ as the localization
  with respect to the fully faithful and essentially surjective
  functors.

  If $\mathcal{C}$ is in the essential image of the fully faithful
  inclusion \[\AlgCat(\mathcal{V}) \hookrightarrow
  \widehat{\txt{Alg}}_{\txt{cat}}(\widehat{\mathcal{P}}(\mathcal{V})),\]
  then the natural map $\mathcal{C} \to \widehat{L}\mathcal{C}$ is
  fully faithful and essentially surjective. But then
  $\iota_{0}\widehat{L}\mathcal{C} \simeq \iota \mathcal{C}$, so
  $\iota_{0} \widehat{L}\mathcal{C}$ is an (essentially) small space,
  and the mapping objects in $\widehat{L}\mathcal{C}$ are in the
  essential image of $\mathcal{V}$ in $\widehat{\mathcal{P}}(\mathcal{V})$. Thus
  $\widehat{L}\mathcal{C}$ is in the essential image of
  $\AlgCat(\mathcal{V})$, and so the functor $\widehat{L}$
  restricts to a functor $L \colon \AlgCat(\mathcal{V}) \to
  \CatIV$, since $\CatIV$ is equivalent to the full subcategory of
  $\LCatI^{\widehat{\mathcal{P}}(\mathcal{V})}$ spanned by
  objects in the essential image of $\AlgCat(\mathcal{V})$.
\end{proof}

\subsection{Properties of the Localized $\infty$-Category}\label{subsec:proploccat}
In this subsection we observe that the localized \icat{} $\CatIV$
inherits the naturality properties of $\AlgCat(\mathcal{V})$. We first
show that $\CatIV$ is functorial in $\mathcal{V}$:

\begin{propn}\label{propn:EnrAdj}
  Let \[\AlgCat \to \widehat{\txt{Mon}}^{\txt{lax}}_{\infty}\] be a
  coCartesian fibration corresponding to the functor
  $\AlgCat(\blank)$.  Define $\catname{Enr}_{\infty}$ to be the full
  subcategory of $\AlgCat$ whose objects are the
  complete enriched \icats{}. Then the restricted projection
  \[\txt{Enr}_{\infty} \to
  \widehat{\txt{Mon}}^{\txt{lax}}_{\infty}\]
  is a coCartesian fibration, and the inclusion $\txt{Enr}_{\infty}
  \hookrightarrow \AlgCat$ admits a left adjoint over
  $\widehat{\txt{Mon}}^{\txt{lax}}_{\infty}$.
\end{propn}

This follows from a general result about fibrewise localizations of
coCartesian fibrations that we prove first:
\begin{lemma}
  Suppose $\mathcal{E} \to \Delta^{1}$ is a coCartesian
  fibration, and $\mathcal{E}'$ is a full subcategory of
  $\mathcal{E}$ such that the inclusion $\mathcal{E}'_{1}
  \hookrightarrow \mathcal{E}_{1}$ admits a left adjoint $L \colon
  \mathcal{E}_{1} \to \mathcal{E}'_{1}$. Then the restriction
  $\mathcal{E}' \to \Delta^{1}$ is also a
  coCartesian fibration.
\end{lemma}
\begin{proof}
  We must show that for each $x \in \mathcal{E}'_{0}$ there exists a
  coCartesian arrow with source $x$ over $0 \to 1$ in
  $\Delta^{1}$. Suppose $\phi \colon x \to y$ is such a coCartesian
  arrow in $\mathcal{E}$, and let $y \to Ly$ be the unit of the
  adjunction. Then the composite $x \xto{\phi} y \to Ly$ is a
  coCartesian arrow in $\mathcal{E}'$: by
  \cite[Proposition 2.4.4.3]{HTT} it suffices to show that for
  all $z \in \mathcal{E}'_{1}$ the map $\Map_{\mathcal{E}'}(Ly,z) \to
  \Map_{\mathcal{E}'}(x, z)$ is an equivalence, which is clear since
  $\Map_{\mathcal{E}'}(Ly,z) \simeq \Map_{\mathcal{E}}(y,z)$ as $z \in
  \mathcal{E}'_{1}$, $\Map_{\mathcal{E}'}(x,z) \simeq
  \Map_{\mathcal{E}}(x,z)$ as $\mathcal{E}'$ is a full subcategory of
  $\mathcal{E}$, and $x \to y$ is a coCartesian morphism in
  $\mathcal{E}$.
\end{proof}

\begin{lemma}\label{lem:LocCoCartLoc}
  Let $\mathcal{E} \to \mathcal{B}$ be a locally coCartesian
  fibration and $\mathcal{E}^{0}$ a full subcategory of
  $\mathcal{E}$ such that for each $b \in \mathcal{B}$
  the induced map on fibres $\mathcal{E}^{0}_{b} \hookrightarrow
  \mathcal{E}_{b}$ admits a left adjoint $L_{b} \colon \mathcal{E}_{b}
  \to \mathcal{E}^{0}_{b}$. Assume these localization functors are compatible in the sense that
  the following condition is satisfied:
  \begin{itemize}
  \item[($*$)] Suppose $f \colon b \to b'$ is a morphism in
    $\mathcal{B}$ and $e$ is an object of $\mathcal{E}_{b}$. Let $e
    \to  e'$ and $L_{b}e \to e''$ be locally coCartesian arrows lying over
    $f$, and let $L_{b'}e' \to
    L_{b'}e''$ be the unique morphism such that the diagram
     \[ 
   \begin{tikzpicture} 
 \matrix (m) [matrix of math nodes,row sep=3em,column sep=2.5em,text height=1.5ex,text depth=0.25ex] %
 {
   e & e' & L_{b'}e' \\
   L_{b}e & e'' & L_{b'}e'' \\
   };
 \path[->,font=\scriptsize] %
 (m-1-1) edge (m-1-2)
 (m-1-2) edge (m-1-3)
 (m-2-1) edge (m-2-2)
 (m-2-2) edge (m-2-3)
 (m-1-1) edge (m-2-1)
 (m-1-2) edge (m-2-2)
 (m-1-3) edge (m-2-3);
 \end{tikzpicture}
 \]
commutes. Then the morphism $L_{b'}e' \to L_{b'}e''$ is an equivalence.
  \end{itemize}
 Then 
  \begin{enumerate}[(i)]
  \item the composite map $\mathcal{E}^{0} \to \mathcal{B}$ is also a
    locally coCartesian fibration,
  \item the inclusion $\mathcal{E}^{0} \hookrightarrow \mathcal{E}$
    admits a left adjoint $L \colon \mathcal{E} \to \mathcal{E}^{0}$
    relative to $\mathcal{B}$.
  \end{enumerate}
\end{lemma}
\begin{proof}
  (i) is immediate from the previous lemma, and then (ii) follows from
  \cite[Proposition 7.3.2.11]{HA} --- condition (2) of this result
  is satisfied since, in the notation of condition ($*$), a locally coCartesian
  arrow in $\mathcal{E}^{0}$ over $f$ with source $L_{b}e$ is given by the
  composite $L_{b}e \to e'' \to L_{b'}e''$.
\end{proof}

\begin{propn}\label{propn:CoCartLoc}
  Let $\mathcal{E} \to \mathcal{B}$ be a coCartesian fibration and
  $\mathcal{E}^{0}$ a full subcategory of $\mathcal{E}$. Suppose
  that for each $b \in \mathcal{B}$ the induced map on fibres
  $\mathcal{E}^{0}_{b} \hookrightarrow \mathcal{E}_{b}$ admits a left
  adjoint $L_{b} \colon \mathcal{E}_{b} \to \mathcal{E}^{0}_{b}$ and
  that the functors $\phi_{!} \colon \mathcal{E}_{b} \to
  \mathcal{E}_{b'}$ corresponding to morphisms $\phi \colon b \to b'$ in
  $\mathcal{B}$ preserve the fibrewise local equivalences. Then
  \begin{enumerate}[(i)]
  \item the composite map $\mathcal{E}^{0} \to \mathcal{B}$ is a coCartesian fibration,
  \item the inclusion $\mathcal{E}^{0} \hookrightarrow \mathcal{E}$
    admits a left adjoint $L \colon \mathcal{E} \to \mathcal{E}^{0}$
    over $\mathcal{B}$, and $L$ preserves coCartesian arrows.
  \end{enumerate}
\end{propn}
\begin{proof}
  Lemma~\ref{lem:LocCoCartLoc} implies (ii) and also that
  $\mathcal{E}^{0} \to \mathcal{E} \to \mathcal{B}$ is a locally
  coCartesian fibration, since for a coCartesian fibration condition ($*$) says
  precisely that fibrewise local equivalences are preserved by the
  functors $\phi_{!}$. By \cite[Proposition 2.4.2.8]{HTT} it remains to show that locally coCartesian
  morphisms are closed under composition. Suppose $f \colon b \to b'$
  and $g \colon b' \to b''$ are morphisms in $\mathcal{B}$, and that
  $e \in \mathcal{E}^{0}_{b}$. Let $e \to e'$ be a coCartesian arrow
  in $\mathcal{E}$ over $f$, and let $e' \to e''_{1}$ and $L_{b'}e'
  \to e''_{2}$ be coCartesian arrows in $\mathcal{E}$ over $g$. Then a
  locally coCartesian arrow over $f$ in $\mathcal{E}^{0}$ is given by
  $e \to e' \to L_{b'}e'$ and a locally coCartesian arrow over $g$ is
  given by $L_{b'}e' \to e''_{2} \to L_{b''}e''_{2}$. We have a
  commutative diagram
  \[ %
  \begin{tikzpicture} %
\matrix (m) [matrix of math nodes,row sep=3em,column sep=2.5em,text height=1.5ex,text depth=0.25ex] %
{
  e & e' & e''_{1} & L_{b''}e''_{1} \\
    & L_{b'}e' & e''_{2} & L_{b''}e''_{2} \\
  };
\path[->,font=\scriptsize] %
(m-1-1) edge (m-1-2)
(m-1-2) edge (m-1-3)
(m-1-3) edge (m-1-4)
(m-2-2) edge (m-2-3)
(m-2-3) edge (m-2-4)
(m-1-1) edge (m-2-2)
(m-1-2) edge (m-2-2)
(m-1-3) edge (m-2-3)
(m-1-4) edge (m-2-4);
\end{tikzpicture}%
\]%
Here the composite along the top row is a locally coCartesian arrow
for $gf$, and the composite along the bottom is the composite of
locally coCartesian arrows for $g$ and $f$. By condition ($*$) of \ref{lem:LocCoCartLoc}, the rightmost
vertical morphism is an equivalence, hence the composite map $e \to L_{b''}e''_{2}$
is locally coCartesian.
\end{proof}

\begin{lemma}\label{lem:LaxMonPresFFES}
  Suppose $\phi \colon \mathcal{V}^{\otimes} \to
  \mathcal{W}^{\otimes}$ is a lax monoidal functor. Then the induced
  functor \[\phi_{*}  \colon \AlgCatV \to
  \AlgCat(\mathcal{W})\]
  preserves fully faithful and essentially surjective morphisms.
\end{lemma}
\begin{proof}
  It is obvious from the definitions that $\phi_{*}$
  preserves fully faithful functors. To see that it preserves
  essentially surjective ones we note that if two points of
  $\iota_{0}\mathcal{C}$ are equivalent as objects of $\mathcal{C}$
  then they are also equivalent as objects of $\phi_{*}\mathcal{C}$,
  since the map $I_{\mathcal{W}} \to \phi(I_{\mathcal{V}})$ induces a
  functor $E^{1}_{\mathcal{W}} \to \phi_{*}E^{1}_{\mathcal{V}}$.
\end{proof}

\begin{proof}[Proof of Proposition~\ref{propn:EnrAdj}]
  The result follows by combining Proposition \ref{propn:CoCartLoc}
  and Lemma~\ref{lem:LaxMonPresFFES}.
\end{proof}

\begin{cor}
  $\CatIV$ is functorial in $\mathcal{V}$ with respect to lax monoidal
  functors of monoidal \icats{}.
\end{cor}
\begin{proof}
  The coCartesian fibration $\txt{Enr}_{\infty} \to
  \widehat{\txt{Mon}}^{\txt{lax}}_{\infty}$ of
  Proposition~\ref{propn:EnrAdj} classifies a functor
  $\widehat{\txt{Mon}}^{\txt{lax}}_{\infty} \to \CatI$ that sends
  a monoidal \icat{} $\mathcal{V}$ to $\CatIV$.
\end{proof}

\begin{lemma}\label{lem:CatIVstrmoncolim}
  Suppose $\mathcal{V}$ and $\mathcal{W}$ are
  monoidal \icats{} compatible with small colimits, and $F
  \colon \mathcal{C}^{\otimes} \to \mathcal{D}^{\otimes}$ is a 
  monoidal functor such that $F_{[1]} \colon \mathcal{V} \to \mathcal{W}$
  preserves colimits. Then the induced functor $F_{*}\colon \CatIV \to
  \CatI^{\mathcal{W}}$ preserves colimits.
\end{lemma}
\begin{proof}
  This functor $F_{*}$ is the composite \[\CatIV \hookrightarrow
  \AlgCat(\mathcal{V}) \xto{F_{*}^{\txt{Alg}}}
  \AlgCat(\mathcal{W}) \xto{L_{\mathcal{W}}}
  \CatI^{\mathcal{W}},\] where $L_{\mathcal{W}}$ is the completion
  functor for $\mathcal{W}$ and we write $F_{*}^{\txt{Alg}}$ for the
  functor on $\AlgCat$ induced by composition with $F$ for clarity. By
  Lemma~\ref{lem:LaxMonPresFFES} the functor $F^{\txt{Alg}}_{*}$
  preserves local equivalences, so
  $F_{*}^{\txt{Alg}}L_{\mathcal{V}}\mathcal{C}$ and
  $F^{\txt{Alg}}_{*}\mathcal{C}$ are locally equivalent for all
  $\mathcal{C}$; it follows that $L_{\mathcal{W}} \circ
  F^{\txt{Alg}}_{*} \circ L_{\mathcal{V}} \simeq L_{\mathcal{W}} \circ
  F^{\txt{Alg}}_{*}$. If $\alpha \mapsto \mathcal{C}_{\alpha}$ is a
  diagram in $\CatIV$ then its colimit is $L_{\mathcal{V}}(\colim
  \mathcal{C}_{\alpha})$ where this colimit is computed in
  $\AlgCat(\mathcal{V})$. Thus we have
  \[
  \begin{split}
    F_{*}(\colim \mathcal{C}_{\alpha}) & \simeq L_{\mathcal{W}} F^{\txt{Alg}}_{*}
    L_{\mathcal{V}}(\colim \mathcal{C}_{\alpha}) \simeq
    L_{\mathcal{W}}F^{\txt{Alg}}_{*}(\colim
    \mathcal{C}_{\alpha}) \\
    & \simeq \colim
    L_{\mathcal{W}}F^{\txt{Alg}}_{*}(\mathcal{C}_{\alpha}) \simeq
    \colim F_{*}\mathcal{C}_{\alpha}.\qedhere
  \end{split}\]
\end{proof}

\begin{propn}
  The restriction of the functor $\CatI^{(\blank)}$ to $\MonPr$
  factors through $\PresI$.
\end{propn}
\begin{proof}
  This follows from Lemma~\ref{lem:CatIVstrmoncolim} and
  Corollary~\ref{cor:CatIVpres}.
\end{proof}

\begin{propn}\label{propn:BoxPrComplete}
  Suppose $\mathcal{V}$ and $\mathcal{W}$ are
  monoidal \icats{} and let $\mathcal{A}$ be a complete
  $\mathcal{V}$-\icat{} and $\mathcal{B}$ a complete
  $\mathcal{W}$-\icat{}. Then $\mathcal{A} \boxtimes \mathcal{B}$ is a
  complete $\mathcal{V} \times \mathcal{W}$-\icat{}.
\end{propn}

This follows from the following observation:
\begin{lemma}\label{lem:BoxPrIota}
  Suppose $\mathcal{V}$ and $\mathcal{W}$ are
  monoidal \icats{} and let $\mathcal{A}$ be a
  $\mathcal{V}$-\icat{} and $\mathcal{B}$ a
  $\mathcal{W}$-\icat{}. Then $\iota_{\bullet}(\mathcal{A} \boxtimes
  \mathcal{B})$ is naturally equivalent to $\iota_{\bullet}\mathcal{A} \times
  \iota_{\bullet} \mathcal{B}$, and $\iota(\mathcal{A} \boxtimes
  \mathcal{B})$ is naturally equivalent to $\iota\mathcal{A} \times
  \iota\mathcal{B}$,
\end{lemma}
\begin{proof}
    The ``external product'' $\boxtimes$ is clearly the Cartesian
  product in the \icat{} $\AlgCat$, and so it is easy to see that
  for any $\mathcal{V} \times \mathcal{W}$-\icat{}
  $\mathcal{C}$ we have \[\Map(\mathcal{C}, \mathcal{A} \boxtimes
  \mathcal{B}) \simeq \Map(\pi_{1,*}\mathcal{C}, \mathcal{A}) \times
  \Map(\pi_{2,*}\mathcal{C}, \mathcal{B}),\] where $\pi_{1}$ and
  $\pi_{2}$ denote the projections from $\mathcal{V} \times
  \mathcal{W}$ to $\mathcal{V}$ and $\mathcal{W}$, respectively.
  Moreover, $\pi_{i,*}E^{S} \simeq E^{S}$ for all $S$ (since
  $\pi_{i}$ obviously preserves the unit of the monoidal structure). Thus
  \[ \iota_{\bullet}(\mathcal{A} \boxtimes \mathcal{B}) \simeq
  \iota_{\bullet}\mathcal{A} \times \iota_{\bullet}\mathcal{B}.\]
  Since colimits of simplicial objects commute with products it
  follows that $\iota (\mathcal{A} \boxtimes \mathcal{B}) \simeq \iota
  \mathcal{A} \times \iota \mathcal{B}$.
\end{proof}

\begin{proof}[Proof of Proposition~\ref{propn:BoxPrComplete}]
  By Lemma~\ref{lem:BoxPrIota} we have a natural map
  \[ \iota_{0}(\mathcal{A} \boxtimes \mathcal{B}) \simeq
  \iota_{0}\mathcal{A} \times \iota_{0}\mathcal{B} \to \iota
  \mathcal{A} \times \iota \mathcal{B} \simeq \iota (\mathcal{A}
  \boxtimes \mathcal{B}).\] This is an equivalence if $\mathcal{A}$
  and $\mathcal{B}$ are complete, i.e. $\mathcal{A} \boxtimes
  \mathcal{B}$ is indeed also complete.
\end{proof}

\begin{cor}\label{cor:CatIVlaxmon}
  $\CatI^{(\blank)}$ is a lax monoidal functor with respect to the
  Cartesian product of monoidal \icats{}.
\end{cor}
\begin{proof}
  By Proposition~\ref{propn:BoxPrComplete} the complete enriched \icats{} are
  closed under the exterior product in $\AlgCat$, and so the
  definition of the lax monoidal structure on the functor
  $\AlgCat(\blank)$ implies that the restriction to $\CatI^{(\blank)}$
  is also lax monoidal.
\end{proof}

\begin{cor}\label{cor:CatIVOMon}
  Let $\mathcal{O}$ be a symmetric \iopd{}, and suppose
  $\mathcal{V}$ is an $\mathcal{O} \otimes
  \mathbb{E}_{1}$-monoidal \icat{}. Then $\CatIV$ is an
  $\mathcal{O}$-monoidal \icat{}. In particular, if
  $\mathcal{V}$ is an $\mathbb{E}_{n}$-monoidal \icat{} then
  $\CatIV$ is $\mathbb{E}_{n-1}$-monoidal, and if
  $\mathcal{V}$ is symmetric monoidal then so is $\CatIV$.
\end{cor}
\begin{proof}
  This follows by the same proof as that of
  Corollaries~\ref{cor:AlgCatOMon} and \ref{cor:AlgCatEnMon}.
\end{proof}

\begin{remark}\label{rmk:enrkcat}
  If $\mathcal{V}$ is an $\mathbb{E}_{n}$-monoidal \icat{},
  we can therefore iterate the enrichment functor $k$ times for $k
  \leq n$ to obtain \icats{} $\Cat_{(\infty,k)}^{\mathcal{V}}$ of
  \defterm{$(\infty,k)$-categories enriched in $\mathcal{V}$}.
\end{remark}

\begin{propn}\label{propn:FFESMonLoc}
  Suppose $\mathcal{V}$ is an $\mathbb{E}_{2}$-monoidal \icat{}. Then
  the localization $L \colon \AlgCatV \to \CatIV$ is monoidal.
\end{propn}
\begin{proof}
  We must show that if $f \colon \mathcal{C} \to \mathcal{C}'$ and $g
  \colon \mathcal{D} \to \mathcal{D}'$ are fully faithful and
  essentially surjective functors in $\AlgCatV$, then their tensor
  product $f \otimes g \colon \mathcal{C} \otimes \mathcal{D} \to
  \mathcal{C}' \otimes \mathcal{D}'$  is also fully faithful and
  essentially surjective. By definition, the tensor product
  $\mathcal{C} \otimes \mathcal{C}'$ is given
  by $\mu_{*}(\mathcal{C} \boxtimes \mathcal{C}')$, where $\mu$ is the
  tensor product functor $\mathcal{V} \times \mathcal{V} \to
  \mathcal{V}$, which is monoidal since $\mathcal{V}$ is
  $\mathbb{E}_{2}$-monoidal. 

  By Lemma~\ref{lem:LaxMonPresFFES} it therefore suffices to check
  that the external product $f \boxtimes g$ is fully faithful and
  essentially surjective in $\AlgCat(\mathcal{V} \times
  \mathcal{V})$. It is obvious that $f \boxtimes g$ is fully faithful,
  and it is essentially surjective since $\iota(f \boxtimes g)$ is
  naturally equivalent to $\iota f \times \iota g$ by
  Lemma~\ref{lem:BoxPrIota}.
\end{proof}

Combining this with Proposition~\ref{propn:moncomploc}, we get:
\begin{cor}\label{cor:FFESLocMon}
  Suppose $\mathcal{V}$ is an $\mathbb{E}_{2}$-monoidal \icat{}. Then
  the localization $L \colon \AlgCatV \to \CatIV$ is a monoidal
  functor.
\end{cor}

\begin{propn}
  When restricted to $\MonPr$, the functor $\CatI^{(\blank)}$ is lax
  monoidal with respect to the tensor product of presentable \icats{}.
\end{propn}
\begin{proof}
  This follows from Corollary~\ref{cor:AlgCatLaxMonPr}, since by
  Proposition~\ref{propn:BoxPrComplete} the complete enriched \icats{} are
  closed under the exterior product.
\end{proof}

\begin{propn}\label{propn:AdjCatIV}
  Suppose $\mathcal{V}$ and $\mathcal{W}$ are
  presentably monoidal \icats{}
  and $F \colon \mathcal{V}^{\otimes} \to
  \mathcal{W}^{\otimes}$ is a monoidal functor such that the
  underlying functor $f \colon \mathcal{V} \to \mathcal{W}$ preserves
  colimits. Let $g \colon \mathcal{W} \to \mathcal{V}$ be a right
  adjoint of $f$, and let $G \colon \mathcal{W}^{\otimes} \to
  \mathcal{V}^{\otimes}$ be the lax monoidal structure on $g$ given by
  Proposition~\ref{propn:rightadjlaxmon}. Then:
  \begin{enumerate}[(i)]
  \item The functor $G_{*} \colon \AlgCat(\mathcal{W}) \to
    \AlgCat(\mathcal{V})$ preserves complete objects.
  \item The functors
    \[ F_{*} : \CatIV \rightleftarrows \CatI^{\mathcal{W}} : G_{*} \]
    are adjoint.
  \end{enumerate}
\end{propn}
\begin{proof}
  Since $F$ is monoidal and $f$ preserves colimits, it is clear
  that for any $\mathcal{S}$-\icat{} $\mathcal{C}$ we have
  $F_{*}(I_{\mathcal{V}} \otimes \mathcal{C}) \simeq I_{\mathcal{W}}
  \otimes \mathcal{C}$. Hence for any $\mathcal{W}$-\icat{}
  $\mathcal{D}$ we have natural equivalences
  \[ \Map_{\AlgCat(\mathcal{W})}(E^{n}, \mathcal{D}) \simeq \Map_{\AlgCat(\mathcal{W})}(F_{*}E^{n}, \mathcal{D})
  \simeq \Map_{\AlgCat(\mathcal{V})}(E^{n}, G_{*}\mathcal{D}),\]
  and so in particular $\iota G_{*}\mathcal{D} \simeq \iota
  \mathcal{D}$ and $G_{*}\mathcal{D}$ is complete if $\mathcal{D}$
  is. This proves (i).

  To prove (ii), observe that using Lemma~\ref{lem:StrMonAdjAlgCat} we
  have natural equivalences
  \[ \Map_{\CatI^{\mathcal{W}}}(L_{\mathcal{W}}F_{*}\mathcal{C},
  \mathcal{D}) \simeq
  \Map_{\AlgCat(\mathcal{W})}(F_{*}\mathcal{C}, \mathcal{D})
  \simeq \Map_{\AlgCatV}(\mathcal{C}, G_{*}\mathcal{D}) \simeq
  \Map_{\CatIV}(\mathcal{C}, L_{\mathcal{V}}G_{*}\mathcal{D}).\qedhere\]
\end{proof}

\begin{propn}\label{propn:MonLocCatIV}
  Let $\mathcal{V}$ be a presentably monoidal \icat{} and
  suppose $L \colon \mathcal{V} \to \mathcal{W}$ is a monoidal
  accessible localization with fully faithful right adjoint $i \colon
  \mathcal{W} \hookrightarrow \mathcal{V}$. Let 
  $i^{\otimes} \colon \mathcal{W}^{\otimes} \hookrightarrow
  \mathcal{V}^{\otimes}$ and $L^{\otimes} \colon
  \mathcal{V}^{\otimes} \to \mathcal{W}^{\otimes}$ be as in
  Proposition~\ref{propn:moncomploc}. Suppose $L$ exhibits
  $\mathcal{W}$ as the localization of $\mathcal{V}$ with respect to a
  set of morphisms $S$. Then the resulting adjunction
  \[ L^{\otimes}_{*} : \CatI^{\mathcal{W}} \rightleftarrows
  \CatIV : i^{\otimes}_{*} \] exhibits $\CatI^{\mathcal{W}}$ as the
  localization of $\CatIV$ with respect to
  $\Sigma(S)$. Moreover, if $\mathcal{V}$ is at least
  $\mathbb{E}_{2}$-monoidal then this localization is again monoidal.
\end{propn}
\begin{proof}
  The adjunction exists by combining Lemma~\ref{lem:monlocadj} and
  Proposition~\ref{propn:AdjCatIV}. The functor $i^{\otimes}_{*}$ is
  fully faithful since the functor on categorical algebras induced by
  $i^{\otimes}$ is fully faithful by
  Proposition~\ref{propn:monlocalgcat} and preserves complete objects
  by Proposition~\ref{propn:AdjCatIV}(i). Thus this adjunction is
  a localization. The remaining statements follow by the same
  argument as in the proof of Proposition~\ref{propn:monlocalgcat}.
\end{proof}

\section{Some Applications}\label{sec:appl}
In this section we describe some simple applications of our machinery:
In \S\ref{subsec:nkcat} we use iterated enrichment to define \icats{}
of $n$-groupoids and $(n,k)$-categories for all $n$ and $0 \leq k \leq
n$ and prove the ``homotopy hypothesis'' in this context. Then in
\S\ref{subsec:ncat} we show that enriching in a monoidal
$(n,1)$-category gives an $(n+1,1)$-category, and use this to prove
the Baez-Dolan stabilization hypothesis for $k$-tuply monoidal
$n$-categories, and finally in \S \ref{subsec:EnAlg} we prove that for
any monoidal \icat{} $\mathcal{V}$ there is a fully faithful embedding
of associative algebras in $\mathcal{V}$ into pointed
$\mathcal{V}$-\icats{}.

\subsection{$(n,k)$-Categories as Enriched $\infty$-Categories}\label{subsec:nkcat}
In this subsection we explain how to define $(n,k)$-categories in the
context of enriched \icats{}, and deduce some simple results that
describe the resulting homotopy theories as localizations, including a
version of the ``homotopy hypothesis''.

We begin by inductively defining $n$-groupoids and $(n,k)$-categories:
\begin{defn}
  Assuming we have already defined $\Cat_{(n,1)}$, let $\Gpd_{n} \hookrightarrow \Cat_{(n,1)}$ be the full subcategory
  of objects local with respect to the obvious map $[1] \to E^{0}$; we refer
  to the objects of $\Gpd_{n}$ as \emph{$n$-groupoids}. Then we define
  the \icat{} $\Cat_{(n+k,k)}$ of \emph{$(n+k,k)$-categories} to be
  the \icat{} $\Cat_{(\infty,k)}^{\Gpd_{n}}$ of
  $(\infty,k)$-categories enriched in $\Gpd_{n}$. To start off the
  induction we define $0$-groupoids to be sets, i.e. we define
  $\Gpd_{0} := \Set$. We also extend the notation by setting
  $\Cat_{(n,0)} := \Gpd_{n}$.
\end{defn}

\begin{remark}
  Since the objects of $\Cat_{(n,1)}$ are already local with respect
  to $E^{1} \to E^{0}$ we can equivalently define $\Gpd_{n}$ as the
  full subcategory of objects local with respect to either of the
  inclusions $[1] \to E^{1}$. Thus an $(n,1)$-category is an
  $n$-groupoid precisely if all of its $1$-morphisms are equivalences.
\end{remark}

\begin{remark}
  Observe that the \icat{} $\Cat_{(n,n)}$ is defined by iterated
  enrichment starting with sets: $\Cat_{(n,n)} :=
  \Cat_{(\infty,n)}^{\Set}$. For $n < \infty$ we will refer to
  $(n,n)$-categories as \emph{$n$-categories} and write $\Cat_{n} :=
  \Cat_{(n,n)}$. The comparison results in \cite{enrcomp} imply that
  this \icat{} of $n$-categories is equivalent to Tamsamani's homotopy
  theory of $n$-categories \cite{Tamsamani}.
\end{remark}

\begin{remark}
  As observed by Bartels and Dolan (cf. \cite{BaezShulman}), the
  definition can also be extended to allow $n = -2$ and $n = -1$: We
  can take $\Cat_{(-2,1)} = \Gpd_{-2} := *$; then $\Cat_{(-1,1)}$
  consists of the empty category and $E^{0}$. These are both
  $-1$-groupoids, so $\Gpd_{-1} \simeq \Cat_{(-1,1)}$. Next it is easy
  to identify $\Cat_{(0,1)}$ with partially ordered sets, so
  $\Gpd_{0}$ consists of partially ordered sets where all morphisms
  are isomorphisms. These are equivalent to partially ordered sets
  with only identity morphisms, i.e. just sets, so $\Gpd_{0} \simeq
  \Set$ as before.
\end{remark}

For $n = \infty$ we define $(\infty,k)$-categories by starting with
spaces instead:
\begin{defn}
  Let $\Cat_{(\infty,0)} := \mathcal{S}$, and $\Cat_{(\infty,k)} :=
  \Cat_{(\infty,k)}^{\mathcal{S}}$.
\end{defn}

We now wish to identify $\Cat_{(n,k)}$ as a localization of
$\Cat_{(\infty,k)}$, starting from the following trivial observation:
\begin{lemma}\ 
  \begin{enumerate}[(i)]
  \item $\Set \hookrightarrow \mathcal{S}$ is the full subcategory of
    objects local with respect to the maps $S^{n} \to *$ for $n > 0$.
  \item $\mathcal{S} \hookrightarrow \CatI$ is the full subcategory of
    objects local with respect to the map $[1] \to E^{0}$.
  \end{enumerate}
\end{lemma}
\begin{proof}
  (i) is obvious, and (ii) is easy to prove if we take complete Segal
  spaces as our model for \icats{} --- a Segal space is local with
  respect to $[1] \to E^{0}$ \IFF{} it is constant.
\end{proof}

Combining this with Proposition~\ref{propn:MonLocCatIV} we immediately
get the following:
\begin{propn}\label{propn:ncatlocs}\ 
  \begin{enumerate}[(i)]
  \item The inclusion $\Cat_{(n,k)} \hookrightarrow \Cat_{(n,k+1)}$  induced
    by the inclusion $\Gpd_{n-k} \hookrightarrow \Cat_{(n-k,1)}$
    exhibits $\Cat_{(n,k)}$ as the localization with respect to
    $\Sigma^{k}[1] \to \Sigma^{k}E^{0}$.
  \item The inclusion $\Cat_{(n,k)} \hookrightarrow \Cat_{(n,l)}$, $k
    < l \leq n$, exhibits $\Cat_{(n,k)}$ as the localization with
    respect to $\Sigma^{i}[1] \to \Sigma^{i}E^{0}$, $i =
    k,k+1,\ldots,l-1$.
  \item The inclusion $\Cat_{(\infty,k)} \hookrightarrow
    \Cat_{(\infty,k+1)}$ induced by the inclusion $\mathcal{S}
    \hookrightarrow \CatI$ exhibits $\Cat_{(\infty,n)}$ as the
    localization with respect to $\Sigma^{k}[1] \to \Sigma^{k}E^{0}$.
  \item The inclusion $\Cat_{(\infty,k)} \hookrightarrow
    \Cat_{(\infty,n)}$ for $k < n$ exhibits $\Cat_{(\infty,k)}$ as the
    localization with respect to $\Sigma^{i}[1] \to \Sigma^{i}E^{0}$,
    $i = k,k+1,\ldots,n-1$.
  \item The inclusion $\Cat_{n} \hookrightarrow \Cat_{(\infty,n)}$
    induced by the inclusion $\Set \hookrightarrow \mathcal{S}$
    exhibits $\Cat_{n}$ as the localization of $\Cat_{(\infty,n)}$
    with respect to to $\Sigma^{n}S^{k} \to \Sigma^{n}*$
    for $k > 0$.
  \end{enumerate}
\end{propn}

\begin{thm}\label{thm:Catnkloc}
  The composite functor $\Cat_{(n,k)} \hookrightarrow \Cat_{n}
  \hookrightarrow \Cat_{(\infty,n)}$ factors through
  $\Cat_{(\infty,k)}$, and the resulting inclusion $\Cat_{(n,k)}
  \hookrightarrow \Cat_{(\infty,k)}$ exhibits $\Cat_{(n,k)}$ as the
  localization with respect to $\Sigma^{k}S^{j} \to \Sigma^{k}*$ for
  $j > n-k$.
\end{thm}

For the proof we need the following observation:
\begin{lemma}\label{lem:suspensionissuspension}
  Let $\kappa \colon \CatI \to \mathcal{S}$ denote the left adjoint to
  the inclusion $\mathcal{S} \hookrightarrow \CatI$. Then if $X$ is a
  space, the space $\kappa \Sigma X$ is the (unreduced) suspension of $X$.
\end{lemma}
\begin{proof}
  We take complete Segal spaces as our model for \icats{}; then the
  inclusion of $\mathcal{S}$ corresponds to the inclusion of
  \emph{constant} simplicial spaces and $\kappa$ corresponds to
  geometric realization. Let $\simp_{s}$ denote the subcategory of
  $\simp$ where the morphisms are the \emph{surjective} morphisms of
  simplicial sets. Let $S(X) \colon \simp_{s}^{\op} \to \mathcal{S}$
  be the semisimplicial space with $S(X)_{0} = \{0,1\}$, $S(X)_{1} =
  X$ with $d_{1}(X) = 0$ and $d_{0}(X) = 1$, and $S(X)_{n} =
  \emptyset$ for $n > 1$. If $j$ denotes the inclusion
  $\simp^{\op}_{s} \to \simp^{\op}$ then it is easy to see that the
  left Kan extension $j_{!}S(X)$ is a (complete) Segal
  space. Moreover, using the adjunction $j_{!} \dashv j^{*}$ it is
  clear that $j_{!}S(X)$ satisfies the universal property of $\Sigma
  X$. Thus $\kappa \Sigma X$ is the colimit of the functor
  $j_{!}S(X)$, i.e. the left Kan extension $q_{!}j_{!}S(X)$ along $q
  \colon \simp^{\op} \to *$. But this is equivalent to $(qj)_{!}S(X)$,
  which is the colimit of the semisimplicial space $S(X)$. Using the
  standard model-categorical approach to homotopy colimits we can
  describe this as the quotient of $\Delta^{1} \times X$ where we
  identify $\{0\} \times X$ and $\{1\} \times X$ with points, which is
  precisely the unreduced suspension of the space $X$.
\end{proof}

\begin{proof}[Proof of Theorem~\ref{thm:Catnkloc}]
  From Proposition~\ref{propn:ncatlocs} we see that $\Cat_{(n,k)}$ is
  the localization of $\Cat_{(\infty,n)}$ with respect to
  $\Sigma^{i}[1] \to \Sigma^{i}E^{0}$, $i = k,k+1,\ldots,n-1$ and
  $\Sigma^{n}S^{j} \to \Sigma^{n}*$ for $j > 0$. On the other hand,
  $\Cat_{(\infty,k)}$ is the localization of $\Cat_{(\infty,n)}$ with
  respect to just the first class of maps, so the inclusion
  $\Cat_{(n,k)} \hookrightarrow \Cat_{(\infty,n)}$ certainly factors
  through $\Cat_{(\infty,k)}$. To prove the result it therefore
  suffices to show that the image of $\Sigma^{n}S^{j} \to \Sigma^{n}*$
  under the localization $\Cat_{(\infty,n)} \to \Cat_{(\infty,k)}$ is
  $\Sigma^{k}S^{j+n-k} \to \Sigma^{k}*$. This follows by induction
  from the case $k = 0$, which is a special case of
  Lemma~\ref{lem:suspensionissuspension}.
\end{proof}

In the case $k = 0$, this gives a version of the ``homotopy
hypothesis'' in our setting:
\begin{cor}[Homotopy Hypothesis]
  There is an inclusion $\Gpd_{n} \hookrightarrow \mathcal{S}$ that
  exhibits $\Gpd_{n}$ as the localization of $\mathcal{S}$ with
  respect to the maps $S^{j} \to *$, $j > n$. In other words, the
  \icat{} $\Gpd_{n}$ of $n$-groupoids is equivalent to the \icat{}
  $\mathcal{S}^{\leq n}$ of \emph{$n$-types}, i.e. spaces whose
  homotopy groups vanish in degrees $> n$.
\end{cor}

\subsection{Enriching in $(n,1)$-Categories and Baez-Dolan
  Stabilization}\label{subsec:ncat}
In this subsection we prove that enriching in an $(n,1)$-category $\mathcal{V}$
gives an $(n+1,1)$-category of $\mathcal{V}$-\icats{}. We begin by
recalling the appropriate definition of an $(n,1)$-category in the
context of \icats{}:

\begin{defn}
  An \icat{} $\mathcal{C}$ is an \emph{$(n,1)$-category} if the
  mapping spaces $\Map_{\mathcal{C}}(X,Y)$ are $(n-1)$-types for all
  $X,Y \in \mathcal{C}$, i.e. $\pi_{k}\Map_{\mathcal{C}}(X,Y) = 0$ for
  $k \geq n$. In other words, there are no non-trivial $k$-morphisms
  in $\mathcal{C}$ for $k > n$.
\end{defn}

\begin{remark}
  Using the equivalence $\CatI^{\mathcal{S}} \simeq \CatI$ of
  Theorem~\ref{thm:CatISisCatI} and the case $k = 1$ of
  Theorem~\ref{thm:Catnkloc} we can identify $(n,1)$-categories in
  this sense with those defined in the previous subsection.
\end{remark}

\begin{remark}
  Suppose $\mathcal{V}$ is a monoidal \icat{} such that
  $\mathcal{V}$ is an $(n,1)$-category. Then clearly
  $\mathcal{V}^{\otimes}$ is also an $(n,1)$-category. The phrase
  \emph{monoidal $(n,1)$-category} is thus unambiguous.
\end{remark}

\begin{propn}\label{propn:iotantype}
  Suppose $\mathcal{V}$ is a monoidal $(n,1)$-category and
  $\mathcal{C}$ is a $\mathcal{V}$-\icat{}. Then the space $\iota \mathcal{C}$
  is an $n$-type.
\end{propn}
\begin{proof}
  Let $s \colon \pi_{0}(\iota_{0}\mathcal{C}) \to
  \iota_{0}\mathcal{C}$ be a section of the projection
  $\iota_{0}\mathcal{C} \to \pi_{0}\iota_{0}\mathcal{C}$. Then the
  Cartesian morphism $s^{*}\mathcal{C} \to \mathcal{C}$ is fully
  faithful and essentially surjective, and so induces an equivalence
  $\iota(s^{*}\mathcal{C}) \to \iota \mathcal{C}$ by
  Proposition~\ref{propn:FFESiotaEq}. Without loss of generality we
  may therefore assume that the space $\iota_{0}\mathcal{C}$ is
  discrete.

  The simplicial space $\iota_{\bullet}\mathcal{C}$ is a groupoid
  object by Corollary~\ref{cor:iotagpd}. By \cite[Corollary
  6.1.3.20]{HTT} this groupoid object is effective, and so we have a
  pullback diagram
  \nolabelcsquare{\iota_{1}\mathcal{C}}{\iota_{0}\mathcal{C}}{\iota_{0}\mathcal{C}}{\iota
    \mathcal{C}.}  If $X$ is a point of $\iota_{0}\mathcal{C}$, we get
  a pullback diagram
  \nolabelcsquare{\iota_{1}\mathcal{C}_{\{X\}}}{\iota_{0}\mathcal{C}}{\{X\}}{\iota
    \mathcal{C},} where $\iota_{1}\mathcal{C}_{\{X\}}$ is the fibre of
  $\iota_{1}\mathcal{C} \to \iota_{0}\mathcal{C}$ at $X$. Since the
  map $\iota_{0}\mathcal{C} \to \iota \mathcal{C}$ is surjective on
  components, by considering the long exact sequence of homotopy
  groups associated to this fibre sequence we see that $\iota
  \mathcal{C}$ is an $n$-type provided the spaces
  $\iota_{1}\mathcal{C}_{\{X\}}$ are $(n-1)$-types for all $X \in
  \iota_{0}\mathcal{C}$.

  The space $\iota_{1}\mathcal{C}_{\{X\}}$ is a union of components of
  $\iota_{1}\mathcal{C}$, so it suffices to show that
  $\iota_{1}\mathcal{C}$ is an $(n-1)$-type. Since
  $\iota_{0}\mathcal{C}$ is discrete, i.e. a $0$-type, by \cite[Lemma
  5.5.6.14]{HTT} this is equivalent to proving that the fibres of the map
  $\iota_{1}\mathcal{C} \to \iota_{0}\mathcal{C} \times
  \iota_{0}\mathcal{C}$ are $(n-1)$-types. But by Proposition~\ref{propn:eqeq} we can identify
  the fibre $\iota_{1}\mathcal{C}_{X,Y}$ at $(X,Y) \in
  \iota_{0}\mathcal{C}^{\times 2}$ with the space $\Map(I,
  \mathcal{C}(X,Y))_{\text{eq}}$ that is the union of the components
  of $\Map(I, \mathcal{C}(X,Y))$ corresponding to equivalences. Since
  $\mathcal{V}$ is by assumption an $n$-category, the space $\Map(I,
  \mathcal{C}(X,Y))$ is necessarily an $(n-1)$-type, hence so is $\Map(I,
  \mathcal{C}(X,Y))_{\text{eq}}$.
\end{proof}

\begin{thm}\label{thm:enrncat}
  Suppose $\mathcal{V}$ is a monoidal $(n,1)$-category. Then
  $\CatIV$ is an $(n+1,1)$-category.
\end{thm}
\begin{proof}
  We need to show that if $\mathcal{C}$ and $\mathcal{D}$ are complete
  $\mathcal{V}$-\icats{} then the space \[\Map_{\CatIV}(\mathcal{C},
  \mathcal{D}) \simeq
  \Map_{\AlgCat(\mathcal{V})}(\mathcal{C}, \mathcal{D})\] is
  an $n$-type. By Proposition~\ref{propn:iotantype}, the space
  $\iota_{0}\mathcal{D} \simeq \iota \mathcal{D}$ is an $n$-type,
  hence the space $\Map_{\mathcal{S}}(\iota_{0}\mathcal{C},
  \iota_{0}\mathcal{D})$ is as well. It follows from \cite[Lemma
  5.5.6.14]{HTT} that, in order to prove that
  $\Map_{\AlgCat(\mathcal{V})}(\mathcal{C}, \mathcal{D})$ is
  an $n$-type, it suffices to show that the fibres of the map
  \[ \Map_{\AlgCat(\mathcal{V})}(\mathcal{C}, \mathcal{D})
  \to \Map_{\mathcal{S}}(\iota_{0}\mathcal{C}, \iota_{0}\mathcal{D})\]
  induced by the projection $\AlgCat(\mathcal{V}) \to
  \mathcal{S}$ are $n$-types.

  Since the projection $\AlgCat(\mathcal{V})
  \to \mathcal{S}$ is a Cartesian fibration, by \cite[Proposition
  2.4.4.2]{HTT} we can identify the fibre of this map at
  $f \colon \iota_{0}\mathcal{C} \to \iota_{0}\mathcal{D}$ with
  \[\Map_{\Alg_{\simp^{\op}_{\iota_{0}\mathcal{C}}}(\mathcal{V})}(\mathcal{C},
  f^{*}\mathcal{D}).\] This space is the fibre of
  \[\Map_{\simp^{\op}}(\simp^{\op}_{\iota_{0}\mathcal{C}} \times
  \Delta^{1},\mathcal{V}^{\otimes}) \to
  \Map_{\simp^{\op}}(\simp^{\op}_{\iota_{0}\mathcal{C}},
  \mathcal{V}^{\otimes}) \times
  \Map_{\simp^{\op}}(\simp^{\op}_{\iota_{0}\mathcal{C}},
  \mathcal{V}^{\otimes})\] at $(\mathcal{C}, f^{*}\mathcal{D})$. Since
  $n$-types are closed under all limits by \cite[Proposition
  5.5.6.5]{HTT}, it suffices to show that the spaces
  $\Map_{\simp^{\op}}(\simp^{\op}_{\iota_{0}\mathcal{C}},
  \mathcal{V}^{\otimes})$ and
  $\Map_{\simp^{\op}}(\simp^{\op}_{\iota_{0}\mathcal{C}} \times
  \Delta^{1}, \mathcal{V}^{\otimes})$ are $n$-types. Now these spaces are
  fibres of $\Map(\simp^{\op}_{\iota_{0}\mathcal{C}},
  \mathcal{V}^{\otimes}) \to \Map(\simp^{\op}, \mathcal{V}^{\otimes})$
  and $\Map(\simp^{\op}_{\iota_{0}\mathcal{C}} \times \Delta^{1},
  \mathcal{V}^{\otimes}) \to \Map(\simp^{\op},
  \mathcal{V}^{\otimes})$, so by the same argument it's enough to show
  that these mapping spaces are $n$-types. But $\mathcal{V}^{\otimes}$
  is by assumption an $(n,1)$-category, so this holds by
  \cite[Proposition 2.3.4.18]{HTT}.
\end{proof}

\begin{cor}
  The \icat{} $\Cat_{n}$ of $n$-categories is an
  $(n+1,1)$-category. More generally, the \icat{} $\Cat_{(n,k)}$ of
  $(n,k)$-categories is an $(n+1,1)$-category.
\end{cor}
\begin{proof}
  Since $\Set$ is obviously a monoidal $(1,1)$-category, applying
  Theorem~\ref{thm:enrncat} inductively we see that $\Cat_{n}$ is an
  $(n+1,1)$-category. Similarly, if we know that $\Gpd_{n}$ is an
  $(n+1,1)$-category it follows by induction that $\Cat_{(n+k,k)} $ is
  an $(n+k+1,1)$-category. In particular $\Cat_{(n+1,1)}$ is an
  $(n+2,1)$-category, and so its full subcategory $\Gpd_{n+1}$ of
  $(n+1)$-groupoids is also an $(n+2,1)$-category. Since $\Gpd_{0} =
  \Set$ is a $(1,1)$-category we see by induction that $\Cat_{(n,k)}$ is an
  $(n+1,1)$-category for all $(n,k)$.
\end{proof}

It follows that if $\mathcal{V}$ is a symmetric monoidal
$(n,1)$-category, then $\mathbb{E}_{k}$-algebras in $\CatIV$ are
equivalent to $\mathbb{E}_{\infty}$-algebras for $k$ sufficiently large:
\begin{cor}\label{cor:Enmonenrinncat}
  Let $\mathcal{V}$ be a symmetric monoidal
  $(n,1)$-category. Then 
  \begin{enumerate}[(i)]
  \item the map $\mathbb{E}_{k} \to \bbGamma^{\op}$ induces
    an equivalence
    \[\Alg^{\Sigma}_{\mathbb{E}_{k}}(\CatIV) \isoto
    \Alg^{\Sigma}_{\bbGamma^{\op}}(\CatIV)\] for $k \geq n+1$,
  \item the stabilization map $i \colon \mathbb{E}_{k} \to
    \mathbb{E}_{k+1}$ (defined in \cite[\S 5.1.1]{HA})
    induces an equivalence
    \[ i^{*} \colon
    \Alg^{\Sigma}_{\mathbb{E}_{k+1}}(\CatIV) \to
    \Alg^{\Sigma}_{\mathbb{E}_{k}}(\CatIV)\] for $k \geq
    n+1$.
  \end{enumerate}
\end{cor}
\begin{proof}
  (i) is immediate from \cite[Corollary 5.1.1.7]{HA}, and (ii) follows
  by the 2-out-of-3 property.
\end{proof}

We end this subsection by observing that when $\mathcal{V}$
is the monoidal \icat{} of $n$-categories, this yields the Baez-Dolan
stabilization hypothesis, by the same proof as Lurie's version for
$(n,1)$-categories \cite[Example 5.1.2.3]{HTT}:
\begin{defn}
  A \emph{$k$-tuply monoidal $n$-category} is an
  $\mathbb{E}_{k}$-algebra in $\Cat_{n}$, i.e. an
  $\mathbb{E}_{k}$-monoidal $n$-category.
\end{defn}

\begin{cor}[Baez-Dolan Stabilization Hypothesis]
  The stabilization map $i \colon \mathbb{E}_{k} \to
  \mathbb{E}_{k+1}$ induces an equivalence
  \[ i^{*} \colon \Alg^{\Sigma}_{\mathbb{E}_{k+1}}(\Cat_{n})
  \to \Alg^{\Sigma}_{\mathbb{E}_{k}}(\Cat_{n})\] for $k \geq
  n+2$, i.e. $k$-tuply monoidal $n$-categories stabilize at $k = n+2$.
\end{cor}
\begin{proof}
  Apply Corollary~\ref{cor:Enmonenrinncat} to $\Cat_{n}$.
\end{proof}

\begin{remark}
  The Baez-Dolan stabilization hypothesis was originally stated by
  Baez and Dolan in \cite{BaezDolanTQFT}. A version of it was proved by
  Simpson~\cite{SimpsonBDStab}, who showed that for $k \geq n+2$ a
  $k$-tuply monoidal $n$-category can be ``delooped'' to a
  $(k+1)$-tuply monoidal $n$-category; the \icatl{} version above extends
  this by showing that this construction gives an equivalence of
  \icats{}.
\end{remark}

\subsection{$\mathbb{E}_{n}$-Algebras as Enriched
  $(\infty,n)$-Categories}\label{subsec:EnAlg}
In ordinary enriched category theory, it is obvious that assocative
algebra objects in a monoidal category $\mathbf{V}$ are equivalent to
$\mathbf{V}$-categories with a single object. Similarly, if
$\mathcal{V}$ is a monoidal \icat{}, we can identify the
\icat{} $\Alg_{\simp^{\op}}(\mathcal{V})$ of associative
algebra objects with the full subcategory of $\AlgCatV$ spanned by
$\mathcal{V}$-\icats{} whose space of objects is a point. In this
subsection we will prove that after localizing with respect to the
fully faithful and essentially surjective $\mathcal{V}$-functors we
still get a fully faithful functor from
$\Alg_{\simp^{\op}}(\mathcal{V})$ provided we consider
\emph{pointed} $\mathcal{V}$-\icats{}. It then follows by induction
that, if $\mathcal{V}$ is at least $\mathbb{E}_{n}$-monoidal, the same
is true for the natural map from $\mathbb{E}_{n}$-algebras to pointed
enriched $(\infty,n)$-categories.

\begin{defn}
  Let $\mathcal{V}$ be a monoidal \icat{}. By
  Proposition~\ref{propn:UnitAlg}, the unit object of $\mathcal{V}$ is
  the initial object in the \icat{}
  $\Alg_{\simp^{\op}}(\mathcal{V})$ of associative algebra
  objects. The inclusion $j \colon \Alg_{\simp^{\op}}(\mathcal{V})
  \hookrightarrow \AlgCatV$ therefore factors through
  $\AlgCatV_{E^{0}/}$. Composing this with the localization functor we
  get a functor $B \colon
  \Alg_{\simp^{\op}}(\mathcal{V}) \to
  (\CatIV)_{E^{0}/}$.
\end{defn}

\begin{thm}\label{thm:E1algff}
  Let $\mathcal{V}$ be a monoidal \icat{}. Then:
  \begin{enumerate}[(i)]
  \item The functor $B \colon \Alg_{\simp^{\op}}(\mathcal{V}) \to
    (\CatIV)_{E^{0}/}$ is fully faithful.
  \item If $\mathcal{V}$ is $\mathbb{E}_{2}$-monoidal, then
    $B$ is a monoidal functor.
  \item  $B$ admits a right adjoint $\Omega \colon (\CatIV)_{E^{0}/}
    \to \Alg_{\simp^{\op}}(\mathcal{V})$.
  \item If $\mathcal{V}$ is presentably $\mathbb{E}_{2}$-monoidal,
    then $\Omega$ is a lax monoidal functor.
  \end{enumerate}
\end{thm}

For the proof of (iii) we first make some simple observations:
\begin{lemma}\label{lem:CartFibInclRAdj}
  Let $\pi \colon \mathcal{E} \to \mathcal{B}$ be a Cartesian
  fibration. For any $B \in \mathcal{B}$, the functor $\mathcal{E}_{B}
  \to \mathcal{E}_{B/} := \mathcal{E} \times_{\mathcal{B}}
  \mathcal{B}_{B/}$ admits a right adjoint.
\end{lemma}
\begin{proof}
  First suppose $B$ is an initial object of $\mathcal{B}$. Then there
  is an obvious map $\mathcal{B}^{\triangleleft} \to \mathcal{B}$ that
  sends $-\infty$ to $B$ and is the identity when restricted to
  $\mathcal{B}$. Let $\pi' \colon \mathcal{E}' \to \mathcal{B}^{\triangleleft}$ be
  the pullback of $\pi$ along this map; then $\pi'$ is a Cartesian
  fibration. Since the obvious projection $\mathcal{B}^{\triangleleft} \to
  (\Delta^{0})^{\triangleleft} = \Delta^{1}$ is clearly a
  Cartesian fibration, the composite functor $\mathcal{E}' \to
  \Delta^{1}$ is also a Cartesian fibration. But this is clearly also
  the coCartesian fibration associated to the inclusion
  $\mathcal{E}_{B} \hookrightarrow \mathcal{E}$, hence this functor
  does indeed have a right adjoint.

  For the general case we reduce to the case already proved by pulling
  back along the forgetful functor $\mathcal{B}_{B/} \to \mathcal{B}$.
\end{proof}

\begin{lemma}\label{lem:UndercatAdj}
  Suppose given an adjunction 
  \[ F : \mathcal{C} \rightleftarrows \mathcal{D} : G, \]
  and suppose $D \in \mathcal{D}$ is an object such that the counit map
  $FGD \to D$ is an equivalence. Then the induced functor
  $\mathcal{D}_{D/} \to \mathcal{C}_{GD/}$ admits a left adjoint,
  given by $\mathcal{C}_{GD/} \to \mathcal{D}_{FGD/} \simeq \mathcal{D}_{D/}$.
\end{lemma}
\begin{proof}
  The (dual of) the argument in the proof of \cite[Lemma 5.2.5.2]{HTT}
  applies under our assumptions without assuming any colimits exist in
  $\mathcal{D}$.
\end{proof}

\begin{proof}[Proof of Theorem~\ref{thm:E1algff}]
  To prove (i), let $R$ and $S$ be two $\simp^{\op}$-algebras in $\mathcal{V}$. We
  have a fibre sequence
  \[ \Map_{(\CatIV)_{E^{0}/}}(BR,BS) \to \Map_{ \CatIV}(BR, BS) \to
  \Map_{\CatIV}(E^{0}, BS). \] Since $BS$ is the completion
  $L_{\mathcal{V}}j(S)$ of $S$ regarded as a $\mathcal{V}$-\icat{}, we
  have equivalences \[\Map_{\CatIV}(BR, BS) \simeq
  \Map_{\AlgCat(\mathcal{V})}(j(R), BS)\] and
  \[\Map_{\CatIV}(E^{0}, BS) \simeq
  \Map_{\AlgCat(\mathcal{V})}(E^{0}, BS).\] The
  projection $\iota_{0} \colon \AlgCat(\mathcal{V}) \to
  \mathcal{S}$ gives a commutative diagram
  \nolabelcsquare{\Map_{\AlgCat(\mathcal{V})}(j(R),
    BS)}{\Map_{\AlgCat(\mathcal{V})}(E^{0},
    BS)}{\Map_{\mathcal{S}}(*,
    \iota_{0}BS)}{\Map_{\mathcal{S}}(*,
    \iota_{0}BS)} where the right vertical map is an
  equivalence by Lemma~\ref{lem:iota0} and the bottom horizontal map
  is the identity, since $E^{0} \to j(R)$ is the identity on
  $\iota_{0}$. Thus we can identify the fibre of the top horizontal
  map at the functor $E^{0} \to BS$ corresponding to a point
  $p \colon * \to \iota_{0}BS$ with the corresponding fibre
  of the left vertical map, which is
  $\Map_{\Alg_{\simp^{\op}}(\mathcal{V})}(R,
  p^{*}BS)$ by \cite[Proposition 2.4.4.2]{HTT}.

  Take $p$ to be the underlying map of spaces of the completion functor
  $j(S) \to BS$; since this is fully faithful the induced map $j(S)
  \to p^{*}BS$ is an equivalence, and in particular
  \[ \Map_{\Alg_{\simp^{\op}}(\mathcal{V})}(R, S) \isoto
  \Map_{\Alg_{\simp^{\op}}(\mathcal{V})}(R, p^{*}BS). \] Thus the map
  $\Map_{\Alg_{\simp^{\op}}(\mathcal{V})}(R, S) \to
  \Map_{(\CatIV)_{E^{0}/}}(BR, BS)$ is also an equivalence, which
  completes the proof of (i).

  We now prove (ii). It is clear from the definition of
  the monoidal structures that the functor $\Alg_{\simp^{\op}}(\mathcal{V}) \to
  \AlgCat(\mathcal{V})_{E^{0}/}$ is monoidal. Since it follows
  from Corollary~\ref{cor:FFESLocMon} that the localization
  $\AlgCat(\mathcal{V})_{E^{0}/} \to (\CatIV)_{E^{0}/}$ is
  monoidal (by regarding the overcategores as \icats{} of
  $\mathbb{E}_{0}$-algebras, for example), it follows that $B$ is
  monoidal.

  To prove (iii), we first observe that the adjunction
  $\AlgCat(\mathcal{V}) \rightleftarrows \CatIV$ descends to an
  adjunction  $\AlgCat(\mathcal{V})_{E^{0}/} \rightleftarrows
  (\CatIV)_{E^{0}/}$ by Lemma~\ref{lem:UndercatAdj}. It therefore
  suffices to show that the functor $j \colon
  \Alg_{\simp^{\op}}(\mathcal{V}) \to \AlgCat(V)_{E^{0}/}$ admits a
  right adjoint. To see this we first show that the obvious functor
  $\AlgCat(V)_{E^{0}/} \to \AlgCat(V) \times_{\mathcal{S}}
  \mathcal{S}_{*}$ is an equivalence. It is clear that this functor is
  essentially surjective, so it suffices to show that for
  $\mathcal{C}, \mathcal{D}$ in $\AlgCat(V)_{E^{0}/}$ the induced map
  \[ \Map_{E^{0}/}(\mathcal{C}, \mathcal{D}) \to \Map(\mathcal{C},
  \mathcal{D}) \times_{\Map(\iota_{0}\mathcal{C},
    \iota_{0}\mathcal{D})} \Map_{*/}(\iota_{0}\mathcal{C},
  \iota_{0}\mathcal{D})\]
  is an equivalence. Consider the following commutative diagram:
  \[ %
\begin{tikzpicture} %
  \matrix (m) [matrix of math nodes,row sep=3em,column sep=2.5em,text height=1.5ex,text depth=0.25ex] %
{ 
  \Map_{E^{0}/}(\mathcal{C}, \mathcal{D}) &  \Map(\mathcal{C},
  \mathcal{D}) \\
  \Map_{*/}(\iota_{0}\mathcal{C}, \iota_{0}\mathcal{D}) &
  \Map(\iota_{0}\mathcal{C}, \iota_{0}\mathcal{D}) \\
  * & \Map(*, \iota_{0}\mathcal{D}).\\
 }; %
\path[->,font=\footnotesize] %
(m-1-1) edge (m-1-2)
(m-2-1) edge (m-2-2)
(m-3-1) edge (m-3-2)
(m-1-1) edge (m-2-1)
(m-1-2) edge (m-2-2)
(m-2-1) edge (m-3-1)
(m-2-2) edge (m-3-2)
;
\end{tikzpicture}%
\]
Here the bottom square is clearly a pullback square, and the outer
rectangle is a pullback square because of the natural equivalence
$\Map(E^{0}, \mathcal{D}) \isoto \Map(*, \iota_{0}\mathcal{D})$. Thus
the top square is also a pullback square. (iii) therefore follows by
applying Lemma~\ref{lem:CartFibInclRAdj}.

Finally, (iv) now follows from Proposition~\ref{propn:rightadjlaxmon}.
\end{proof}

\begin{remark}\label{rmk:imgofalg}
  A pointed $\mathcal{V}$-\icat{} $\mathcal{C}$ is in the essential
  image of the functor $B$ \IFF{} $\iota \mathcal{C}$ is connected,
  since then the functor $p^{*}\mathcal{C} \to \mathcal{C}$ induced by
  the chosen point $p : * \to \iota_{0}\mathcal{C}$ is fully faithful
  and essentially surjective, and $p^{*}\mathcal{C}$ is a
  $\simp^{\op}$-algebra. In other words, $\simp^{\op}$-algebras in
  $\mathcal{V}$ are equivalent to pointed
  $\mathcal{V}$-\icats{} with a single object up to homotopy.
\end{remark}

\begin{defn}
  Let $\mathcal{V}$ be an $\mathbb{E}_{2}$-monoidal
  \icat{}. A \emph{monoidal $\mathcal{V}$-\icat{}} is a
  $\simp^{\op}$-algebra in $\CatIV$.
\end{defn}

\begin{cor}
  Let $\mathcal{V}$ be an $\mathbb{E}_{2}$-monoidal
  \icat{}. Then monoidal $\mathcal{V}$-\icats{} are equivalent to
  pointed $\mathcal{V}$-$(\infty,2)$-categories with a single object
  (up to homotopy).
\end{cor}
\begin{proof}
  By definition $\mathcal{V}$-$(\infty,2)$-categories are \icats{}
  enriched in $\CatIV$, so this follows from Remark~\ref{rmk:imgofalg}.
\end{proof}

\begin{remark}
  In particular, taking $\mathcal{V}$ to be $\Gpd_{n}$, we see that
  monoidal $(n,k)$-categories are equivalent to pointed
  $(n+1,k+1)$-categories with a single object. Taking $\mathcal{V}$ to
  be $\mathcal{S}$, this remains true for $n = \infty$.
\end{remark}

\begin{defn}
  If $\mathcal{C}$ is a $\mathcal{V}$-\icat{} and $X$ is an object of
  $\mathcal{C}$, we write $\Omega_{X}\mathcal{C} \in
  \Alg_{\simp^{\op}}(\mathcal{V})$ for the value of the functor
  $\Omega$ on the corresponding map $E^{0} \to
  \mathcal{C}$. This is the \emph{endomorphism algebra} of $X$.
\end{defn}

By applying Theorem~\ref{thm:E1algff} inductively we can generalize it
to the $\mathbb{E}_{n}$-monoidal setting:
\begin{defn}
  By Proposition~\ref{propn:nssymmeq} monoidal \icats{} are equivalent
  to $\mathbb{E}_{1}$-monoidal \icats{}, and
  $\simp^{\op}$-algebras in a monoidal \icat{} are equivalent to
  $\mathbb{E}_{1}$-algebras in the associated
  $\mathbb{E}_{1}$-monoidal \icat{}. Since
  $\mathbb{E}_{n} \otimes \mathbb{E}_{m} \simeq
  \mathbb{E}_{n+m}$ for all $n,m$, by Theorem~\ref{thm:E1algff}(ii) we get maps
\[ \Alg^{\Sigma}_{\mathbb{E}_{n}}(\mathcal{V}) \simeq
\Alg^{\Sigma}_{\mathbb{E}_{n-1}}(\Alg^{\Sigma}_{\mathbb{E}_{1}}(\mathcal{V}))
\to \Alg^{\Sigma}_{\mathbb{E}_{n-1}}((\CatIV)_{E^{0}/}).\]
We can identify $(\CatIV)_{E^{0}/}$ with
$\Alg^{\Sigma}_{\mathbb{E}_{0}}(\CatIV)$, so
\[
\begin{split}
\Alg^{\Sigma}_{\mathbb{E}_{n-1}}((\CatIV)_{E^{0}/})
& \simeq
\Alg^{\Sigma}_{\mathbb{E}_{n-1}}(\Alg^{\Sigma}_{\mathbb{E}_{0}}(\CatIV)) \\
& \simeq \Alg^{\Sigma}_{\mathbb{E}_{n-1} \otimes
  \mathbb{E}_{0}}(\CatIV) \\ & \simeq
\Alg^{\Sigma}_{\mathbb{E}_{n-1}}(\CatIV).
\end{split}
\]
Thus we have maps
\[ \Alg^{\Sigma}_{\mathbb{E}_{n}}(\mathcal{V}) 
\to \Alg^{\Sigma}_{\mathbb{E}_{n-1}}(\CatIV)
\to \cdots \to
\Alg^{\Sigma}_{\mathbb{E}_{1}}(\Cat_{(\infty,n-1)}^{\mathcal{V}})
\to (\Cat_{(\infty,n)}^{\mathcal{V}})_{E^{0}/}.\]
\end{defn}

Applying Theorem~\ref{thm:E1algff} (and the symmetric counterparts of
some of the results we used in its proof) inductively, we get the
following:
\begin{cor}
  Suppose $\mathcal{V}$ is an $\mathbb{E}_{n}$-monoidal \icat{}.
  \begin{enumerate}[(i)]
  \item The functor
    \[B^{n} \colon \Alg^{\Sigma}_{\mathbb{E}_{n}}(\mathcal{V}) \to
    (\Cat^{\mathcal{V}}_{(\infty,n)})_{E^{0}/}\] is fully faithful.
  \item If $\mathcal{V}$ is $\mathbb{E}_{n+1}$-monoidal, then $B^{n}$
    is a monoidal functor.
  \item If $\mathcal{V}$ is presentably $\mathbb{E}_{n}$-monoidal,
    then $B^{n}$ admits a right adjoint
    $\Omega^{n} \colon (\Cat^{\mathcal{V}}_{(\infty,n)})_{E^{0}/} \to
    \Alg^{\Sigma}_{\mathbb{E}_{n}}(\mathcal{V})$.
  \item If $\mathcal{V}$ is presentably $\mathbb{E}_{n+1}$-monoidal,
    then $\Omega^{n}$ is a lax monoidal functor.
  \end{enumerate}
\end{cor}

\begin{defn}
  Let $\mathcal{V}$ be an $\mathbb{E}_{n+1}$-monoidal
  \icat{}; then $\CatIV$ is $\mathbb{E}_{n}$-monoidal by
  Corollary~\ref{cor:CatIVOMon}. An \emph{$\mathbb{E}_{n}$-monoidal
    $\mathcal{V}$-\icat{}} is an $\mathbb{E}_{n}$-algebra in $\CatIV$.
\end{defn}

\begin{cor}
  Let $\mathcal{V}$ be an $\mathbb{E}_{n+1}$-monoidal
  \icat{}. Then $\mathbb{E}_{n}$-monoidal $\mathcal{V}$-\icats{} are
  equivalent to pointed $\mathcal{V}$-$(\infty,n+1)$-categories with a
  single object and only identity $j$-morphisms for $j = 1$, \ldots, $n-1$.  
\end{cor}

\begin{remark}
  In particular, taking $\mathcal{V}$ to be $\Gpd_{n}$, we see that
  $\mathbb{E}_{m}$-monoidal $(n,k)$-categories are equivalent to
  pointed $(n+m,k+m)$-categories with a single object and only
  identity $j$-morphisms for $j = 1,\ldots,m-1$. Taking $\mathcal{V}$
  to be $\mathcal{S}$, this remains true for $n = \infty$.
\end{remark}

\begin{defn}
  If $\mathcal{C}$ is a $\mathcal{V}$-$(\infty,n)$-category and $X$ is
  an object of $\mathcal{C}$, we write $\Omega^{n}_{X}\mathcal{C}$ for
  $\Omega^{n}$ applied to the corresponding map $E^{0} \to
  \mathcal{C}$. This is the \emph{endomorphism
    $\mathbb{E}_{n}$-algebra} of $X$.
\end{defn}
\begin{remark}
  If $\mathcal{C}$ is a $\mathcal{V}$-$(\infty,n)$-category and $X$ is
  an object of $\mathcal{C}$, the underlying object in $\mathcal{V}$
  of the $\mathbb{E}_{n}$-algebra $\Omega^{n}_{X}\mathcal{C}$ is the
  endomorphisms of the $(n-1)$-fold identity map of the identity map
  of \ldots of the identity map of $X$.
\end{remark}

\section{Enriching $\infty$-Categories Tensored over a Monoidal
  $\infty$-Category}\label{sec:selfenr}

Suppose $\mathbf{V}$ is a monoidal category and $\mathbf{C}$ is an
ordinary category that is right-tensored over $\mathbf{V}$, i.e. there
is a functor \[(\blank) \otimes (\blank) \colon \mathbf{C} \times
\mathbf{V} \to \mathbf{C},\] compatible with the tensor product of
$\mathbf{V}$. If for every $C \in \mathbf{C}$ the functor $C \otimes
(\blank)$ has a right adjoint $F(C, \blank) \colon \mathbf{C} \to
\mathbf{V}$, then it is easy to see that we can enrich $\mathbf{C}$ in
$\mathbf{V}$, with the morphism object from $C$ to $D$ given by
$F(C,D) \in \mathbf{V}$. In particular, if the monoidal category
$\mathbf{V}$ is left-closed, then it is enriched in itself. Our goal
in this section is to prove the analogous statement in the context of
enriched \icats{}, which will allow us to construct a number of
interesting examples of these.

To prove this we will consider a variant of Lurie's definition of
enriched \icats{} from \cite[\S 4.2.1]{HA}. After introducing the
natural \gnsiopds{} that parametrize modules in \S\ref{subsec:Mod}
(and proving that the resulting \icats{} of modules are equivalent to
those of \cite{HA}), we review Lurie's definition in
\S\ref{subsec:LurEnr}. It is easy to see that an \icat{}
right-tensored over a monoidal \icat{} with adjoints as above defines
an enriched \icat{} in this sense; by applying Lurie's construction of
enriched strings from \cite[\S 4.7.2]{HA}, which we review in
\S\ref{subsec:EnrStr}, we can quite easily extract a categorical
algebra from this in \S\ref{subsec:ExtCatAlg}.

\subsection{Modules}\label{subsec:Mod}

\begin{defn}
  Write $\BM$ for the category of simplices
  $\Simp(\Delta^{1})$ of the simplicial set $\Delta^{1}$. The objects
  of $\BM$ can be described as sequences of integers
  $(i_{0},\ldots,i_{k})$ where $0 \leq i_{0} \leq \cdots \leq i_{k}
  \leq 1$, and there is a unique morphism $(i_{0},\ldots,i_{k}) \to
  (i_{\phi(0)},\ldots, i_{\phi(m)})$ for every map $\phi \colon [m]
  \to [k]$ in $\simp$. The obvious projection $\BM \to
  \simp^{\op}$ exhibits $\BM$ as a double \icat{}. If
  $\mathcal{M}$ is a generalized non-symmetric \iopd{}, a
  \emph{bimodule} in $\mathcal{M}$ is a $\BM$-algebra in
  $\mathcal{M}$. We write $\txt{Bimod}(\mathcal{M})$ for the \icat{} of
  $\Alg_{\BM}(\mathcal{M})$ of bimodules in $\mathcal{M}$.
\end{defn}

\begin{defn}
  The obvious inclusions $i,j \colon
  \simp^{\op}_{\{0\}},\simp^{\op}_{\{1\}} \to \BM$ are maps of
  generalized \iopds{}. We say a bimodule $M$ in a \gnsiopd{}
  $\mathcal{M}$ is an \emph{$i^{*}M$-$j^{*}M$-bimodule}. If $A$ and
  $B$ are associative algebras in a generalized \nsiopd{}
  $\mathcal{M}$, we write $\txt{Bimod}_{A,B}(\mathcal{M})$ for the
  fibre of the projection $(i^{*},j^{*}) \colon
  \txt{Bimod}(\mathcal{M}) \to \Alg_{\simp^{\op}}(\mathcal{M}) \times
  \Alg_{\simp^{\op}}(\mathcal{M})$ at $(A,B)$, i.e. the \icat{} of
  $A$-$B$-bimodules.
\end{defn}

\begin{defn}
  Let $\LM$ denote the full subcategory of $\BM$ spanned by the
  objects of the form $(0,\ldots,0,1)$ and $(0,\ldots, 0)$. The
  restricted projection $\LM \to \simp^{\op}$ exhibits $\LM$ as a
  double \icat{}. A \emph{left module} in a generalized non-symmetric
  \iopd{} $\mathcal{M}$ is an $\LM$-algebra in $\mathcal{M}$. We write
  $\txt{LMod}(\mathcal{M})$ for the \icat{} $\Alg_{\LM}(\mathcal{M})$
  of left modules in $\mathcal{M}$.
\end{defn}

\begin{defn}
  Let $\RM$ denote the full subcategory of $\BM$ spanned by the
  objects of the form $(0,1\ldots,1)$ and $(1,\ldots, 1)$. The
  restricted projection $\RM \to \simp^{\op}$ exhibits $\RM$ as a
  double \icat{}. A \emph{right module} in a generalized non-symmetric
  \iopd{} $\mathcal{M}$ is an $\RM$-algebra in $\mathcal{M}$. We write
  $\txt{RMod}(\mathcal{M})$ for the \icat{} $\Alg_{\RM}(\mathcal{M})$
  of right modules in $\mathcal{M}$.
\end{defn}

\begin{defn}
  The obvious inclusions $i \colon \simp^{\op}_{\{0\}} \hookrightarrow
  \LM$ and $j \colon \simp^{\op}_{\{1\}} \hookrightarrow \RM$ are maps
  of \gnsiopds{}. If $M \colon \LM \to \mathcal{M}$ is a left module
  in a \gnsiopd{} $\mathcal{M}$, we say $M$ is a \emph{left
    $i^{*}M$-module}. Similarly, if $M$ is a right module in
  $\mathcal{M}$ we say that it is a \emph{right $j^{*}M$-module}. If
  $A$ is an associative algebra in $\mathcal{M}$, we write
  $\txt{LMod}_{A}(\mathcal{M})$ and $\txt{RMod}_{A}(\mathcal{M})$ for
  the fibres of the projections $i^{*} \colon \txt{LMod}(\mathcal{M})
  \to \Alg_{\simp^{\op}}(\mathcal{M})$ and $j^{*} \colon
  \txt{RMod}(\mathcal{M}) \to \Alg_{\simp^{\op}}(\mathcal{M})$ at $A$,
  respectively.
\end{defn}

It is easy to describe the \iopd{} localizations of the \gnsiopds{}
$\BM$, $\LM$, and $\RM$ in terms of multicategories:
\begin{defn}
  Let $\mathbf{BM}$ be the multicategory with objects $\mathfrak{l}$,
  $\mathfrak{m}$ and $\mathfrak{r}$ and multimorphisms
  \[
  \mathbf{BM}(\mathfrak{l},\ldots,\mathfrak{l},\mathfrak{m},\mathfrak{r},\ldots,\mathfrak{r};
  \mathfrak{m}) = *\]
  \[ \mathbf{BM}(\mathfrak{l},\ldots,\mathfrak{l};
  \mathfrak{l}) = *\]
  \[ \mathbf{BM}(\mathfrak{r},\ldots,\mathfrak{r};
  \mathfrak{r}) = *\]
  (where there can be zero $\mathfrak{l}$'s and $\mathfrak{r}$'s in
  the lists), with all other multimorphism sets empty. We then define
  $\mathbf{LM}$ to be the full submulticategory of $\mathbf{BM}$ with
   objects $\mathfrak{l}$ and $\mathfrak{m}$ and $\mathbf{RM}$ to
  be the full submulticategory with objects $\mathfrak{r}$ and
  $\mathfrak{m}$.  
\end{defn}

\begin{propn}\label{propn:BMloceq}
  There are obvious maps $\BM \to \mathbf{BM}^{\otimes}$,  $\LM \to
  \mathbf{LM}^{\otimes}$, and  $\RM \to
  \mathbf{RM}^{\otimes}$. These induce equivalences of \nsiopds{}
  $L_{\txt{gen}}\BM \isoto \mathbf{BM}^{\otimes}$, $L_{\txt{gen}}\LM
  \isoto \mathbf{LM}^{\otimes}$ and $L_{\txt{gen}}\RM \isoto
  \mathbf{RM}^{\otimes}$.
\end{propn}
\begin{proof}
  It is easy to see that these maps satisfy the criterion of
  Corollary~\ref{cor:OpdLocCof}.
\end{proof}

\begin{cor}
  Let $\mathcal{O}$ be a \nsiopd{}. Then there are natural
  equivalences $\txt{Bimod}(\mathcal{O}) \simeq
  \Alg_{\mathbf{BM}}(\mathcal{O})$,
  $\txt{LMod}(\mathcal{O}) \simeq
  \Alg_{\mathbf{LM}}(\mathcal{O})$ and
  $\txt{RMod}(\mathcal{O}) \simeq
  \Alg_{\mathbf{RM}}(\mathcal{O})$.
\end{cor}

\begin{remark}\label{rmk:LurieMods}
  The symmetric \iopds{} used in \cite{HA} to define bimodules, left
  modules, and right modules clearly arise from the symmetrizations of the
  multicategories $\mathbf{BM}$, $\mathbf{LM}$ and $\mathbf{RM}$,
  respectively. By Proposition~\ref{propn:SymApprox} it therefore
  follows that the \icats{} of modules defined here are equivalent to
  those defined in \cite{HA}.
\end{remark}

\subsection{Lurie's Enriched $\infty$-Categories}\label{subsec:LurEnr}
In this section we describe a variant of Lurie's definition of
enriched \icats{} in \cite[\S 4.2.1]{HA}.

\begin{defn}
  A \defterm{weakly enriched \icat{}} is a fibration of \gnsiopds{} $q
  \colon \mathcal{M} \to \RM$ such that the fibres $\mathcal{M}_{(0)}$
  and $\mathcal{M}_{(1)}$ are contractible. We write
  $\mathcal{M}_{\mathfrak{r}}^{\otimes}$ for the \nsiopd{}
  $j^{*}\mathcal{M} \to \simp^{\op}$ and $\mathcal{M}_{\mathfrak{m}}$
  for the fibre $\mathcal{M}_{(0,1)}$ and say that $q$ \emph{exhibits
    $\mathcal{M}_{\mathfrak{m}}$ as weakly enriched in
    $\mathcal{M}_{\mathfrak{r}}$}.
\end{defn}

\begin{ex}
  Let $\mathcal{O}$ be any \nsiopd{}. The pullback
  $\pi^{*}\mathcal{O} \to \RM$ along the projection $\pi
  \colon \RM \to \simp^{\op}$ exhibits $\mathcal{O}_{[1]}$ as weakly
  enriched in $\mathcal{O}$.
\end{ex}

\begin{ex}
  Let $q \colon \mathcal{M} \to \RM$ be a weakly enriched \icat{} such
  that $q$ is also a coCartesian fibration. Then we say that $q$
  exhibits $\mathcal{M}_{\mathfrak{m}}$ as \emph{right-tensored} over
  $\mathcal{M}_{\mathfrak{r}}$, which is a monoidal \icat{}. Clearly,
  an \icat{} $\mathcal{C}$ is right-tensored over a monoidal \icat{}
  $\mathcal{V}$ \IFF{} there exists an $\RM$-algebra $F
  \colon \RM \to \CatI^{\times}$ such that $F(0,1) \simeq \mathcal{C}$
  and $j^{*}F$ is an associative algebra corresponding to
  $\mathcal{V}^{\otimes}$.
\end{ex}

\begin{defn}
  Let $q \colon \mathcal{M} \to \RM$ be a weakly enriched
  \icat{}. Given $C_{1},\ldots, C_{n} \in \mathcal{M}_{\mathfrak{r}}$
  and $M,N \in \mathcal{M}_{\mathfrak{m}}$, we write
  \[\Map_{\mathcal{M}_{\mathfrak{m}}}(M \boxtimes (C_{1},\ldots,C_{n}),
  N)\] for $\Map_{\mathcal{M}}^{\phi}((M, C_{1},\ldots,C_{n}), N)$,
  where $\phi \colon [n+1] \to [1]$ is the unique active map.
\end{defn}

\begin{defn}
  A \emph{pseudo-enriched \icat{}} is a weakly enriched \icat{} $q
  \colon \mathcal{M} \to \RM$ such that
  $\mathcal{M}_{\mathfrak{r}}^{\otimes}$ is a monoidal \icat{}, and
  for all $C_{1},\ldots,C_{n} \in \mathcal{M}_{\mathfrak{r}}$ ($n =
  0,1,\ldots$) and $M,N \in \mathcal{M}_{\mathfrak{m}}$, the canonical
  map
  \[ \Map_{\mathcal{M}_{\mathfrak{m}}}(M \boxtimes
  (C_{1}\otimes\cdots\otimes C_{n}), N) \to
  \Map_{\mathcal{M}_{\mathfrak{m}}}(M \boxtimes (C_{1}, \ldots,
  C_{n}), N) \] is an equivalence.
\end{defn}

\begin{remark}\label{rmk:PsEnrI}
  Taking $n = 0$ in this definition, we see that in a pseudo-enriched
  \icat{} $\mathcal{M}$ we have
  \[ \Map_{\mathcal{M}_{\mathfrak{m}}}(M \boxtimes I, N) \simeq
  \Map_{\mathcal{M}_{\mathfrak{m}}}(M, N).\]
\end{remark}

\begin{ex}
  The pullback $\pi^{*}\mathcal{O} \to \RM$ exhibits
  $\mathcal{O}_{[1]}$ as pseudo-enriched in $\mathcal{O}$ \IFF{}
  $\mathcal{O}$ is a monoidal \icat{}.
\end{ex}

\begin{ex}
  If a weakly enriched \icat{} $q \colon \mathcal{M} \to \RM$ is a
  coCartesian fibration, then it is clearly pseudo-enriched.
\end{ex}

\begin{defn}
  Let $\mathcal{M} \to \RM$ be a pseudo-enriched \icat{}. Suppose
  $M$ and $N$ are objects of $\mathcal{M}_{\mathfrak{m}}$; a
  \emph{morphism object} for $M,N$ is an object $F(M,N) \in
  \mathcal{M}_{\mathfrak{r}}$ together with a map $\alpha \in
  \Map_{\mathcal{M}_{\mathfrak{m}}}(M \boxtimes F(M,N), N)$ such that
  for every $C \in \mathcal{M}_{\mathfrak{r}}$ composition with
  $\alpha$ induces an equivalence
  \[ \Map_{\mathcal{M}_{\mathfrak{r}}}(C, F(M,N)) \to
  \Map_{\mathcal{M}_{\mathfrak{m}}}(M \boxtimes C, N).\] We say that
  $\mathcal{M} \to \RM$ is a \emph{Lurie-enriched \icat{}} if there exists
  a morphism object in $\mathcal{M}_{\mathfrak{r}}$ for all $M,N \in
  \mathcal{M}_{\mathfrak{m}}$.
\end{defn}

\begin{remark}\label{rmk:LurEnrI} 
  From Remark~\ref{rmk:PsEnrI} we see that in a Lurie-enriched \icat{}
  $\mathcal{M}$ there is a natural equivalence
  \[ \Map_{\mathcal{M}_{\mathfrak{r}}}(I, F(M,N)) \isoto
  \Map_{\mathcal{M}_{\mathfrak{m}}}(M, N).\]
\end{remark}

\begin{ex}
  A monoidal \icat{} $\mathcal{V}$ is \emph{left-closed}
  \IFF{} for every $C \in \mathcal{V}$ the functor $C \otimes (\blank)
  \colon \mathcal{V} \to \mathcal{V}$ has a right adjoint. If
  $\mathcal{V}$ is a monoidal \icat{}, the pullback
  $\pi^{*}\mathcal{V}^{\otimes}$ exhibits $\mathcal{V}$ as
  Lurie-enriched in $\mathcal{V}$ \IFF{}
  $\mathcal{V}$ is left-closed.
\end{ex}

\begin{ex}\label{ex:RightTensLurEnr}
  More generally, suppose the \icat{} $\mathcal{C}$ is right-tensored
  over the monoidal \icat{} $\mathcal{V}$. The associated
  coCartesian weakly enriched \icat{} $q \colon \mathcal{M} \to \RM$
  is Lurie-enriched \IFF{} for every $M \in \mathcal{C}$, the
  right-tensoring functor $M \otimes (\blank) \colon \mathcal{V} \to
  \mathcal{C}$ has a right adjoint $F(M,\blank)$ (so that
  $\Map_{\mathcal{V}}(V, F(M,N)) \simeq \Map_{\mathcal{C}}(M \otimes
  V, N)$.  
\end{ex}

\begin{remark}
  We use \emph{right} modules rather than the left modules used in
  \cite[\S 4.2.1]{HA} so that the composition maps of morphism objects
  are compatible with those for categorical algebras: If $\mathcal{M}$
  is a Lurie-enriched \icat{} in our sense, then for a triple $A,B,C$
  of objects in $\mathcal{M}_{\mathfrak{m}}$ we get a composition map
  $F(A,B) \otimes F(B,C) \to F(A,C)$, whereas \cite[Definition
  4.2.1.28]{HA} gives composition maps $F(B,C) \otimes F(A,B) \to
  F(A,C)$. This is why we get Lurie-enriched \icats{} from
  \emph{left}-closed monoidal \icats{} rather than right-closed ones
  as in \cite[Example 4.2.1.32]{HA}.
\end{remark}

\begin{defn}
  Let $\EnrLur$ be the full subcategory of $(\OpdInsg)_{/\RM}$ spanned
  by the Lurie-enriched \icats{}. Pullback along the inclusion
  $\simp^{\op} \to \RM$ induces a projection $\EnrLur \to \MonI$; we
  write $\EnrLurV$ for the fibre at $\mathcal{V}^{\otimes} \in
  \MonI$. This is the \icat{} of \emph{Lurie-$\mathcal{V}$-enriched
    \icats{}}.
\end{defn}

We expect that the \icat{} $\EnrLurV$ is equivalent to the \icat{}
$\CatIV$ of complete categorical algebras in $\mathcal{V}$ defined
above, but we will not attempt to prove this here. 

\subsection{Enriched Strings}\label{subsec:EnrStr}
We now describe the analogue for our variant definition of Lurie's
construction of an \icat{} of \emph{enriched strings} in \cite[\S
4.7.2]{HA}.

\begin{defn}
  Let $\Po$ denote the full subcategory of $\Fun([1], \simp)$
  spanned by the inert morphisms. In other words, an object of
  $\Po$ is an inert morphism $\alpha \colon [i] \to [n]$, or
  equivalently an object $[n] \in \simp$ together with a subinterval
  $\{j, j+1, \ldots, j+i\} \subseteq [n]$. A morphism from $\alpha$ to
  $\beta \colon [j] \to [m]$ is a commutative diagram
  \csquare{{[i]}}{{[n]}}{{[j]}}{{[m]}.}{\alpha}{\psi}{\phi}{\beta}
  Note that, since $\alpha$ and $\beta$ are inert, a morphism $\psi$
  factoring $\phi \circ \alpha$ through $\beta$ is uniquely
  determined, if it exists. The inclusions $i_{0}, i_{1} \colon [0]
  \hookrightarrow [1]$ taking the unique object of $[0]$ to $0, 1$,
  respectively, induce functors $\Phi, \Theta \colon \Po \to
  \simp$. We write $\Po'$ for the full subcategory of
  $\Po$ spanned by the (necessarily inert) morphisms $[0] \to [n]$.
\end{defn}

\begin{defn}
  We define a map $\chi \colon \simp^{\op} \to \RM$ by sending $[n]$
  to the object $(0,1,\ldots,1)$ over $[n+1]$ and $\phi \colon [m] \to
  [n]$ to the coCartesian map over $[0] \star \phi \colon [m+1] \to
  [n+1]$. Thus the composite $\simp^{\op} \to \RM \to \simp^{\op}$ is
  given by $[0] \star \blank$.
\end{defn}

\begin{defn}
  Suppose $\mathcal{M} \to \RM$ is
  a weakly enriched \icat{}. Define simplicial sets
  $\overline{\StrM}^{\txt{en}}$ and $\StrM$ over
  $\simp^{\op}$ by the universal properties
  \[ \Hom_{\simp^{\op}}(K, \overline{\StrM}^{\txt{en}}) \simeq
  \Hom_{\RM}(K \times_{\simp^{\op}}
  \Po^{\op}, \mathcal{M}),\]
  \[ \Hom_{\simp^{\op}}(K, \StrM) \simeq
  \Hom_{\RM}(K \times_{\simp^{\op}} (\Po')^{\op}, \mathcal{M}),\]
  where the map $\Po^{\op} \to \RM$ is given
  by the composite
  \[ \Po^{\op} \xto{\Phi} \simp^{\op} \xto{\chi} \RM.\]
\end{defn}

\begin{lemma}
  The \icat{} $\StrM$ is equivalent to $\simp^{\op}_{\mathcal{M}_{\mathfrak{m}}}$.
\end{lemma}
\begin{proof}
  This is immediate from Remark~\ref{rmk:DopSdescr}.
\end{proof}

\begin{defn}
  Let $\Po_{[n]}$ denote the fibre of $\Theta \colon \Po \to \simp$ at
  $[n]$, i.e. the full subcategory of $\simp_{/[n]}$ spanned by inert
  morphisms. A morphism in $\Po_{[n]}$ is thus a commutative diagram
  \opctriangle{{[i]}}{{[j]}}{{[n]}}{\phi}{\alpha}{\beta} where
  $\alpha$ and $\beta$ are inert --- if such a morphism exists then
  the morphism $\phi$ is clearly uniquely determined by $\alpha$ and
  $\beta$, and must also be inert. We can thus equivalently describe
  $\Po_{[n]}$ as the category associated to the partially ordered set
  of subintervals of $[n]$. We write $\Phi_{[n]}$ for
  $\Phi|_{\Po_{[n]}}$.
\end{defn}

\begin{defn}
  The unique map $[0] \to [-1]$ in $\simp_{+}^{\op}$ induces a natural
  transformation $[0] \star (\blank) \to [-1] \star (\blank) = \id$ of
  functors $\simp^{\op} \to \simp^{\op}$; this is given by
  $d_{0}\colon [n] \to [n-1]$ for all $n = 1,\ldots$. Since $d_{0}$ is
  inert, we can define a natural transformation $\overline{\chi} \colon
  \Delta^{1} \times \simp^{\op} \to \RM$ by taking the coCartesian
  lift of this starting at $\chi$. Thus $\overline{\chi}_{[n]}$ is given by
  $d_{0} \colon (0,1,\ldots,1) \to (1,\ldots,1)$.
\end{defn}


\begin{defn}
  Let $q \colon \mathcal{M} \to \RM$ be a weakly enriched \icat{}.
  An \defterm{enriched $n$-string} in
  $\mathcal{M}$ is a functor $\sigma \colon \Po^{\op}_{[n]}\to
  \mathcal{M}$ such that:
  \begin{enumerate}[(1)]
  \item The composite $q \circ \sigma$ is 
    \[ \Po_{[n]}^{\op} \xto{\Phi_{[n]}^{\op}} \simp^{\op} \xto{\chi} \RM.\]
  \item If \opctriangle{{[i]}}{{[j]}}{{[n]}}{\phi}{\alpha}{\beta} is a
    morphism in $\Po_{[n]}$ such that $\alpha(0) = \beta(0)$ (or
    equivalently $\phi(0) = 0$), then $\sigma(\phi)$ is inert.
    (Notice that if $\phi \colon [n] \to [m]$ is an inert map in
    $\simp^{\op}$, then $[0] \star \phi$ is inert \IFF{} $\phi(0) = 0$,
    so these are precisely the maps $\phi$ so that $\sigma(\phi)$ lies
    over an inert map in $\simp^{\op}$.)
  \item Let $\sigma \to \sigma'$ be a coCartesian lift of
    $\overline{\chi}|_{\Delta^{1} \times \Po^{\op}_{[n]}}$. Then for any
    morphism $\phi$ in $\Po_{[n]}$, the morphism $\sigma'(\phi)$ is
    inert in $\mathcal{M}_{\mathfrak{r}}^{\otimes}$.
  \end{enumerate}
\end{defn}

\begin{remark}
  An enriched 0-string is a map $* \simeq \Po_{[0]}^{\op} \to
  \mathcal{M}$ over the map $* \to \RM$ sending $*$ to
  $(0,1)$, i.e. just an object of $\mathcal{M}_{\mathfrak{m}}$. An
  enriched 1-string corresponds to a map $(M, C) \to N$ over $d_{1}
  \colon (0,1,1) \to (0,1)$; if $\mathcal{M}$ is a Lurie-enriched
  \icat{} then this is equivalent to a map $C \to F(M,N)$. In general,
  an enriched $n$-string corresponds to a sequence of maps \[(M_{0},
  C_{1}, \ldots, C_{n}) \to (M_{1}, C_{2}, \ldots, C_{n}) \to \cdots
  \to (M_{n-1}, C_{n}) \to M_{n},\] together with coherence data,
  where each map is the identity on the components after the first
  two. If $\mathcal{M}$ is a Lurie-enriched \icat{}, then this is
  equivalent to a sequence of maps
  \[ C_{1} \to F(M_{0}, M_{1}), \quad C_{2} \to F(M_{1},M_{2}), \quad
  \dots \quad C_{n} \to F(M_{n-1}, M_{n}).\]
\end{remark}

\begin{defn}
  The fibre of $\overline{\StrM}^{\txt{en}}$ at $[n]$ is clearly
  $\Fun_{\RM}(\Po^{\op}_{[n]}, \mathcal{M})$. We write $\StrMen$ for
  the full subcategory of $\overline{\StrM}^{\txt{en}}$ spanned by the
  enriched $n$-strings for all $n$.
\end{defn}

\begin{propn}\label{propn:StrMentoD}
  Let $q \colon \mathcal{M} \to \RM$ be a weakly enriched
  \icat{}. Then:
  \begin{enumerate}[(i)]
  \item The projection $p \colon \StrMen \to
    \simp^{\op}$ is a categorical fibration.
  \item For every $X \in \StrMen$ and every inert
    morphism $\alpha \colon p(X) \to [n]$ in $\simp^{\op}$ there
    exists a $p$-coCartesian morphism $X \to \alpha_{!}X$ over
    $\alpha$.
  \item Let $\overline{\alpha} \colon X \to Y$ be a morphism in
    $\StrMen$ such that $p(\overline{\alpha}) \colon
    [m] \to [n]$ is an inert morphism. Then $\overline{\alpha}$ is
    $p$-coCartesian \IFF{} for all $\phi \colon [k] \to [n]$ in
    $\Po^{\op}_{[n]}$ the induced map $X(\alpha \circ \phi) \to
    Y(\phi)$ is an equivalence.
  \item Suppose $q$ is a coCartesian fibration. Then so is $p$, and a
    morphism $\overline{\alpha} \colon X \to Y$ in $\StrMen$ over $\alpha
    \colon [m] \to [n]$ in $\simp^{\op}$ is
    $p$-coCartesian \IFF{} for every $\phi \colon [k] \to [n]$ in
    $\Po^{\op}_{[n]}$ the induced map $X(\alpha \circ \phi) \to
    Y(\phi)$ is $q$-coCartesian.
  \end{enumerate}
\end{propn}
\begin{proof}
  As \cite[Proposition 4.7.2.23]{HA}.
\end{proof}

\begin{propn}\label{propn:StrMSegCond} 
  Let $q \colon \mathcal{M} \to \RM$ be a weakly enriched
  \icat{}. Then the projection $p \colon \StrMen \to \simp^{\op}$
  satisfies the Segal condition, i.e. for each $[n]$, the map
  \[ \StrMen_{[n]} \to \StrMen_{[1]} \times_{\StrMen_{[0]}} \cdots
  \times_{\StrMen_{[0]}}\StrMen_{[1]} \]
  is an equivalence.
\end{propn}
\begin{proof}
  As \cite[Proposition 4.7.2.13]{HA}.
\end{proof}

\begin{propn}\label{propn:StrMentoDM}
  Let $q \colon \mathcal{M} \to \simp^{\op}$ be a weakly enriched
  \icat{}. Then:
  \begin{enumerate}[(i)]
  \item The projection $r \colon \StrMen \to \simp^{\op}_{\mathcal{M}}$ is a
    categorical fibration.
  \item Given $X \in \StrMen$ and an inert morphism $\alpha \colon
    r(X) \to Y$ in $\simp^{\op}_{\mathcal{M}}$, there exists an
    $r$-coCartesian morphism $X \to \alpha_{!}X$ over $\alpha$.
  \item Suppose $X \in \StrMen$, $\alpha \colon r(X) \to Y$ is a
    morphism in $\simp^{\op}_{\mathcal{M}}$, and $\alpha_{0}\colon [m]
    \to [n]$ is the underlying morphism in $\simp^{\op}$. Then a
    morphism $\overline{\alpha} \colon X \to \overline{Y}$ over $\alpha$ is
    $r$-coCartesian \IFF{} $\overline{\alpha}$ induces an equivalence
    $X(\alpha \circ \phi) \to \overline{Y}(\phi)$ for all $\phi \colon [k]
    \to [n]$ in $\Po^{\op}_{[n]}$.
  \item Suppose $q$ is a coCartesian fibration, and let $r_{0}$ denote
    the projection $\simp^{\op}_{\mathcal{M}} \to \simp^{\op}$. Given
    $X \in \StrMen$ and an $r_{0}$-coCartesian morphism $\alpha \colon
    r(X) \to Y$ in $\simp^{\op}_{\mathcal{M}}$, there exists an
    $r$-coCartesian morphism $\overline{\alpha} \colon X \to \alpha_{!}X$
    in $\StrMen$ over $\alpha$. Moreover, if $X \in \StrMen$ and
    $\alpha \colon r(X) \to Y$ in $\simp^{\op}_{\mathcal{M}}$ is
    $r_{0}$-coCartesian over $\alpha_{0} \colon [m] \to [n]$ in
    $\simp^{\op}$, then a morphism $X \to \overline{Y}$ in $\StrMen$ over
    $\alpha$ is $r$-coCartesian \IFF{} the induced map $X(\alpha \circ
    \phi) \to \overline{Y}(\phi)$ in $\mathcal{M}$ is $q$-coCartesian for
    all $\phi \in \Po^{\op}_{[n]}$.
  \end{enumerate}
\end{propn}
\begin{proof}
  As \cite[Lemma 4.7.2.27]{HA}.
\end{proof}

\begin{defn}
  Define $\overline{\StrM}^{\txt{en},+} \to \simp^{\op}$ by
  \[ \Hom_{\simp^{\op}}(K, \overline{\StrM}^{\txt{en},+}) \cong
  \Hom_{\RM}(\Delta^{1} \times K \times_{\simp^{\op}} \Po^{\op},
  \mathcal{M}),\] where the map $\Delta^{1} \times \Po^{\op}
  \to \RM$ is the composite of $\id \times \Phi^{\op} \colon
  \Delta^{1}\times \Po^{\op} \to \Delta^{1} \times \simp^{\op}$ with
  the natural transformation $\overline{\chi}$.
\end{defn}

\begin{defn}
  Suppose $q \colon \mathcal{M} \to \RM$ is a weakly enriched
  \icat{}. Let $\StrM^{\txt{en},+}$ denote the full subcategory of
  $\overline{\StrM}^{\txt{en},+} \to \simp^{\op}$ spanned by objects
  $F \colon \Delta^{1} \times \Po^{\op}_{[n]} \to \mathcal{M}$ such
  that $F|_{\{0\} \times \Po^{\op}_{[n]}}$ is an enriched $n$-string
  and $F$ is a $q$-left Kan extension of $F|_{\{0\} \times
    \Po^{\op}_{[n]}}$.
\end{defn}

\begin{lemma}
  The projection $\StrM^{\txt{en},+} \to \StrM^{\txt{en}}$ is a
  trivial fibration.
\end{lemma}
\begin{proof}
  Immediate from \cite[Proposition 4.3.2.15]{HTT}.
\end{proof}

\begin{defn}
  Let $i \colon \simp^{\op} \to \Delta^{1} \times \Po^{\op}$ be the
  functor that sends $[n]$ to $(1, \id \colon [n] \to [n])$. Then
  composition with $i$ induces a functor $\StrM^{\txt{en},+} \to
  \mathcal{M}_{\mathfrak{r}}^{\otimes}$ over $\simp^{\op}$.
\end{defn}

\begin{lemma}\label{lem:StremPlprInert}
  Let $q \colon \mathcal{M} \to \RM$ be a weakly enriched \icat{}. The
  functor $\StrM^{\txt{en},+} \to \mathcal{M}_{\mathfrak{r}}^{\otimes}$ preserves
  inert morphisms.
\end{lemma}
\begin{proof}
  This is obvious from the definitions and
  Proposition~\ref{propn:StrMentoD}.
\end{proof}

\subsection{Extracting a Categorical Algebra}\label{subsec:ExtCatAlg}
In this subsection we will extract a categorical algebra from a coCartesian
Lurie-enriched \icat{}, and consider some examples of enriched \icats{}
that arise in this way.

\begin{defn}
  Suppose $q \colon \mathcal{M} \to \RM$ is a weakly enriched
  \icat{}. Let $\StrM^{\txt{en}}_{\iota}$ be defined by the pullback
  \nolabelcsquare{\StrM^{\txt{en}}_{\iota}}{\StrM^{\txt{en}}}{\simp^{\op}_{\iota
      \mathcal{M}_{\mathfrak{m}}}}{\simp^{\op}_{\mathcal{M}_{\mathfrak{m}}}.}
\end{defn}

\begin{lemma}\label{lem:StrMicoCart}
  Suppose $q \colon \mathcal{M} \to \RM$ is a coCartesian weakly
  enriched \icat{}. Then the projection $\StrM^{\txt{en}}_{\iota} \to
  \simp^{\op}_{\iota \mathcal{M}}$ is a coCartesian fibration.
\end{lemma}
\begin{proof}
  This follows immediately from Proposition~\ref{propn:StrMentoDM}
  since the projection $\pi \colon \simp^{\op}_{\iota \mathcal{M}} \to
  \simp^{\op}$ is a left fibration, so all morphisms in
  $\simp^{\op}_{\iota \mathcal{M}}$ are $\pi$-coCartesian.
\end{proof}

\begin{remark}
  We expect that Lemma~\ref{lem:StrMicoCart} is also true for
  pseudo-enriched \icats{} that are not coCartesian fibrations, but
  since this is not needed for the examples we are interested in we
  will not consider this generalization here.
\end{remark}

\begin{defn}
  Suppose $q \colon \mathcal{M} \to \RM$ is a Lurie-enriched
  \icat{}. Let $\StrM^{\txt{en}}_{\txt{eq}}$ be the full subcategory
  of $\StrM^{\txt{en}}_{\iota}$ spanned by enriched $n$-strings
  $\sigma \colon \Po^{\op}_{[n]} \to \mathcal{M}$ such that for $i =
  1,\ldots, n$, the map \[\sigma(\{i,i+1\} \hookrightarrow
  [n]) \simeq (M_{i}, C_{i}) \to M_{i+1} \simeq \sigma(\{i+1\}
  \hookrightarrow [n])\] exhibits $C_{i}$ as the morphism object
  $F(M_{i},M_{i+1})$.
\end{defn}

\begin{propn}
  Suppose $q \colon \mathcal{M} \to \RM$ is a coCartesian
  Lurie-enriched \icat{}. Then the projection
  $\StrM^{\txt{en}}_{\txt{eq}} \to \simp^{\op}_{\iota \mathcal{M}}$ is
  a trivial fibration.
\end{propn}
\begin{proof}
  The universal property of the morphism object $F(M,N)$ implies that
  the universal map $(M, F(M,N)) \to N$ is the final object in the
  fibre of $\StrMen \to \simp^{\op}_{\mathcal{M}}$ over $(M,N)$. The
  Segal condition (Proposition~\ref{propn:StrMSegCond}) implies that
  $\StrM^{\txt{en}}_{\txt{eq}}$ is precisely the full subcategory of
  $\StrM^{\txt{en}}_{\iota}$ spanned by the objects that are final in
  their fibre. It therefore follows by \cite[Proposition
  2.4.4.9(1)]{HTT} that the projection $\StrM^{\txt{en}}_{\txt{eq}}
  \to \simp^{\op}_{\iota \mathcal{M}}$ is a trivial fibration.
\end{proof}

\begin{defn}
  Suppose $q \colon \mathcal{M} \to \RM$ is a Lurie-enriched
  \icat{}. Let $\StrM^{\txt{en},+}_{\txt{eq}}$ be defined by the
  pullback square
  \nolabelcsquare{\StrM^{\txt{en},+}_{\txt{eq}}}{\StrM^{\txt{en},+}}{\StrM^{\txt{en}}_{\txt{eq}}}{\StrM^{\txt{en}}.}
\end{defn}

\begin{thm}
  Suppose $q \colon \mathcal{M} \to \RM$ is a coCartesian
  Lurie-enriched \icat{}. The
  composite \[\overline{\mathcal{M}}_{\mathfrak{m}} \colon \simp^{\op}_{\iota
    \mathcal{M}} \xleftarrow{\sim} \StrM^{\txt{en}}_{\txt{eq}}
  \xleftarrow{\sim} \StrM^{\txt{en},+}_{\txt{eq}} \to
  \mathcal{M}_{\mathfrak{r}}^{\otimes}\] is a categorical algebra in
  $\mathcal{M}_{\mathfrak{r}}$.
\end{thm}
\begin{proof}
  It follows from Lemma~\ref{lem:StremPlprInert} that this map
  preserves inert morphisms, so it is a categorical algebra. 
\end{proof}

\begin{remark}
  We expect that, as suggested by Remark~\ref{rmk:LurEnrI}, in the
  situation above the underlying \icat{} of
  $\overline{\mathcal{M}}_{\mathfrak{m}}$ is equivalent to
  $\mathcal{M}_{\mathfrak{m}}$. This would imply that the categorical
  algebra $\overline{\mathcal{M}}_{\mathfrak{m}}$ is in fact complete.
  We will not prove this here, however, as this requires developing
  more of the theory of Lurie-enriched \icats{} than is
  appropriate here.
\end{remark}

Using Example~\ref{ex:RightTensLurEnr} we can restate this as:
\begin{cor}
  Suppose $\mathcal{V}$ is a monoidal \icat{} and
  $\mathcal{C}$ is an \icat{} that is right-tensored over
  $\mathcal{V}$ so that the tensor product $C \otimes
  (\blank)$ has a right adjoint $F(C, \blank) \colon \mathcal{C} \to
  \mathcal{V}$ for all $C \in \mathcal{C}$. Then $\mathcal{C}$ is
  enriched in $\mathcal{V}$; more precisely, there is a
  categorical algebra $\overline{\mathcal{C}} \colon \simp^{\op}_{\iota
    \mathcal{C}} \to \mathcal{V}^{\otimes}$ such that
  $\overline{\mathcal{C}}(C,D) \simeq F(C,D)$.
\end{cor}

This construction allows us to construct several interesting examples
of enriched \icats{}:
\begin{cor}
  Suppose $\mathcal{V}$ is a left-closed monoidal
  \icat{}. Then $\mathcal{V}$ is enriched in itself; more precisely,
  there exists a categorical algebra $\overline{\mathcal{V}}
  \colon \simp^{\op}_{\iota \mathcal{V}} \to \mathcal{V}^{\otimes}$
  such that $\overline{\mathcal{V}}(V,W)$ in $\mathcal{V}$ is the
  internal hom from $V$ to $W$.
\end{cor}

\begin{ex}
  Suppose $\mathcal{V}$ is a presentably $\mathbb{E}_{2}$-monoidal
  \icat{}; then $\CatIV$ is a presentably monoidal \icat{}, and so is
  in particular right-closed. Thus there exists a
  $\mathcal{V}$-\itcat{} $\overline{\Cat}_{\infty}^{\mathcal{V}}$ of
  $\mathcal{V}$-\icats{}. More generally, if $\mathcal{V}$ is
  presentably $\mathbb{E}_{k}$-monoidal (or presentably symmetric
  monoidal), there exists a $\mathcal{V}$-$(\infty,n+1)$-category
  $\overline{\Cat}_{(\infty,n)}^{\mathcal{V}}$ of
  $\mathcal{V}$-$(\infty,n)$-categories for all $n < k$. For example, taking
  $\mathcal{V}$ to be $\mathcal{S}$ there is an
  $(\infty,n+1)$-category $\overline{\Cat}_{(\infty,n)}^{\mathcal{S}}$
  of $(\infty,n)$-categories, and taking $\mathcal{V}$ to be
  $\Set$ there is an $(n+1)$-category
  $\overline{\Cat}_{n}$ of $n$-categories.
\end{ex}

\begin{remark}
  Several homotopy theories that can easily be constructed as 
  spectral presheaves $\Fun^{\Sp}(\mathcal{A}^{\op},
  \overline{\Sp})$, where $\mathcal{A}$ is a small spectral
  category, can (conjecturally) be identified with more familiar
  homotopy theories:
  \begin{enumerate}[(i)]
  \item Suppose $G$ is a finite group, and let $\mathcal{B}^{G}$
    denote the \emph{Burnside (2,1)-category} of $G$; this has objects
    finite $G$-sets, 1-morphisms spans of finite $G$-sets, and
    2-morphisms isomorphisms of spans. We can regard this as a
    category enriched in symmetric monoidal groupoids, via the
    coproduct, and hence as an \icat{} enriched in
    $E_{\infty}$-spaces. Group completion of $E_{\infty}$-spaces is a
    lax monoidal functor from $E_{\infty}$-spaces to (connective)
    spectra, so applying this to the mapping spaces in
    $\mathcal{B}^{G}$ gives a spectral \icat{}
    $\mathcal{B}^{G}_{+}$. The presheaf spectral \icat{}
    $\Fun^{\Sp}(\mathcal{B}^{G,\op}_{+}, \overline{\Sp})$ is the
    spectral \icat{} of \emph{genuine $G$-spectra} --- a version of
    this comparison has recently been proved by Guillou and
    May~\cite{GuillouMay1,GuillouMay2,GuillouMay3} using enriched
    model categories. (It has also been observed by Barwick that (as
    group-completion is a left adjoint) it is not necessary to
    group-complete the mapping spaces in $\mathcal{B}^{G}$ to describe
    $G$-spectra; this is the basis for the \icatl{} description of
    $G$-spectra in \cite{BarwickMackey}.)
  \item Let $\mathcal{B}$ denote the \emph{global Burnside
      (2,1)-category of finite groups}. This has objects finite
    groups, 1-morphisms from $G$ to $H$ are finite free $H$-sets
    equipped with a compatible $G$-action, and 2-morphisms are
    isomorphisms of these. This can also be regarded as enriched in
    symmetric monoidal groupoids via coproducts, and by
    group-completing we obtain a spectral \icat{}
    $\mathcal{B}_{+}$. The presheaf spectral \icat{}
    $\Fun^{\Sp}(\mathcal{B}^{\op}_{+}, \overline{\Sp})$ is the
    spectral \icat{} of \emph{global equivariant spectra} for finite
    groups, as studied by Schwede~\cite{SchwedeGlobal}.
  \end{enumerate}
\end{remark}

\begin{cor}
  Suppose $\mathcal{V}$ is a presentably monoidal \icat{},
  and $\mathcal{C}$ is a right $\mathcal{V}$-module in
  $\PresI$ (with respect to the tensor product of presentable \icats{}).
  Then $\mathcal{C}$ is enriched in $\mathcal{V}$.
\end{cor}

\begin{ex}
  By \cite[Proposition 4.8.2.18]{HA}, presentable stable \icats{} are
  precisely $\Sp$-modules in $\PresI$, hence presentable stable \icats{}
  are enriched in spectra. But any stable \icat{} is a full
  subcategory of its Ind-completion, hence it follows that all stable
  \icats{} are enriched in spectra.
\end{ex}

\begin{ex}
  In \cite[\S 6]{DAG7}, Lurie defines \emph{R-linear \icats{}} for an
  $\mathbb{E}_{2}$-ring spectrum $R$ to be left
  $\txt{LMod}_{R}$-modules in $\PresI$. If we instead consider right
  $\txt{LMod}_{R}$-modules we get \icats{} enriched in left
  $R$-modules from $R$-linear \icats{}. Moreover, if $R$ is at least
  $\mathbb{E}_{3}$-monoidal (so that $\txt{LMod}_{R}$ is at least
  $\mathbb{E}_{2}$-monoidal), then these two notions coincide.
\end{ex}

\appendix

\section{Technicalities on $\infty$-Operads}\label{sec:algcolims}
In this appendix we collect the more technical results we need about
\nsiopds{}.

\subsection{Monoidal Envelopes}\label{subsec:monenv}
In this subsection we describe the non-symmetric version of
Lurie's \emph{monoidal envelope} of an \iopd{} $\mathcal{O}$, which
gives a monoidal structure on the \icat{} $\mathcal{O}_{\txt{act}}$ of
active morphisms in $\mathcal{O}$ that we will make use of below to
define operadic colimits.

\begin{defn}
  Let $\txt{Act}(\simp^{\op})$ be the full subcategory of
  $\Fun(\Delta^{1}, \simp^{\op})$ spanned by the active morphisms. If
  $\mathcal{M}$ is a \gnsiopd{}, we define $\txt{Env}(\mathcal{M})$ to
  be the fibre product
  \[ \mathcal{M} \times_{\Fun(\{0\}, \simp^{\op})}
  \txt{Act}(\simp^{\op}).\]
\end{defn}

\begin{propn}\label{propn:EnvDbl}
  The map $\txt{Env}(\mathcal{M}) \to \simp^{\op}$ induced
  by evaluation at $1$ in $\Delta^{1}$ is a double \icat{}.
\end{propn}
\begin{proof}
  As \cite[Proposition 2.2.4.4]{HA}.
\end{proof}

\begin{propn}\label{propn:EnvAdj}
  Suppose $\mathcal{N}$ is a double \icat{} and $\mathcal{M}$ is a
  \gnsiopd{}. The inclusion $\mathcal{M} \to \txt{Env}(\mathcal{M})$ induces an
  equivalence
  \[ \Fun^{\otimes}(\txt{Env}(\mathcal{M}), \mathcal{N}) \to
  \Alg_{\mathcal{M}}(\mathcal{N}).\]
\end{propn}
\begin{proof}
  As \cite[Proposition 2.2.4.9]{HA}.
\end{proof}

\begin{cor}
  Suppose $\mathcal{O}$ is a \nsiopd{}. Then
  $\txt{Env}(\mathcal{O})$ is a monoidal \icat{}, and if
  $\mathcal{C}^{\otimes}$ is a monoidal \icat{} then 
  \[ \Fun^{\otimes}(\txt{Env}(\mathcal{O}), \mathcal{C}^{\otimes}) \simeq
  \Alg_{\mathcal{O}}(\mathcal{C}).\]
\end{cor}
\begin{proof}
  The only object of $\simp$ that admits an active map from $[0]$ is
  $[0]$, hence for any \gnsiopd{} $\mathcal{M}$ we have
  $\txt{Env}(\mathcal{M})_{[0]} \simeq \mathcal{M}_{[0]}$. In
  particular $\txt{Env}(\mathcal{O})_{[0]} \simeq *$ for a
  \nsiopd{} $\mathcal{O}$ , so the result follows from
  Proposition~\ref{propn:EnvDbl} and Proposition~\ref{propn:EnvAdj}.
\end{proof}

\begin{defn}
  If $\mathcal{O}$ is a \nsiopd{}, the monoidal \icat{}
  $\txt{Env}(\mathcal{O})$ is called the \emph{monoidal envelope}
  of $\mathcal{O}$. This gives a monoidal structure on the
  subcategory $\mathcal{O}_{\txt{act}}$ of
  $\mathcal{O}$ determined by the active morphisms. We
  denote this tensor product on $\mathcal{O}_{\txt{act}}$ by
  $\oplus$.
\end{defn}

\subsection{Operadic Colimits}
We wish to prove that, under reasonable hypotheses, if
$\mathcal{V}$ is a monoidal \icat{} and $f \colon
\mathcal{O} \to \mathcal{P}$ is a morphism of
\nsiopds{} then the functor
\[ f^{*}\colon \Alg_{\mathcal{P}}(\mathcal{V}) \to
\Alg_{\mathcal{O}}(\mathcal{V}) \] given by
composition with $f$ has a left adjoint. This depends on an existence
theorem for \emph{operadic left Kan extensions}, which makes use of
the concept of \emph{operadic colimits} that we introduce in this
subsection.

\begin{defn}
  Suppose $q \colon \mathcal{O} \to \simp^{\op}$ is a
  \nsiopd{}. Given a diagram $p \colon K \to
  \mathcal{O}_{\txt{act}}$ we write
  $\mathcal{O}^{\txt{act}}_{[1],p/} := \mathcal{O}_{[1]}
  \times_{\mathcal{O}}
  (\mathcal{O}_{\txt{act}})_{p/}$. A diagram $\overline{p}
  \colon K^{\triangleright} \to \mathcal{O}_{\txt{act}}$ is
  a \defterm{weak operadic colimit diagram} if the induced map $\psi
  \colon \mathcal{O}^{\txt{act}}_{[1],\overline{p}/} \to
  \mathcal{O}^{\txt{act}}_{[1],p/}$ is a categorical equivalence.

  A diagram $\overline{p} \colon K^{\triangleright} \to
  \mathcal{O}_{\txt{act}}$ is an \defterm{operadic colimit
    diagram} if the composite functors
  \[ K^{\triangleright} \to \mathcal{O}_{\txt{act}}
  \xto{\blank \oplus X} \mathcal{O}_{\txt{act}}\]
  \[ K^{\triangleright} \to \mathcal{O}_{\txt{act}}
  \xto{X \oplus \blank} \mathcal{O}_{\txt{act}}\]
  are weak operadic colimit diagrams for all $X \in
  \mathcal{O}$.
\end{defn}

\begin{remark}
  By \cite[Proposition 2.1.2.1]{HTT}, the map $\psi$ in the definition
  of weak operadic colimits is always a left fibration, hence it is a
  categorical equivalence \IFF{} it is a trivial Kan fibration.
\end{remark}



\begin{lemma}\label{lem:opdcoliminproduct}
  Suppose $\mathcal{O}$ and $\mathcal{P}$ are
  \nsiopds{}, and $\overline{p} \colon K^{\triangleright} \to
  \mathcal{O}_{\txt{act}}$ and $\overline{q} \colon
  L^{\triangleright} \to \mathcal{P}_{\txt{act}}$ are
  weak operadic colimit diagrams. Then the composite 
\[ \overline{r} \colon (K \times_{\simp^{\op}} L)^{\triangleright} \to
K^{\triangleright} \times_{\simp^{\op}} L^{\triangleright} \to
\mathcal{O} \times_{\simp^{\op}} \mathcal{P} \]
is also a weak operadic colimit diagram. Moreover, if $\overline{p}$ and
$\overline{q}$ are operadic colimit diagrams, so is $\overline{r}$.
\end{lemma}
\begin{proof}
  Let $r := \overline{r}|_{K\times_{\simp^{\op}} L}$. Then we must
  show that the map $(\mathcal{O}_{[1]} \times
  \mathcal{P}_{[1]})^{\txt{act}}_{\overline{r}/} \to
  (\mathcal{O}_{[1]} \times \mathcal{P}_{[1]})^{\txt{act}}_{r/}$ is a
  categorical equivalence. We have a commutative diagram
  \factortriangle{(\mathcal{O}_{[1]} \times
    \mathcal{P}_{[1]})^{\txt{act}}_{(\overline{p},\overline{q})/}}{(\mathcal{O}_{[1]}
    \times \mathcal{P}_{[1]})^{\txt{act}}_{r/}}{(\mathcal{O}_{[1]}
    \times \mathcal{P}_{[1]})^{\txt{act}}_{\overline{r}/.}}{}{}{} We
  clearly have an equivalence $(\mathcal{O}_{[1]} \times
  \mathcal{P}_{[1]})^{\txt{act}}_{(\overline{p},\overline{q})/} \simeq
  \mathcal{O}^{\txt{act}}_{[1],\overline{p}/} \times
  \mathcal{P}^{\txt{act}}_{[1],\overline{q}/}$, and so the top
  horizontal map is the product of the equivalences
  $\mathcal{O}^{\txt{act}}_{[1],\overline{p}/} \to
  \mathcal{O}^{\txt{act}}_{[1],p/}$ and
  $\mathcal{P}^{\txt{act}}_{[1],\overline{q}/} \to
  \mathcal{P}^{\txt{act}}_{[1],q/}$ and hence is an equivalence. By
  the 2-out-of-3 property it therefore suffices to show that the left
  diagonal map in the diagram is an equivalence. But this is true
  because the inclusion $(K \times L)^{\triangleright} \hookrightarrow
  K^{\triangleright} \times L^{\triangleright}$ is right anodyne. (By
  \cite[Proposition 4.1.2.1]{HTT} it suffices to prove this inclusion
  is cofinal, and the criterion of ``Theorem A'', \cite[Theorem
  4.1.3.1]{HTT}, clearly holds in this case.) It is then clear from
  the definition of the monoidal structure on $(\mathcal{O}
  \times_{\simp^{\op}} \mathcal{P})_{\txt{act}}$ that if
  $\overline{p}$ and $\overline{q}$ are operadic colimits, then so is
  $\overline{r}$.
\end{proof}

\begin{propn}
  Let $\mathcal{O}$ be a \nsiopd{}, and suppose given
  finitely many operadic colimit diagrams $\overline{p}_{i} \colon
  K^{\triangleright}_{i} \to \mathcal{O}_{\txt{act}}$, $i =
  0, \ldots, n$. Let
  $K = \prod_{i} K_{i}$, and let $\overline{p}$ be the composite
  \[ K^{\triangleright} \to \prod_{i} K_{i}^{\triangleright} \to
  \prod_{i} \mathcal{O}_{\txt{act}} \simeq
  \txt{Env}(\mathcal{O})_{[n]} \xto{\oplus}
  \mathcal{O}_{\txt{act}}.\]
  Then $\overline{p}$ is an operadic colimit diagram.
\end{propn}
\begin{proof}
  As \cite[Proposition 3.1.1.8]{HA}.
\end{proof}

\begin{lemma}
  Suppose $K$ is a sifted simplicial set, and $\mathcal{V}$ is a monoidal \icat{} that is compatible with
  $K$-indexed colimits. Then $\phi_{!} \colon
  \mathcal{V}^{\otimes}_{[n]} \to \mathcal{V}^{\otimes}_{[m]}$ preserves
  $K$-indexed colimits for all $\phi$ in $\simp^{\op}$.
\end{lemma}
\begin{proof}
  As \cite[Lemma 3.2.3.7]{HA}.
\end{proof}

\begin{propn}\label{propn:wocdinmon}
  Let $\mathcal{V}$ be a monoidal \icat{}, and let $\overline{p}
  \colon K^{\triangleright} \to \mathcal{V}^{\otimes}_{[m]}$ be a diagram. Then
  $\overline{p}$ is a weak operadic colimit diagram \IFF{} the composite
  \[ K^{\triangleright} \to \mathcal{V}^{\otimes}_{[m]} \xto{r_{!}}
  \mathcal{V} \]
  is a colimit diagram, where $r$ is the unique active map $[m] \to [1]$.
\end{propn}
\begin{proof}
  This follows as in the proof of \cite[Proposition 3.1.1.7]{HA}.
\end{proof}

\begin{cor}\label{cor:ocdinmon}
    Let $\mathcal{V}$ be a monoidal \icat{}, and let $\overline{p}
  \colon K^{\triangleright} \to \mathcal{V}^{\otimes}_{[m]}$ be a diagram. Then
  $\overline{p}$ is an operadic colimit diagram \IFF{} for every object $Y
  \in \mathcal{V}^{\otimes}$ the
  composites
  \[ K^{\triangleright} \to \mathcal{V}^{\otimes}_{[m]} \xto{ \blank \oplus Y}
  \mathcal{V}^{\otimes}_{[n+m]} \xto{r_{!}} \mathcal{C}\]
  \[ K^{\triangleright} \to \mathcal{V}^{\otimes}_{[m]} \xto{Y \oplus
    \blank } \mathcal{V}^{\otimes}_{[n+m]} \xto{r_{!}} \mathcal{V}\]
  are colimit diagrams in $\mathcal{V}$, $Y$ lies over $[n]$
  in $\simp^{\op}$ and $r$ is the unique active map $[n+m] \to
  [1]$.
\end{cor}

\begin{propn}\label{propn:ocdnattr}
  Let $q \colon \mathcal{O} \to \simp^{\op}$ be a \nsiopd{}, and suppose given a map
  $\overline{h} \colon \Delta^{1} \times K^{\triangleright} \to
  \mathcal{O}_{\txt{act}}$; write $\overline{h}_{i} :=
  \overline{h}|_{\{i\}\times K^{\triangleright}}$, $i = 0,1$. Suppose that
  \begin{enumerate}[(a)]
  \item For every vertex $x$ of $K^{\triangleright}$, the restriction
    $\overline{h}|_{\Delta^{1}\times \{x\}}$ is a $q$-coCartesian edge of
    $\mathcal{O}$.
  \item The composite map \[\Delta^{1} \times \{\infty\}
    \hookrightarrow \Delta^{1} \times K^{\triangleright} \xto{\overline{h}}
    \mathcal{O} \xto{q} \simp^{\op}\]
    is an equivalence in $\simp^{\op}$.
  \end{enumerate}
  Then $\overline{h}_{0}$ is a weak operadic colimit diagram \IFF{}
  $\overline{h}_{1}$ is a weak operadic colimit diagram. Moreover, if
  $\mathcal{O}$ is a monoidal \icat{}, then $\overline{h}_{0}$ is
  an operadic colimit diagram \IFF{} $\overline{h}_{1}$ is an operadic
  colimit diagram.
\end{propn}
\begin{proof}
  As \cite[Proposition 3.1.1.15]{HA}.
\end{proof}

\begin{cor}\label{cor:StrMonPrOCD}
  Let $\mathcal{V}$ and $\mathcal{W}$ be
  monoidal \icats{} compatible with small colimits, and suppose $F
  \colon \mathcal{V}^{\otimes} \to \mathcal{W}^{\otimes}$ is a
  monoidal functor such that $F_{[1]} \colon \mathcal{V} \to \mathcal{W}$
  preserves colimits. Then composition with $F$ preserves operadic
  colimit diagrams.
\end{cor}
\begin{proof}
  Suppose $\overline{p} \colon K^{\triangleright}\to \mathcal{V}^{\otimes}$ is an operadic
  colimit diagram. We wish to show that the composite map
  $K^{\triangleright} \to \mathcal{W}^{\otimes}$ is also an operadic
  colimit diagram. By Proposition \ref{propn:ocdnattr} we may assume
  that $\overline{p}$
  lands in a fibre $\mathcal{V}^{\otimes}_{[m]}$. We now apply
  Corollary \ref{cor:ocdinmon} to conclude that it suffices to show
  that the composites
  \[ K^{\triangleright} \to \mathcal{V}^{\otimes}_{[m]} \to
  \mathcal{W}^{\otimes}_{[m]} \xto{\blank
    \oplus Y} \mathcal{W}^{\otimes}_{[n+m]} \xto{r_{!}} \mathcal{W}\]
  \[ K^{\triangleright} \to \mathcal{V}^{\otimes}_{[m]} \to
  \mathcal{W}^{\otimes}_{[m]} \xto{Y \oplus \blank}
  \mathcal{W}^{\otimes}_{[n+m]} \xto{r_{!}} \mathcal{W},\] where $r$
  is the unique active map $[n+m] \to [1]$, are colimit diagrams, for
  all $[n]$ and all $Y \in \mathcal{V}^{\otimes}_{[n]}$. Observe that
  the functors $r_{!}(\blank \oplus Y)$ and $r_{!}(Y \oplus \blank)$
  are equivalently given by $\mu_{!}(r'_{!}(\blank) \oplus
  r''_{!}(Y))$ and $\mu_{!}(r''_{!}(Y) \oplus r'_{!}(\blank))$, where
  $r' \colon [m] \to [1]$, $r'' \colon [n] \to [1]$ and $\mu \colon
  [2] \to [1]$ are the unique active maps between these objects. Since
  $\mu_{!}$ preserves colimits in each variable in both
  $\mathcal{V}^{\otimes}$ and $\mathcal{W}^{\otimes}$, it suffices to
  show that
  \[ K^{\triangleright} \to \mathcal{W}^{\otimes}_{[m]} \xto{r'_{!}}
  \mathcal{W}\] is a colimit diagram. But we have a commutative
  diagram
  \csquare{\mathcal{V}^{\otimes}_{[m]}}{\mathcal{W}^{\otimes}_{[m]}}{\mathcal{V}}{\mathcal{W}}{F^{\otimes}_{[m]}}{r'_!}{r'_!}{F}
  so this is true since $K^{\triangleright} \to
  \mathcal{V}^{\otimes}_{[m]} \to \mathcal{V}$ is a colimit diagram
  and $F_{[1]}$ preserves colimits.
\end{proof}

\begin{propn}\label{propn:OpdColimsExist}
  Let $q \colon \mathcal{V}^{\otimes} \to \simp^{\op}$ be a monoidal
  \icat{} compatible with $K$-indexed colimits for some simplicial set
  $K$. Suppose given a diagram $\overline{p} \colon K^{\triangleright} \to
  \mathcal{V}^{\otimes}_{\txt{act}}$ that sends the cone point
  $\infty$ to an object in $\mathcal{V}$. Let $\overline{q}
  \colon K^{\triangleright} \to \mathcal{V}^{\otimes}$ be a
  coCartesian lift of $\overline{p}$ along the active maps to $[1]$. Then
  $\overline{p}$ is an operadic colimit diagram \IFF{} $\overline{q}$ is a
  colimit diagram. In particular, given a diagram $p \colon
  K^{\triangleright} \to \mathcal{V}^{\otimes}_{\txt{act}}$ there
  exists an operadic colimit diagram $\overline{p} \colon
  K^{\triangleright} \to \mathcal{V}^{\otimes}_{\txt{act}}$ extending
  $p$ that sends $\infty$ to an object of $\mathcal{V}$.
\end{propn}
\begin{proof}
  As \cite[Proposition 3.1.1.20]{HA}.
\end{proof}

\subsection{Operadic Kan Extensions}\label{subsec:OpdKanExt}
We now discuss operadic Kan extensions in the non-symmetric case. Here
we work in slightly more generality than for the corresponding results
in \cite{HA} --- the proof of Lurie's existence theorem can also be
used to construct operadic Kan extensions along a restricted class of
morphisms of \gnsiopds{}.

\begin{defn}
  Let $\mathcal{C}$ be an \icat{}. A \defterm{$\mathcal{C}$-family of
    (generalized) \nsiopds{}} is a categorical fibration
  $\pi \colon \mathcal{O} \to \simp^{\op} \times
  \mathcal{C}$ such that:
  \begin{enumerate}[(i)]
  \item For $C \in \mathcal{C}$, $X \in \mathcal{O}_{C}$, and
    $\alpha$ an inert morphism in $\simp^{\op}$ from the
    image of $X$, there exists a coCartesian morphism $X \to Y$ over
    $\alpha$ in $\mathcal{O}_{C}$.
  \item For $X \in \mathcal{O}_{C}$ with image $[n] \in
    \simp^{\op}$ let $p_{X} \colon K_{[n]}^{\triangleleft} \to
    \mathcal{O}$ be a coCartesian lift of $p_{[n]} \colon
    K_{[n]} \to \simp^{\op}$ (or
    consider a lift of $\mathcal{G}^{\txt{ns}}_{[n]/} \to \simp^{\op}$
    for a \gnsiopd{}). Then $p_{X}$ is a $\pi$-limit diagram.
  \item For each $C \in \mathcal{C}$, the induced map
    $\mathcal{O}_{C} \to \simp^{\op}$ is a
    (generalized) \nsiopd{}.
  \end{enumerate}
  A $\Delta^{1}$-family will also be referred to as a
  \defterm{correspondence of (generalized) \nsiopds{}}.
\end{defn}

\begin{defn}
  We say a correspondence $\mathcal{M} \to \simp^{\op} \times
  \Delta^{1}$ of generalized \nsiopds{} is \emph{constant
    over $[0]$} if the restriction $\mathcal{M}_{[0]} \to \Delta^{1}$
  is a coCartesian fibration whose associated functor $\Delta^{1} \to
  \CatI$ is an equivalence.
\end{defn}

\begin{defn}
  Let $\mathcal{M} \to \simp^{\op} \times \Delta^{1}$ be a
  correspondence from a \gnsiopd{} $\mathcal{A}$ to a \gnsiopd{}
  $\mathcal{B}$ that is constant over $[0]$ and such that
  $\mathcal{A}_{[0]}$ and $\mathcal{B}_{[0]}$ are Kan complexes, let
  $\mathcal{O}$ be a \nsiopd{}, and let $\overline{F} \colon
  \mathcal{M} \to \mathcal{O}$ be a map of \gnsiopds{}. The
  map $\overline{F}$ is an \defterm{operadic left Kan extension} of $F =
  \overline{F}|_{\mathcal{A}}$ if for every $B \in \mathcal{B}_{[1]}$ the
  composite map
  \[ \left((\mathcal{M}_{\txt{act}})_{/B}
    \times_{\mathcal{M}}
    \mathcal{A}\right)^{\triangleright} \to
  (\mathcal{M}_{/B})^{\triangleright} \to
  \mathcal{M} \xto{\overline{F}} \mathcal{O} \]
  is an operadic colimit diagram.
\end{defn}

\begin{thm}\label{thm:OpdcKanExt}\ 
  \begin{enumerate}[(i)]
  \item Suppose given a $\Delta^{1}$-family of \gnsiopds{}
    $\mathcal{M} \to \simp^{\op} \times \Delta^{1}$ constant over
    $[0]$ and such that $\mathcal{M}_{[0],i}$ is a Kan complex for $i
    = 0,1$, a \nsiopd{} $\mathcal{O}$ and a commutative
    diagram of \gnsiopd{} family maps \csquare{\mathcal{M}
      \times_{\Delta^{1}}
      \{0\}}{\mathcal{O}}{\mathcal{M}}{\simp^{\op}.}{f}{}{}{}
    Then there exists an operadic left Kan extension $\overline{f}$ of $f$
    \IFF{} for every $B$ in $\mathcal{M} \times_{\Delta^{1}} \{1\}$,
    the diagram
    \[ (\mathcal{M}_{\txt{act}})_{/B} \times_{\Delta^{1}} \{0\} \to
    \mathcal{M} \times_{\Delta^{1}} \{0\} \xto{f}
    \mathcal{O} \] can be extended to an operadic colimit
    diagram lifting
  \[ \left((\mathcal{M}_{\txt{act}})_{/B} \times_{\Delta^{1}}
    \{0\}\right)^{\triangleright} \to \mathcal{M} \to \simp^{\op}.\]
\item Suppose given a $\Delta^{n}$-family of \gnsiopds{} $\mathcal{M}
  \to \simp^{\op} \times \Delta^{n}$ with $n \geq 1$ such that all
  sub-$\Delta^{1}$-families are constant over $[0]$ and the fibres
  $\mathcal{M}_{[0],i}$ are all Kan complexes, a
  \nsiopd{} $\mathcal{O}$, and a commutative diagram of
  \gnsiopd{} family maps \csquare{\mathcal{M} \times_{\Delta^{n}}
    \Lambda^n_0}{\mathcal{O}}{\mathcal{M}}{\simp^{\op}}{f}{}{}{}
  such that the restriction of $f$ to $\mathcal{M} \times_{\Delta^{n}}
  \Delta^{\{0,1\}}$ is an operadic left Kan extension of
  $f|_{\mathcal{M} \times_{\Delta^{n}} \{0\}}$. Then there exists a
  morphism $\overline{f} \colon \mathcal{M} \to \mathcal{O}$
  extending $f$.
\end{enumerate}
\end{thm}
\begin{proof}
  As \cite[Theorem 3.1.2.3]{HA}.
\end{proof}

\subsection{Free Algebras}
Let $\mathcal{V}$ be a monoidal \icat{} compatible with
small colimits and let $f \colon \mathcal{A} \to \mathcal{B}$ be a
functor of \gnsiopds{} that is an equivalence over $[0]$ and such that
$\mathcal{A}_{[0]}$ and $\mathcal{B}_{[0]}$ are Kan complexes.  Using the
existence theorem for operadic left Kan extensions, we can now
construct an adjoint to the functor
\[ f^{*} \colon \Alg_{\mathcal{B}}(\mathcal{V}) \to
\Alg_{\mathcal{A}}(\mathcal{V}) \] given by composition with
$f$. This is given by forming \emph{free} algebras:

\begin{defn}
  Let $\mathcal{A}$ and $\mathcal{B}$ be \gnsiopds{}, let
  $\mathcal{O}$ be a \nsiopd{}, and let $f \colon
  \mathcal{A} \to \mathcal{B}$ be a map of \gnsiopds{} that is an equivalence over $[0]$ and such that
$\mathcal{A}_{[0]}$ and $\mathcal{B}_{[0]}$ are Kan complexes. Suppose $A \in
  \Alg_{\mathcal{A}}(\mathcal{O})$, $B \in
  \Alg_{\mathcal{B}}(\mathcal{O})$, and $\phi
  \colon A \to f^{*}B$ is a map of $\mathcal{A}$-algebras in
  $\mathcal{O}$. For $b \in \mathcal{B}_{[1]}$, let
  $(\mathcal{A}_{\txt{act}})_{/b} := \mathcal{A} \times_{\mathcal{B}}
  (\mathcal{B}_{\txt{act}})_{/b}$. Then $A$ and $B$ induce maps
  $\alpha, \beta \colon (\mathcal{A}_{\txt{act}})_{/b} \to
  \mathcal{O}_{\txt{act}}$ and $\phi$ determines a natural
  transformation $\eta \colon \alpha \to \beta$. The map $\beta$
  clearly extends to $\overline{\beta} \colon
  (\mathcal{A}_{\txt{act}})_{/b} \to
  (\mathcal{O}_{\txt{act}})_{/B(b)}$. Since the projection
  \[(\mathcal{O}_{\txt{act}})_{/B(b)} \to
  \mathcal{O}_{\txt{act}} \times_{\simp^{\op}_{\txt{act}}}
  (\simp^{\op}_{\txt{act}})_{/[n]}\] (where $b$ lies over $[n] \in
  \simp^{\op}$) is a right fibration, we can lift $\eta$ to an
  essentially unique map $\overline{\eta} \colon \overline{\alpha} \to \overline{\beta}$
  (over $\simp^{\op}$). We say that $\phi$ \defterm{exhibits $B$ as a
    free $\mathcal{B}$-algebra generated by $A$} if for every $b \in
  \mathcal{B}_{[1]}$ the map $\overline{\alpha}$ determines an operadic
  $q$-colimit diagram $(\mathcal{A}_{\txt{act}})_{/b}^{\triangleright}
  \to \mathcal{O}$.
\end{defn}

\begin{remark}\label{rmk:FreeIsKan}
  The map $\phi \colon A \to f^{*}B$ above determines a map
  \[ H \colon (\mathcal{A} \times \Delta^{1})
  \amalg_{\mathcal{A} \times \{1\}} \mathcal{B}
  \to \mathcal{O} \times \Delta^{1}.\]
  Choose a factorization of $H$ as 
  \[ H \colon (\mathcal{A} \times \Delta^{1})
  \amalg_{\mathcal{A} \times \{1\}} \mathcal{B}
  \xto{H'} \mathcal{M} \xto{H''} \mathcal{O}
  \times \Delta^{1},\]
  where $H'$ is a categorical equivalence and $\mathcal{M}$
  is an \icat{}. The composite map $\mathcal{M} \to
  \simp^{\op} \times \Delta^{1}$ exhibits
  $\mathcal{M}$ as a correspondence of \nsiopds{}. Then the
  map $\phi$ exhibits $B$ as a free $\mathcal{B}$-algebra
  generated by $A$ \IFF{} $H''$ is an operadic left Kan extension.
\end{remark}

\begin{propn}
  Suppose $\phi \colon A \to f^{*}B$ exhibits $B$ as a free
  $\mathcal{B}$-algebra in $\mathcal{O}$ generated
  by $A$. Then for every $B' \in
  \Alg_{\mathcal{B}}(\mathcal{O})$
  composition with $\phi$ induces a homotopy equivalence
  \[
  \Map(B, B') \to \Map(A, f^{*}B').\]
\end{propn}
\begin{proof}
  As \cite[Proposition 3.1.3.2]{HA}.
\end{proof}

\begin{propn}
  Suppose $A \in
  \Alg_{\mathcal{A}}(\mathcal{O})$. Then
  there exists a free $\mathcal{B}$-algebra $B$ generated by
  $A$ \IFF{} for every $b \in \mathcal{B}_{[1]}$ the induced
  map
  \[ (\mathcal{A}_{\txt{act}})_{/b} \to
  \mathcal{A}_{\txt{act}} \xto{A} \mathcal{O} \]
  can be extended to an operadic colimit diagram lying over
  \[ (\mathcal{A}_{\txt{act}})_{/b}^{\triangleright} \to
  \mathcal{B}_{\txt{act}} \to \simp^{\op}_{\txt{act}}.\]
\end{propn}
\begin{proof}
  As \cite[Proposition 3.1.3.3]{HA}.
\end{proof}

\begin{cor}
  Let $\mathcal{O}$ be a \nsiopd{}, and suppose $f \colon
  \mathcal{A} \to \mathcal{B}$ is a map of
  \gnsiopds{} that is an equivalence over $[0]$ and such that
$\mathcal{A}_{[0]}$ and $\mathcal{B}_{[0]}$ are Kan complexes. The functor $f^{*} \colon
  \Alg_{\mathcal{B}}(\mathcal{O})  \to
  \Alg_{\mathcal{A}}(\mathcal{O})$
  admits a left adjoint $f_{!}$, provided that for every $\mathcal{A}$-algebra
  $A$ in $\mathcal{O}$ and every $b \in
  \mathcal{B}_{*}$, the diagram
  \[ (\mathcal{A}_{\txt{act}})_{/b} \to
  \mathcal{A}_{\txt{act}} \xto{A} \mathcal{O} \]
  can be extended to an operadic colimit diagram lying over
  \[ (\mathcal{A}_{\txt{act}})_{/b}^{\triangleright} \to
  \mathcal{B}_{\txt{act}} \to \simp^{\op}_{\txt{act}}.\]
\end{cor}
\begin{proof}
  As \cite[Corollary 3.1.3.4]{HA}.
\end{proof}

Combining this with Proposition \ref{propn:OpdColimsExist}, we get:
\begin{thm}\label{thm:FreeFtrExists}
  Suppose $\mathcal{V}$ is a monoidal \icat{} compatible
  with $\kappa$-small colimits for some uncountable regular cardinal
  $\kappa$, and $f \colon \mathcal{A} \to \mathcal{B}$ is a map of
  \gnsiopds{} that is an equivalence over $[0]$ and such that
$\mathcal{A}_{[0]}$ and $\mathcal{B}_{[0]}$ are Kan complexes, with
  $\mathcal{A}$ and $\mathcal{B}$ essentially $\kappa$-small. Then the
  functor $f^{*} \colon
  \Alg_{\mathcal{B}}(\mathcal{V}) \to
  \Alg_{\mathcal{A}}(\mathcal{V})$ admits a
  left adjoint $f_{!}$.
\end{thm}

\begin{lemma}\label{lem:StrMonPrFree}
  Suppose $\mathcal{V}$ and $\mathcal{W}$ are
  monoidal \icats{} which are compatible with small colimits, and let $F \colon
  \mathcal{V}^{\otimes} \to \mathcal{W}^{\otimes}$ be a
  monoidal functor such that $F_{[1]} \colon \mathcal{V} \to
  \mathcal{W}$ preserves colimits. Then for every \gnsiopd{}
  $\mathcal{M}$ the induced functor
  \[F_{*} \colon \Alg_{\mathcal{M}}(\mathcal{V}) \to
  \Alg_{\mathcal{M}}(\mathcal{W})\] preserves free algebras,
  i.e. for all maps of \gnsiopds{} $f \colon \mathcal{N} \to
  \mathcal{M}$ that are equivalences over $[0]$ and such that
  $\mathcal{M}_{[0]}$ and $\mathcal{N}_{[0]}$ are Kan complexes, the
  natural map $f_{!}F_{*} \to F_{*}f_{!}$ (adjoint to $F_{*} \to
  F_{*}f^{*}f_{!} \simeq f^{*}F_{*}f_{!}$) is an equivalence.
\end{lemma}
\begin{proof}
  This follows immediately from Corollary \ref{cor:StrMonPrOCD}.
\end{proof}

We can also give a more explicit description of the left adjoint
$\tau_{\mathcal{M},!}$, where $\mathcal{M}$ is a \gnsiopd{} such that
$\mathcal{M}_{[0]}$ is a Kan complex. Recall that by
Proposition~\ref{propn:TrivAlgEq} if $\mathcal{O}$ is a \nsiopd{} then
we have $\Alg_{\mathcal{M}_{\txt{triv}}}(\mathcal{O}) \simeq
\Fun(\mathcal{M}_{[1]}, \mathcal{O}_{[1]})$. We can therefore regard
$\tau_{\mathcal{M},!}$ as a functor
\[ \Fun(\mathcal{M}_{[1]}, \mathcal{O}_{[1]}) \to
\Alg_{\mathcal{M}}(\mathcal{O}).\]

\begin{defn}
  For $[n] \in \simp^{\op}$ and $X \in
  \mathcal{M}_{[1]}$, let
  $\mathcal{P}_{X,n}^{\mathcal{M}}$ be the full subcategory
  of $\mathcal{M}_{\txt{triv}}
  \times_{\mathcal{M}} \mathcal{M}_{/X}$ of
  morphisms $Y \to X$ over the active map $[n] \to [1]$.
\end{defn}

Suppose $\mathcal{V}$ is a monoidal \icat{} and $F \colon
\mathcal{M}_{[1]} \to \mathcal{V}$ is a functor. Let
$\overline{F}$ be the associated $\mathcal{M}_{\txt{triv}}$-algebra in
$\mathcal{V}$. We have a canonical map $h \colon
\mathcal{P}^{\mathcal{M}}_{X,n} \times \Delta^{1} \to \mathcal{M}$, a
natural transformation from $\mathcal{P}^{\mathcal{M}}_{X,n} \to
\mathcal{M}_{\txt{triv}} \hookrightarrow \mathcal{M}$ to the constant
functor at $X$. Since $\mathcal{V}^{\otimes}\to \simp^{\op}$ is
coCartesian, from $\overline{F} \circ h$ we get a coCartesian natural
transformation $\overline{h}$ from a functor $g \colon
\mathcal{P}^{\mathcal{M}}_{X,n} \to \mathcal{V}$ to
the constant functor at $F(X)$. We let
$\mathrm{P}^{n}_{\mathcal{M},X}(F)$ denote a colimit of $g$, if it
exists.

\begin{propn}\label{propn:FreeAlgMonad}
  Suppose $\mathcal{V}$ is a monoidal \icat{} compatible
  with $\kappa$-small colimits, and $\mathcal{M}$ is a $\kappa$-small
  \gnsiopd{} such that $\mathcal{M}_{[0]}$ is a Kan complex. Suppose moreover that $A$ is an
  $\mathcal{M}$-algebra in $\mathcal{V}$ and $F \colon
  \mathcal{M}_{[1]} \to \mathcal{V}$ is a
  functor. Then a map $F \to (\tau_{\mathcal{M}})^{*}A$ is adjoint to
  an equivalence $\tau_{\mathcal{M},!}F \isoto A$ if and only if for
  every $X \in \mathcal{M}_{[1]}$ the maps
  $\txt{P}^{n}_{\mathcal{M},X}(F) \to A(X)$ exhibit $A(X)$ as a
  coproduct
  \[  \coprod_{[n] \in \simp^{\op}}
  \txt{P}^{n}_{\mathcal{M},X}(F) \to A(X)\]
\end{propn}
\begin{proof}
  As \cite[Proposition 3.1.3.13]{HA}.
\end{proof}

\subsection{Colimits of Algebras in Monoidal $\infty$-Categories}
In this subsection we show that colimits exist in the \icats{}
$\Alg_{\mathcal{O}}(\mathcal{V})$ for all
small \nsiopds{} $\mathcal{O}$ when $\mathcal{V}$ is a
monoidal \icat{} compatible with small colimits. We first consider the
case of sifted colimits:

\begin{lemma}
  Suppose $K$ is a sifted simplicial set and
  $\mathcal{V}$ is a
  monoidal \icat{} that is compatible with $K$-indexed
  colimits. Then for every $\phi \colon [n] \to [m]$ in
  $\simp^{\op}$ the associated functor $\phi_{!} \colon
  \mathcal{V}_{[n]}^{\otimes} \to \mathcal{V}^{\otimes}_{[m]}$ preserves
  $K$-indexed colimits.
\end{lemma}
\begin{proof}
  As \cite[Lemma 3.2.3.7]{HA}.
\end{proof}

\begin{lemma}\label{lem:ColimCoCartEdge}
  Suppose $p \colon X \to S$ is a coCartesian fibration, and let $\overline{r} \colon K^{\triangleright}
  \to \Fun(\Delta^{1}, X)$ be a colimit diagram such that for every $i \in K$
  the edge $\overline{r}(i,0) \to \overline{r}(i,1)$ is coCartesian. Then the
  edge $\overline{r}(\infty,0) \to \overline{r}(\infty,1)$ is also coCartesian.
\end{lemma}
\begin{proof}
  Since colimits in functor categories are pointwise, we must show
  that for all $x \in X$ the diagram
  \nolabelcsquare{\Map_{X}(\colim_{i} \overline{r}(i,1),
    x)}{\Map_{X}(\colim_{i} \overline{r}(i,0), x)}{\Map_{S}(\colim_{i}
    p\overline{r}(i,1), p(x))}{\Map_{S}(\colim_{i} p \overline{r}(i,0), p(x))}
  is Cartesian, which is clear since limits commute.
\end{proof}

To describe sifted colimits of algebras, we need the following result,
which is due to Jacob Lurie --- we thank him for explaining the proof
to us.
\begin{thm}\label{thm:MysteryResult}
  Let $K$ be a weakly contractible simplicial set. Suppose $p \colon X
  \to S$ is a coCartesian fibration such that for all $s \in S$ the
  fibre $X_{s}$ admits $K$-indexed colimits, and for all edges $f
  \colon s \to t$ in $S$ the functor $f_{!} \colon X_{s} \to X_{t}$
  preserves $K$-indexed colimits. Then for any map $g \colon T \to S$,
  \begin{enumerate}[(i)]
  \item the \icat{} $\Fun_{S}(T, X)$ admits $K$-indexed colimits,
  \item a map $K^{\triangleright} \to \Fun_{S}(T, X)$ is a colimit
    diagram \IFF{} for all $t \in T$ the composite
    \[ K^{\triangleright} \to \Fun_{S}(T, X) \to X_{g(t)} \]
    is a colimit diagram,
  \item if $E$ is a set of edges of $T$, the full subcategory of
    $\Fun_{S}(T, X)$ spanned by functors that take the edges in $E$
    to coCartesian edges of $X$ is closed under $K$-indexed colimits in
    $\Fun_{S}(T, X)$.
  \end{enumerate}
\end{thm}
\begin{proof}
  The \icat{} $\Fun_{S}(T, X)$ is a fibre of the functor $p_{*} \colon
  \Fun(T, X) \to \Fun(T, S)$ induced by composition with $p$. The
  functor $p_{*}$ is a coCartesian fibration by \cite[Proposition
  3.1.2.1]{HTT}. Since the functors $f_{!}$ preserve $K$-indexed
  colimits, by \cite[Proposition 4.3.1.10]{HTT} a diagram $\overline{q}
  \colon K^{\triangleright} \to \Fun_{S}(T, X)$ is a colimit diagram
  \IFF{} the composite $\overline{q}' \colon K^{\triangleright} \to
  \Fun_{S}(T, X) \to \Fun(T, X)$ is a $p_{*}$-colimit diagram. By
  \cite[Corollary 4.3.1.11]{HTT}, $K$-indexed $p_{*}$-colimits exist
  in $\Fun(T, X)$, which proves (i).

  Moreover, a diagram in $\Fun(T,X)$ is a colimit diagram \IFF{} it is
  a $p_{*}$-colimit diagram and its image in $\Fun(T, S)$ is a colimit
  diagram. Since $\overline{q}'$ lands in one of the fibres of $p_{*}$, the
  projection to $\Fun(T, S)$ is constant, which means it is a colimit
  as $K$ is weakly contractible. Thus $\overline{q}'$ is a $p_{*}$-colimit
  diagram \IFF{} it is a colimit diagram in $\Fun(T,X)$. By
  \cite[Corollary 5.1.2.3]{HTT} this means that $\overline{q}'$ is a
  colimit diagram \IFF{} for all $t \in T$ the induced maps
  $K^{\triangleright} \to X$ are colimit diagrams. A diagram in $X$ is
  a colimit \IFF{} it is a $p$-colimit and the projection to $S$ is a
  colimit. Since $K$ is weakly contractible, applying
  \cite[Proposition 4.3.1.10]{HTT} we see that this is true \IFF{} the
  induced map $K^{\triangleright} \to X_{g(t)}$ is a colimit diagram
  in $X_{g(t)}$. This proves (ii).

  Suppose $e \colon t \to t'$ is an edge of $T$ and $q \colon K \to
  \Fun_{S}(T, X)$ is a diagram such that for all vertices $k \in K$
  the functor $q(k) \colon T \to X$ takes $e$ to a $p$-coCartesian
  edge of $X$. Let $\overline{q} \colon K^{\triangleright} \to \Fun_{S}(T,
  X)$ be a colimit diagram extending $q$. To prove (iii) we must show
  that the functor $\overline{q}(\infty)$ also takes $e$ to a coCartesian
  edge of $X$. From our description of colimits in $\Fun_{S}(T, X)$ it
  follows that this is equivalent to showing that coCartesian edges of
  $X$ are closed under colimits, which is true by
  Lemma~\ref{lem:ColimCoCartEdge}.
\end{proof}

\begin{cor}\label{cor:AlgSiftedColim}
  Suppose $K$ is a sifted simplicial set and
  $\mathcal{V}$ is a monoidal \icat{} that
  is compatible with $K$-indexed colimits. Then for any generalized
  \nsiopd{} $p \colon \mathcal{M} \to \simp^{\op}$, we have:
  \begin{enumerate}[(i)]
  \item The \icat{}
    $\Fun_{\simp^{\op}}(\mathcal{M},
    \mathcal{V}^{\otimes})$ admits $K$-indexed colimits.
  \item A map $K^{\triangleright} \to
    \Fun_{\simp^{\op}}(\mathcal{M},
    \mathcal{V}^{\otimes})$ is a colimit diagram \IFF{} for every $X
    \in \mathcal{M}$ the induced diagram
    $K^{\triangleright} \to \mathcal{V}^{\otimes}_{p(X)}$ is a colimit diagram.
  \item The full subcategory
    $\Alg_{\mathcal{M}}(\mathcal{V})$ of
    $\Fun_{\simp^{\op}}(\mathcal{M},
    \mathcal{V}^{\otimes})$ is stable under $K$-indexed colimits.
  \item A map $K^{\triangleright} \to
    \Fun_{\simp^{\op}}(\mathcal{M},
    \mathcal{V}^{\otimes})$ is a colimit diagram \IFF{}, for every $X
    \in \mathcal{M}_{[1]}$, the induced diagram
    $K^{\triangleright} \to \mathcal{V}$ is a colimit
    diagram.
  \item The restriction functor
    $\Alg_{\mathcal{M}}(\mathcal{V})
    \to \Fun(\mathcal{M}_{[1]}, \mathcal{V})$
    detects $K$-indexed colimits.
  \end{enumerate}
\end{cor}
\begin{proof}
  Sifted simplicial sets are weakly contractible by \cite[Proposition
  5.5.8.7]{HTT} so (i)--(iii) follow from
  Theorem~\ref{thm:MysteryResult} (which is implicit in the proof of
  \cite[Proposition 3.2.3.1]{HA}). Then (iv) and (v) follow as in the
  proof of \cite[Proposition 3.2.3.1]{HA}.
\end{proof}

We now use this to show that the adjunction $\tau_{\mathcal{M},!}
\dashv \tau_{\mathcal{M}}^{*}$ is monadic; we first check that
$\tau_{\mathcal{M}}^{*}$ is conservative:
\begin{lemma}\label{lem:AlgCons}
  Suppose $\mathcal{V}$ is a monoidal \icat{} and
  $\mathcal{M}$ is a \gnsiopd{}. Then the forgetful functor
  \[\tau_{\mathcal{M}}^{*} \colon
  \Alg_{\mathcal{M}}(\mathcal{V}) \to
  \Alg_{\mathcal{M}_{\txt{triv}}}(\mathcal{V})
  \simeq \Fun(\mathcal{M}_{[1]}, \mathcal{V})\] is conservative.
\end{lemma}
\begin{proof}
  The \icat{} $\Alg_{\mathcal{M}}(\mathcal{V})$
  is a full subcategory of $\Fun_{\simp^{\op}}(\mathcal{M},
  \mathcal{V}^{\otimes})$. Therefore a map of algebras $f \colon A \to
  B$ is an equivalence in
  $\Alg_{\mathcal{M}}(\mathcal{V})$ \IFF{} it
  is an equivalence in $\Fun_{\simp^{\op}}(\mathcal{M},
  \mathcal{V}^{\otimes})$. Applying Theorem~\ref{thm:MysteryResult} to
  $\Delta^{0}$-indexed colimits, we see that a morphism $f \colon A
  \to B$ is an equivalence in $\Fun_{\simp^{\op}}(\mathcal{M}, \mathcal{V}^{\otimes})$ \IFF{} $f_{X} \colon A(X) \to B(X)$ is an
  equivalence in $\mathcal{V}^{\otimes}$ for all $X \in
  \mathcal{M}$. Thus equivalences are detected after
  restricting to $\mathcal{M}_{\txt{triv}}$.
\end{proof}

\begin{cor}\label{cor:AlgMonad}
  Suppose $\mathcal{V}$ is a monoidal \icat{} compatible with small
  colimits, and $\mathcal{M}$ is a \gnsiopd{} such that
  $\mathcal{M}_{[0]}$ is a Kan complex. Then the adjunction
  \[ (\tau_{\mathcal{M}})_{!} \colon
  \Alg_{\mathcal{M}_{\txt{triv}}}(\mathcal{V})
  \rightleftarrows
  \Alg_{\mathcal{M}}(\mathcal{V})
  \colon (\tau_{\mathcal{M}})^{*} \] is monadic.
\end{cor}
\begin{proof}
  We showed that the functor $\tau_{\mathcal{M}}^{*}$ is conservative
  in Lemma~\ref{lem:AlgCons}, and that it preserves sifted colimits in
  Corollary~\ref{cor:AlgSiftedColim}. The adjunction
  $(\tau_{\mathcal{M}})_{!} \dashv \tau_{\mathcal{M}}^{*}$ is
  therefore monadic by \cite[Corollary 5.5.2.9]{HTT}.
\end{proof}

\begin{cor}\label{cor:AlgOcolim}
  Suppose $\mathcal{V}$ is a monoidal \icat{} compatible
  with small colimits and $\mathcal{M}$ is a \gnsiopd{} such that
  $\mathcal{M}_{[0]}$ is a Kan complex. Then
  $\Alg_{\mathcal{M}}(\mathcal{V})$ has all small
  colimits. Moreover, if $\mathcal{V}$ is presentable, so is
  $\Alg_{\mathcal{M}}(\mathcal{V})$.
\end{cor}

This is an immediate consequence of the following general facts about
monadic adjunctions:
\begin{lemma}\label{lem:MonadColims}
  Suppose $F \colon \mathcal{C} \rightleftarrows \mathcal{D} :\! U$ is
  a monadic adjunction such that $\mathcal{C}$ has all small colimits,
  $\mathcal{D}$ has sifted colimits, and $U$ preserves sifted colimits. Then
  $\mathcal{D}$ has all small colimits.
\end{lemma}
\begin{proof}
  Since $\mathcal{D}$ by assumption has all sifted colimits, it
  suffices to prove that $\mathcal{D}$ has finite coproducts. Since
  $\mathcal{C}$ has coproducts and $F$ preserves colimits, the \icat{}
  $\mathcal{D}$ has coproducts for objects in the essential image of
  $F$.

  Let $A^{1}$, \ldots, $A^{n}$ be a finite collection of objects in
  $\mathcal{D}$. By \cite[Proposition 4.7.4.14]{HA}, there exist
  simplicial objects $A^{i}_{\bullet}$ in $\mathcal{D}$ such that each
  $A^{i}_{k}$ is in the essential image of $F$ and $|A^{i}_{\bullet}|
  \simeq A^{i}$. Since coproducts of elements in the essential image
  of $F$ exist, we can form a simplicial diagram
  $\coprod_{i}A^{i}_{\bullet}$. By \cite[Lemma 5.5.2.3]{HTT}, the
  geometric realization $|\coprod_{i}A^{i}_{\bullet}|$ is a coproduct
  of the $A^{i}$'s.
\end{proof}

\begin{propn}\label{propn:MonadPres}
  Suppose $F \colon \mathcal{C} \rightleftarrows \mathcal{D} :\! U$ is
  a monadic adjunction such that $\mathcal{C}$ is
  $\kappa$-presentable, $\mathcal{D}$ has small colimits,
  and the right adjoint $U$ preserves $\kappa$-filtered colimits. Then
  $\mathcal{D}$ is $\kappa$-presentable.
\end{propn}
\begin{proof}
  Since $\mathcal{C}$ is $\kappa$-presentable, every object of
  $\mathcal{C}$ is a colimit of $\kappa$-compact objects. Since $U$
  preserves $\kappa$-filtered colimits, $F$ preserves $\kappa$-compact
  objects by Lemma~\ref{lem:ladjcompact}. Therefore every object in
  the essential image of $F$ is a colimit of $\kappa$-compact
  objects. But by \cite[Proposition 4.7.4.14]{HA}, every object of
  $\mathcal{D}$ is a colimit of objects in the essential image of $F$,
  so every object of $\mathcal{D}$ is a colimit of $\kappa$-compact
  objects. Since by assumption $\mathcal{D}$ has all small colimits,
  this implies that $\mathcal{D}$ is $\kappa$-presentable.
\end{proof}
 
\begin{proof}[Proof of Corollary~\ref{cor:AlgOcolim}]
  Apply Lemma~\ref{lem:MonadColims} and
  Proposition~\ref{propn:MonadPres} to the monadic adjunction
  $\tau_{\mathcal{M},!} \dashv \tau_{\mathcal{M}}^{*}$.
\end{proof}

\begin{propn}\label{propn:StrMonColim1}
  Let $\mathcal{M}$ be a \gnsiopd{} such that $\mathcal{M}_{[0]}$ is a
  Kan complex, and let $\mathcal{V}$ and
  $\mathcal{W}$ be monoidal \icats{} compatible with small
  colimits. Suppose $F \colon \mathcal{V}^{\otimes} \to
  \mathcal{W}^{\otimes}$ is a  monoidal functor such that $F_{[1]}
  \colon \mathcal{V} \to \mathcal{W}$ preserves colimits. Then the
  induced functor \[F_{*} \colon
  \Alg_{\mathcal{M}}(\mathcal{V}) \to
  \Alg_{\mathcal{M}}(\mathcal{W})\] preserves colimits.
\end{propn}
\begin{proof}
  Write $F_{*}^{\txt{triv}}$ for the induced functor
  $\Alg_{\mathcal{M}_{\txt{triv}}}(\mathcal{V}) \to
  \Alg_{\mathcal{M}_{\txt{triv}}}(\mathcal{W})$. Under the
  equivalences $\Alg_{\mathcal{M}_{\txt{triv}}}(\mathcal{V})
  \simeq \Fun(\mathcal{M}_{[1]}, \mathcal{V})$ and
  $\Alg_{\mathcal{M}_{\txt{triv}}}(\mathcal{W}) \simeq
  \Fun(\mathcal{M}_{[1]}, \mathcal{W})$ this corresponds to
  composition with $F_{[1]}$, and so preserves colimits. Clearly
  $\tau^{*}_{\mathcal{M}}F_{*} \simeq
  F_{*}^{\txt{triv}}\tau^{*}_{\mathcal{M}}$. Since
  $\tau^{*}_{\mathcal{M}}$ detects sifted colimits, it follows that
  $F_{*}$ preserves sifted colimits. To prove that it preserves all
  colimits, it thus remains to prove it preserves finite coproducts.

  Since $F$ is a monoidal functor, by Lemma
  \ref{lem:StrMonPrFree} the functor $F_{*}$
  preserves free algebras, i.e. $F_{*} \tau_{\mathcal{M},!} \simeq
  \tau_{\mathcal{M},!}F_{*}^{\txt{triv}}$. Therefore
  $F_{*}$ preseves colimits of free algebras.  Let $A$ and $B$ be objects of
  $\Alg_{\mathcal{M}}(\mathcal{V})$ and let
  $A_{\bullet}$ and $B_{\bullet}$ be free resolutions of $A$ and
  $B$. Then we have natural equivalences
  \[
  \begin{split}
    F_{*}(A \amalg B) &   \simeq F_{*}(|A_{\bullet} \amalg B_{\bullet}|)
    \simeq |F_{*}(A_{\bullet} \amalg B_{\bullet})| \simeq
    |F_{*}(A_{\bullet}) \amalg F_{*}(B_{\bullet})| \\
    &  \simeq
    |F_{*}(A_{\bullet})| \amalg |F_{*}(B_{\bullet})| \simeq
    F_{*}(|A_{\bullet}|) \amalg F_{*}(|B_{\bullet}|) \simeq F_{*}(A)
    \amalg F_{*}(B),
  \end{split}
  \]
  so $F_{*}$ does indeed preserve coproducts.
\end{proof}

\begin{propn}\label{propn:rightadjlaxmon}
  Suppose $\mathcal{V}$ and $\mathcal{W}$ are
  presentably monoidal \icats{}
  and $F \colon \mathcal{V}^{\otimes} \to
  \mathcal{W}^{\otimes}$ is a  monoidal functor such that the
  underlying functor $F_{[1]} \colon \mathcal{V} \to \mathcal{W}$ preserves
  colimits. Let $g \colon \mathcal{W} \to \mathcal{V}$ be a right
  adjoint of $F_{[0]}$. Then there exists a lax monoidal functor
  $G \colon \mathcal{W}^{\otimes} \to \mathcal{V}^{\otimes}$
  extending $g$ such that for any small \nsiopd{} $\mathcal{O}$ we have an adjunction
  \[ F_{*} \colon
  \Alg_{\mathcal{O}}(\mathcal{V}) \rightleftarrows
  \Alg_{\mathcal{O}}(\mathcal{W}) \,\colon\! G_{*}.\]
\end{propn}

\begin{proof}
  By Proposition~\ref{propn:StrMonColim1} the functor $F_{*}
  \colon \Alg_{\mathcal{O}}(\mathcal{V}) \to
  \Alg_{\mathcal{O}}(\mathcal{W})$ is
  colimit-preserving, and by Corollary~\ref{cor:AlgOcolim} these
  \icats{} of $\mathcal{O}$-algebras are presentable. It
  follows by \cite[Corollary 5.5.2.9]{HTT} that $F_{*}$ has
  a right adjoint \[R_{\mathcal{O}} \colon
  \Alg_{\mathcal{O}}(\mathcal{W}) \to
  \Alg_{\mathcal{O}}(\mathcal{V}).\] Moreover,
  since $F_{*}$ is natural in $\mathcal{O}$ so is
  $R^{\mathcal{O}}$, by \cite[Corollary
  5.2.2.5]{HTT}. Taking the underlying spaces of the \icats{} of
  algebras, we see that $R_{(\blank)}$ induces a natural
  transformation $\rho \colon \Map(\blank, \mathcal{W}^{\otimes}) \to
  \Map(\blank, \mathcal{V}^{\otimes})$ of functors $(\OpdIns)^{\op} \to
  \mathcal{S}$. The full subcategory $\mathcal{W}^{\otimes}_{\kappa}$
  of $\mathcal{W}^{\otimes}$ spanned by objects coming from the full
  subcategory $\mathcal{W}^{\kappa} \subseteq \mathcal{W}$ spanned by
  $\kappa$-compact objects is a small \nsiopd{}. Applying
  $R_{\mathcal{W}^{\otimes}_{\kappa}}$ to the inclusion
  $\mathcal{W}_{\kappa}^{\otimes} \to \mathcal{W}^{\otimes}$ gives
  compatible maps $G^{\kappa} \colon
  \mathcal{W}^{\otimes}_{\kappa} \to \mathcal{V}^{\otimes}$. Combining
  these gives $G \colon \mathcal{W}^{\otimes} \to
  \mathcal{V}^{\otimes}$. Since every map $\mathcal{O} \to
  \mathcal{W}^{\otimes}$ where $\mathcal{O}$ is a small
  \nsiopd{} factors through $\mathcal{W}^{\otimes}_{\kappa}$ for some
  $\kappa$, we see that $\rho$ is given by composition with
  $G$. Moreover, the functor $R_{(\blank)}$ must also be
  given by composition with $G$, since
  $\Alg_{\mathcal{O}}(\mathcal{W})$ is the
  \icat{} associated to the simplicial space
  $\Map(\mathcal{O} \otimes \Delta^{\bullet},
  \mathcal{W}^{\otimes})$.

  It remains to show that $G$ is indeed a lax monoidal
  extension of $g$. This follows from taking $\mathcal{O}$ to be
  the trivial \nsiopd{} $\simp^{\op}_{\text{int}}$: then
  $\Alg_{\simp^{\op}_{\text{int}}}(\mathcal{V}) \simeq
  \mathcal{V}$ and
  $\Alg_{\simp^{\op}_{\text{int}}}(\mathcal{W}) \simeq
  \mathcal{W}$, and under these identifications $F_{*}$ corresponds to $F_{[1]}$ and
  $G_{*}$ to the functor $G_{[1]}$. Thus
  $g$ and $G_{[1]}$ are both right adjoint to $F$ and
  so must be equivalent.
\end{proof}

In the case of monoidal localizations we can explicitly identify this
lax monoidal structure on the right adjoint:
\begin{lemma}\label{lem:monlocadj}
  Suppose $\mathcal{O}$ is a small \nsiopd{},
  $\mathcal{V}$ is a monoidal \icat{} and $L \colon
  \mathcal{V} \to \mathcal{W}$ is a monoidal localization with fully
  faithful right adjoint $i \colon \mathcal{W} \hookrightarrow
  \mathcal{V}$. Then the monoidal functor $L^{\otimes}$ and the lax
  monoidal inclusion $i^{\otimes} \colon \mathcal{W}^{\otimes} \hookrightarrow
  \mathcal{V}^{\otimes}$ of Proposition~\ref{propn:moncomploc} induce
  an adjunction
  \[ L^{\otimes}_{*} : \Alg_{\mathcal{O}}(\mathcal{V}) \rightleftarrows
  \Alg_{\mathcal{O}}(\mathcal{W}) : i^{\otimes}_{*}. \]
  Moreover, $i^{\otimes}_{*}$ is fully faithful.
\end{lemma}
\begin{proof}
  Since $L^{\otimes}$ is left adjoint to $i^{\otimes}$, it is
  easy to see that we get an adjunction
  \[ L^{\otimes}_{*} : \Fun_{\simp^{\op}}(\mathcal{O},
  \mathcal{V}^{\otimes}) \rightleftarrows
  \Fun_{\simp^{\op}}(\mathcal{O}, \mathcal{W}^{\otimes}) :
  i^{\otimes}_{*}.\]
  But this clearly restricts to an adjunction between the full
  subcategories $\Alg_{\mathcal{O}}(\mathcal{V})$ and
  $\Alg_{\mathcal{O}}(\mathcal{W})$, as
  required. 

  To prove that $i^{\otimes}_{*}$ is fully faithful, it suffices to
  show that for every $\mathcal{O}$-algebra $A$ in
  $\mathcal{W}$ the counit $L^{\otimes}_{*}i^{\otimes}_{*}A \to A$ is
  an equivalence. By Lemma~\ref{lem:AlgCons} we need only show that
  the induced natural transformation of functors $\mathcal{O}_{[1]} \to \mathcal{W}$ is an
  equivalence, i.e. that for every $X \in \mathcal{O}_{[1]}$ the map $LiA(X)
  \to A(X)$ is an equivalence in $\mathcal{W}$, which is true since
  $i$ is fully faithful.
\end{proof}

\subsection{Approximations of $\infty$-Operads}
In this subsection we use Lurie's theory of \emph{approximations} to
give a criterion for a map $\mathcal{M} \to \mathcal{O}$ to 
exhibit a \nsiopd{} $\mathcal{O}$ as the operadic
localization $L_{\txt{gen}}\mathcal{M}$ of a \gnsiopd{} $\mathcal{M}$.

\begin{defn}\label{defn:approx}
  Suppose $\mathcal{M}$ is a \gnsiopd{}, $\mathcal{O}$ is a
  \nsiopd{},  and $f \colon \mathcal{M} \to \mathcal{O}$ is a
  fibration of \gnsiopds{}. Then $f$ is an \defterm{approximation} if for
  all $C \in \mathcal{M}$ and $\alpha \colon X \to f(C)$ active in
  $\mathcal{O}$ there exists an $f$-Cartesian morphism
  $\overline{\alpha} \colon \overline{X} \to C$ lifting $\alpha$, and a
  \defterm{weak approximation} if given $C \in \mathcal{M}$ and
  $\alpha \colon X \to f(C)$ an arbitrary morphism in
  $\mathcal{O}$, the full subcategory of \[\mathcal{M}_{/C}
  \times_{\mathcal{O}_{/f(C)}}
  \mathcal{O}_{X//f(C)}\] corresponding to pairs $(\beta
  \colon C' \to C, \gamma \colon X \to f(C'))$ with $\gamma$ inert is
  weakly contractible. More generally, a map $f \colon \mathcal{M} \to
  \mathcal{O}$ is a (\defterm{weak}) \defterm{approximation}
  if it factors as a composition \[\mathcal{M} \xto{f'} \mathcal{M}'
  \xto{f''} \mathcal{O}\] where $f'$ is an equivalence of
  \gnsiopds{} and $f''$ is a categorical fibration that is a (weak)
  approximation.
\end{defn}

\begin{propn}
  An approximation is a weak approximation.
\end{propn}
\begin{proof}
  As \cite[Lemma 2.3.3.10]{HA}.
\end{proof}

\begin{propn}\label{propn:wkapproxcrit}
  A fibration of \gnsiopds{} $f \colon \mathcal{M} \to
  \mathcal{O}$, where $\mathcal{O}$ is a
  \nsiopd{}, is a weak approximation \IFF{} for every object $C \in
  \mathcal{M}$ and every active morphism $\alpha \colon X \to f(C)$ in
  $\mathcal{O}$, the \icat{} $\mathcal{M}_{/C}
  \times_{\mathcal{O}_{/f(C)}} \{X\}$
  is weakly contractible.
\end{propn}
\begin{proof}
  As \cite[Proposition 2.3.3.11]{HA}.
\end{proof}

\begin{propn}
  Let $f \colon \mathcal{M} \to
  \mathcal{O}$ be a fibration of \gnsiopds{}, where
  $\mathcal{O}$ is a \nsiopd{}. If
  $\mathcal{O}_{[1]}$ is a Kan complex, then $f$ is a weak
  approximation \IFF{} $f$ is an approximation.
\end{propn}
\begin{proof}
  As \cite[Corollary 2.3.3.17]{HA}.
\end{proof}

\begin{thm}\label{thm:Approx}
  Suppose $f \colon \mathcal{M} \to \mathcal{O}$ is a weak
  approximation such that $f_{[1]} \colon \mathcal{M}_{[1]} \to
  \mathcal{O}$ is a categorical equivalence. Then for
  any \nsiopd{} $\mathcal{P}$, the induced map
  \[ f^{*} \colon \Alg_{\mathcal{O}}(\mathcal{P})
  \to \Alg_{\mathcal{M}}(\mathcal{P}) \]
  is an equivalence.
\end{thm}
\begin{proof}
  As \cite[Theorem 2.3.3.23]{HA}.
\end{proof}

\begin{cor}\label{cor:wapproxloceq}
  Suppose $f \colon \mathcal{M} \to \mathcal{O}$ is a weak
  approximation such that $f_{[1]}$ is a categorical equivalence. Then
  the induced map of \nsiopds{} $L_{\txt{gen}}\mathcal{M} \to
  \mathcal{O}$ is an equivalence.
\end{cor}

\begin{propn}\label{propn:ApproxFree}
  Suppose $f \colon \mathcal{O} \to \mathcal{P}$
  is a map of \nsiopds{}, and $\mathcal{P}_{[1]}$ is a Kan
  complex. The commutative diagram
  \csquare{\Alg_{\mathcal{P}}(\mathcal{S})}{\Alg_{\mathcal{O}}(\mathcal{S})}{\Fun(\mathcal{P},
    \mathcal{S})}{\Fun(\mathcal{O},
    \mathcal{S})}{f^{*}}{\tau_{\mathcal{P}}^{*}}{\tau_{\mathcal{O}}^{*}}{f_{[1]}^*}
  induces a natural transformation $\alpha \colon
  \tau_{\mathcal{O},!}  \circ f^{*}_{[1]} \to f^{*} \circ
  \tau_{\mathcal{P},!}$. If $\alpha$ induces an equivalence
  $\tau_{\mathcal{O},!}f_{[1]}^{*}A \isoto
  f^{*}\tau_{\mathcal{P},!}A$ where $A$ is the constant
  functor $\mathcal{P} \to \mathcal{S}$ with value
  $*$, then $f$ is an approximation.
\end{propn}
\begin{proof}
  As \cite[Proposition 2.3.4.8]{HA}.
\end{proof}

\begin{cor}\label{cor:OpdLocFree}
  Let $\mathcal{O}$ be a \nsiopd{} such that
  $\mathcal{O}_{[1]}$ is a Kan complex, and let $f \colon
  \mathcal{M} \to \mathcal{O}$ be a map of
  \gnsiopds{} such that $f_{[1]} \colon \mathcal{M}_{[1]} \to
  \mathcal{O}_{[1]}$ is an equivalence. Write $A$ for the
  constant functor $\mathcal{M}_{[1]} \simeq
  \mathcal{O}_{[1]} \to \mathcal{S}$ with
  value $*$. If the natural map $\tau_{\mathcal{M},!}A \to
  f^{*}\tau_{\mathcal{O},!}A$ is an equivalence, then $f$
  exhibits $\mathcal{O}$ as the operadic localization of $\mathcal{M}$.
\end{cor}
\begin{proof}
  Applying Proposition~\ref{propn:ApproxFree} to the induced map $f'
  \colon L_{\txt{gen}}\mathcal{M} \to \mathcal{O}$, we see
  that this map is an approximation and induces an equivalence
  $L\mathcal{M}_{[1]} \to \mathcal{O}$. By Theorem~\ref{thm:Approx},
  it follows that $f'$ is an equivalence.
\end{proof}

\begin{cor}\label{cor:OpdLocCof}
  Let $\mathcal{O}$ be a \nsiopd{} such that $\mathcal{O}_{[1]}$
  is a Kan complex, and $f \colon \mathcal{M} \to
  \mathcal{O}$ be a map of \gnsiopds{} such that $f_{[1]}
  \colon \mathcal{M}_{[1]} \to \mathcal{O}_{[1]}$ is an equivalence and
  $\mathcal{M}_{[0]}$ is a Kan complex. If the induced map
  $(\mathcal{M}_{\txt{act}})_{/x} \to
  (\mathcal{O}_{\txt{act}})_{/x}$ is cofinal for all $x \in
  \mathcal{M}_{[1]} \simeq \mathcal{O}_{[1]}$, then $f$ exhibits
  $\mathcal{O}$ as the operadic localization of
  $\mathcal{M}$.
\end{cor}
\begin{proof}
  By Corollary~\ref{cor:OpdLocFree} it suffices to show that the
  natural map of $\mathcal{M}$-algebras $\tau_{\mathcal{M},!}A \to
  f^{*}\tau_{\mathcal{O},!}A$ is an equivalence. Since
  $\tau_{\mathcal{M}}^{*}$ detects equivalences by
  Lemma~\ref{lem:AlgCons}, to see this it
  suffices to show that for all $x \in \mathcal{M}_{[1]}$ the map of spaces
  $(\tau_{\mathcal{M},!}A)(x) \to
  (\tau_{\mathcal{O},!}A)(x)$ is an equivalence. Since
  $\mathcal{M}_{[0]}$ is a Kan complex, we can describe
  $\tau_{\mathcal{M},!}A$ using the results of \S\ref{subsec:OpdKanExt}.
  We thus see that this map can be identified with the map
  \[ \colim_{(\mathcal{M}_{\txt{act}})_{/x}} * \to
  \colim_{(\mathcal{O}_{\txt{act}})_{/x}} * \]
  of colimits induced by $(\mathcal{M}_{\txt{act}})_{/x} \to
  (\mathcal{O}_{\txt{act}})_{/x}$. If this map is cofinal,
  then the induced map on colimits is an equivalence.
\end{proof}
\begin{remark}
  The same argument shows that for any presentably monoidal \icat{}
  $\mathcal{V}$ the natural map $\tau_{\mathcal{M},!}F \to
  f^{*}\tau_{\mathcal{O},!}F$ is an equivalence for any
  functor $F \colon \mathcal{M}_{[1]} \to \mathcal{V}$. It follows
  that $\tau_{\mathcal{M},!}$ and $\tau_{\mathcal{O},!}$ are
  given by the same monad on $\Fun(\mathcal{M}_{[1]}, \mathcal{V})$,
  hence the \icats{} of algebras
  $\Alg_{\mathcal{M}}(\mathcal{V})$ and
  $\Alg_{\mathcal{O}}(\mathcal{V})$ must be
  equivalent, since they are both \icats{} of algebras for this
  monad. An alternative proof of Corollary~\ref{cor:OpdLocCof} (not
  using the notion of approximation) results by embedding any small
  \nsiopd{} $\mathcal{P}$ in a presentably monoidal \icat{}
  $\widehat{\mathcal{P}}$ and showing that
  $\Alg_{\mathcal{M}}(\mathcal{P})$ and
  $\Alg_{\mathcal{O}}(\mathcal{P})$ are the same
  subcategory of $\Alg_{\mathcal{M}}(\widehat{\mathcal{P}})
  \simeq
  \Alg_{\mathcal{O}}(\widehat{\mathcal{P}})$.
\end{remark}

\setcounter{section}{4}

\section{Erratum}
Our proof of Corollary~\ref{cor:OpdLocFree} does not make sense: it
assumes that we know we have an equivalence
$\mathcal{M}_{[1]} \simeq (L_{\text{gen}}\mathcal{M})_{[1]}$, which we
have not established. We thank Hongyi Chu for pointing this out to us.

Moreover, the statement of Corollary~\ref{cor:OpdLocCof} is also
incorrect: The formula for $\tau_{\mathcal{M},!}$ does not involve a
colimit over $(\mathcal{M}_{\text{act}})_{/x}$, but over the fibre
product
$\mathcal{M}_{\text{triv,act}/x} := \mathcal{M}_{\text{triv,act}}
\times_{\mathcal{M}_{\text{act}}}(\mathcal{M}_{\text{act}})_{/x}$. This
\icat{} has as its objects active morphisms $\alpha \colon M \to x$ in
$\mathcal{M}$ and as morphisms commutative triangles
\opctriangle{M}{M'}{x}{i}{\alpha}{\beta} where $\alpha$ and $\beta$
are active and $i$ lies over an inert morphism in $\simp^{\op}$. This
forces $i$ to lie over an identity morphism in $\simp^{\op}$, and thus
if $\mathcal{M}_{[1]}$ is a space, the morphism $i$ is forced to be an
equivalence. Hence in this case the \icat{}
$\mathcal{M}_{\text{triv,act}/x}$ is an $\infty$-groupoid. The
correct statement of Corollary~\ref{cor:OpdLocCof} is then:
\begin{cor}\label{cor:OpdLocCofCorrect}
  Let $\mathcal{O}$ be a \nsiopd{} such that $\mathcal{O}_{[1]}$
  is a space, and let $f \colon \mathcal{M} \to
  \mathcal{O}$ be a map of \gnsiopds{} such that $f_{[1]}
  \colon \mathcal{M}_{[1]} \to \mathcal{O}_{[1]}$ is an equivalence and
  $\mathcal{M}_{[0]}$ is a space. If the induced map of spaces
  $\mathcal{M}_{\text{\textup{triv,act}}/x} \to \mathcal{O}_{\text{\textup{triv,act}}/f(x)}$
  is an equivalence for all $x \in
  \mathcal{M}_{[1]} \simeq \mathcal{O}_{[1]}$, then $f$ exhibits
  $\mathcal{O}$ as the operadic localization of
  $\mathcal{M}$.
\end{cor}
We can prove this without invoking Proposition~\ref{propn:ApproxFree}
or Corollary~\ref{cor:OpdLocFree}, by instead using the following
non-symmetric analogue of (a special case of) \cite{HA}*{Proposition
  2.3.3.11}, whose proof also goes through in the non-symmetric setting.
\begin{propn}\label{propn:nicewapproxcrit}
  Let $\mathcal{O}$ be a \nsiopd{}, $\mathcal{M}$ a \gnsiopd{}, and $f
  \colon \mathcal{M} \to \mathcal{O}$ a morphism of \gnsiopds{}. Then
  $f$ is a weak approximation \IFF{} for every object $C \in
  \mathcal{M}$ and every active morphism $\alpha \colon X \to f(C)$ in
  $\mathcal{O}$, the fibre product $\mathcal{M}_{/C}
  \times_{\mathcal{O}_{/f(C)}} \{\alpha\}$ is weakly contractible.\qed
\end{propn}

\begin{proof}[Proof of Corollary~\ref{cor:OpdLocCofCorrect}]
  We check that the criterion of
  Proposition~\ref{propn:nicewapproxcrit} holds. Since
  $\mathcal{M}_{[1]}$ and $\mathcal{C}_{[1]}$ are by assumption
  spaces, the \icat{}
  $\mathcal{M}_{/C} \times_{\mathcal{O}_{/f(C)}} \{\alpha\}$ is a
  space, and can be identified with the fibre at $\alpha$ of the
  morphism of spaces
  $\mathcal{M}_{\text{triv,act}/C} \to
  (\mathcal{O}_{\txt{act}})_{/f(C)}$. The criterion is thus equivalent
  to this morphism being an equivalence for all $C \in
  \mathcal{M}$. If $C \to C_{i}$ are inert morphisms over $\rho_{i}$,
  then as $\mathcal{M}_{[0]}$ is a space we have an equivalence
  $\mathcal{M}_{\text{triv,act}/C} \simeq \prod_{i}
  \mathcal{M}_{\text{triv,act}/C_{i}}$, and similarly for
  $\mathcal{O}$. We are hence reduced to showing the morphism is an
  equivalence for $C \in \mathcal{M}_{[1]}$, which is true by
  assumption. The map $f$ is thus a weak approximation, and now
  applying Corollary~\ref{cor:wapproxloceq} completes the proof.
\end{proof}

Corollary~\ref{cor:OpdLocCof} was used twice in the paper, in the
proofs of Theorem~\ref{thm:opdlocDopX} and of
Proposition~\ref{propn:BMloceq}. The proof of
Theorem~\ref{thm:opdlocDopX} also implies the criterion of
Corollary~\ref{cor:OpdLocCofCorrect} (it establishes an equivalence of
\icats{} of which we require the equivalence of underlying spaces),
though it can probably be simplified. The proof of
Proposition~\ref{propn:BMloceq} is also unchanged.

\end{document}